\def\draft{n}
\documentclass[12pt]{amsart}
\usepackage[headings]{fullpage}
\usepackage{amssymb,epic,eepic,epsfig,amsbsy,amsmath,amscd,color,hyperref,bm}
\usepackage{bbm} 
\usepackage{enumitem,fvextra} 
\usepackage{ulem}

\usepackage{tikz,ifthen}
\usetikzlibrary{positioning, math, decorations.markings, arrows.meta}
\usetikzlibrary{calc,shapes,cd}
\usetikzlibrary{knots} 
\usepgfmodule{decorations}

\tikzset{edge/.style={line width=0.8}}
\tikzset{wall/.style={very thick}} 

\tikzset{->-/.style n args={2}{decoration={markings, mark=at position #1 with {\arrow{#2}}}, postaction={decorate}}} 

\ExplSyntaxOn
\cs_new_eq:NN \ifstreqF \str_if_eq:nnF
\cs_new_eq:NN \ifstreqTF \str_if_eq:nnTF
\ExplSyntaxOff

\tikzset{-o-/.code 2 args={\ifstreqF{#2}{} 
{\ifstreqTF{#2}{>}
   {\pgfkeysalso{decoration={markings,mark=at position #1 with {\arrow[scale=0.8]{#2}}}
                    ,postaction={decorate}}
    }
   {\ifstreqTF{#2}{<}
       {\pgfkeysalso{decoration={markings,mark=at position #1 with {\arrow[scale=0.8]{#2}}}
                    ,postaction={decorate}}
        }
       {\pgfkeysalso{decoration={markings,
                    mark=at position #1 with 
                    {\draw[black, fill={#2}] circle[radius=2pt];}}
                    ,postaction={decorate}}
        }
     }
  }}}

\newcommand{\diskpic}[1]{ 
\raisebox{-0.4\height}{\begin{tikzpicture}
    \begin{scope}[scale=1.2]
    \tikzmath{\rad=0.4;\d=0.3;}
    \fill[gray!20!white] (0,0)circle(\rad);
    #1
    \end{scope}
\end{tikzpicture}
}} 

\newcommand{\squarepic}[1]{ 
\raisebox{-0.4\height}{\begin{tikzpicture}[scale=1.2]
    \tikzmath{\rad=0.4;\d=0.3;}
    \fill[gray!20!white] (-\rad,-\rad) rectangle (\rad,\rad);
    #1
\end{tikzpicture}
}} 

\newcommand{\crossround}[2]{
\diskpic{
    \ifthenelse{\equal{#2}{p}}
     {\draw[edge, #1] (-\d,\d) -- (\d,-\d);
    \ifthenelse{\equal{#1}{<-}}
      {\draw[edge] (\d,\d) -- (0.05,0.05);
       \draw[edge, <-] (-\d,-\d) -- (-0.05,-0.05);}
      {\draw[edge, #1] (0.05,0.05)-- (\d,\d);
       \draw[edge] (-\d,-\d) -- (-0.05,-0.05);}}
    {
      \draw[edge, #1] (-\d,-\d) -- (\d,\d);
      \ifthenelse{\equal{#1}{<-}}
     {\draw[edge, <-] (-\d,\d) -- (-0.05,0.05);
    \draw[edge] (\d,-\d) -- (0.05,-0.05);}
    {\draw[edge] (-\d,\d) -- (-0.05,0.05);
    \draw[edge,#1] (0.05,-0.05)-- (\d,-\d);}
}    }}

\newcommand{\smoothA}[2]{
\diskpic{
\clip (0,0)circle(\rad);
\ifthenelse{\equal{#1}{<}}{
\draw[edge, -o-={.2}{#1}] (-\d,\d)..controls (0,0)..(\d,\d);
}{
\draw[edge, -o-={.9}{#1}] (-\d,\d)..controls (0,0)..(\d,\d);}
\ifthenelse{\equal{#2}{<}}{
\draw[edge,-o-={.2}{#2}] (-\d,-\d)..controls (0,0)..(\d,-\d);
}{
\draw[edge,-o-={.9}{#2}] (-\d,-\d)..controls (0,0)..(\d,-\d);
}}}

\newcommand{\smoothB}[2]{
\diskpic{
\clip (0,0)circle(\rad);
\draw[edge, -o-={.9}{#1}] (-\d,-\d)..controls (0,0)..(-\d,\d);
\draw[edge,-o-={.2}{#2}] (\d,-\d)..controls (0,0)..(\d,\d);
}}

\newcommand{\emptyd}[1][]{
\raisebox{-0.4\height}{\begin{tikzpicture}[scale=0.5]
    \tikzmath{\yo=-0.3; \xw=2; \yw=1.1;}
    \fill[gray!20!white] (0,\yo) rectangle (\xw,\yw);
    \ifthenelse{\equal{#1}{w}}{\draw[wall] (0,\yo) -- (0,\yw); \draw[wall] (\xw,\yo) -- (\xw,\yw);}{}
\end{tikzpicture}
}}

\newcommand{\circlediag}[2][]{
\raisebox{-0.4\height}{\begin{tikzpicture}[scale=0.5]
    \tikzmath{\yo=-0.3; \xw=2; \yw=1.1;}
    \fill[gray!20!white] (0,\yo) rectangle (\xw,\yw);
    \ifthenelse{\equal{#1}{w}}{\draw[wall] (0,\yo) -- (0,\yw); \draw[wall] (\xw,\yo) -- (\xw,\yw);}{}
    \draw[edge, -o-={0.1}{#2}] (\xw/2,{(\yo + \yw)/2}) circle (0.4);
\end{tikzpicture}
}}

\tikzset{knotarrow/.pic={ \draw[edge, <-] (0,0) -- +(-.001,0);}}

\newcommand{\kink}[1][]{\raisebox{-0.5\height}{
\begin{tikzpicture}[pics/arrow/.style={code={%
  \draw[line width=0pt,-{Computer Modern Rightarrow[line
  width=0.8pt,width=1ex,length=.6ex]}] (-0.5ex,0) -- (0.5ex,0);}}, scale=0.5]
   \fill[gray!20!white] (0,-0.6) rectangle (2,1); 
   \ifthenelse{\equal{#1}{w}}{\draw[wall] (0,-0.6) -- (0,1); \draw[wall] (2,-0.6) -- (2,1);}{}
    \begin{knot}[clip width=5, clip radius=5pt, consider self intersections,
    end tolerance=3pt, background color=gray!20!white, fill opacity=0.5]
       \strand[edge] (2,0)
          to[out=left, in=right] pic[pos=0.5,sloped]{arrow} (1.4,0) 
          to[out=left, in=down] (0.6,0.4)
          to[out=up, in=left] (1,0.8)
          to[out=right, in=up]  (1.4,0.4)
          to[out=down, in=right]  (0.6,0)
          to[out=left, in=right]  (0,0);
    \end{knot}
  \end{tikzpicture}}}

\newcommand{\twokinks}[1]{\squarepic{
\tikzmath{\h=-0.2;}
\draw[edge, -o-={0.7}{#1}] (-\rad,\h) --(-\rad/4,\h);
\draw[edge, -o-={0.7}{#1}] (\rad/4,\h) --(\rad,\h);
\draw[edge,-o-={0.5}{<}] (-\rad/4,\h) ..controls (-\rad,\rad) and (\rad,\rad).. (\rad/4,\h);
}}

\newcommand{\horizontaledge}[1]{\squarepic{
\draw[edge, -o-={0.8}{#1}] (-\rad,0) --(\rad,0);
}}

\newcommand{\Iweb}[2]{
\raisebox{-.2in}{
\begin{tikzpicture}[scale=1.4]
    \tikzmath{\d=0.3;\h=0.2;}
    \fill[gray!20!white] (0,0)circle(0.4);
    \draw[edge, decoration={markings, mark=at position 0.7 with {\arrow{#1}}},postaction={decorate}]  (-\d,\d)-- (-\h,0);
    \draw[edge, decoration={markings, mark=at position 0.7 with {\arrow{#1}}},postaction={decorate}]  (-\d,-\d)-- (-\h,0);
    \draw[edge, decoration={markings, mark=at position 0.8 with {\arrow{#1}}},postaction={decorate}] (\h,0) -- (\d,-\d);
    \draw[edge, decoration={markings, mark=at position 0.8 with {\arrow{#1}}},postaction={decorate}] (\h,0) -- (\d,\d);
    \draw[edge, decoration={markings, mark=at position 0.5 with {\arrow{#1}}},postaction={decorate}] (\h,0) -- (-\h,0) node[midway, above] {\small{#2}};
\end{tikzpicture}
}   }

\newcommand{\sinksourcethree}[2][]{
\raisebox{-.35in}{\begin{tikzpicture}
        \tikzmath{\x1=0; \s=0.35; \b=-.004;
         \y1=\b+\s-0.08; \t=\b+4*\s; \y2=\y1+\s; \y3=\y2+\s; \y4=\y1+3*\s; \x2=1.4; }
         \fill[gray!20!white] (0,0)rectangle(\x2,\y4+.2);
        \ifthenelse{\equal{#1}{w}}{\draw[wall] (0,0) -- (0,{\y4+.2});\draw[wall] (\x2,0) -- (\x2,{\y4+.2});}{}
         \coordinate (V0) at (0.6,\t/2);
          \coordinate (V1) at (0.8,\t/2);
         \draw[edge, -o-={.35}{#2}]  (\x1,\y4) to [out=0, in=120] (V0); 
         \draw[edge, -o-={.4}{#2}] (\x1,\y1) to [out=0, in=-120] (V0); 
         \draw[edge, -o-={.65}{#2}]  (V1) to [out=60, in=180] (\x2,\y4); 
          \draw[edge, -o-={.65}{#2}] (V1) to [out=-60, in=180] (\x2,\y1);
        \draw[edge, -o-={.4}{#2}] (\x1,\y2) to [out=0, in=-150] (V0);
         \draw[edge, -o-={.58}{#2}] (V1) to [out=-30, in=180] (\x2,\y2);
         \node at ($(0.2,\y2)!.5!(0.2,\y4)$) {\vdots};
         \node at ($(\x2-0.2,\y2)!.5!(\x2-0.2,\y4)$) {\vdots};
         \end{tikzpicture}
}
}

\newcommand{\sinksourcethreewithtb}[3][]{
\raisebox{-.35in}{\begin{tikzpicture}[scale=0.4]
        \tikzmath{\x1=0; \s=0.8; \b=-.1;
         \y1=\b+\s-0.2; \t=\b+4*\s; \y2=\y1+\s; \y3=\y2+\s; \y4=\y1+3*\s; \x2=3.5; }
         \fill[gray!20!white] (0,0)rectangle(\x2,\y4+.5);
        \ifthenelse{\equal{#1}{w}}{\draw[wall] (0,0) -- (0,{\y4+.5});\draw[wall] (\x2,0) -- (\x2,{\y4+.5});}{}
         \coordinate (V0) at (1.5,\t/2);
          \coordinate (V1) at (2,\t/2);
         \draw[edge, -o-={.35}{#3}]  (\x1,\y4) to [out=0, in=120] (V0); 
         \draw[edge, -o-={.4}{#3}] (\x1,\y1) to [out=0, in=-120] (V0); 
         \draw[edge, -o-={.65}{#3}]  (V1) to [out=60, in=180] (\x2,\y4); 
          \draw[edge, -o-={.65}{#3}] (V1) to [out=-60, in=180] (\x2,\y1);
     
         \ifthenelse{\equal{#2}{b}}
        {\draw[edge, -o-={.4}{#3}] (\x1,\y2) to [out=0, in=-150] (V0);
         \draw[edge, -o-={.58}{#3}] (V1) to [out=-30, in=180] (\x2,\y2);
         \node at ($(1-0.5,\y2)!.5!(1-0.5,\y4)$) {\vdots};
         \node at ($(\x2-0.5,\y2)!.5!(\x2-0.5,\y4)$) {\vdots};}
         { \draw[edge, -o-={.4}{#3}] (\x1,\y3) to [out=0, in=150] (V0);
         \draw[edge, -o-={.58}{#3}] (V1) to [out=30, in=180] (\x2,\y3);
         \node at ($(1 -0.5,\y1+0.1)!.5!(1-0.5,\y3+0.1)$) {\vdots};
         \node at ($(\x2 -0.5,\y1+0.1)!.5!(\x2-0.5,\y3+0.1)$) {\vdots};}
         \end{tikzpicture}
}
}

\newcommand{\coupon}[6][]{
\raisebox{-.35in}{\begin{tikzpicture}
        \tikzmath{\x1=0; \s=0.35; \b=-.004;
         \y1=\b+\s-0.15; \t=\b+4*\s; \y2=\y1+\s; \y4=\y1+3*\s+0.05; \x2=1.4;} 
         \fill[gray!20!white] (0,0)rectangle(\x2,\y4+.2);
          \node at (\x2/2,\y4/2+.1) [ellipse, minimum height =2.5em, draw] (V) {\small{#2}\hspace*{-.05in}}; 
        \ifthenelse{\equal{#1}{w}}{\draw[wall] (0,0) -- (0,{\y4+.2});\draw[wall] (\x2,0) -- (\x2,{\y4+.2});}{}
         \coordinate (V0) at (0.6,\t/2);
          \coordinate (V1) at (0.8,\t/2);
         \draw[edge, -o-={.35}{#3}]  (\x1,\y4) node[left]{{\small #4}} to [out=0, in=120] (V); 
         \draw[edge, -o-={.35}{#3}] (\x1,\y1) node[left]{{\small #6}} to [out=0, in=-120] (V); 
         \draw[edge, -o-={.65}{#3}]  (V) to [out=60, in=180] (\x2,\y4); 
          \draw[edge, -o-={.65}{#3}] (V) to [out=-60, in=180] (\x2,\y1);
        \draw[edge, -o-={.5}{#3}] (\x1,\y2) node[left]{{\small #5}} to [out=0, in=-155] (V);
         \draw[edge, -o-={.58}{#3}] (V) to [out=-25, in=180] (\x2,\y2);
         \node at ($(0.1,\y2+.1)!.5!(0.1,\y4+.1)$) {\vdots};
         \node at ($(\x2-0.1,\y2+.1)!.5!(\x2-0.1,\y4+.1)$) {\vdots};
         \end{tikzpicture}
}
}

\newcommand{\edgewall}[4][]{
\raisebox{-0.4\height}{
\begin{tikzpicture}
\tikzmath{\xw=0.8;\yw=0.8;}
\fill[gray!20!white] (0,0)rectangle(\xw,\yw);
\draw[wall, #2] (\xw,0)--(\xw,\yw);
\ifthenelse{\equal{#1}{}}{}{\draw[wall, #1] (0,0) -- (0,\yw);}
\draw[edge, -o-={.5}{#3}] (0,\yw/2) -- (\xw,\yw/2)node[right]{{\small #4}};
\end{tikzpicture}
}}

\newcommand{\walledgewall}[5]{
\raisebox{-0.4\height}{
\begin{tikzpicture}
\tikzmath{\xw=1.2;\yw=0.8;}
\fill[gray!20!white] (0,0)rectangle(\xw,\yw);
\draw[wall, #2] (\xw,0)--(\xw,\yw);
\draw[wall, #1] (0,0) -- (0,\yw);
\draw[edge, -o-={.5}{#3}] (0,\yw/2) node[left]{{\small #5}} -- (\xw,\yw/2)node[right]{{\small #4}};
\end{tikzpicture}
}}

\newcommand{\wallcuspswall}[6]{
\raisebox{-0.4\height}{
\begin{tikzpicture}
\tikzmath{\xw=1.5;\yw=0.8;\yww=0.7;\dx=.05;}
\fill[gray!20!white] (0,0)rectangle(\xw,\yw);
\draw[wall, #2] (\xw,0)--(\xw,\yw);
\draw[wall, #1] (0,0) -- (0,\yw);
\draw[edge, -o-={.5}{#3}] (0,\yw/2) node[left]{{\small #6}} .. controls +(.2,0) and +(-\dx,-.3) ..  (\xw/3, \yww);
\draw[edge, -o-={.5}{#4}] (\xw/3, \yww) .. controls +(\dx,-.3) and +(-\dx,-.3) ..  (\xw*2/3, \yww);
\draw[edge, -o-={.5}{#3}] (\xw*2/3, \yww) .. controls +(\dx,-.3) and +(-.2,0) .. 
(\xw,\yw/2)node[right]{{\small #5}};
\end{tikzpicture}
}}

\newcommand{\twowallpic}[6][0]{ 
\raisebox{-0.4\height}{\begin{tikzpicture}
        \tikzmath{\xwidth=0.7;}
        \fill[gray!20!white] (#1,0)rectangle(\xwidth,1);
        \ifthenelse{\equal{#3}{left}}{\tikzmath{\xw=0;}}{\tikzmath{\xw=\xwidth;}}
        \ifthenelse{\equal{#2}{<-}}{\tikzmath{\yw=0.2;}}{
        \ifthenelse{\equal{#2}{->}}{\tikzmath{\yw=0;}}{\tikzmath{\yw=0.1;}}}
        \ifthenelse{\equal{#2}{n}}{}{\draw[wall, #2] (\xw,0) --(\xw,1);}
        \tikzmath{\ya=\yw+0.6; \yb=\yw+0.2;} 
        \coordinate[label={#3:{\small #4}}] (A) at (\xw,\ya); 
        \coordinate[label={#3:{\small #5}}] (B) at (\xw,\yb);
        #6
\end{tikzpicture}
}\hspace*{-.1in}} 

\newcommand{\walltwo}[5]{
\twowallpic{#1}{left}{#4}{#5}{
\draw[edge, -o-={0.5}{#2}] (A) --(\xwidth,\ya);
\draw[edge, -o-={0.5}{#3}] (B) --(\xwidth,\yb);}}

\newcommand{\wallcross}[6]{
\twowallpic{#1}{left}{#5}{#6}{
\tikzmath{\xcent=\xwidth*4/7;\ycent=\ya/2+\yb/2;}
\ifthenelse{\equal{#2}{p}}
    {\draw[edge,  -o-={0.25}{#3}] (0,\ya) to [out=0, in=180]  (\xwidth,\yb);
    \draw[edge,  -o-={0.5}{#4}] (0,\yb) to [out=0, in=-135] (\xcent-0.08,\ycent-0.08);
    \draw[edge] (\xcent+0.04,\ycent+0.04) to [out=45, in=180] (\xwidth,\ya);
    }   
    {\draw[edge, -o-={0.2}{#3}] (0,\yb) to [out=0, in=180]  (\xwidth,\ya);
    \draw[edge, -o-={0.5}{#4}] (0,\ya) to [out=0, in=135] (\xcent-0.08,\ycent+0.08);
    \draw[edge] (\xcent+0.045,\ycent-0.045) to [out=-45, in=180] (\xwidth,\yb);}
    }}
   
\newcommand{\twowall}[5]{
\twowallpic{#1}{right}{#4}{#5}{
\draw[edge, -o-={0.5}{#2}] (0,\ya)-- (A);
\draw[edge, -o-={0.5}{#3}] (0,\yb)-- (B);
}}

\newcommand{\Hwall}[6]{
\twowallpic{#1}{right}{#5}{#6}{
\draw[edge, -o-={0.25}{#2}, -o-={0.75}{#3}] (0,\ya)-- (A);
\draw[edge, -o-={0.25}{#3}, -o-={0.75}{#2}] (0,\yb)-- (B);
\draw[edge, -o-={0.7}{#4}] (\xw/2, \yb)--(\xw/2,\ya);
}}

\newcommand{\twoonetwowall}[7][]{
\raisebox{-0.4\height}{\begin{tikzpicture}
        \ifthenelse{\equal{#2}{<-}}{\tikzmath{\y0=0.2;\y1=.1;}}{
        \ifthenelse{\equal{#2}{->}}{\tikzmath{\y0=0; \y1=.2;}}{\tikzmath{\y0=0.1; \y1=0.1;}}}
        \tikzmath{\xw=1.5;\yb=\y0+0.1; \ya=\yb+0.4; \ycent=\ya/2+\yb/2; \yh=\ya+ .1+\y1;}
        \fill[gray!20!white] (0,0)rectangle(\xw,\yh);
        \ifthenelse{\equal{#2}{n}}{}{\draw[wall, #2] (\xw,0) --(\xw,\yh);} 
        \ifthenelse{\equal{#1}{w}}{\draw[wall] (0,0) --(0,\yh);}{}
        \draw[edge, -o-={0.8}{#3}] (1.05,\ycent) to[out=30, in=180] (\xw,\ya) node[right] {\small #6};
        \draw[edge, -o-={0.8}{#4}] (1.05,\ycent) to[out=-30, in=180] (\xw,\yb) node[right] {\small #7};
        \draw[edge] (0.45,\ycent) -- (1.05,\ycent);
        \coordinate[label={above:{\small #5}}] (B) at (0.85,\yb+0.2);
        \draw[edge, -o-={0.4}{#3}] (0,\ya) to[out=0, in=150] (0.45,\ycent);
        \draw[edge, -o-={0.4}{#4}] (0,\yb) to[out=0, in=-150]  (0.45,\ycent);
\end{tikzpicture}
}\hspace*{-.1in}}

\newcommand{\forkwall}[7][]{
\raisebox{-0.4\height}{\begin{tikzpicture}
        \ifthenelse{\equal{#2}{<-}}{\tikzmath{\y0=0.2;\y1=.1;}}{
        \ifthenelse{\equal{#2}{->}}{\tikzmath{\y0=0; \y1=.2;}}{\tikzmath{\y0=0.1; \y1=0.1;}}}
        \tikzmath{\xw=1.5;\yb=\y0+0.1; \ya=\yb+0.4; \ycent=\ya/2+\yb/2; \yh=\ya+ .1+\y1;}
        \fill[gray!20!white] (0.55,0)rectangle(\xw,\yh);
        \draw[wall, #2] (\xw,0) --(\xw,\yh); 
        \ifthenelse{\equal{#1}{w}}{\draw[wall] (0,0) --(0,\yh);}{}
        \draw[edge, -o-={0.5}{#3}] (1.05,\ycent) to[out=30, in=180] (\xw,\ya) node[right] {\small #6};
        \draw[edge, -o-={0.5}{#4}] (1.05,\ycent) to[out=-30, in=180] (\xw,\yb) node[right] {\small #7};
        \draw[edge, -o-={0.5}{#5}] (0.55,\ycent) -- (1.05,\ycent);
\end{tikzpicture}
}\hspace*{-.1in}} 

\newcommand{\crosswall}[6]{ 
\twowallpic{#1}{right}{#5}{#6}{
 \tikzmath{\xcent=\xwidth*4/7;\ycent=\ya/2+\yb/2;}
  \ifthenelse{\equal{#2}{p}}
    {\ifthenelse{\equal{#3}{>}}{\draw[edge,  -o-={0.4}{>}] (0,\ya) to [out=0, in=180]  (\xwidth,\yb);}
    {\draw[edge,  -o-={0.2}{#3}] (0,\ya) to [out=0, in=180]  (\xwidth,\yb);}
    \ifthenelse{\equal{#4}{>}}{\draw[edge,  -o-={0.9}{>}] (0,\yb) to [out=0, in=-135] (\xcent-0.08,\ycent-0.08);}
    {\draw[edge,  -o-={0.4}{#4}] (0,\yb) to [out=0, in=-135] (\xcent-0.08,\ycent-0.08);}
    \draw[edge] (\xcent+0.04,\ycent+0.04) to [out=45, in=180] (\xwidth,\ya);}
    {\ifthenelse{\equal{#4}{>}}{\draw[edge, -o-={0.4}{>}] (0,\yb) to [out=0, in=180]  (\xwidth,\ya);}
    {\draw[edge, -o-={0.2}{#4}] (0,\yb) to [out=0, in=180]  (\xwidth,\ya);}
    \ifthenelse{\equal{#3}{>}}{\draw[edge, -o-={0.9}{>}] (0,\ya) to [out=0, in=135] (\xcent-0.08,\ycent+0.08);}
    {\draw[edge, -o-={0.5}{#3}] (0,\ya) to [out=0, in=135] (\xcent-0.08,\ycent+0.08);}
    \draw[edge] (\xcent+0.045,\ycent -0.045) to [out=-45, in=180] (\xwidth,\yb);}
}    }

\newcommand{\doublecrosswall}[5]{ 
\twowallpic[-.5]{#1}{right}{#4}{#5}{
 \tikzmath{\xw=.5;\xcent=\xwidth*4/7;\ycent=\ya/2+\yb/2;}
 \draw[edge,  -o-={0.2}{#2}] (0,\ya) to [out=0, in=180]  (\xwidth,\yb);
 \draw[edge,  -o-={0.4}{#3}] (0,\yb) to [out=0, in=-135] (\xcent-0.08,\ycent-0.08);
 \draw[edge] (\xcent+0.04,\ycent+0.04) to [out=45, in=180] (\xwidth,\ya);
 \draw[edge] (-\xw,\ya) to [out=0, in=180]  (0,\yb);
 \draw[edge] (-\xw,\yb) to [out=0, in=-135] (-\xcent+0.04,\ycent-0.05);
 \draw[edge] (-\xcent+0.2,\ycent+0.08) to [out=45, in=180] (0,\ya);
    }}
    
\newcommand{\crosswithcaps}[1]{
\begin{tikzpicture}
    \tikzmath{\x0=0.2; \xw=1; \y0=.2; \s=0.3; \y2=\y0+2*\s;\yw=\y2+\s+.2;}
    \fill[gray!20!white] (0,0)rectangle(\xw,\yw);
    \draw[wall] (\xw,0) -- (\xw,\yw); 
    \begin{knot}[end tolerance=3pt, clip radius=3pt,background color=gray!20!white]
    \strand[draw, edge, only when rendering/.style={-o-={0.2}{#1}}] (\x0,\y0) .. controls ++(0:0.5) and ++(0:0.5) .. (\x0,\y2);
    \strand[draw, edge, only when rendering/.style={-o-={0.2}{#1}}] (\x0,\y0+\s) .. controls ++(0:0.5) and ++(0:0.5) .. (\x0,\y2+\s);
   \end{knot}
\end{tikzpicture}
}

\newcommand{\crosswithtwoextra}{
\begin{tikzpicture}
   \tikzmath{\x0=0; \xw=1; \y0=.4; \s=0.3; \y1=\y0+\s; \y2=\y1+\s; \y3=\y2+\s; \yw=\y3+.2;}
    \fill[gray!20!white] (0,0)rectangle(\xw,\yw);
    \draw[wall, <-] (\xw,0) -- (\xw,\yw); 
    \draw[edge, -o-={0.2}{<}] (\x0,\y2) -- (\xw,\y2);
    \draw[edge, -o-={0.2}{<}] (\x0,\y3) -- (\xw,\y3);
   \begin{knot}[end tolerance=3pt, clip radius=3pt,background color=gray!20!white]
   \strand[draw, edge, only when rendering/.style={-o-={0.2}{>}}] (\x0,\y0) .. controls ++(0:0.5) and ++(180:0.5) .. (\xw,\y1);
    \strand[edge, draw, only when rendering/.style={-o-={0.2}{>}}] (\x0,\y1) .. controls ++(0:0.5) and ++(180:0.5) .. (\xw,\y0);
   \end{knot}
\end{tikzpicture}
}

\newcommand{\crosswallrecent}[6]{ 
\twowallpic{#1}{right}{#5}{#6}{
 \tikzmath{\xcent=\xwidth*4/7;\ycent=\ya/2+\yb/2;}
  \ifthenelse{\equal{#2}{p}}
    {\draw[edge,  -o-={0.2}{#3}] (0,\ya) to [out=0, in=180]  (\xwidth,\yb);
    \draw[edge,  -o-={0.4}{#4}] (0,\yb) to [out=0, in=-135] (\xcent-0.08,\ycent-0.08);
    \draw[edge] (\xcent+0.04,\ycent+0.04) to [out=45, in=180] (\xwidth,\ya);}   
    {\draw[edge, -o-={0.2}{#4}] (0,\yb) to [out=0, in=180]  (\xwidth,\ya);
    \draw[edge, -o-={0.5}{#3}] (0,\ya) to [out=0, in=135] (\xcent-0.08,\ycent+0.08);
    \draw[edge] (\xcent+0.045,\ycent -0.045) to [out=-45, in=180] (\xwidth,\yb);}
}    }

\newcommand{\crosswallveryold}[6]{
\raisebox{-.15in}{
\begin{tikzpicture}[scale=1.2]
    \fill[gray!20!white] (-0.4,-0.5)rectangle(0.3,.5);
     \coordinate[label={right:{\small #5}}] (A) at (0.3,0.2); 
     \coordinate[label={right:{\small #6}}] (B) at (0.3,-0.2);
    \draw[wall, #1] (0.3,-0.5) -- (0.3,0.5); 
    \ifthenelse{\equal{#2}{p}}
    {\draw[edge,  -o-={0.2}{#3}] (-0.4,0.2) to [out=0, in=180]  (0.3,-0.2);
    \draw[edge,  -o-={0.4}{#4}] (-0.4,-0.2) to [out=0, in=-135] (-0.08,-0.08);
    \draw[edge] (0.04,0.04) to [out=45, in=180] (0.3,0.2);}   
    {\draw[edge, -o-={0.2}{#4}] (-0.4,-0.2) to [out=0, in=180]  (0.3,0.2);
    \draw[edge, -o-={0.5}{#3}] (-0.4,0.2) to [out=0, in=135] (-0.08,0.08);
    \draw[edge] (0.045,-0.045) to [out=-45, in=180] (0.3,-0.2);}
\end{tikzpicture}
}    }

\newcommand{\capwall}[5][]{
\twowallpic{#2}{right}{#4}{#5}{
\ifthenelse{\equal{#1}{w}}{\draw[wall]  (0,0) --(0,1);}{}
\draw[edge, -o-={0.8}{#3}] (\xwidth,\yb) ..controls (0,\yb) and (0,\ya) .. (\xwidth,\ya);
}}

\newcommand{\wallnearcap}[2][0]{
\twowallpic[#1]{}{left}{}{}{
\draw[edge, -o-={0.8}{white}] (\xwidth,\yb)
	..controls (0,\yb) and (0,\ya).. (\xwidth,\ya);	
}}

\newcommand{\wallcup}[5][]{
\twowallpic{#2}{left}{#4}{#5}{
 \ifthenelse{\equal{#1}{w}}{\draw[wall]  (\xwidth,0) --(\xwidth,1);}{}
\draw[edge, -o-={0.8}{#3}] (0,\yb) ..controls (\xwidth,\yb) and (\xwidth,\ya) .. (0,\ya);
}}

\newcommand{\tangleeg}[6]{
\raisebox{-0.4\height}{\begin{tikzpicture}
        \tikzmath{\xw=0.7;\yd=0.25; \s=0.35; \yc=\yd+\s; \yb=\yc+\s; \ya=\yb+\s; \yh=\ya+.25;}
        \fill[gray!20!white] (0,0)rectangle(\xw,\yh);
        \draw[wall]  (0,0) --(0,\yh); \draw[wall] (\xw,0) --(\xw,\yh); 
        \coordinate[label={left:{\small #1}}] (A) at (0,\ya); 
        \coordinate[label={left:{\small #2}}] (B) at (0,\yb);
        \coordinate[label={left:{\small #3}}] (C) at (0,\yc); 
        \coordinate[label={left:{\small #4}}] (D) at (0,\yd);
        \coordinate[label={right:{\small #5}}] (E) at (\xw,\ya-.5*\s); 
        \coordinate[label={right:{\small #6}}] (F) at (\xw,\yd+.5*\s);
        \draw[edge, -o-={0.5}{<}] (A) ..controls ++(0:0.5) and ++(180:0.5) .. (E);
        \draw[edge, -o-={0.5}{>}] (D) ..controls ++(0:0.5) and ++(180:0.5) .. (F);
        \draw[edge, -o-={0.8}{>}] (B) ..controls ++(0:0.5) and ++(0:0.5) .. (C);
\end{tikzpicture}
}\hspace*{-.1in}}

\newcommand{\capnearwall}[1]{
\twowallpic{}{right}{}{}{
\draw[edge, -o-={0.8}{#1}] (0,\yb) ..controls (\xwidth,\yb) and (\xwidth,\ya).. (0,\ya);
}}

\newcommand*{\capcrosswall}[5]{
\raisebox{-.15in}{\begin{tikzpicture}[scale=0.5]
        \fill[gray!20!white] (-0.3,0)rectangle(1,2.2);
        \draw[wall, #1] (1,2.2) --(1,0); 
        \coordinate[label={right:{\small #4}}] (A) at (1,1.5); 
        \coordinate[label={right:{\small #5}}] (B) at (1,0.3);
         \ifthenelse{\equal{#2}{p}}
         {\draw[edge] {(B) -- (0.8,0.6)}; 
         \draw[edge, -o-={.5}{#3}] plot [smooth,tension=1] coordinates {(0.65,0.8) (0.3,1.2) (0,0.7) (0.5,0.5) (A)};} 
         {\draw[edge] {(A) -- (0.75,1.05)}; 
         \draw[edge, -o-={.5}{#3}] plot [smooth,tension=1] coordinates {(0.6,0.9) (0.2,0.6) (0,1.2) (0.5,1.2) (B)};}  
\end{tikzpicture}
}\hspace*{-.1in}}

\newcommand*{\kinkwall}[2]{
\raisebox{-.15in}{\begin{tikzpicture}[scale=0.5]
        \draw[wall] (1.3,1.8) --(1.3,0); 
        \coordinate (A) at (0,1.5); 
        \coordinate (B) at (0,0.3);
         \ifthenelse{\equal{#1}{p}}
         {\draw[edge] {(A) -- (0.25,1.05)}; 
         \draw[edge, -o-={.5}{#2}] plot [smooth,tension=1] coordinates {(0.4,0.9) (0.8,0.6) (1,1.2) (0.5,1.2) (B)};}           
         {\draw[edge] {(B) -- (0.2,0.6)}; 
         \draw[edge, -o-={.5}{#2}] plot [smooth,tension=1] coordinates {(0.35,0.8) (0.7,1.2) (1,0.7) (0.5,0.5) (A)};} 
\end{tikzpicture}
}\hspace*{-.1in}}

\newcommand*{\capwm}[4]{
\raisebox{-.15in}{\begin{tikzpicture}[scale=0.5]
        \ifthenelse{\equal{#2}{u}}{\draw[<-, wall] (0,2.2) --(0,0)}{\draw[->, wall] (0,1.9) --(0,-0.3)};
        \coordinate[label={left:{\small #3}}] (A) at (0,1.5); 
        \coordinate[label={left:{\small #4}}] (B) at (0,0.3);
         \draw[edge, ->-={.7}{#1}] plot [smooth,tension=1] coordinates {(B) (1,0.9) (A)}; 
\end{tikzpicture}
}\hspace*{-.1in}}

\newcommand{\veenearwall}{ 
\raisebox{-0.4\height}{
\begin{tikzpicture}[scale=1.2]
\fill[gray!20!white] (0,0)rectangle(0.6,1);
\draw[wall, <-] (0.6,0)--(0.6,1);
\draw[edge,decoration={markings, mark=at position 0.5 with {\draw[black, fill=white] circle[radius=2pt];}},postaction={decorate}] (0,0.8)--(0.5,0.6);
\draw[edge,decoration={markings, mark=at position 0.5 with {
\draw[black, fill=white] circle[radius=2pt];}},postaction={decorate}] (0,0.4)--(0.5,0.6);
\end{tikzpicture}
}}

\newcommand{\capveewall}[1]{
\raisebox{-0.4\height}{
\begin{tikzpicture}[scale=1.2]
\fill[gray!20!white] (0,0)rectangle(1,1);
\draw[wall, <-] (1,0)--(1,1);
\draw[edge,decoration={markings, 
mark=at position 0.2 with {\draw[black, fill=black] circle[radius=2pt];}},postaction={decorate}] (1,0.3)node[right]{{\small #1}}
	..controls (0.3,0.3) and (0.3,0.8).. (0.8,0.8);
\draw[edge,decoration={markings, 
mark=at position 0.8 with {\draw[black, fill=white] circle[radius=2pt];}},postaction={decorate}] (0.8,0.8) -- (0,0.8);
\end{tikzpicture}
}}

\newcommand{\walltwowall}[8]{ 
\twowallpic{#2}{right}{#7}{#8}{
   \ifthenelse{\equal{#1}{n}}{}{\draw[wall, #1]  (0,0) --(0,1);} 
   \coordinate[label={left:{\small #5}}] (C) at (0,\ya); 
   \coordinate[label={left:{\small #6}}] (D) at (0,\yb);
   \draw[edge,  -o-={0.5}{#3}] (C)  -- (A);
   \draw[edge,  -o-={0.5}{#4}] (D)  -- (B);
    
}    }

\newcommand{\wallcapcupwall}[8]{ 
\twowallpic{#2}{right}{#7}{#8}{
   \ifthenelse{\equal{#1}{n}}{}{\draw[wall, #1]  (0,0) --(0,1);} 
   \coordinate[label={left:{\small #5}}] (C) at (0,\ya); 
   \coordinate[label={left:{\small #6}}] (D) at (0,\yb);
   \draw[edge,  -o-={0.5}{#3}] (C) ..controls (\xwidth*0.4,\ya) and (\xwidth*0.4,\yb) .. (D); 
   \draw[edge,  -o-={0.5}{#4}] (A) ..controls (\xwidth*0.6,\ya) and (\xwidth*0.6,\yb) .. (B);
}    }

\newcommand{\wallsinksourcetwowall}[6]{
\twowallpic{#2}{right}{#5}{#6}{
   \ifthenelse{\equal{#1}{n}}{}{\draw[wall, #1]  (0,0) --(0,1);} 
   \coordinate[label={left:{\small #3}}] (C) at (0,\ya); 
   \coordinate[label={left:{\small #4}}] (D) at (0,\yb);
   \tikzmath{\ym=\ya/2+\yb/2; \xm=\xwidth/2;}
   \draw[edge, decoration={markings, mark=at position 0.7 with {\arrow{<}}},postaction={decorate}] (\xm-0.05,\ym) -- (C);
   \draw[edge, decoration={markings, mark=at position 0.7 with {\arrow{<}}},postaction={decorate}] (\xm-0.05,\ym) -- (D);
   \draw[edge, decoration={markings, mark=at position 0.95 with {\arrow{>}}},postaction={decorate}] (\xm+0.05,\ym) -- (A);
   \draw[edge, decoration={markings, mark=at position 0.95 with {\arrow{>}}},postaction={decorate}] (\xm+0.05,\ym) -- (B);
}    }

\newcommand{\cross}[9]{ 
\raisebox{-0.4\height}{\begin{tikzpicture}
        \tikzmath{\xw=0.7;}
        \fill[gray!20!white] (0,0)rectangle(\xw,1);
        \ifthenelse{\equal{#2}{<-}}{\tikzmath{\yw=0.2;}}{
        \ifthenelse{\equal{#2}{->}}{\tikzmath{\yw=0;}}{\tikzmath{\yw=0.1;}}}
        \ifthenelse{\equal{#1}{n}}{}{\draw[wall, #1] (0,0) --(0,1);} 
        \ifthenelse{\equal{#2}{n}}{}{\draw[wall, #2] (\xw,0) --(\xw,1);} 
        \tikzmath{\ya=\yw+0.6; \yb=\yw+0.2;}
        \coordinate[label={left:{\small #6}}] (C) at (0,\ya); 
        \coordinate[label={left:{\small #7}}] (D) at (0,\yb); 
        \coordinate[label={right:{\small #8}}] (A) at (\xw,\ya); 
        \coordinate[label={right:{\small #9}}] (B) at (\xw,\yb);
       \tikzmath{\xcent=\xw*4/7;\ycent=\ya/2+\yb/2;}
  \ifthenelse{\equal{#3}{p}}
    {\draw[edge,  -o-={0.9}{#4}] (0,\ya) to [out=0, in=180]  (\xw,\yb);
    \draw[edge] (0,\yb) to [out=0, in=-135] (\xcent-0.08,\ycent-0.08);
    \draw[edge, -o-={0.7}{#5}] (\xcent+0.04,\ycent+0.04) to [out=45, in=180] (\xw,\ya);}   
    {\draw[edge, -o-={0.9}{#5}] (0,\yb) to [out=0, in=180]  (\xw,\ya);
    \draw[edge] (0,\ya) to [out=0, in=135] (\xcent-0.08,\ycent+0.08);
    \draw[edge, -o-={0.7}{#4}] (\xcent+0.045,\ycent -0.045) to [out=-45, in=180] (\xw,\yb);}
\end{tikzpicture}    
}    }

\newcommand{\crossextr}[9]{ 
\raisebox{-0.4\height}{\begin{tikzpicture}
        \tikzmath{\xw=0.7;\xww=0.7+0.4;}
        \fill[gray!20!white] (0,0)rectangle(\xww,1);
        \ifthenelse{\equal{#2}{<-}}{\tikzmath{\yw=0.2;}}{
        \ifthenelse{\equal{#2}{->}}{\tikzmath{\yw=0;}}{\tikzmath{\yw=0.1;}}}
        \ifthenelse{\equal{#1}{n}}{}{\draw[wall, #1] (0,0) --(0,1);} 
        \draw[dashed] (\xw,0) --(\xw,1); 
        \ifthenelse{\equal{#2}{n}}{}{\draw[wall, #2] (\xww,0) --(\xww,1);} 
        \tikzmath{\ya=\yw+0.6; \yb=\yw+0.2;}
        \coordinate[label={left:{\small #6}}] (C) at (0,\ya); 
        \coordinate[label={left:{\small #7}}] (D) at (0,\yb); 
        \coordinate[label={right:{\small #8}}] (A) at (\xww,\ya); 
        \coordinate[label={right:{\small #9}}] (B) at (\xww,\yb);
       \tikzmath{\xcent=\xw*4/7;\ycent=\ya/2+\yb/2;}
  \ifthenelse{\equal{#3}{p}}
    {\draw[edge,  -o-={0.9}{#4}] (0,\ya) to [out=0, in=180]  (\xw,\yb);
    \draw[edge] (0,\yb) to [out=0, in=-135] (\xcent-0.08,\ycent-0.08);
    \draw[edge, -o-={0.7}{#5}] (\xcent+0.04,\ycent+0.04) to [out=45, in=180] (\xw,\ya);}   
    {\draw[edge, -o-={0.9}{#5}] (0,\yb) to [out=0, in=180]  (\xw,\ya);
    \draw[edge] (0,\ya) to [out=0, in=135] (\xcent-0.08,\ycent+0.08);
    \draw[edge, -o-={0.7}{#4}] (\xcent+0.045,\ycent -0.045) to [out=-45, in=180] (\xw,\yb);}
    \draw[edge, -o-={0.6}{>}] (\xw,\ya) -- (\xww,\ya);
    \draw[edge, -o-={0.6}{>}] (\xw,\yb) -- (\xww,\yb);
\end{tikzpicture}    
}    }

\newcommand{\crossextl}[9]{ 
\raisebox{-0.4\height}{\begin{tikzpicture}
        \tikzmath{\xw=0.7;\xl=-0.4;}
        \fill[gray!20!white] (\xl,0)rectangle(\xw,1);
        \ifthenelse{\equal{#2}{<-}}{\tikzmath{\yw=0.2;}}{
        \ifthenelse{\equal{#2}{->}}{\tikzmath{\yw=0;}}{\tikzmath{\yw=0.1;}}}
        \ifthenelse{\equal{#1}{n}}{}{\draw[wall, #1] (\xl,0) --(\xl,1);} 
        \draw[dashed] (0,0) --(0,1); 
        \ifthenelse{\equal{#2}{n}}{}{\draw[wall, #2] (\xw,0) --(\xw,1);} 
        \tikzmath{\ya=\yw+0.6; \yb=\yw+0.2;}
        \coordinate[label={left:{\small #6}}] (C) at (\xl,\ya); 
        \coordinate[label={left:{\small #7}}] (D) at (\xl,\yb); 
        \coordinate[label={right:{\small #8}}] (A) at (\xw,\ya); 
        \coordinate[label={right:{\small #9}}] (B) at (\xw,\yb);
       \tikzmath{\xcent=\xw*4/7;\ycent=\ya/2+\yb/2;}
  \ifthenelse{\equal{#3}{p}}
    {\draw[edge,  -o-={0.9}{#4}] (0,\ya) to [out=0, in=180]  (\xw,\yb);
    \draw[edge] (0,\yb) to [out=0, in=-135] (\xcent-0.08,\ycent-0.08);
    \draw[edge, -o-={0.7}{#5}] (\xcent+0.04,\ycent+0.04) to [out=45, in=180] (\xw,\ya);}   
    {\draw[edge, -o-={0.9}{#5}] (0,\yb) to [out=0, in=180]  (\xw,\ya);
    \draw[edge] (0,\ya) to [out=0, in=135] (\xcent-0.08,\ycent+0.08);
    \draw[edge, -o-={0.7}{#4}] (\xcent+0.045,\ycent -0.045) to [out=-45, in=180] (\xw,\yb);}
    \draw[edge, -o-={0.6}{>}] (\xl,\ya) -- (0,\ya);
    \draw[edge, -o-={0.6}{>}] (\xl,\yb) -- (0,\yb);
\end{tikzpicture}    
}    }

\newcommand{\wallcrosswallleft}[9]{ 
\raisebox{-0.4\height}{\begin{tikzpicture}
        \tikzmath{\xwidth=0.8;}
        \fill[gray!20!white] (0,0)rectangle(\xwidth,1);
        \ifthenelse{\equal{#2}{<-}}{\tikzmath{\yw=0.2;}}{
        \ifthenelse{\equal{#2}{->}}{\tikzmath{\yw=0;}}{\tikzmath{\yw=0.1;}}}
        \draw[wall, #1] (0,0) --(0,1); 
        \draw[wall, #2] (\xwidth,0) --(\xwidth,1); 
        \tikzmath{\ya=\yw+0.6; \yb=\yw+0.2;}
        \coordinate[label={left:{\small #6}}] (C) at (0,\ya); 
        \coordinate[label={left:{\small #7}}] (D) at (0,\yb); 
        \coordinate[label={right:{\small #8}}] (A) at (\xwidth,\ya); 
        \coordinate[label={right:{\small #9}}] (B) at (\xwidth,\yb);
       \tikzmath{\xcent=\xwidth*4/7;\ycent=\ya/2+\yb/2;}
  \ifthenelse{\equal{#3}{p}}
    {\draw[edge,  -o-={0.3}{#4}] (0,\ya) to [out=0, in=180]  (\xwidth,\yb);
    \draw[edge,  -o-={0.6}{#5}] (0,\yb) to [out=0, in=-135] (\xcent-0.08,\ycent-0.08);
    \draw[edge] (\xcent+0.04,\ycent+0.04) to [out=45, in=180] (\xwidth,\ya);}   
    {\draw[edge, -o-={0.2}{#5}] (0,\yb) to [out=0, in=180]  (\xwidth,\ya);
    \draw[edge, -o-={0.5}{#4}] (0,\ya) to [out=0, in=135] (\xcent-0.08,\ycent+0.08);
    \draw[edge] (\xcent+0.045,\ycent -0.045) to [out=-45, in=180] (\xwidth,\yb);}
\end{tikzpicture}    
}    }

\newcommand{\crosswallext}[3]{
\raisebox{-.15in}{
\begin{tikzpicture}[scale=1.2]
    \fill[gray!20!white] (-0.4,-0.5)rectangle(0.6,.5);
    \draw[wall, #1] (0.3,-0.5) -- (0.3,0.5);   
    \ifthenelse{\equal{#2}{p}}
    {\draw[edge, -o-={0.2}{#3}] (-0.4,0.2) to [out=0, in=180]  (0.3,-0.2);
    \draw[edge, -o-={0.4}{#3}] (-0.4,-0.2) to [out=0, in=-135] (-0.08,-0.08);
    \draw[edge] (0.05,0.05) to [out=45, in=180] (0.3,0.2);}
    {\draw[edge, -o-={0.2}{#3}] (-0.4,-0.2) to [out=0, in=180]  (0.3,0.2);
    \draw[edge, -o-={0.4}{#3}] (-0.4,0.2) to [out=0, in=135] (-0.08,0.08);
    \draw[edge] (0.045,-0.045) to [out=-45, in=180] (0.3,-0.2);}
    \draw[edge] (0.3,-0.2)--(0.6,-0.2);
    \draw[edge] (0.3,0.2)--(0.6, 0.2);
\end{tikzpicture}
}    }

\newcommand{\capthroughwall}{
\raisebox{-0.4\height}{
\begin{tikzpicture}[scale=1.2]
\fill[gray!20!white] (0,0)rectangle(1,1);
\draw[edge, <-] (0.8,0)--(0.8,1);
\draw[edge,decoration={markings, 
mark=at position 0.8 with {\draw[black, fill=white] circle[radius=2pt];}},postaction={decorate}] (1,0.3)
	..controls (0,0.3) and (0,0.8).. (1,0.8);
\end{tikzpicture}
}}

\newcommand{\capseparatedwall}[2]{
\raisebox{-0.4\height}{
\begin{tikzpicture}[scale=1.2]
\fill[gray!20!white] (0,0)rectangle(0.6,1);
\draw[wall, <-] (0.6,0)--(0.6,1);
\fill[gray!20!white] (1,0)rectangle(1.2,1);
\draw[wall, <-] (1,0)--(1,1);
\draw[edge,decoration={markings, 
mark=at position 0.8 with {\draw[black, fill=white] circle[radius=2pt];}},postaction={decorate}] (0.6,0.3)node[right]{{\small #2}}
	..controls (0,0.3) and (0,0.8).. (0.6,0.8)node[right]{{\small #1}};
\draw[edge] (1,0.8)--(1.2,0.8);	
\draw[edge] (1,0.3)--(1.2,0.3);	
\end{tikzpicture}
}}

\newcommand{\capedgewall}[3]{
\raisebox{-0.4\height}{
\begin{tikzpicture}[scale=1.2]
\fill[gray!20!white] (0,0)rectangle(0.6,1.2);
\draw[wall, <-] (0.6,0)--(0.6,1.2);
\draw[edge,decoration={markings, 
mark=at position 0.2 with {\draw[black, fill=black] circle[radius=2pt];}},postaction={decorate}] (0.6,0.3)node[right]{{\small #3}}
	..controls (0,0.3) and (0,0.6).. (0.6,0.6)node[right]{{\small #2}};
\draw[edge,decoration={markings, 
mark=at position 0.7 with {\draw[black, fill=white] circle[radius=2pt];}},postaction={decorate}] (0,1) -- (0.6,1)node[right]{{\small #1}};
\end{tikzpicture}
}}

\newcommand{\nwallpicold}[7]{ 
\raisebox{-0.4\height}{\begin{tikzpicture}
        \tikzmath{\xwidth=1;\s=0.3;} 
        \ifthenelse{\equal{#2}{left}}{\tikzmath{\xw=0;}}{\tikzmath{\xw=\xwidth;}}
        \ifthenelse{\equal{#1}{<-}}{\tikzmath{\yd=0.2;\yu=0;}}{
        \ifthenelse{\equal{#1}{->}}{\tikzmath{\yd=0;\yu=0.2;}}{\tikzmath{\yd=0;\yu=0;}}}
        \tikzmath{\yh=4*\s+\yd+\yu;}
        \fill[gray!20!white] (0,0)rectangle(\xwidth,\yh);
        \draw[wall, #1] (\xw,0) --(\xw,\yh); 
         \tikzmath{\ya=\yd+3.5*\s; \yc=\yd+0.5*\s;}        
        \ifthenelse{\equal{#3}{b}}{\tikzmath{\yb=\yd+1.5*\s;}}{\tikzmath{\yb=\yd+2.5*\s;}}
        \coordinate[label={[label distance=-1pt] #2:{\small #4}}] (A) at (\xw,\ya); 
        \coordinate[label={[label distance=-1pt] #2:{\small #5}}] (B) at (\xw,\yb);
        \coordinate[label={[label distance=-1pt] #2:{\small #6}}] (C) at (\xw,\yc);
        #7
\end{tikzpicture}
}\hspace*{-.1in}} 

\newcommand{\nwallpic}[8][0.4]{ 
\raisebox{-0.4\height}{\begin{tikzpicture}
        \tikzmath{\xwidth=1;\s=#1;} 
        \ifthenelse{\equal{#3}{left}}{\tikzmath{\xw=0;}}{\tikzmath{\xw=\xwidth;}}
        \ifthenelse{\equal{#2}{<-}}{\tikzmath{\yd=0.2;\yu=0;}}{
        \ifthenelse{\equal{#2}{->}}{\tikzmath{\yd=0;\yu=0.2;}}{\tikzmath{\yd=0;\yu=0;}}}
        \tikzmath{\yh=4*\s+\yd+\yu;}
        \fill[gray!20!white] (0,0)rectangle(\xwidth,\yh);
        \ifthenelse{\equal{#2}{b}}{\draw[wall] (0,0) --(0,\yh);\draw[wall] (\xw,0) --(\xw,\yh);}{
        \ifthenelse{\equal{#2}{n}}{}{\draw[wall, #2] (\xw,0) --(\xw,\yh);}} 
         \tikzmath{\ya=\yd+3.5*\s; \yc=\yd+0.5*\s;}        
        \ifthenelse{\equal{#4}{b}}{\tikzmath{\yb=\yd+1.5*\s;}}{\tikzmath{\yb=\yd+2.5*\s;}}
        \coordinate[label={[label distance=-1pt] #3:{\small #5}}] (A) at (\xw,\ya); 
        \coordinate[label={[label distance=-1pt] #3:{\small #6}}] (B) at (\xw,\yb);
        \coordinate[label={[label distance=-1pt] #3:{\small #7}}] (C) at (\xw,\yc);
        #8
\end{tikzpicture}
}\hspace*{-.1in}}

\newcommand*{\vertexwall}[7][]{
\nwallpic{#2}{right}{#3}{#5}{#6}{#7}{
         \ifthenelse{\equal{#1}{w}}{\draw[wall] (0,0)--(0,\yh);}{}
         \coordinate (V) at (0.1,\yd+2*\s);
         \draw[edge, -o-={.65}{#4}]  (V) to [out=60, in=180] (A); 
         \draw[edge, -o-={.65}{#4}] (V) to [out=-60, in=180] (C); 
         \tikzmath{\dotp=0.8*\xwidth;}
         \ifthenelse{\equal{#3}{b}}
        {\draw[edge, -o-={.58}{#4}] (V) to [out=-30, in=180] (B);
         \node at ($(\dotp,\ya)!.3!(\dotp,\yb)$) {\vdots};}
         {\draw[edge, -o-={.58}{#4}] (V) to [out=30, in=180] (B);
         \node at ($(\dotp,\yb)!.3!(\dotp,\yc)$) {\vdots};}; 
}
}

\newcommand*{\vertexwalltwist}[7][]{
\nwallpic{#2}{right}{b}{#5}{#6}{#7}{
         \ifthenelse{\equal{#1}{w}}{\draw[wall] (0,0)--(0,\yh);}{}
         \coordinate (V) at (0.1,\yd+2*\s);
         \draw[edge, -o-={.85}{#4}]  (V) to [out=-60, in=180] (A); 
         \draw[edge, -o-={.85}{#4}] (V) to [out= 60, in=180] (C); 
         \tikzmath{\dotp=0.8*\xwidth;}
         \draw[edge, -o-={.58}{#4}] (V) to [out=30, in=180] (B);
         \node at ($(\dotp,\ya)!.3!(\dotp,\yb)$) {\vdots}; 
}
}

\newcommand*{\vertexwallcyclic}[2][]{
\raisebox{-0.4\height}{\begin{tikzpicture}
        \tikzmath{\xw=1.2;\xv=0.5; \s=0.4;\yc=0.5*\s; \yb=1.3*\s; \ya=\yb+1.6*\s; \yh=\ya+ .8*\s; \yv=\yb+.8*\s;} 
         \fill[gray!20!white] (0,0)rectangle(\xw,\yh);
        \ifthenelse{\equal{#1}{w}}{\draw[wall] (0,0) --(0,\yh); \draw[wall] (\xw,0) --(\xw,\yh);}{} 
         \coordinate (V) at (\xv,\yv);
         \draw[edge, -o-={.65}{#2}]  (V) .. controls ++(60:0.7) and ++(90:0.5).. (\xv/2,\yv);
         \draw[edge, -o-={.65}{#2}] (\xv/2,\yv) .. controls ++(-90:0.4) and ++(180:0.4).. (\xw,\yc);
         \draw[edge, -o-={.65}{#2}]  (V) to [out=40, in=180] (\xw,\ya); 
         \draw[edge, -o-={.65}{#2}] (V) to [out=-60, in=180] (\xw,\yb); 
         \tikzmath{\dotp=0.85*\xw;}
         \node at ($(\dotp,\yb+.1)!.5!(\dotp,\ya+.1)$) {\vdots};
         \end{tikzpicture}
 }\hspace*{-.1in}}         

\newcommand*{\vertexwallturn}[4]{ 
\raisebox{-0.4\height}{\begin{tikzpicture}
        \tikzmath{\xw=1.2;\xv=1;  \yv= 0.4; \s=0.4;\yc=0.8; \yb=\yc+\s; \ya=\yb+1.6*\s; \yh=\ya+ .5*\s;} 
         \fill[gray!20!white] (0,0)rectangle(\xw,\yh);
         \draw[wall,<-] (\xw,0) --(\xw,\yh); 
         \coordinate (V) at (\xv,\yv);
         \draw[edge, -o-={.7}{#1}] (V) .. controls ++(150:0.4) and ++(180:0.2).. (\xw,\yc) node[right]{\small #4};
         \draw[edge, -o-={.8}{#1}] (V) .. controls ++(180:0.7) and ++(180:0.5).. (\xw,\yb) node[right]{\small #3};
         \draw[edge, -o-={.85}{#1}] (V) .. controls ++(210:1.3) and ++(180:0.8).. (\xw,\ya) node[right]{\small #2};
         \tikzmath{\dotp=0.85*\xw;}
         \node at ($(\dotp,\yb+.1)!.5!(\dotp,\ya+.1)$) {\vdots};
         \end{tikzpicture}
 }\hspace*{-.1in}}  
 
 \newcommand*{\threecaps}[7]{ 
\raisebox{-0.4\height}{\begin{tikzpicture}
        \tikzmath{\xw=1.2; \s=0.4; \yf=0.3; \ye=\yf+1.6*\s; \yd=\ye+\s; \yc=\yd+\s; \yb=\yc+\s; \ya=\yb+1.6*\s; \yh=\ya+ .5*\s;} 
         \fill[gray!20!white] (0,0)rectangle(\xw,\yh);
         \draw[wall,<-] (\xw,0) --(\xw,\yh); 
         \draw[edge, -o-={.7}{#1}] (\xw,\yd) node[right]{\small #5} .. controls ++(180:0.4) and ++(180:0.4).. (\xw,\yc) node[right]{\small #4};
         \draw[edge, -o-={.8}{#1}] (\xw,\ye) node[right]{\small #6} .. controls ++(180:0.7) and ++(180:0.7).. (\xw,\yb) node[right]{\small #3};
         \draw[edge, -o-={.85}{#1}] (\xw,\yf) node[right]{\small #7}  .. controls ++(180:1) and ++(180:1).. (\xw,\ya) node[right]{\small #2};
         \tikzmath{\dotp=0.85*\xw;}
         \node at ($(\dotp,\yb+.1)!.5!(\dotp,\ya+.1)$) {\vdots};
         \node at ($(\dotp,\yf+.1)!.5!(\dotp,\ye+.1)$) {\vdots};
         \end{tikzpicture}
 }\hspace*{-.1in}}

\newcommand*{\onevertexwall}[7]{
\raisebox{-0.4\height}{\begin{tikzpicture}
    \tikzmath{\xwidth=1;\s=0.4;} 
        \tikzmath{\xw=\xwidth;}
        \ifthenelse{\equal{#1}{<-}}{\tikzmath{\yd=0.2;\yu=0;}}{
        \ifthenelse{\equal{#1}{->}}{\tikzmath{\yd=0;\yu=0.2;}}{\tikzmath{\yd=0;\yu=0;}}}
        \tikzmath{\yh=4*\s+\yd+\yu;\lx=-0.3;}
        \fill[gray!20!white] (\lx,0)rectangle(\xwidth,\yh);
        \draw[wall, #1] (\xw,0) --(\xw,\yh); 
         \tikzmath{\ya=\yd+3.5*\s; \yc=\yd+0.5*\s;}        
        \ifthenelse{\equal{#2}{b}}{\tikzmath{\yb=\yd+1.5*\s;}}{\tikzmath{\yb=\yd+2.5*\s;}}
        \coordinate[label={[label distance=-1pt] right:{\small #5}}] (A) at (\xw,\ya); 
        \coordinate[label={[label distance=-1pt] right:{\small #6}}] (B) at (\xw,\yb);
        \coordinate[label={[label distance=-1pt] right:{\small #7}}] (C) at (\xw,\yc);
        \coordinate (V) at (0.1,\yd+2*\s);
         \draw[edge, -o-={.65}{#4}]  (V) to [out=60, in=180] (A); 
         \draw[edge, -o-={.65}{#4}] (V) to [out=-60, in=180] (C); 
         \tikzmath{\dotp=0.8*\xwidth;}
         \ifthenelse{\equal{#2}{b}}
        {\draw[edge, -o-={.58}{#4}] (V) to [out=-30, in=180] (B);
         \node at ($(\dotp,\ya)!.3!(\dotp,\yb)$) {\vdots};}
         {\draw[edge, -o-={.58}{#4}] (V) to [out=30, in=180] (B);
         \node at ($(\dotp,\yb)!.3!(\dotp,\yc)$) {\vdots};};    
         \draw[edge, -o-={.65}{#3}]  (V) to [out=180, in=0] (\lx,\yd+2*\s);
\end{tikzpicture}
}\hspace*{-.1in}}

\newcommand*{\extvertexwall}[7]{
\raisebox{-0.4\height}{\begin{tikzpicture}
    \tikzmath{\xwidth=1;\s=0.4;} 
        \tikzmath{\xw=\xwidth;}
        \ifthenelse{\equal{#1}{<-}}{\tikzmath{\yd=0.2;\yu=0;}}{
        \ifthenelse{\equal{#1}{->}}{\tikzmath{\yd=0;\yu=0.2;}}{\tikzmath{\yd=0;\yu=0;}}}
        \tikzmath{\yh=4*\s+\yd+\yu;\lx=-0.3;}
        \fill[gray!20!white] (\lx,0)rectangle(\xwidth,\yh);
        \draw[wall, #1] (\xw,0) --(\xw,\yh); 
         \tikzmath{\ya=\yd+3.5*\s; \yc=\yd+0.5*\s;}        
        \ifthenelse{\equal{#2}{b}}{\tikzmath{\yb=\yd+1.5*\s;}}{\tikzmath{\yb=\yd+2.5*\s;}}
        \coordinate[label={[label distance=-1pt] right:{\small #5}}] (A) at (\xw,\ya); 
        \coordinate[label={[label distance=-1pt] right:{\small #6}}] (B) at (\xw,\yb);
        \coordinate[label={[label distance=-1pt] right:{\small #7}}] (C) at (\xw,\yc);
        \coordinate (V) at (0.1,\yd+2*\s);
         \draw[edge, -o-={.65}{#4}]  (V) to [out=60, in=180] (A); 
         \draw[edge, -o-={.65}{#4}] (V) to [out=-60, in=180] (C); 
         \tikzmath{\dotp=0.8*\xwidth;}
         \ifthenelse{\equal{#2}{b}}
        {\draw[edge, -o-={.58}{#4}] (V) to [out=-30, in=180] (B);
         \node at ($(\dotp,\ya)!.3!(\dotp,\yb)$) {\vdots};}
         {\draw[edge, -o-={.58}{#4}] (V) to [out=30, in=180] (B);
         \node at ($(\dotp,\yb)!.3!(\dotp,\yc)$) {\vdots};};    
         \draw[edge, -o-={.65}{#3}]  (V) to [out=120, in=-30] (\lx,\yd+3*\s);
         \draw[edge, -o-={.65}{#3}] (V) to [out=-120, in=30] (\lx,\yd+\s); 
         \node at ($(\lx*2/3,\yd+3*\s)!.4!(\lx*2/3,\yd+\s)$) {\vdots};
\end{tikzpicture}
}\hspace*{-.1in}}

\newcommand*{\vertexwallold}[4]{
\raisebox{-.35in}{\begin{tikzpicture}[scale=0.35]
        \tikzmath{\x1=2.5; \s=1; \b=-.1;
         \y1=\b+\s-0.2; \t=\b+4*\s; \y4=\y1+3*\s;}
          \fill[gray!20!white] (-0.3,-0.6)rectangle(\x1,4.2);
        \ifthenelse{\equal{#1}{b}} 
        {\tikzmath{\y2=\y1+\s; \y3=\y2+\s;};}
        {\tikzmath{\y2=\y1+\s; \y3=\y2+\s;};}
        \draw[wall,<-] (\x1,-.6) -- (\x1,{\y4+.5}); 
         \coordinate (V) at (0,\t/2);
         \coordinate[label={[label distance=-1pt] right:{\small #2}}] (A) at (\x1,\y4); 
         \coordinate[label={[label distance=-1pt] right:{\small #4}}] (C) at (\x1,\y1);
         \draw[edge, -o-={.65}{white}]  (V) to [out=60, in=180] (A); 
         \draw[edge, -o-={.65}{white}] (V) to [out=-60, in=180] (C); 
         \ifthenelse{\equal{#1}{b}}
        {\coordinate[label={[label distance=-1pt] right:{\small #3}}] (B) at (\x1,\y2);
        \draw[edge, -o-={.58}{white}] (V) to [out=-30, in=180] (B);
         \node at ($(\x1-0.5,\y2)!.5!(\x1-0.5,\y4)$) {\vdots};}
         {\draw[edge, -o-={.58}{white}] (V) to [out=30, in=180] (\x1,\y3);
          \coordinate[label={[label distance=-1pt] right:{\small #3}}] (B) at (\x1,\y3);
         \node at ($(\x1 -0.5,\y1+0.1)!.5!(\x1-0.5,\y3+0.1)$) {\vdots};}; 
\end{tikzpicture}
}
}

\newcommand*{\vertexnearwall}[3][]{
\raisebox{-.35in}{\begin{tikzpicture}[scale=0.4]
        \tikzmath{\x1=0; \s=1; \b=-.1;
         \y1=\b+\s-0.2; \t=\b+4*\s; \y4=\y1+3*\s;}
          \fill[gray!20!white] (0,0)rectangle(2.5,4.2);
        \ifthenelse{\equal{#2}{b}} 
        {\tikzmath{\y2=\y1+\s; \y3=\y2+\s;};}
        {\tikzmath{\y2=\y1+\s; \y3=\y2+\s;};}
        \draw[wall] (2.5,0) -- (2.5,{\y4+.5}); 
        \ifthenelse{\equal{#1}{}}{}{\draw[wall, #1] (0,0) -- (0,{\y4+.5});}
         \coordinate (V) at (2,\t/2);
         \coordinate (A) at (\x1,\y4); 
         \coordinate (C) at (\x1,\y1);
         \draw[edge, -o-={.35}{#3}]  (A) to [out=0, in=120] (V); 
         \draw[edge, -o-={.4}{#3}] (C) to [out=0, in=-120] (V); 
         \ifthenelse{\equal{#1}{b}}
        {\coordinate (B) at (\x1,\y2);
        \draw[edge, -o-={.4}{#3}] (B) to [out=0, in=-150] (V);
         \node at ($(1-0.5,\y2)!.5!(1-0.5,\y4)$) {\vdots};}
         {\coordinate (B) at (\x1,\y3);
         \draw[edge, -o-={.4}{#3}] (B) to [out=0, in=150] (V);
         \node at ($(1 -0.5,\y1+0.1)!.5!(1-0.5,\y3+0.1)$) {\vdots};}; 
\end{tikzpicture}
}
}

\newcommand*{\wallnvertex}[6]{ 
\raisebox{-.35in}{\begin{tikzpicture}[scale=0.4]
        \tikzmath{\x1=0; \s=1; \b=-.1;
         \y1=\b+\s-0.2; \t=\b+4*\s; \y4=\y1+3*\s;}
          \fill[gray!20!white] (0,0)rectangle(2.5,4.2);
        \ifthenelse{\equal{#2}{b}} 
        {\tikzmath{\y2=\y1+\s; \y3=\y2+\s;};}
        {\tikzmath{\y2=\y1+\s; \y3=\y2+\s;};}
        \draw[wall, #1] (0,0) -- (0,{\y4+.5});
         \coordinate (V) at (2,\t/2);
         \coordinate[label={[label distance=-1pt] left:{\small #4}}] (A) at (\x1,\y4); 
         \coordinate[label={[label distance=-1pt] left:{\small #6}}] (C) at (\x1,\y1);
         \draw[edge, -o-={.35}{#3}]  (A) to [out=0, in=120] (V); 
         \draw[edge, -o-={.4}{#3}] (C) to [out=0, in=-120] (V); 
         \ifthenelse{\equal{#1}{b}}
        {\coordinate[label={[label distance=-1pt] left:{\small #5}}] (B) at (\x1,\y2);
        \draw[edge, -o-={.4}{#3}] (B) to [out=0, in=-150] (V);
         \node at ($(1-0.5,\y2)!.5!(1-0.5,\y4)$) {\vdots};}
         {\coordinate[label={[label distance=-1pt] left:{\small #5}}] (B) at (\x1,\y3);
         \draw[edge, -o-={.4}{#3}] (B) to [out=0, in=150] (V);
         \node at ($(1 -0.5,\y1+0.1)!.5!(1-0.5,\y3+0.1)$) {\vdots};}; 
\end{tikzpicture}
}
}

\newcommand*{\vertexfour}[2][]{
\raisebox{-.35in}{\begin{tikzpicture}[scale=0.4]
        \tikzmath{\x1=2.5; \s=1; 
         \y1=.5; \y2= \y1+1.5*\s; \y3=\y1+ 2.2*\s; \y4=\y1+3.5*\s; \t=\y4+.5;}
          \fill[gray!20!white] (0,0)rectangle(\x1,\t);
        \ifthenelse{\equal{#1}{w}}{\draw[wall] (0,0) -- (0,{\y4+.5});}{}
        \draw[wall] (2.5,0) -- (2.5,{\y4+.5}); 
         \coordinate (V) at (0.5,\t/2);
         \draw[edge, -o-={.8}{#2}]  (V) to [out=60, in=180] (\x1,\y4);
         \draw[edge, -o-={.75}{#2}]  (V) to [out=25, in=180] (\x1,\y3);
         \draw[edge, -o-={.75}{#2}] (V) to [out=-25, in=180] (\x1,\y2);
         \draw[edge, -o-={.8}{#2}] (V) to [out=-60, in=180] (\x1,\y1); 
         \node at ($(\x1-0.3,\y3+0.15)!.5!(\x1-0.3,\y4+0.15)$) {\vdots};
         \node at ($(\x1 -0.3,\y1+0.1)!.5!(\x1-0.3,\y2+0.1)$) {\vdots};        
\end{tikzpicture}
}
}

\newcommand*{\vertexfourcross}[2][]{
\raisebox{-.35in}{\begin{tikzpicture}[scale=0.4]
        \tikzmath{\x1=2.5; \s=1; 
         \y1=.5; \y2= \y1+1.5*\s; \y3=\y1+ 2.2*\s; \y4=\y1+3.5*\s; \t=\y4+.5;}
          \fill[gray!20!white] (0,0)rectangle(\x1,\t);
        \ifthenelse{\equal{#1}{w}}{\draw[wall] (0,0) -- (0,{\y4+.5});}{}
        \draw[wall] (2.5,0) -- (2.5,{\y4+.5}); 
         \coordinate (V) at (0.5,\t/2);
         \draw[edge, -o-={.8}{#2}]  (V) to [out=60, in=180] (\x1,\y4);
         \draw[edge, -o-={.9}{#2}]  (V) to [out=25, in=155] (\x1,\y2-.1);
         
         \draw[edge] (V) to [out=-25, in=-155] (1.7,\y2+0.1);
         \draw[edge, -o-={.75}{#2}] (1.9,\y2+.4) to [out=25, in=180] (\x1,\y3);
         
         \draw[edge, -o-={.8}{#2}] (V) to [out=-60, in=180] (\x1,\y1); 
         \node at ($(\x1-0.3,\y3+0.15)!.5!(\x1-0.3,\y4+0.15)$) {\vdots};
         \node at ($(\x1 -0.3,\y1+0.1)!.5!(\x1-0.3,\y2+0.1)$) {\vdots};        
\end{tikzpicture}
}
}

\newcommand*{\vertexnearwalll}[6]{
\raisebox{-.35in}{\begin{tikzpicture}[scale=0.4]
        \tikzmath{\x1=0; \s=1;
         \y0=0; \y1=\y0+\s; \y3=\y1+3*\s+.2; \x2=2.5;}
          \fill[gray!20!white] (0,0)rectangle(2.5,\y3+0.5);
        \ifthenelse{\equal{#2}{b}} 
        {\tikzmath{\y2=\y1+\s;};}
        {\tikzmath{\y2=\y1+2*\s;};}
        \draw[wall] (\x2,0) -- (\x2,{\y3+.5}); 
        \draw[wall, #1] (0,0) -- (0,\y3+.5);
         \coordinate (V) at (2,\y1+1.5*\s);
         \coordinate[label={[label distance=-1pt] left:{\small #4}}] (A) at (\x1,\y3); 
         \coordinate[label={[label distance=-1pt] left:{\small #5}}] (B) at (\x1,\y2);
         \coordinate[label={[label distance=-1pt] left:{\small #6}}] (C) at (\x1,\y1);
         \draw[edge, -o-={.35}{#3}]  (A) to [out=0, in=120] (V); 
         \draw[edge, -o-={.4}{#3}] (C) to [out=0, in=-120] (V); 
         \ifthenelse{\equal{#2}{b}}
        {\draw[edge, -o-={.4}{#3}] (B) to [out=0, in=-150] (V);
         \node at ($(1-0.5,\y2)!.5!(1-0.5,\y3)$) {\vdots};}
         {\draw[edge, -o-={.4}{#3}] (B) to [out=0, in=150] (V);
         \node at ($(1 -0.5,\y1+0.1)!.5!(1-0.5,\y2+0.1)$) {\vdots};}; 
 \end{tikzpicture}
}
}

\newcommand*{\vertexnearwalledge}[9]{
\raisebox{-.35in}{\begin{tikzpicture}[scale=0.4]
        \tikzmath{\x1=0; \s=1;
         \y0=1; \y1=\y0+\s; \y3=\y1+3*\s+.2; \x2=2.5;}
          \fill[gray!20!white] (0,0)rectangle(2.5,\y3+0.5);
        \ifthenelse{\equal{#2}{b}} 
        {\tikzmath{\y2=\y1+\s;};}
        {\tikzmath{\y2=\y1+2*\s;};}
        \draw[wall, #1] (\x2,0) -- (\x2,{\y3+.5}); 
        \draw[wall, #1] (0,0) -- (0,\y3+.5);
         \coordinate (V) at (2,\y1+1.5*\s);
         \coordinate[label={[label distance=-1pt] left:{\small #4}}] (A) at (\x1,\y3); 
         \coordinate[label={[label distance=-1pt] left:{\small #5}}] (B) at (\x1,\y2);
         \coordinate[label={[label distance=-1pt] left:{\small #6}}] (C) at (\x1,\y1);
         \draw[edge, -o-={.35}{#3}]  (A) to [out=0, in=120] (V); 
         \draw[edge, -o-={.4}{#3}] (C) to [out=0, in=-120] (V); 
         \ifthenelse{\equal{#2}{b}}
        {\draw[edge, -o-={.4}{#3}] (B) to [out=0, in=-150] (V);
         \node at ($(1-0.5,\y2)!.5!(1-0.5,\y3)$) {\vdots};}
         {\draw[edge, -o-={.4}{#3}] (B) to [out=0, in=150] (V);
         \node at ($(1 -0.5,\y1+0.1)!.5!(1-0.5,\y2+0.1)$) {\vdots};}; 
         \coordinate[label={[label distance=-1pt] left:{\small #7}}] (D) at (\x1,\y0);
         \coordinate[label={[label distance=-1pt] right:{\small #8}}] (E) at (\x2,\y0);
         \draw[edge, -o-={.5}{#9}] (D) --(E);
 \end{tikzpicture}
}
}

\newcommand*{\vertexnearwallcap}[6]{
\raisebox{-.35in}{\begin{tikzpicture}[scale=0.4]
        \tikzmath{\x1=0; \s=1;
         \y0=1; \y1=\y0+0.3; \y3=\y1+3*\s+.2; \x2=2.5;}
          \fill[gray!20!white] (0,0)rectangle(2.5,\y3+0.5);
         \tikzmath{\y2=\y1+\s;}
        \draw[wall, #1] (\x2,0) -- (\x2,\y3+.5); 
        \draw[wall, #1] (0,0) -- (0,\y3+.5);
         \coordinate (V) at (2,\y1+1.5*\s);
         \coordinate[label={[label distance=-1pt] left:{\small #3}}] (A) at (\x1,\y3); 
         \coordinate[label={[label distance=-1pt] left:{\small #4}}] (B) at (\x1,\y2);
         \draw[edge, -o-={.35}{#2}]  (A) to [out=0, in=120] (V); 
         \draw[edge, -o-={.4}{#2}] (B) to [out=0, in=-150] (V);
         \node at ($(1-0.5,\y2)!.5!(1-0.5,\y3)$) {\vdots};
         \coordinate[label={[label distance=-1pt] right:{\small #5}}] (E) at (\x2,\y0);
         \draw[edge, -o-={.5}{#6}] (V) to [out=-150, in=180] (E);
 \end{tikzpicture}
}
}

\newcommand*{\nedgewall}[6]{
\nwallpic{#1}{right}{#2}{#4}{#5}{#6}{ 
         \draw[edge, -o-={.5}{#3}]  (0,\ya) -- (A);
         \draw[edge, -o-={.5}{#3}]  (0,\yb) -- (B); 
         \draw[edge, -o-={.5}{#3}] (0,\yc) -- (C);
         \ifthenelse{\equal{#2}{b}}
        {\node at ($(\xwidth*0.75,\ya)!.3!(\xwidth*0.75,\yb)$) {\vdots};}
        {\node at ($(\xwidth*0.75,\yb)!.3!(\xwidth*0.75,\yc)$) {\vdots};};
}
}

\newcommand*{\phalftwist}[5]{ 
\nwallpic{#1}{right}{b}{#3}{#4}{#5}{
\begin{knot}[consider self intersections, background color=gray!20!white, clip radius=3pt, end tolerance=3pt, clip radius=3pt]
     \strand[draw, edge, only when rendering/.style={-o-={.2}{#2}}] (0,\ya) .. controls +(0:.8) and +(180:.5) ..  (C); 
     \strand[draw, edge, only when rendering/.style={-o-={.2}{#2}}]  (0,\yb+\s) .. controls +(0:.5) and +(180:.5) ..  (B); 
     \strand[draw, edge, only when rendering/.style={-o-={.2}{#2}}]  (0,\yc) .. controls +(0:.5) and +(180:.5) ..  (\xw,\ya);
\end{knot}         
        \node at ($(\xwidth*0.9,\ya)!.3!(\xwidth*0.9,\yb)$) {\vdots};
        \node at ($(\xwidth*0.1,\yb+.1)!.3!(\xwidth*0.1,\yc+.1)$) {\vdots};
}
}

\newcommand*{\nhalftwist}[5]{ 
\nwallpic{#1}{right}{b}{#3}{#4}{#5}{
\begin{knot}[consider self intersections, background color=gray!20!white, clip radius=3pt, end tolerance=3pt,
 clip radius=3pt]
      \strand[draw, edge, only when rendering/.style={-o-={.2}{#2}}]  (0,\yc) .. controls +(0:.5) and +(180:.5) ..  (\xw,\ya);
      \strand[draw, edge, only when rendering/.style={-o-={.2}{#2}}]  (0,\yb+\s) .. controls +(0:.5) and +(180:.5) ..  (B); 
      \strand[draw, edge, only when rendering/.style={-o-={.2}{#2}}] (0,\ya) .. controls +(0:.8) and +(180:.5) ..  (C); 
\end{knot}         
       \node at ($(\xwidth*0.9,\ya)!.3!(\xwidth*0.9,\yb)$) {\vdots};
       \node at ($(\xwidth*0.1,\yb+.1)!.3!(\xwidth*0.1,\yc+.1)$) {\vdots};
}
}

\newcommand*{\nedgewalltall}[6]{
\nwallpic[0.4]{#1}{right}{#2}{#4}{#5}{#6}{
         \draw[edge, -o-={.5}{#3}]  (0,\ya) -- (A);
         \draw[edge, -o-={.5}{#3}]  (0,\yb) -- (B); 
         \draw[edge, -o-={.5}{#3}] (0,\yc) -- (C);
         \ifthenelse{\equal{#2}{b}}
        {\node at ($(\xwidth*0.75,\ya)!.3!(\xwidth*0.75,\yb)$) {\vdots};}
        {\node at ($(\xwidth*0.75,\yb)!.3!(\xwidth*0.75,\yc)$) {\vdots};};
}
}

\newcommand*{\wallnedge}[6]{
\nwallpic{#1}{left}{#2}{#4}{#5}{#6}{
         \draw[edge, -o-={.5}{#3}]  (\xwidth,\ya) -- (A);
         \draw[edge, -o-={.5}{#3}]  (\xwidth,\yb) -- (B); 
         \draw[edge, -o-={.5}{#3}] (\xwidth,\yc) -- (C);
         \ifthenelse{\equal{#2}{b}}
        {\node at ($(\xwidth*0.25,\ya)!.3!(\xwidth*0.25,\yb)$) {\vdots};}
        {\node at ($(\xwidth*0.25,\yb)!.3!(\xwidth*0.25,\yc)$) {\vdots};};
}
}

\newcommand*{\nedgewallold}[5]{
\raisebox{-.35in}{\begin{tikzpicture}[scale=0.35]
        \tikzmath{\x1=2.5; \s=1; \b=-.1;
         \y1=\b+\s-0.2; \t=\b+4*\s; \y4=\y1+3*\s;}
          \fill[gray!20!white] (-0.3,-0.6)rectangle(\x1,4.2);
        \ifthenelse{\equal{#2}{b}} 
        {\tikzmath{\y2=\y1+\s; \y3=\y2+\s;};}
        {\tikzmath{\y2=\y1+\s; \y3=\y2+\s;};}
        \draw[wall,#1] (\x1,-.5) -- (\x1,\y4+.7); 
         \coordinate[label={[label distance=-1pt] right:{\small #3}}] (A) at (\x1,\y4); 
         \coordinate[label={[label distance=-1pt] right:{\small #5}}] (C) at (\x1,\y1);
         \draw[edge, -o-={.5}{white}]  (0,\y4) -- (A); 
         \draw[edge, -o-={.5}{white}] (0,\y1) -- (C); 
         \ifthenelse{\equal{#2}{b}}
        {\coordinate[label={[label distance=-1pt] right:{\small #4}}] (B) at (\x1,\y2);
        \draw[edge, -o-={.5}{white}] (0,\y2) -- (B); 
         \node at ($(\x1-0.5,\y2)!.5!(\x1-0.5,\y4)$) {\vdots};}
         {\coordinate[label={[label distance=-1pt] right:{\small #4}}] (B) at (\x1,\y3);
          \draw[edge, -o-={.5}{white}] (0,\y3) -- (B);
         \node at ($(\x1 -0.5,\y1+0.1)!.5!(\x1-0.5,\y3+0.1)$) {\vdots};}; 
\end{tikzpicture}
}
}

\newcommand*{\wallnedgewall}[9]{
\raisebox{-.35in}{\begin{tikzpicture}[scale=0.35]
        \tikzmath{\x0=-0.3; \x1=2.5; \s=1; \b=-.1;
         \y1=\b+\s-0.2; \t=\b+4*\s; \y4=\y1+3*\s;}
          \fill[gray!20!white] (\x0,-0.6)rectangle(\x1,4.2);
        \ifthenelse{\equal{#2}{b}} 
        {\tikzmath{\y2=\y1+\s; \y3=\y2+\s;};}
        {\tikzmath{\y2=\y1+\s; \y3=\y2+\s;};}
        \draw[wall,#1] (\x0,-.5) -- (\x0,\y4+.7); 
        \draw[wall,#1] (\x1,-.5) -- (\x1,\y4+.7); 
         \coordinate[label={[label distance=-1pt] right:{\small #7}}] (A) at (\x1,\y4); 
         \coordinate[label={[label distance=-1pt] right:{\small #9}}] (C) at (\x1,\y1);
         \coordinate[label={[label distance=-1pt] left:{\small #4}}] (A0) at (\x0,\y4); 
         \coordinate[label={[label distance=-1pt] left:{\small #6}}] (C0) at (\x0,\y1);
         \draw[edge, -o-={.5}{#3}]  (\x0,\y4) -- (A); 
         \draw[edge, -o-={.5}{#3}] (\x0,\y1) -- (C); 
         \ifthenelse{\equal{#2}{b}}
        {\coordinate[label={[label distance=-1pt] right:{\small #8}}] (B) at (\x1,\y2);
         \coordinate[label={[label distance=-1pt] left:{\small #5}}] (B0) at (\x0,\y2);
        \draw[edge, -o-={.5}{#3}] (\x0,\y2) -- (\x1,\y2); 
         \node at ($(\x1-0.5,\y2)!.6!(\x1-0.6,\y4)$) {\vdots};}
         {\coordinate[label={[label distance=-1pt] right:{\small #8}}] (B) at (\x1,\y3);
          \coordinate[label={[label distance=-1pt] left:{\small #5}}] (B0) at (\x0,\y3);
          \draw[edge, -o-={.5}{#3}] (\x0,\y3) -- (B);
         \node at ($(\x1 -0.5,\y1+0.1)!.6!(\x1-0.5,\y3+0.1)$) {\vdots};}; 
\end{tikzpicture}
}
}

\newcommand*{\wallnedgewallcap}[7]{
\raisebox{-.35in}{\begin{tikzpicture}[scale=0.35]
        \tikzmath{\x0=-0.3; \x1=2.5; \s=1;
        \yc0=0.4; \yc1=\yc0+\s;
         \y1=\yc1+\s; \t=4*\s; \y2=\y1+3*\s;}
         \fill[gray!20!white] (\x0,-0.6)rectangle(\x1,\y2+.7);
        \draw[wall,<-] (\x0,-.6) -- (\x0,\y2+.7); 
        \draw[wall,<-] (\x1,-.6) -- (\x1,\y2+.7); 
         \coordinate[label={[label distance=-1pt] right:{\small #4}}] (A) at (\x1,\y2); 
         \coordinate[label={[label distance=-1pt] right:{\small #5}}] (B) at (\x1,\y1);
         \coordinate[label={[label distance=-1pt] left:{\small #2}}] (A0) at (\x0,\y2); 
         \coordinate[label={[label distance=-1pt] left:{\small #3}}] (B0) at (\x0,\y1);
         \draw[edge, -o-={.5}{>}]  (A0) -- (A); 
         \draw[edge, -o-={.5}{>}] (B0) -- (B); 
         \node at ($(\x1-0.5,\y1)!.6!(\x1-0.6,\y2)$) {\vdots};
         \coordinate[label={[label distance=-1pt] right:{\small #7}}] (C) at (\x1,\yc0);
         \coordinate[label={[label distance=-1pt] right:{\small #6}}] (D) at (\x1,\yc1);
         \draw[edge, -o-={.8}{>}] (C) ..controls ++(180:1.5) and ++(180:1.5) .. (D);
\end{tikzpicture}
}
}

\newcommand*{\vertexwallext}[1]{
\raisebox{-.35in}{\begin{tikzpicture}[scale=0.35]
       \fill[gray!20!white] (-0.3,-0.6)rectangle(4.2,4.2);
        \tikzmath{\x1=2.5; \s=1; \b=-.1;
         \y1=\b+\s-0.2; \t=\b+4*\s; \y4=\y1+3*\s;}
        \ifthenelse{\equal{#1}{b}}{\tikzmath{\y2=\y1+\s; \y3=\y2+\s;}}
        {\tikzmath{\y2=\y1+\s; \y3=\y2+\s;}};
        \draw[wall,<-] (\x1,-.6) -- (\x1,{\y4+.5}); 
         \coordinate (V) at (0,\t/2);
         \draw[edge]  (V) to [out=60, in=180] (\x1,\y4); 
         \draw[edge, -o-={.5}{white}]  (\x1,\y4) to coordinate[pos=0.5](M1) (\x1+1.5,\y4);
         \draw[edge] (V) to [out=-60, in=180] (\x1,\y1); 
         \draw[edge, -o-={.5}{white}]  (\x1,\y1) to coordinate[pos=0.5](M4) (\x1+1.5,\y1);
         \ifthenelse{\equal{#1}{b}}
        {\draw[edge] (V) to [out=-30, in=180] (\x1,\y2);
         \draw[edge, -o-={.5}{white}]  (\x1,\y2) to coordinate[pos=0.5](M3) (\x1+1.5,\y2); 
         \node at ($(\x1-0.5,\y2)!.5!(\x1-0.5,\y4)$) {\vdots};
          \node at ($(\x1+0.5,\y2)!.5!(\x1+0.5,\y4)$) {\vdots};}
         {\draw[edge] (V) to [out=30, in=180] (\x1,\y3);
         \draw[edge, -o-={.5}{white}] (\x1,\y3) to coordinate[pos=0.5](M2) (\x1+1.5,\y3); 
         \node at ($(\x1 -0.5,\y1+0.1)!.5!(\x1-0.5,\y3+0.1)$) {\vdots};
         \node at ($(\x1+0.5,\y1+.1)!.5!(\x1+0.5,\y3+.1)$) {\vdots};}; 
\end{tikzpicture}
}
}

\newcommand*{\wallvertex}[2][0]{
\raisebox{-.30in}{\begin{tikzpicture}[scale=0.35]
        \tikzmath{\x1=2.5; \s=1; \b=-.1;
         \y1=\b+\s-0.2; \t=\b+4*\s; \y4=\y1+3*\s;}
          \fill[gray!20!white] (#1-.4,-0.4)rectangle(\x1,4.2);
        \ifthenelse{\equal{#2}{b}} 
        {\tikzmath{\y2=\y1+\s; \y3=\y2+\s;};}
        {\tikzmath{\y2=\y1+\s; \y3=\y2+\s;};}
        \draw[wall,
        ] (-.4,-.4) -- (-.4,{\y4+.5}); 
         \coordinate (V) at (0,\t/2);
         \coordinate (A) at (\x1,\y4); 
         \coordinate (C) at (\x1,\y1);
         \draw[edge, -o-={.6}{white}]  (V) to [out=60, in=180] (A); 
         \draw[edge, -o-={.6}{white}] (V) to [out=-60, in=180] (C); 
         \ifthenelse{\equal{#2}{b}}
        {\coordinate (B) at (\x1,\y2);
        \draw[edge, -o-={.6}{white}] (V) to [out=-30, in=180] (B);
         \node at ($(\x1-0.5,\y2)!.5!(\x1-0.5,\y4)$) {\vdots};}
         {\draw[edge, -o-={.6}{white}] (V) to [out=30, in=180] (\x1,\y3);
          \coordinate (B) at (\x1,\y3);
         \node at ($(\x1 -0.5,\y1+0.1)!.5!(\x1-0.5,\y3+0.1)$) {\vdots};}; 
\end{tikzpicture}
}
}

\newcommand*{\walledges}[4]{
\raisebox{-.3in}{\begin{tikzpicture}[scale=0.35]
        \tikzmath{\x1=0; \s=1; \y1=\s-0.3; \t=-.1+4*\s; \y4=\y1+3*\s;} 
         \fill[gray!20!white] (0,-0.4)rectangle(\x1+2,\y4+.5);
        \ifthenelse{\equal{#3}{b}}
        {\tikzmath{\y2=\y1+\s; \y3=\y2+\s;}}
        {\tikzmath{\y2=\y1+\s; \y3=\y2+\s;}}
         \draw[wall, <-] (\x1,-.4) -- (\x1,\y4+.5); 
         \coordinate[label={
         left:{\small #2}}] (A) at (\x1,\y4); 
         \coordinate[label={left:{\small #4}}] (C) at (\x1,\y1);
         \draw[edge, -o-={.5}{white}] (\x1,\y4) -- (\x1+2,\y4);
         \draw[edge, -o-={.5}{white}] (\x1,\y1) -- (\x1+2,\y1);
         \ifthenelse{\equal{#1}{b}}
         {\coordinate[label={left:{\small #3}}] (B) at (\x1,\y2); 
           \draw[edge, -o-={.5}{white}] (\x1,\y2) -- (\x1+2,\y2);
          \node at ($(\x1+0.5,\y4)!.5!(\x1+0.5,\y2)$) {\vdots}}
         {\coordinate[label={left:{\small #3}}] (B) at (\x1,\y3);
          \draw[edge, -o-={.5}{white}] (\x1,\y3) -- (\x1+2,\y3);
          \node at ($(\x1+0.5,\y3)!.5!(\x1+0.5,\y1)$) {\vdots}}; 
\end{tikzpicture}
}
}

\newcommand*{\vertexnearw}[2]{%
\raisebox{-.35in}{\begin{tikzpicture}[scale=0.3]
        \tikzmath{\x1=2.5; \s=1; \b=0;
         \y1=\b+\s-0.2; \t=\b+4*\s; \tt = \t/2; \y4=\y1+3*\s;}
         \tikzmath{\y2=\y1+\s; \y3=\y2+\s;}
        \draw[->,wall] (\x1,-.4) -- (\x1,{\y4+.6}); 
         \coordinate (V) at (2,\y4/2);
         \draw[edge, ->-={.2}{#1}] (0,\y1) to coordinate[pos=0.2](M1) (V);
         \draw[edge, ->-={.2}{#1}] (0,\y4) to coordinate[pos=0.2](M4) (V);
         \ifthenelse{\equal{#2}{t}}
         {\draw[edge, ->-={.2}{#1}] (0,\y3) to coordinate[pos=0.2](M3) (V);
          \node at ($(M1)!.6!(M3)$) {\small \vdots}}
         {\draw[edge, ->-={.2}{#1}] (0,\y2) to coordinate[pos=0.2](M2) (V);
          \node at ($(M4)!.5!(M2)$) {\small \vdots}};
\end{tikzpicture}
}
}

\newcommand{\bigonpic}[6]{ 
\raisebox{-0.4\height}{\begin{tikzpicture}
        \tikzmath{\xdim=#4; \ydim=#5; \h=0.5; \o=\ydim/2; \t1=\ydim/3; \t2=\ydim*2/3;} 
        \ifthenelse{\equal{#3}{b}}{
         \fill[gray!20!white, rounded corners=1.5mm]  (\xdim/2, -\h)-- (0,0)--(0,\ydim) -- (\xdim/2, \ydim+\h) -- 
         (\xdim,\ydim)  -- (\xdim,0) -- cycle;
          \draw[wall,  -o-={.3}{#1}, rounded corners=1.5mm] (\xdim/2, -\h)-- (0,0)--(0,\ydim) -- (\xdim/2, \ydim+\h);
       \draw[wall, -o-={.3}{#2}, rounded corners=1.5mm] (\xdim/2, -\h)-- (\xdim,0)--(\xdim,\ydim) -- (\xdim/2, \ydim+\h);
        \draw[black, fill=white] (\xdim/2,\ydim+\h) circle[radius=2pt];
        }
        {\draw[wall,  -o-={.51}{#1}, fill=gray!20!white]  
        (\xdim/2, -\h) [rounded corners=1.5mm] --  (0,0) [sharp corners] --(0,\ydim) 
        arc [start angle =180, end angle =0, radius=\xdim/2] 
         [rounded corners=1.5mm]  --(\xdim,0) [sharp corners] -- cycle;} 
        \draw[black, fill=white] (0.5*\xdim,-\h) circle[radius=2pt];
          #6
\end{tikzpicture}
}\hspace*{-.1in}} 

\newcommand{\bigonpicc}[6]{ 
\raisebox{-0.4\height}{\begin{tikzpicture}
        \tikzmath{\xdim=#4; \ydim=#5; \h=0.5; \o=\ydim/2; \t1=\ydim/3; \t2=\ydim*2/3;} 
        \ifthenelse{\equal{#3}{b}}{
         \fill[gray!20!white, rounded corners=1.5mm]  (\xdim/2, -\h)-- (0,0)--(0,\ydim) -- (\xdim/2, \ydim+\h) -- 
         (\xdim,\ydim)  -- (\xdim,0) -- cycle;
          \draw[wall,  -o-={.3}{#1}, rounded corners=1.5mm] (\xdim/2, -\h)-- (0,0)--(0,\ydim) -- (\xdim/2, \ydim+\h);
       \draw[wall, -o-={.3}{#2}, rounded corners=1.5mm] (\xdim/2, -\h)-- (\xdim,0)--(\xdim,\ydim) -- (\xdim/2, \ydim+\h);
        \draw[black, fill=white] (\xdim/2,\ydim+\h) circle[radius=2pt];
        }
        {\draw[wall,  -o-={.51}{#1}, fill=gray!20!white]  
        (\xdim/2, \ydim+\h) [rounded corners=1.5mm] --  (0,\ydim) [sharp corners] --(0,0) 
        arc [start angle =-180, end angle =0, radius=\xdim/2] 
         [rounded corners=1.5mm]  --(\xdim,\ydim) [sharp corners] -- cycle;} 
        \draw[black, fill=white] (0.5*\xdim,\ydim+\h) circle[radius=2pt];
          #6
\end{tikzpicture}
}\hspace*{-.1in}}

\newcommand{\bigonedge}[6]{ 
\bigonpic{#1}{#2}{#3}{.5}{.5}{
\coordinate[label={left:{\small #5}}] (A) at (0,\o);  
\coordinate[label={right:{\small #6}}] (B) at (\xdim,\o);
\draw[edge, -o-={0.6}{#4}] (A) -- (B);
}
}

\newcommand{\bigontwoedges}[9]{ 
\bigonpic{#1}{#2}{#3}{.5}{.5}{
\coordinate[label={left:{\small #6}}] (A) at (0,\ydim); 
\coordinate[label={left:{\small #7}}] (B) at (0,\ydim/4);   
\coordinate[label={right:{\small #8}}] (C) at (\xdim,\ydim);
\coordinate[label={right:{\small #9}}] (D) at (\xdim,\ydim/4);
\draw[edge, -o-={0.55}{#4}] (A) -- (C);
\draw[edge, -o-={0.55}{#5}] (B) -- (D);
}
}

\newcommand{\bigontwonodes}[9]{ 
\bigonpic{#1}{#2}{#3}{.9}{.7}{
\coordinate[label={left:{\small #6}}] (A) at (0,\ydim); 
\coordinate[label={left:{\small #7}}] (B) at (0,0.1);   
\coordinate[label={right:{\small #8}}] (C) at (\xdim,\ydim);
\coordinate[label={right:{\small #9}}] (D) at (\xdim,0.1);
\ifthenelse{\equal{#4}{-}}{
\draw (A)--(C);}{
\node [thick, circle, draw, fill=white, radius=0.5] (N) at (\xdim/2,\ydim) {\hspace*{-.1in}#4\hspace*{-.1in}};
\draw[edge] (A) -- (N);
\draw[edge] (N) -- (C);
}
\node [thick, circle, draw, fill=white, radius=0.5] (M) at (\xdim/2,0.1) {\hspace*{-.1in}#5\hspace*{-.1in}};
\draw[edge] (B) -- (M);
\draw[edge] (M) -- (D);
}
}

\newcommand{\bigontwonodesc}[9]{ 
\bigonpicc{#1}{#2}{#3}{.9}{.7}{
\coordinate[label={left:{\small #6}}] (A) at (0,\ydim); 
\coordinate[label={left:{\small #7}}] (B) at (0,0.1);   
\coordinate[label={right:{\small #8}}] (C) at (\xdim,\ydim);
\coordinate[label={right:{\small #9}}] (D) at (\xdim,0.1);
\ifthenelse{\equal{#4}{-}}{
\draw (A)--(C);}{
\node [thick, circle, draw, fill=white, radius=0.5] (N) at (\xdim/2,\ydim) {\hspace*{-.1in}#4\hspace*{-.1in}};
\draw[edge] (A) -- (N);
\draw[edge] (N) -- (C);
}
\node [thick, circle, draw, fill=white, radius=0.5] (M) at (\xdim/2,0.1) {\hspace*{-.1in}#5\hspace*{-.1in}};
\draw[edge] (B) -- (M);
\draw[edge] (M) -- (D);
}
}

\newcommand{\bigonpunctc}[6]{ 
\bigonpicc{#1}{#2}{#3}{.9}{.7}{
\coordinate[label={left:{\small #5}}] (B) at (0,\ydim-0.1);   
\coordinate[label={right:{\small #6}}] (D) at (\xdim,\ydim-0.1);
\draw[black, fill=white] (\xdim/2,0.1) circle[radius=2pt];
\node [thick, circle, draw, fill=white, radius=0.5] (M) at (\xdim/2,\ydim-0.1) {\hspace*{-.1in}#4\hspace*{-.1in}};
\draw[edge] (B) -- (M);
\draw[edge] (M) -- (D);
}
}

\newcommand{\bigoncap}[6]{ 
\bigonpic{#1}{#2}{#3}{.5}{.5}{
\coordinate[label={right:{\small #5}}] (C) at (\xdim,\ydim);
\coordinate[label={right:{\small #6}}] (D) at (\xdim,\ydim/4);
\draw[edge, -o-={0.8}{#4}] (C) ..controls (0,\ydim) and (0,\ydim/4) .. (D);
}
}

\newcommand{\bigonalpha}[6]{ 
\bigonpic{#1}{#2}{#3}{0.8}{.6}{
\coordinate[label={left:{\small #5}}] (A) at (0,\o); 
\coordinate[label={right:{\small #6}}] (B) at (\xdim,\o);
\node [thick, circle, draw, fill=white, radius=0.5] (N) at (\xdim/2,\o) {\hspace*{-.1in}#4\hspace*{-.1in}};
\draw[edge] (0,\o) -- (N);
\draw[edge] (N) -- (\xdim,\o);
}
}

\usetikzlibrary{cd}

\usepackage[all]{xy}
\usepackage{graphicx}
\usepackage{texdraw}
\usepackage{url}
\usepackage{bbm}
\usepackage{mathrsfs}
\usepackage{soul}
\usepackage[final]{pdfpages}

\def\printname#1{
        \if\draft y
                \smash{\makebox[0pt]{\hspace{-0.5in}
                        \raisebox{8pt}{\tt\tiny #1}}}
        \fi
}

\def\printname#1{
        \if\draft y
                \smash{\makebox[0pt]{\hspace{-0.5in}
                        \raisebox{8pt}{\tt\tiny #1}}}
        \fi
}

\newlength{\standardunitlength}
\catcode`\@=11
\long\def\@makecaption#1#2{
     \vskip 10pt

\setbox\@tempboxa\hbox{
       \small\sf{\bfcaptionfont #1. }\ignorespaces #2}
     \ifdim \wd\@tempboxa >\captionwidth {
         \rightskip=\@captionmargin\leftskip=\@captionmargin
         \unhbox\@tempboxa\par}
       \else
         \hbox to\hsize{\hfil\box\@tempboxa\hfil}
     \fi}
\font\bfcaptionfont=cmssbx10 scaled \magstephalf
\newdimen\@captionmargin\@captionmargin=2\parindent
\newdimen\captionwidth\captionwidth=\hsize
\catcode`\@=12

\def\lbl#1{\label{#1}\printname{#1}}

                        \theoremstyle{plain}

\newtheorem{theorem}{Theorem}[section]
\newtheorem*{theorem*}{Theorem}
\newtheorem{lemma}[theorem]{Lemma}
\newtheorem{corollary}[theorem]{Corollary}
\newtheorem*{corollary*}{Corollary}
\newtheorem{proposition}[theorem]{Proposition}
\newtheorem*{proposition*}{Proposition}
\newtheorem{conjecture}[theorem]{Conjecture}
\newtheorem*{conjecture*}{Conjecture}
\newtheorem{question}[theorem]{Question}

\newtheorem{definition}[theorem]{Definition}

\theoremstyle{definition}

\newtheorem{remark}[theorem]{Remark}
\newtheorem{example}[theorem]{Example}
\newtheorem*{example*}{Example}

\newcommand{\bcon}{\begin{conjecture}}
\newcommand{\econ}{\end{conjecture}}
\newcommand{\bcor}{\begin{corollary}}
\newcommand{\ecor}{\end{corollary}}
\newcommand{\bdf}{\begin{definition}}
\newcommand{\edf}{\end{definition}}
\newcommand{\benu}{\begin{enumerate}}
\newcommand{\eenu}{\end{enumerate}}
\newcommand{\beq}{\begin{equation}}
\newcommand{\eeq}{\end{equation}}
\newcommand{\bexa}{\begin{example}}
\newcommand{\eexa}{\end{example}}
\newcommand{\bexe}{\begin{exercise}}
\newcommand{\eexe}{\end{exercise}}
\newcommand{\bfac}{\begin{fact}}
\newcommand{\efac}{\end{fact}}
\newcommand{\bite}{\begin{itemize}}
\newcommand{\eite}{\end{itemize}}
\newcommand{\blem}{\begin{lemma}}
\newcommand{\elem}{\end{lemma}}
\newcommand{\bmat}{\begin{pmatrix}}
\newcommand{\emat}{\end{pmatrix}}
\newcommand{\bprb}{\begin{problem}}
\newcommand{\eprb}{\end{problem}}
\newcommand{\bpro}{\begin{proposition}}
\newcommand{\epro}{\end{proposition}}

\newcommand{\bque}{\begin{question}}
\newcommand{\eque}{\end{question}}
\newcommand{\brem}{\begin{remark}}
\newcommand{\erem}{\end{remark}}
\newcommand{\bthm}{\begin{theorem}}
\newcommand{\ethm}{\end{theorem}}

\newcommand{\bpr}{\begin{proof}}
\newcommand{\epr}{\end{proof}}
\newcommand{\ignore}[1]{}

\newcommand{\lb}{\label}
\newcommand{\la}{\langle}
\newcommand{\ra}{\rangle}
\newcommand{\comment}[1]{\,}
\newcommand{\wh}{\widehat}
\newcommand{\cal}{\mathcal}
\newcommand{\p}{\partial}

\newcommand{\Z}{\mathbb Z}
\newcommand{\Q}{\mathbb Q}
\newcommand{\K}{\mathcal K}
\newcommand{\R}{\mathbb R}
\newcommand{\C}{\mathbb C}

\newcommand{\ve}{\varepsilon}

\newcommand{\diagg}[2]{\raisebox{-.5\height}{\includegraphics[height=#2]{#1}}}

\renewcommand{\S}{\mathcal S}

\def\dl{\partial_\ell}
\def\dr{\partial_r}
\newcommand{\cev}[1]{\reflectbox{\ensuremath{\vec{\reflectbox{\ensuremath{#1}}}}}}
\newcommand{\ca}{{\cev{a}  }}
\def\ep{\epsilon}
\def\hR{\widehat{\mathcal R}}
\def\OqM{\cO_q(M(n))}
\def\Oq{{\cO_q(SL(n))}}
\def\OqZ{\cO_{q,\BZ}(SL(n))}
\def\Zv{{\BZ[v^{\pm1}]}}
\def\Uqq{{U_q(sl_n) } }
\def\tUqq{\widetilde{\Uqq}}
\def\Qv{{\BQ(v)}}
\def\End{\mathrm{End}}
\def\dl{\partial_\ell}
\def\dr{\partial_r}
\def\ba{{\mathbf a}}
\def\buu{{\mathbf u}}
\def\bX{{\mathbf X}}
\def\bY{{\mathbf Y}}

\def\BN{\mathbb N}
\def\BZ{\mathbb Z}
\def\BR{\mathbb R}

\def\BQ{\mathbb Q}

\def\la{\langle}
\def\ra{\rangle}

\def\cS{\mathscr S}
\def\ot{\otimes}

\def\cE{\mathcal E}

\def\cP{\mathcal P}

\def\Id{\mathrm{Id}}
\def\fS{\mathfrak{S}}
\def\fC{\mathfrak C}

\def\embed{\hookrightarrow}

\def\im{\mathrm{Im}}

\def\cT{\mathcal T}

\def\bQ{\overline Q}

\def\tw{\tilde{\ww}}

\def\embed{\hookrightarrow}

\def\im{\mathrm{Im}}

\def\vk{\varkappa}

\def\pS{\partial \Sigma}

\def\al{\alpha}
\def\ve{\varepsilon}
\def\be { \begin{equation} }
\def\ee { \end{equation} }

\def\bD{{\underline \Delta }}

\def\im{\mathrm {Im}}

\def\bS{\bar \Sigma}

\def\wt{\widetilde}

\newcommand\no[1]{}

      \def\nc{\newcommand}

    \nc\FIGc[3]{\begin{figure}[htpb]
    \includegraphics[height=#3]{#1-eps-converted-to.pdf}
    \caption{#2}
    \label{f.#1}
    \end{figure}}

\newcommand\incl[2]{{\includegraphics[height=#1]{#2-eps-converted-to.pdf}}}

\def\cA{\mathcal A}
\def\cB{\mathcal B}

\def\fB{\mathfrak B}
\def\pr{\mathrm{pr}}
\def\cO{\mathcal O}

\def\cR{\mathcal R}

\def\bomu{{\boldsymbol \mu}}

\def\cN{\mathcal N}
\def\MN{(M,\cN)}

   \def\hR{\widehat{\mathcal R}}
\def\OqM{\cO_q(M(n))}
\def\OqZ{\cO_{q,\BZ}(SL(n))}
\def\Zv{{\BZ[v^{\pm1}]}}

\def\tUqq{\widetilde{\Uqq}}
\def\Qv{{\BQ(v)}}
\def\End{\mathrm{End}}
\def\dl{\partial_\ell}
\def\dr{\partial_r}
\def\ba{{\mathbf a}}
\def\buu{{\mathbf u}}
\def\bX{{\mathbf X}}
\def\bY{{\mathbf Y}}
\def\pM{\partial M}
\def\vk{\varkappa}
\def\drQ{\partial_r(Q)}
\def\pal{\partial \al}
\def\bpp{\beta^\perp}

\def\hR{\widehat{\mathcal R}}
\def\OqM{\cO_q(M(n))}
\def\OqZ{\cO_{q,\BZ}(SL(n))}
\def\Zv{{\BZ[v^{\pm1}]}}

\def\tUqq{\widetilde{\Uqq}}
\def\Qv{{\BQ(v)}}
\def\End{\mathrm{End}}
\def\dl{\partial_\ell}
\def\dr{\partial_r}
\def\ba{{\mathbf a}}
\def\buu{{\mathbf u}}
\def\bX{{\mathbf X}}
\def\bY{{\mathbf Y}}
\def\sgn{\mathrm{sgn}}
\def\id{\mathrm{id}}
\def\RT{\mathsf{RT}_0}
\def\RTp{{\mathsf{RT}}}
\def\fM{\mathfrak{M}}

\def\Uq{U_q(sl_n)}
\def\Hom{\mathrm{Hom}}

\begin{document}

\title{Stated SL($n$)-Skein Modules and Algebras}

\author[Thang  T. Q. L\^e]{Thang  T. Q. L\^e}
\address{School of Mathematics, 686 Cherry Street,
 Georgia Tech, Atlanta, GA 30332, USA}
\email{letu@math.gatech.edu}
\author[Adam S. Sikora]{Adam S. Sikora}
\address{Depatrment of Mathematics, University at Buffalo}
\email{asikora@buffalo.edu}
 
\date{\today}

\thanks{The first author was supported in part by National Science Foundation. \\
2020 {\em Mathematics Classification:} 
Primary: 
57K31, 
Secondary: 57K16,  
20G42, 
17B37, 
\\
}

  \newcommand{\blue}[1]{{\color{blue} #1}}
  \newcommand{\brown}[1]{{\color{brown} #1}}
  
  \def\uot{{\underline{\otimes}}}

\begin{abstract}
We develop a theory of stated SL($n$)-skein modules, $\S_n(M,\cN),$ of $3$-manifolds $M$ marked with intervals $\cN$ in their boundaries. These skein modules, generalizing stated SL($2$)-modules of the first author, 
stated SL($3$)-modules of Higgins', and SU(n)-skein modules of the second author, consist of linear combinations of framed, oriented graphs, called $n$-webs, with ends in $\cN$,  considered up to skein relations of the $\Uq$-Reshetikhin-Turaev functor on tangles, involving coupons representing the anti-symmetrizer and its dual. 

We prove the Splitting Theorem asserting that cutting of a marked $3$-manifold $M$ along a disk resulting in a $3$-manifold $M'$ yields a homomorphism $\S_n(M)\to \S_n(M')$ for all $n$. That result allows to analyze the skein modules of $3$-manifolds through the skein modules of their pieces. 

The theory of stated skein modules is particularly rich for thickened surfaces $M=\Sigma \times (-1,1),$ in whose case, $\S_n(M)$ is an algebra, denoted by $\S_n(\Sigma).$ One of the main results of this paper asserts that the skein algebra of the ideal bigon is isomorphic with $\Oq$ and it provides simple geometric interpretations of the product, coproduct, counit, the antipode, and the cobraided structure on $\Oq.$ (In particular, the coproduct is given by a splitting homomorphism.) We show that for surfaces with boundary $\Sigma$ every splitting homomorphism is injective and that $\S_n(\Sigma)$ is a free module with a basis induced from the Kashiwara-Lusztig canonical bases.

Additionally, we show that a splitting of a thickened bigon near a marking defines a right $\Oq$-comodule structure on $\S_n(M),$ or dually, a left $\Uq$-module structure. Furthermore, we show that the skein algebra of surfaces $\Sigma_1, \Sigma_2$ glued along two sides of a triangle is isomorphic with the braided tensor product $\S_n(\Sigma_1)\uot \S_n(\Sigma_2)$ of Majid.
These results allow for geometric interpretation of further concepts in the theory of quantum groups, for example, of the braided products 
and of Majid's transmutation operation.

Building upon the above results, we prove that the factorization homology 
with coefficients in the category of representations of $\Uq$  is equivalent to the category of left modules over $\S_n(\Sigma)$ for surfaces $\Sigma$ with $\p \Sigma=S^1$.
We also establish isomorphisms of our skein algebras  
with the quantum moduli spaces of Alekseev-Schomerus and with the internal algebras of the skein categories for these surfaces and $\mathfrak g=sl(n)$.

\end{abstract}

\maketitle

\tableofcontents{}

    \newcommand{\orange}[1]{{\color{orange}#1}}
    \newcommand{\gray}[1]{{\color{gray}#1}}
   \def\bi{{\bar i}}
   \def\bj{{\bar j}}
    \def\bF{\bar F}
    \def\pF{\partial {F}}
    \def\inv{\mathrm{inv}}
 \def\ptfS{\partial \tilde{\fS}}
 \def\Cut{{\mathsf{cut}}} 
 \def\CutD{\Cut_D}
    
%
\section{Introduction}
\label{s.intro}
%

%
\subsection{Motivation}
%

Moduli spaces of flat connections on surfaces and their quantizations play a pivotal role in quantum field theory. For example, they appear as the classical phase spaces of the Yang-Mills and Chern-Simons theories, \cite{AB, Wi}, and are quantized by these theories. More rigorous quantizations are achieved in mathematics through the Topological Quantum Field Theories, \cite{Tu2}, Kauffman bracket skein algebras, \cite{Pr, Tu1, PS}, 
the lattice gauge theory, \cite{AGS,BR1}, and more recently, through (quantum) cluster algebras, \cite{FG2, JLSS}, and factorization homology, \cite{BBJ}. These quantizations and relations between them are a subject of current active research and are of central importance to Quantum Topology. 

Based on ideas of \cite{BW}, the first author extended the notion of the Kauffman bracket skein algebras (quantizing $SL(2,\C)$-character varieties) to their stated version, built of links and arcs with stated ends, \cite{Le-triang}. His approach made it possible to analyze skein algebras of surfaces through surface triangulations and provided a conceptual framework for the Bonahon-Wang theory, \cite{BW, Le-qteich}. 

On the other hand, the second author introduced the notion of SL($n$)-skein modules of $3$-manifolds and proved that they quantize their SL($n$)-character varieties, \cite{Si}. Based on these two developments, we develop a theory of stated SL($n$)-skein modules of $3$-manifolds.
Our theory generalizes the recent work of Higgins for $n=3,$~\cite{Hi}, however it is not a straightforward generalization of the $n=2,3$ cases, because the SL($n$)-skein theories for $n=2$ and $3$ rely on explicit bases of skein algebras 
which are unknown for higher $n$. In fact, one of the main achievements of this work is a construction of bases for the SL($n$)-skein algebras of surfaces with non-empty boundary for all $n$.  

We discuss the relation between our stated SL($n$)-skein algebras and other quantizations of the $SL(n)$-character varieties of surfaces in Subsection \ref{ss.intro-facthom} and Section \ref{s.facthom_etc}.

%
\subsection{Skein Modules of marked $3$-manifolds}
%

In this paper we will work with a commutative ring of coefficients $R$ with a distinguished invertible $v=q^{\frac{1}{2n}}.$ 
A marked $3$-manifold is a pair $(M, \cN),$ where $M$ is a smooth oriented $3$-manifold with (possibly empty) boundary $\p M$ and $\cN \subset \p M$ consists of open intervals, called {\bf markings}.

An  $n$-web $\alpha$ in $\MN$ is a disjoint union of an oriented link and a directed ribbon graph whose every vertex is either $1$-valent end in $\cN$ or an internal $n$-valent sink or source, cf. Fig. \ref{f.webexample}. Each web is equipped with a transversal vector field called its framing, which at each end $e$ points in the direction of the marking containing $e$, cf. Subsec. \ref{ss.marked}. 

\begin{figure}[h] 
   \centering \diagg{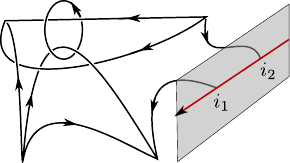}{1.2in}
   \caption{An example of a $3$-web with two endpoints in a marking (in red) stated by $i_1,i_2$.}
   \label{f.webexample}
\end{figure}

A {\bf state} of a web $\alpha$ is an assignment of a label from $\{1,\dots, n\}$ to each of its ends.

The stated $SL(n)$-skein module, $\S_n(M,\cN)$, of $(M,\cN)$ is the space of all $R$-linear combination of stated $n$ webs in $(M,\cN)$, up to internal skein relations \eqref{e.pm}-\eqref{e.sinksource} and boundary skein relations \eqref{e.vertexnearwall}-\eqref{e.crossp-wall}. 
These relations mimic those satisfied by the Reshetikhin-Turaev functor on tangles, with $n$-vertices representing the anti-symmetrizer tensor and its dual. More specifically, the internal relations are based on the skein relations of \cite{Si}. (It may be useful to recall here the premise of \cite{Si}: that although $n$-webs seem unnecessary from the point of view of study of quantum invariants of links in manifolds, they allow for a very effcient formulation of the necessary skein relations.)

However, our specific relations involve a novel sign modification 
which leads to a major technical benefit: the half-edges around each $n$-valent vertex have a cyclic ordering only, rather than a linear ordering required in \cite{Si}, cf. Subsections \ref{ss.unbased-t} and \ref{ss.half-ribbon}. An additional benefit of this modification is that it makes skein relations invariant under the orientation reversal of the webs. (That is reversal of all loop orientations and edge directions.)

The boundary skein relations of $\S_n(M,\cN)$ are new and generalize those of \cite{Le-triang} and \cite{Hi}.

%
\subsection{Splitting Homomorphisms}
%

An important property of stated skein modules is that they behave in a simple manner under the splitting of 3-manifolds along disks. 
Specifically, for a marked 3-manifold $\MN$ with a properly embedded closed disk $D$ in $M-\cN$, let $M'$ be $M$ with an open collar neighborhood of $D$ removed. Then $M'$ is a $3$-manifolds with copies $D_1,D_2$ of $D$ in its boundary, whose gluing together leads to an epimorphism 
$pr: M'\to M.$ Given an oriented open interval $\beta\subset D$, the splitting of $\MN$ along $(D,\beta)$, denoted by  $\Cut_{(D,\beta)}\MN$, is  the marked 3-manifold $(M', \cN')$, where $\cN'= \cN \cup \beta_1 \cup \beta_2$, where $\beta_i\subset D_i$ are the connected components of $pr^{-1}(\beta)$, cf. Fig. \ref{f.intro.splitexample}. Note that for any stated $n$-web $\alpha$ in $\MN$ transversal to $D$ with $\alpha\cap D \subset \beta,$ the inverse image $pr^{-1}(\alpha)$ is an $n$-web in $(M', \cN')$ stated at all its ends except those at $\beta_1\cup \beta_2.$
Given any map $s:\al \cap \beta \to \{1,\dots, n\}$ let $\al(s)$ be $\pr^{-1}(\al)$ with each of its ends $x\in \pr^{-1}(\al) \cap (\beta_1 \cup \beta_2)$ stated by $s(\pr(x))$.

\begin{figure}[h] 
   \centering \diagg{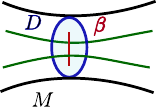}{0.7in} $\xrightarrow{\Theta_{D,\beta}} \sum_{i,j=1}^n$
   \centering \diagg{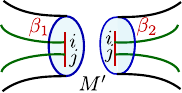}{0.7in}
    {}
   \caption{An example of a splitting of an $n$-web (in green) intersecting the splitting disk $D$ twice.}
   \label{f.intro.splitexample}
\end{figure}
 
The following result generalizes that of \cite{Le-triang, BL} for the Kauffman bracket skein modules ($n=2$) and of \cite{Hi} for $n=3$:

\begin{theorem*}[Splitting Theorem \ref{t.splitting}]
There is a unique $R$-module homomorphism
$$ \Theta_{(D,\beta)}: \S_n(M,\cN) \to \S_n(\Cut_{(D,\beta)}\MN)$$
sending every stated $(D,\beta)$-transverse $n$-web $\alpha$ in $\MN$ to the sum of all of its lifts,
$$\Theta_{(D,\beta)}(\al) = \sum_{s: \al \cap \beta \to \{\pm\}} \al(s).$$
\end{theorem*}

%
\subsection{Basic properties of stated skein modules}
\label{ss.i-sym}

We discuss symmetries and other properties of stated skein modules of marked $3$-manifolds in Subsections \ref{ss.symmetries}-\ref{ss.edgeauto}.  In particular, we observe that for every marked $3$-manifold $\MN$, the orientation reversal of webs $\alpha\to \cev{\alpha}$ defines an $R$-module automorphism:
$$\cev {\,\cdot\,}: \S_n(M,\cN)\to \S_n(M,\cN),$$
where an \underline{orientation} of a web consists of orientations of all its loop components and directions of all its edges.

Let $\overline{(M,\cN)}$ denote $M$ and $\cN$ with reversed orientations. Let $\bar R$ be the ring $R$ with the distinguished element $v^{-1}$ instead of $v$. For an $n$-web $\al$ of $\MN$ let $\overline{\alpha}$ be the $n$-web in $\overline{(M,\cN)}$ obtained from $\al$ by negating its framing, $f \to -f$.

\begin{theorem*}[Thm. \ref{t.framing-rev}] 
(1) Any ring isomorphism $\vk: R\to \bar R$ sending $v$ to $v^{-1}$ extends to an isomorphism of $R$-modules
$\vk_{(M,\cN)}: \S_n(M,\cN,R)\xrightarrow{\cong} \S_n(\overline{M,\cN}, \bar R)$ sending every stated $n$-web $\alpha$ to $\overline{\alpha}$, where $\S_n(\overline{M,\cN}, \bar R)$ is an $R$-module via $\vk: R\to \bar R.$\\
(2) The composition   $ \vk_{(\overline{M,\cN})} \circ \vk_{(M,\cN)}$ is  the identity on $\S_n(M,\cN,R)$.
\end{theorem*}

For every marking $\beta$ of $\MN$ there is an $R$-module automorphism $g_\beta$ of $\S_n(M,\cN)$, called a {\bf marking automorphism}, sending stated $n$-webs $\alpha$ to
$$g_\beta(\alpha)=\prod_{x\in \alpha\cap \beta} (-1)^{n-1}q^{2s(x)-n-1}\cdot \alpha$$ 
where $s(x)$ is the state of the endpoint $x$ of $\alpha.$  
\def\htw{\mathrm{htw}}
\def\htwf{\raisebox{-25pt}{\incl{2cm}{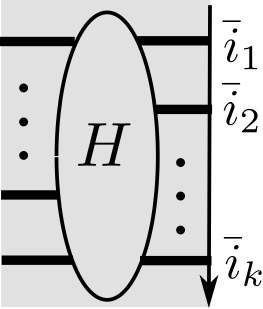}}}
\def\htwn{\raisebox{-25pt}{\incl{2cm}{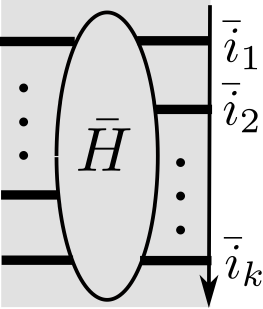}}}
\def\UL{U^L}

There is an additional automorphism of $\S_n\MN$ associated with each marking of $\MN:$

\begin{proposition*}[Prop. \ref{p.twist}]
 For any marking $\beta$ in $\cN$ there exist unique $R$-linear isomorphisms, called the {\bf half-twist automorphisms},
 $$\htw_{\beta}, \wt\htw_\beta :\S_n(M,\cN) \to \S_n(M,\cN)$$
 sending any stated $n$-web $\alpha$ in $(M,\cN)$ with $k$ endpoints on $\beta$ to
 $$\htw_{\beta}   \left( \nedgewalltall{<-}{b}{}{$i_k$}{$i_2$}{$i_1$}\right)= 
\left( \prod_{j=1}^k c_{\overline{i_j}}\right) \cdot
\nedgewalltall{->}{b}{}{$\overline{i_k}$} {$\overline{i_2}$}{$\overline{i_1}$} 
= \left( \prod_{j=1}^k c_{\overline{i_j}}\right) \cdot \htwf$$
and to 
$$\wt\htw_{\beta}   \left(\nedgewalltall{<-}{b}{}{${i_k}$}{${i_2}$}{${i_1}$}\right)= 
\left( \prod_{j=1}^k c_{i_j}\right) \cdot
\nedgewalltall{->}{b}{}{$\overline{i_k}$} {$\overline{i_2}$}{$\overline{i_1}$} = \left( \prod_{j=1}^k c_{i_j}\right) \cdot \htwf,$$
where $c_i\in R$'s are defined in \eqref{e.a} in Subsec. \ref{ss.notation}, and $\overline{i}=n+1-i$. $H$ is the positive half-twist -- see further details in Subsection \ref{ss.htw}. (The directions of the horizontal edges are arbitrary.)
\end{proposition*}

\subsection{Stated SL($n$)-skein algebras of surfaces}

\def\OqR{\mathcal O_q(sl_n;R)}
\def\OqQ{\mathcal O_q(sl_n;\Qv)}

The theory of stated SL($n$)-skein modules is particularly rich for thickened surfaces $M=\Sigma \times (-1,1).$ 
It is most convenient to consider it in the context of {\bf punctured bordered surfaces} ({\bf pb surfaces} for short) which are of the form $\Sigma = \bS- \cP,$ where $\bS$ is a compact oriented surface and $\cP \subset \bS$ is a finite set, called the {\bf ideal points} of $\Sigma$, which meets each connected component of $\p \bS.$ Then $\pS$ is a union of open intervals. These intervals are called {\bf boundary edges}.

In each boundary edge $e$ choose a point $b_e$.
Let $\S_n({\Sigma})= \S_n({\Sigma} \times (-1,1), \cN)$, where $\cN$ is the union of all $b_e \times (-1,1)$, each having the natural orientation of the interval $(-1,1)$. 

For the monogon $\fM$, which is the closed disk with a boundary point removed, we have

\begin{theorem*}[Thm. \ref{t.monogon}.]  The map 
$\mu: R\to S_n(\fM)$ given by $\mu(r)= r\cdot \emptyset$ is an $R$-algebra isomorphism.
\end{theorem*}

Despite its simple statement, the proof of the above result is non-trivial -- see the comment at the end of this subsection.

The bigon, $\fB$, is a closed disk with two boundary points removed.  In Lemma \ref{l.S_n-gens} we show that the $R$-algebra $\S_n(\fB)$ is generated by the arcs $a_{ij}=\bigonedge{}{}{b}{>}{$i$}{$j$}$ for $1\leq i,j\leq n.$ 
Splitting $\fB$ along an interior ideal arc connecting its two ideal vertices defines an algebra $R$-homomorphism $$ \Delta: \S_n(\fB) \to \S_n(\fB) \ot \S_n(\fB).$$

Let $\epsilon: \S_n(\fB) \to R$ be the composition 
 $$\epsilon: \S_n(\fB) \xrightarrow{\wt\htw_{e_r}}\S_n(\fB) \xrightarrow{\iota_*} \S_n(\fM)\simeq R,$$
where $\htw_{e_r}$ is the half-twist automorphism defined above and $\iota_*$ is the algebra homomorphism induced by an embedding $\fB\hookrightarrow \fM$ filling in one of the two ideal points of $\fB$, (depicted always on top of $\fB$ in this paper).
\def\counit{\raisebox{-16pt}{\incl{1.6 cm}{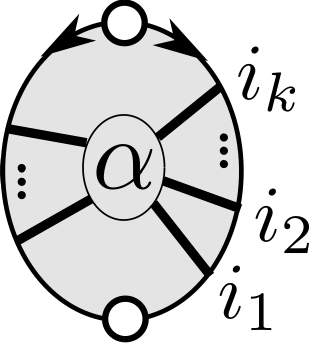}}}
 \def\counitb{\raisebox{-16pt}{\incl{1.6 cm}{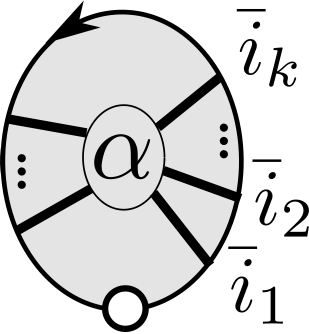}}}
 \def\counitc{\raisebox{-12pt}{\incl{1.2 cm}{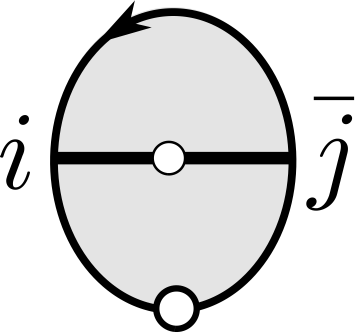}}}
On generators,
$$\epsilon(a^i_j)= \epsilon(\ca^i_j)= c_{{\bar  j}}\, \,  \counitc 
=\delta_{i,j}.$$

Let $\Oq$ be the quantized coordinate ring algebra of $sl_n$. This Hopf algebra is the restricted dual of the quantized enveloping algebra, $\Uq$, cf. \cite[9.2.2]{KS}. For technical convenience, we consider
$\Oq$ as defined over $\Q[v^{\pm 1}].$

\begin{theorem*}[Thm. \ref{t.Hopf}]
(a) The algebra $\S_n(\fB)$ has the structure of a Hopf algebra over $R$ with the coproduct $\Delta$, the counit $\ep$, and the antipode $S$ such that 
$$S(a^i_j) = (-q)^{i-j}\,\,  \ca^\bj _\bi\ \text{for}\ i,j=1,\dots, n.$$ 

(b) The map $\Psi (u^i_j) = a^i_j$ extends to a unique Hopf algebra isomorphism  
$$\OqR:=\Oq\otimes R\xrightarrow{\Psi} \S_n(\fB) .$$
\end{theorem*}

The above theorems generalize statements for $n=2$ in \cite{CL} and $n=3$ in \cite{Hi}.
However, it is not a straightforward generalization. The proofs of \cite{CL, Hi} relied on specific bases of $S_n(\Sigma)$ for $n=2,3$ which can be obtained through the confluence method, cf. \cite{SW}.
That method does not work for higher $n$ and a construction of bases for $n>3$ is an important and still open problem. (We discuss our progress on that problem below.) We were able to establish the above theorem without constructing a basis of $\S_n(\fB).$

\subsection{Geometric interpretation of the cobraided structure on $\Oq$}
\label{ss.i-pbsurf}

The Hopf algebra $\OqR$ is {\bf dual quasitriangular} (see  \cite[Section 2.2]{Maj}, 
 \cite[Section 10]{KS}, \cite[Section 10.3]{ES}), also known as {\bf cobraided} (see e.g. \cite[Section VIII.5]{Kass}). This means it has an $R$-form (also known as  co-$R$-matrix), which is a bilinear form
 $$\rho : \OqR \otimes \OqR \to R$$ satisfying certain properties, with the help of which one can make the category of $\OqR$-modules a braided category. 
 {}
 The following generalizes \cite[Theorem 3.5]{CL} from  $n=2$ to all $n$:
 
 \def\letterx{  \raisebox{-9pt}{\incl{.9 cm}{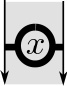}} } 
\def\lettery{  \raisebox{-9pt}{\incl{.9 cm}{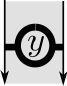}} }
\def\letterxy{  \raisebox{-9pt}{\incl{.9 cm}{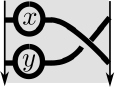}} }

\begin{theorem*}[Thm. \ref{t.cobraid}] Under the above identification $\S_n(\cB)\simeq \OqR$ 
 the $R$-form $\rho$ has the following geometric description
 $$\rho \left (\letterx  \ot \lettery \right)  = \epsilon \left(  \letterxy \right),$$  
for any $x,y\in \OqR.$    
 \end{theorem*}
 
Above we identified $\fB$ with $[-1,1] \times (-1,1)$ by stretching its top and bottom ideal points into horizontal intervals.

\subsection{Relation to Reshetikhin-Turaev theory}
\label{ss.intro-RT}
\def\fCn{{\mathfrak C_n}}
 \def\sgnl{{\sgn_l}}
 \def\sgnr{{\sgn_r}}
 \def\boeta{{\boldsymbol{\eta}}}
 \def\boi{{\bm{i}}}
\def\boj{{\bm{j}}}
\def\bobeta{{\bm{\beta}}}
\def\boal{{\bm{\alpha}}}

Sections \ref{s.quantumg}-\ref{s.RT} mostly summarize the background in quantum groups and in Reshetikhin-Turaev theory necessary for this paper. Section \ref{s.quantumg} however also introduces a novel modification of the Reshetikhin-Turaev functor, utilized throughout the paper. Let us briefly describe its connection to stated skein modules.

When the bigon $\fB$ is identified with $[-1,1] \times (-1,1)$, the webs on $\fB$ can be thought as oriented framed tangles with coupons given by $n$-valent sinks and sources. 

The {\bf sign} $\sgn(e)$ of an endpoint  $e\in \pal$ is positive if the direction of $\al$ goes from left to right at $e$, and negative otherwise. Let $\sgnl(\al)$ (respectively: $\sgnr(\al)$) be the sequence of signs of left (respectively: right) endpoints of $\al$ appearing from the bottom to the top.  

Let $V=\Q(v)^n$ be the defining representation of $\Uq$ with its standard basis $\{e_1,\dots, e_n\}$. Let 
$\{e_1^*,\dots, e_n^*\}$ be the dual basis of $V^*$ and let
$\{f_1,\dots, f_n\}$ be a basis of $V^*$ given by
$f_i=(-1)^{i-1}q^{i-\frac{n+1}{2n}}e^{\bar i}$ for $i=1,\dots, n.$ For any sign sequence $\boeta=(\eta_1,\dots, \eta_k)$, let 
$$V^{\boeta}:= V^{\eta_1} \ot \dots \ot V^{\eta_k},$$
where $V^+= V \ \text{and}\ V^-= V^*$. The above bases $\{e_1,\dots, e_n\}$ and $\{f_1,\dots, f_n\}$
induce the {\bf tensor basis} of $V^{\boeta}$  indexed by the elements of $\{1,\dots, n\}^k,$ cf. Subsec. \ref{ss.Gamma}. 

Note that $V^{\boeta}$ is a $\Uq$-module for every sign sequence $\boeta.$
The key to our work is a modified version of the Reshetikhin-Turaev functor, introduced in Section \ref{s.RT}, which associates to each web $\alpha$ in $\fB$ a $\Uq$-module homomorphism $\RTp(\alpha): V^{\sgnl(\alpha)}\to V^{\sgnr(\alpha)}$. (It is a sign modification of the standard $\Uq$-Reshetikhin-Turaev functor, \cite{RT}, which requires that the half-edges incident to each $n$-valent vertex in $\alpha$ are linearly ordered.)

We show that the benefit of utilizing the basis $\{f_1,\dots, f_n\}$ of $V^*$ rather than the dual basis to $\{e_1,\dots, e_n\}$ is that it makes the modified R-T functor values 
$\RTp(\alpha)$ 
independent of the orientation of tangles.
 
The following result relates our skein algebra of $\fB$ to the Reshetikhin-Turaev theory:

\begin{proposition*}[Prop. \ref{p.counit=RT}]  Let $\boal$ be an $n$-web $\alpha$ on $\fB$ stated by $\boi=(i_1,\dots,i_k)$ on the left and $\boj=(j_1,\dots,j_l)$ on the right. Then $\ep(\bf\al)$ is equal to the $(\boi,\boj)$-entry of the matrix of the modified Reshetikhin-Turaev operator $\RTp(\al)$ of $\alpha$ in the above tensor bases.
\end{proposition*}

\subsection{Module and Co-module structures}
\def\tUL{{\widetilde \UL}}

Given a marking $\beta$ of a marked 3-manifold $\MN$, consider its closed disk neighborhood $D$ in $\pM$, disjoint from the other markings of $\MN$. By pushing the interior of $D$ inside $M$ we get a new disk $D'$ which is properly embedded in $M$. Splitting $\MN$ along $D'$, we get a new marked 3-manifold $(M', \cN')$ isomorphic to $\MN$, and another marked 3-manifold bounded by $D$ and $D'$. The latter,
after removing the common boundary of $D$ and $D'$, is isomorphic to the thickening of the 
bigon. 
Hence, this construction yields an $R$-linear splitting map
$$\Delta_\beta: \S_n\MN \to \S_n\MN\otimes \Oq.$$
which defines a right coaction of $\Oq$ on $\S_n\MN$. Such coactions at different markings commute.

The Hopf algebra $\cal O_q(SL(n);\Z[v^{\pm 1}])$ has a Hopf dual given by a completion $\tUqq$ of the Lusztig's integral version, $\tUL,$ of $\UL$, which is a Hopf algebra over $\Z[v^{\pm 1}],$ cf. Subsection \ref{ss.half-ribbon}and \cite[Sec. 1.3]{Lu-root1}. The duality between these two Hopf algebras turns any right $\Oq$-comodule $W$ to a left $\tUL$-module as follows: 
For $u\in \tUL$ and $ x\in W$,
$$u * x =  \sum x_{(1)} \la f_{(2)}, u\ra,  \quad \text{where}\  \Delta(x) =\sum x_{(1)}  \ot f_{(2)}$$
(in the Sweedler notation)
 is the $\Oq$-coaction map. We make this left  $\tUL$-action on $\S_n\MN$ explicit in Subsec. \ref{ss.coaction}.

The Hopf algebra $\tUL$ contains distinguished charmed element $g$ and the half-ribbon element $X\in \tUL$.
We prove that the action of these elements on $\S_n\MN$ at a marking $\beta$ is exactly the marking automorphism $g_\beta$ and the half-twist homomorphism $\htw_\beta$ of Subsection \ref{ss.i-sym}.

\subsection{Glueing over an ideal triangle} 
\label{ss.i.gluing}

\def\fT{{\mathfrak T}}
\def\glue{{\mathsf{glue}}}
\def\uast{{\underline{\ast}}}
\def\sF{\Sigma}
\def\SF{{\S_n(\sF)}}

The standard ideal triangle $\fT\subset \BR^2$ is the closed triangle with vertices $(-1,0),(1,0)$ and $(0,1)$
 with these vertices removed.  We will denote its sides by $e_1, e_2$, and $\p_b \fT$ as in Figure \ref{f.brtensor0copy}. Suppose $a_1, a_2$ are two distinct boundary edges of a (possibly disconnected) pb surface.
 Define 
$$\Sigma_{a_1\triangle a_2} = (\Sigma \sqcup \fT)/(e_1 =a_1, e_2 = a_2),$$ 
as in Figure \ref{f.brtensor0copy}. Define the $R$-linear homomorphism $\glue_{a_1,a_2}:\S_n(\Sigma) \to \S_n(\Sigma_{a_1\triangle a_2})$ by continuing the strands of any web $\al$ with endpoints on $a_1$  and $a_2$ until they reach $\p_b \fT $, as in Figure  \ref{f.brtensor0copy} (right).  

\begin{figure}[htpb]
    \includegraphics[height=2.8cm]{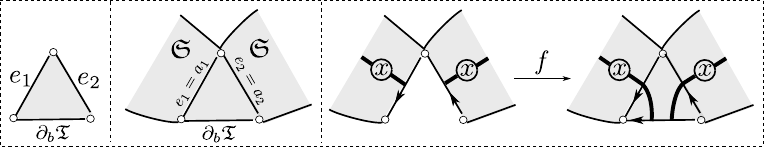} 
\caption{Left: The standard ideal triangle $\fT$. Middle: Glueing $\Sigma$ and $\fT$ by $a_1=e_1$ and $a_2=e_2$ to get $\Sigma_{a_1\triangle a_2}$. Right: tangle diagram $x\in \S_n(\Sigma)$ and its image $\glue_{a_1,a_2}(x)\in \S_n(\Sigma_{a_1\triangle a_2})$}
    \lbl{f.brtensor0copy}
    \end{figure}

\begin{proposition*}[Prop. \ref{p.tri-glue}, see \cite{CL} for $n=2$ and \cite{Hi} for $n=3$]
The map $$\glue_{a_1,a_2}:\S_n(\Sigma) \to \S_n(\Sigma_{a_1\triangle a_2})$$ is an $R$-linear isomorphism.
\end{proposition*}

We construct an explicit inverse map to $\glue_{a_1,a_2}$ in Subsec. \ref{ss.braidedtensor}. 

Although the bijective map $\glue_{a_1,a_2}$ is not an algebra isomorphism with respect to the standard skein algebra product on $\S_n(\Sigma_{a_1\triangle a_2})$, we show in Subsec. \ref{ss.braidedtensor} that it is one with respect to the {\em self braided tensor product} which we will define right now. 

There are two right $\Oq$-comodule algebra structures on $\S_n(\Sigma)$ given by
$$\Delta_i:= \Delta_{a_i}: \S_n(\Sigma) \to \S_n(\Sigma)\ot \Oq,\quad i=1,2,$$
which  commute. Define the $R$-linear map
$\bD:  \S_n(\Sigma)  \to \S_n(\Sigma) \ot \Oq$  by
$$\bD(x) = \sum x' \ot u_1 u_2,$$
in Sweedler's notation, where
$$(\Delta_1\otimes \Id_\Oq)\circ \Delta_2(x) = \sum x_{(1)}\ot u_{(2)} \ot u_{(3)}.$$
For $x,y\in \S_n(\Sigma)$ define a new product by
$$y\uast x = \sum  y_{(1)} x_{(1)} \rho(u_{(2)} \ot w_{(2)})$$
where
$$\Delta_2 (y) =\sum y_{(1)}  \ot u_{(2)} ,\quad  \Delta_1(x) =\sum x_{(1)}  \ot w_{(2)},$$
and $\rho$ is the $R$-form.

It is proved in \cite{CL} that $\bD$ and $\uast$ together give $\SF$ a right $\Oq$-comodule algebra structure for $n=2$. That proof extends to all $n$. 

Denote by $\uot \S_n(\Sigma)$ the $R$-module $\S_n(\Sigma)$ with this $\Oq$-comodule algebra structure. On the other hand  $\S_n(\Sigma_{a_1\triangle a_2})$ has a right $\Oq$-comodule algebra structure coming from the boundary edge $\partial_b \fT$. Here is a stronger version of the proposition above.

\begin{theorem*}[Thm. \ref{r.braidedtensor}]\label{r.braidedtensorcopy}
The map $\glue_{a_1,a_2} : \uot \S_n(\Sigma) \to \S_n({\Sigma}_{a_1 \triangle a_2})$ is an
isomorphism of right $\Oq$-comodule algebras.
\end{theorem*}

When $\Sigma=\Sigma_1 \sqcup\Sigma_2$ and $a_i \subset \Sigma_i$ for $i=1,2$ then each $\S_n(\Sigma_i)$ is a right $\Oq$-comodule algebra via the coaction coming from the edge $a_i$ and $\uot(\S_n(\Sigma))$ is the well-known {\em braided tensor product} $\S_n(\Sigma_1)$ and $\S_n(\Sigma_2)$ of the two $\Oq$-module algebras $\S_n(\Sigma_1)$ and $\S_n(\Sigma_2)$ of Majid, cf. \cite[Lemma 9.2.12]{Maj}. 
An analogous braided tensor product in the context of lattice gauge theory appears in \cite{AGS} (quantizing \cite{FR}) and, in the context of factorization homology, in
\cite[Cor. 6.11]{BBJ}, cf. Subsec. \ref{ss.intro-facthom}.


\subsection{On injectivity of splitting homomorphism}

A pb surface ${\Sigma}$ is {\bf essentially bordered} if every connected component of it has non-empty boundary.

\begin{proposition*}[Prop. \ref{r.inj}]
Suppose $\Sigma$ is an essentially bordered pb surface. Then for  any interior ideal arc $c$ of $\Sigma$, the splitting homomorphism $\Theta_c: \S_n(\Sigma) \to \S_n(\Cut_c\Sigma)$ is injective.
\end{proposition*}

\begin{conjecture*}[Conj. \ref{con.inj}] For any punctured bordered surface $\Sigma$ and any interior ideal arc $c$ the splitting homomorphism $\Theta_c$ is injective as well.
\end{conjecture*}

The conjecture is true when $n=2$ by \cite{CL} and for $n=3$ by Higgins \cite{Hi}. In both cases the proofs rely on explicit bases of $\S_n(\Sigma)$. Proposition \ref{r.inj} shows the conjecture is true if $\Sigma$ has non-trivial boundary. 
We will establish an alternative, weaker version of this conjecture for all pb surfaces in Subsec. \ref{ss.i-ker}.

%
\subsection{Skein algebras of surfaces with boundary} 
\label{ss.i-surf-boundary}

Let $\Sigma_{g,p}$ denote the surface of genus $g$ with $p-1$ punctures and $\p \Sigma_{g,p}=S^1.$
Let $\Sigma_{g,p}^*$ be $\Sigma_{g,p}$ with a boundary point removed. Hence, $\Sigma_{0,2}^*$ is a punctured monogon. 

Utilizing the results of Subsec. \ref{ss.i.gluing}, we show in Proposition \ref{p.tm} that $\S_n(\Sigma_{0,2}^*)\simeq \Oq$ as an $R$-module with the $\Oq$-comodule structure, $\Delta_{\p \S_n(\Sigma_{0,2}^*)},$ coinciding with the adjoint $\Oq$-coaction on $\Oq$, \cite[Example 1.6.14]{Maj}. Furthermore, we prove that the product on
$\S_n(\Sigma_{0,2}^*)$ coincides with the {\em braided} (or, {\em covantarised}) {\em product} of Majid, \cite[Example 9.4.10]{Maj}. (That result was shown for $n=2$ in \cite{CL}.)
Consequently, our theory provides simple geometric proofs of the associativity of the braided product on $\Oq$ and of $\Oq$ being an $\Oq$-comodule algebra. (The proofs of these facts are quite technical and involved in \cite{Maj}.) Furthermore, our theory generalizes these statements to the boundary $\Oq$-coaction on the skein algebra of any essentially bordered punctured surface.
We discuss a finite presentation of $\S_n(\Sigma_{0,2}^*)$ in Subsec. \ref{ss.transmut}.

Let $\Sigma$ be now any essentially bordered pb surface. 
A collection $A=\{a_1,\dots , a_r\}$ of disjoint closed oriented arcs properly embedded in $\Sigma$ is {\bf saturated} if 
\begin{itemize}
\item[(i)] each connected component of ${\Sigma} \setminus \bigcup_{i=1}^r a_i$ contains exactly one ideal point (interior or boundary) of ${\Sigma}$, and
\item[(ii)] $A$ is maximal with respect to the above  condition.
\end{itemize}

Let $U(a_1),\dots, U(a_n)$ be a collection of disjoint open tubular neighborhood of a saturated collection of arcs $a_1,\dots, a_n$, respectively. Each $U(a_i)$ is homeomorphic with $a_i\times (-1,1)$ (by an orientation preserving homeomorphism) and we require that $(\p a_i) \times (-1,1)\subset \p \Sigma.$
Let $U(A)= \bigcup_{i=1}^k U(a_i)$.

\def\tri{{\mathsf{tri}}}
\begin{theorem*}[Thm. \ref{t.Aisom}]

(1) We have $r= r(\Sigma):=\# \p \Sigma-\chi(\Sigma),$
where $\# \p \Sigma$ is the number of boundary components of $\Sigma$ and $\chi$ denotes the Euler characteristics.\\
(2) The embedding $U(A) \embed {\Sigma}$ induces an $R$-module isomorphism $f_A: \S_n(U(A)) \to \S_n(\Sigma)$. 
\end{theorem*}

Note that each $U(a_i) = a_i\times (-1,1)$ is naturally a directed bigon, with its sides $(\p a_i) \times (-1,1)$ oriented in the direction of $(-1,1)$ and that we have an $R$-linear isomorphism
$$\Oq^{\ot r} \xrightarrow{\Psi^{\ot r}} \S_n(U(A)) \xrightarrow{ f_A} \S_n(\Sigma).$$

Since $\Oq$ has a Kashiwara-Lusztig's canonical basis over $\Z[v^{\pm 1}]$, cf. \cite{Kash, Lu-bases}, we have 
 
\begin{corollary*} 
$\S_n(\Sigma)$ is a free $R$-module with a basis given by the image of tensor product of Kashiwara-Lusztig's canonical bases on $\Oq^{\ot r}$ under $f_A\circ \Psi^{\ot r}$.
\end{corollary*}

We apply the above method to show that $\S_n(\Sigma_{1,1}^*)\simeq \Oq^{\otimes 2}$ (as an $R$-module) and to describe the product on it in Subsec. \ref{ss.torus}. Furthermore, we explain a construction of finite presentations for $\S_n(\Sigma)$, for every essentially bordered $\Sigma$.

\subsection{Kernel of the splitting homomorphism.}
\label{ss.i-ker}
\def\bSF{\bar {\S}_n({\Sigma})}

Suppose ${\Sigma}$ is a connected pb surface with an ideal point $p$. Then a trivial ideal arc $c_p$ at $p$ cuts $\Sigma$ into a monogon and a new pb surface ${\Sigma}_p$ which has $c_p$ as its boundary edge, see Figure \ref{f.proj-copy}.

\begin{figure}[htpb]
    \includegraphics[height=2cm]{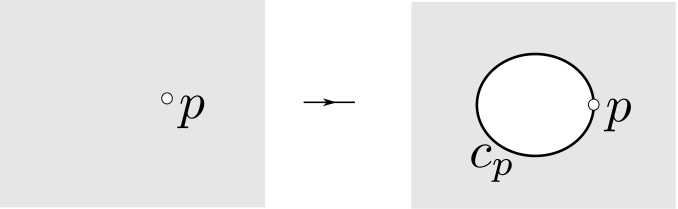}
    \caption{From ${\Sigma}$ to ${\Sigma}_p$. Here $p$ is an interior ideal point. The picture when $p$ is a boundary ideal point is similar.}
 \lbl{f.proj-copy}
    \end{figure}   
    
Theorem \ref{t.K} shows that the kernel of 
$$ \Theta_p: \SF \xrightarrow{ \Theta_{c_p}} \S_n({\Sigma}_p) \ot_R \S_n(\fM)  \xrightarrow \cong \S_n({\Sigma}_p)$$
does not depend on the choice of $p$. Let us denote it by $\K({\Sigma})$.
Then the quotient $\bSF:= \SF/\K(\Sigma)$ is called the {\bf projected stated skein algebra of ${\Sigma}$}. 
By Proposition \ref{r.inj}, $\K$ is trivial and $\bSF= \SF$ if $\pS\neq \emptyset$, 

\begin{corollary*}[Cor. \ref{c.inj}] The splitting homomorphism descends to an injective algebra homomorphism
$$
\bar \Theta_c: \bSF\to \bar \S_n(\Cut_c\, \Sigma)=  \S_n(\Cut_c\, \Sigma).
$$
\end{corollary*}

The following is an alternative characterization of projected skein algebras:
\begin{theorem*}
For any $\Sigma$, $p$ and $c_p$ as above,  $\bSF$ coincides with the subalgebra of $\S_n(\Sigma_p)$ coinvariant under the coaction $\Delta_{c_p}: \S_n(\Sigma_p)\to \S_n(\Sigma_p)\otimes \S_n(\fB)$ at $c_p$: 
$$\bSF=\{x\in \S_n(\Sigma_p): \Delta_{c_p}(x)=x\otimes 1\}.$$
\end{theorem*}

Let $\Sigma=\overline\Sigma -\cal P$, where $\cal P$ is a finite subset of compact surface $\overline\Sigma$, as in Subsec. \ref{ss.i-pbsurf}. 
Generalizing the setup of Subsec. \ref{ss.i-surf-boundary}, consider a collection $A$ of disjoint, oriented, arcs in $\Sigma$, each with endpoints in $\p \Sigma\cup \cal P,$ satisfying the conditions (i) and (ii) above.  We show in Subsec. \ref{ss.kernel}  that such $A$ defines an identification of $\bSF$ with $\Oq^{\ot r}$ and, hence, it determines a basis of $\bSF$.

%
\subsection{The Image of the Splitting homomorphism}
\def\fli{\mathrm{fl}}

Let $c$ be an interior oriented ideal arc of a pb surface $\Sigma$.
Denote the two copies of $c$ in $\Cut_c\,\Sigma$ by $a_1$ and $a_2$. We have the splitting $R$-algebra homomorphism
$$\Theta_c: \S_n(\Sigma)\to \S_n(\Cut_c\, \Sigma).$$
and $\S_n(\Cut_c\, \Sigma)$ is a $\Oq$-bi-comodule with the right and left coactions
\begin{align*} \Delta_{a_1}: \S_n(\Cut_c\, \Sigma)\to \S_n(\Cut_c\, \Sigma)\otimes \Oq \\
_{a_2}\Delta: \S_n(\Cut_c\, \Sigma)\to \Oq \ot  \S_n(\Cut_c\, \Sigma)
\end{align*} 
respectively, where $\Oq$ is identified with the skein algebra of the bigon directed by the orientation of $c$.
 {}
Recall that the Hochshild cohomology module is defined by
$$HH^0(\S_n(\Cut_c\, \Sigma))= \{ x \in \S_n(\Cut_c\, \Sigma) \mid  \Delta_{a_1}(x)= \fli \circ {}_{a_2}\Delta(x) \},$$
where $\fli$ is the transposition  
$$\fli : \Oq\otimes \S_n(\Sigma)\to \S_n(\Sigma)\otimes \Oq,\quad \fli(x\otimes y)=y\otimes x.$$

\bthm[\cite{CL,KQ} for $n=2$ and \cite{Hi} for $n=3$]
The image of $\Theta_c$ is equal to $HH^0(\S_n(\Cut_c\, \Sigma))$.
\ethm

We prove it by considering the projected version $\Theta_c: \bSF\to HH^0(\S_n(\Cut_c\, \Sigma))$, which has the same image as non-projected one. Furthermore, we construct an explicit inverse map 
$\nabla: HH^0(\S_n(\Cut_c\, \Sigma))\to \bSF.$

%
\subsection{Relation to Factorization Homology, skein categories and lattice field theory}
\label{ss.intro-facthom}
%

Factorization homology is an invariant of oriented $n$-dimensional manifolds introduced by Beilinson and Drinfeld \cite{BD} in the setting of conformal field theory and then in \cite{Lu, AF, AFT} in the topological context. For $n=2$, it associates with every surface $\Sigma$ and every balanced braided category 
$\cal A$, a certain category denoted by $\int_\Sigma \cal A.$ When the base ring is a {\em field}, using reconstruction theory  Ben-Zvi, Brochier, and Jordan \cite{BBJ} showed that  $\int_\Sigma \cal A$ is equivalent to the category of left modules over a certain algebra ${\cal A}_\Sigma$, which is isomorphic to the quantum moduli spaces of Alekseev-Grosse-Schomerus and Buffenoir-Roche \cite{AGS, AS, BR1,BR2, BFK}, and also to the internal algebras of the skein categories of Walker and Johnson-Freyd, \cite{Wa,JF,Co}.

Building upon the above theory of stated skein algebras, we prove that for surfaces $\Sigma$ with $\p \Sigma=S^1$, the algebra $A_{\Sigma}$, for the representation category $\cal A$ of $U_q(sl(n))$,
is isomorphic with $\S_n(\Sigma^*)$, where $\Sigma^*$ is $\Sigma$ with a boundary point removed. 
The $n=2$ case follows also from the results of \cite{Fa, Ko, LY1}; see further discussion in Subsec. \ref{ss.skeincat}.

Although our construction of the stated skein modules was motivated by its rich theory 
developed in this paper, the above result provides a further justification for that construction.

On the one hand, factorization homology of \cite{BBJ} is more general in that it can defined for all semi-simple Lie algebras $\mathfrak g$ and it can be viewed as quantizing the entire moduli stacks of representations, rather than just the character varieties.

On the other hand, our approach has its own advantages. First our theory is defined over any ground ring, a commutative domain, while the factorization homology approach defines the algebra ${\cal A}_\Sigma$ over a field. For example, our theory works over the cyclotomic ring (and of course cyclotomic field), an important quantization case. Over the ring $\BZ[q, q^{-1}]$ our theory highlights some integral results in quantum group theory, like relations to canonical basis.  Furthermore, we define the stated skein module not only for surfaces but also for all $3$-dimensional manifolds, and we worked out the theory for surfaces with multiple boundary components and multiple markings.

The stated skein algebra of a surface is defined explicitly, via generators (which are geometric objects and relations, making the theory elementary, while in factorization homology the algebra $A_\Sigma$  is defined up to an isomorphism only. 

The cutting homomorphism in our theory, though related to the excision in factorization homology, is different from the latter. Our cutting homomorphism and gluing over triangle operations make the study of the stated skein algebra easy by cutting surfaces into triangles.

An important application of our approach is it allows the first author and T. Yu \cite{LY3} to prove the existence of the quantum trace map which quantizes the classical Fock and Goncharov trace map \cite{FG1}  and generalizes the quantum $SL_2$-trace map of Bonahon and Wong \cite{BW} to all $SL_n$. The concrete geometric nature of the generators of the stated skein algebra allows in many cases to present a set of elements which generates a quantum space inside the stated skein algebra. This eventually leads to various versions of quantum traces.

\no{The existence of our stated skein algebras over $\Z[q^\frac1{2n}]$ allows to construct quantum trace  
homomorphisms of the stated skein algebras into quantum tori, over any ground ring. It was done for $n=2$ by Bonahon-Wang, \cite{BW}, and more generally in \cite{CL, LY2}. 
The construction of quantum trace was generalized to all $n$ in \cite{LY3}, see further details in Subsec. \ref{ss.facthom}.}

The above works relate our algebras to the theory of quantum cluster algebras, which provide alternative quantizations of character varieties. Further connections to quantum cluster algebras are through \cite{CS, Sh, JLSS}.

\no{In Subsections \ref{ss.skeincat}-\ref{ss.latticegauge-quantmoduli} we also establish isomorphisms of our skein algebras  
with the quantum moduli spaces of Alekseev-Grosse-Schomerus and Buffenoir-Roche,  \cite{AGS, AS, BR1,BR2, BFK}, and with the internal algebras of the skein categories of Walker and Johnson-Freyd, \cite{Wa,JF},
for surfaces $\Sigma$ with $\p \Sigma=S^1$ and $\mathfrak g=sl(n)$.
}



%
\subsection{Compatibility with stated Kauffman bracket skein modules}
\label{ss.i-sKB}

The stated Kauffman bracket skein algebras (of  surfaces) of the first author \cite{Le-triang} were generalized to stated skein modules of marked $3$-manifolds in \cite{BL} (cf. also \cite{LY1}). We are going to prove that these modules are isomorphic with our $SL(2)$-skein modules, $\S_2(M,\cN)$.  

To relate these modules to ours, let us replace the variable $q$ of \cite{Le-triang} with $q^{1/2}$ and denote the resulted stated Kauffman bracket skein module by $\cS(M,\cN)_{q^{1/2}}$.  Recall that it is built of non oriented $2$-webs without sinks nor sources, stated by sings $\pm$.

\begin{theorem*}[Thm. \ref{t.su2}] Suppose $(M,\cN)$ is a marked $3$-manifold. \\
(1) There is a unique $R$-linear isomorphism $\Lambda: \cS(M, \cN)_{q^{1/2}}\to \cal S_2(M,\cN)$ which maps framed links $\al$ to a stated $2$-webs by assigning arbitrary orientation to them, and changing the minus state to $1$ and the plus state to $2$.\\ 
(2) The splitting homomorphism of \cite{Le-triang, BL} coincides with ours through $\Lambda$. 
\end{theorem*}

\subsection{Compatibility with the SU(n)-skein modules}
\label{ss.i-SLn}
\def\UM{UM}

As mentioned already, a partial motivation for our definition of $\S_n\MN$ were the $SU(n)$-skein modules of $3$-manifolds introduced by the second author in \cite{Si}. These skein modules are built of  {\bf based $n$-webs} in $M$ which are defined as our $n$-webs in $(M,\emptyset)$, except that the half-edges incident to any of their $n$-valent vertices are linearly ordered. In particular, the based $n$-webs have no endpoints and $SU(n)$-skein modules have no boundary skein relations. In Subsection \ref{ss.SLn} we show that for any $3$-manifold $M$ and any $n$ our $\S_n(M, \emptyset)$ is isomorphic with the SU(n)-skein module of $M$. That isomorphism is straightforward for $n$ odd, but it requires a choice of a spin structure on $M$ for $n$ even.

%
\subsection{Compatibility with Higgins' SL3 skein algebras}
%

In his recent work \cite{Hi} Higgins introduced his version of stated $SL_3$-skein algebras, denoted by $\S_q^{SL_3}(\Sigma)$, of punctured bordered surfaces $\Sigma$.  His skein algebra is the $R$-module freely generated by $3$-webs stated by $-1,0,1$, subject to his system of skein relations. 
We show that in that our theory recovers Higgins' for $n=3$ in Subsection \ref{ss.higgins}.

%
\subsection{Relation to the Frohman-Sikora SU(3)-skein algebras}
%

C. Frohman and the second author considered in \cite{FS} the ``reduced $SU(3)$-skein algebra'' of marked surfaces built of unstated $3$-webs, subject to the $SU(3)$-skein relations of \cite{Ku}, extended by certain boundary skein relations, which depend on an invertible parameter $a\in R$.
We denote that algebra by $\cal S_{FS}(\Sigma)$ for the value $1$ of that parameter.

For an unstated $3$-web $\alpha$ in $\Sigma$, let $\eta_+(\alpha)$ (respectively: $\eta_-(\alpha)$) denote $\alpha$ stated with threes (respectively: ones) at all its ends. In Subsection \ref{ss.FS-SU3} we show that these operations extend to injective homomorphisms $\eta_+,\eta_-: \cal S_{FS}(\Sigma)\to \S_3(\Sigma)$. Furthermore, their images are direct summands of $\S_3(\Sigma)$.

%
\subsection{Thanks}

{\bf Acknowledgements.} The authors would like to thank F. Costantino, D. Douglas, V. Higgins, J. Korinman, and T. Yu for helpful discussions. 

The first author is partially supported by NSF grant DMS-1811114. Many results of the paper were proved and discussed during the visit of both authors to Toulouse in Spring 2017,  supported by the Centre International de Mathématiques et d'Informatique’s Excellence program during the Thematic Semester ``Invariants in low-dimensional geometry and topology”.

The authors presented the results of this paper in the form of talks or mini-courses at many conferences, including at Noncommutative Geometry Seminar, IMPAN, March 2017 and the special session at the AMS meeting at Hunter College in May 2017 and would like to thank the organizers for the opportunities to present their work.

%
\section{Quantum groups associated to $sl_n$}
\label{s.quantumg}
%

\subsection{Notations and conventions} 
\label{ss.notation}
We use the notations $\C, \BR, \BQ, \BZ, \BN$ for respectively the sets of complex numbers, reals, rationals numbers, integers, and non-negative integers. We emphasize that our $\BN$ contains 0. The number $n$ (in $sl_n$) is a fixed integer $\ge 2$.

The ground ring $R$ is a commutative ring with unit containing a distinguished invertible element $v$. The element $q= v^{2n}$ is the usual quantum parameter. 
The basic example is $R=\BZ[v^{\pm1}]$, the ring of Laurent polynomials in $v$ with integer coefficients. 
All fractional powers of $q$ in our papers are defined via the obvious integral powers of $v= q^{1/2n}$. 

For a non-negative integer $m$ we define the quantum integer $[m]$ and its factorials by
$$[m] = \frac{q^m - q^{-m}}{q - q^{-1}}, \quad [m]!= \prod_{i=1}^m [i], 
\quad [0]!=1.$$

We will often use the following scalars:
\begin{align}
t_0^{1/2}&= q^{\frac{n^2-1}{2n}}:= v^{n^2-1}, \quad 
t= (-1)^{n-1} t_0= (-1)^{n-1} q^{\frac{n^2-1}{n}}, \ t^{n/2} =  (-1)^{\frac{(n-1)n}2} q^{\frac{n^2-1}{2}} \label{e.t}\\
a &= (-v)^{n(n-1)/2} t^{-n/2}=q^{(1-n)(2n+1)/4}=  q^{\frac{1-n^2}{4}-\frac{n^2-n}{4}},\quad 
d_i = i- \frac{n+1}{2}   
\label{e.a} \\
c_i&= (-1)^{n-i} q^{\frac{n+1}{2}-i} q^{\frac{n^2-1}{2n}} =  (-1)^{n-i} q^{-d_i} t_0^{1/2}.  \label{e.ci}
\end{align}

Note that
\beq\label{e.prodc} 
\prod_{i=1}^n c_i  = t^{n/2}= (-1)^{\binom n 2 } q^\frac{n^2-1}{2} 
\ \text{and}\ c_i\cdot c_{\bar i}=t, \ \text{for}\ i=1,\dots, n,
\eeq
where $\bar i$ is the {\bf conjugate of $i$}, defined as $n+1-i.$

In fact, our entire theory works (up to a normalization) for any invertible $a,t, c_1,...,c_n$ satisfying Eq. \eqref{e.prodc}, cf. Subsection \ref{ss.uniqueness}. However, our particular choice of these constants makes our theory invariant under the orientation reversal of $3$-manifolds and the orientation reversal of webs in it, cf. Corollary \ref{c.orient-rev} and Theorem \ref{t.framing-rev}.

Let $S_n$ be the group of permutations of $\{1,\dots, n\}$. 
The {\bf length} of $\sigma\in S_n$ is the minimal number of factors in the decomposition of $\sigma$ into elementary transpositions $(i,i+1),$ $i=1,...,n-1.$
Alternatively, it is
\beq\label{e.length}
\ell(\sigma)= |\{ (i,j) | 1\le i<j\le n, \ \sigma(i) > \sigma(j)\}|.
\eeq
 The longest element $w_0\in S_n$ is the permutation $w_0(i) = \bi$.
 
We use the convention that
\beq
(-q)^{\ell(\sigma)} =0\  \text{if} \ \sigma: \{1,\dots, n\}\to \{1,\dots, n\} \ \text{is not a permutation}.
\label{e.ell}
\eeq
We also use Kronecker's delta notation and its sibling:
\begin{align*}
\delta_{i,j}  =\begin{cases} 1 \quad  & \text{ if } j=i \\
0  & \text{ if }  j \neq i
\end{cases}, \qquad 
\delta_{j>i}  =\begin{cases} 1 \quad  & \text{ if } j>i \\
0  & \text{ if }  j \le i.
\end{cases}
\end{align*}

 \def\hR{\hat \cR}
 \def\eva{\mathrm{ev}}
\def\coev{\mathrm{coev}}
\def\tev{\widetilde{\eva}_0}
\def\tcoev{\widetilde{\coev}_0}
\def\Heav{\mathrm{Heav}}
\def\UU{{\mathcal C_n  }}
\def\Cn{{\mathcal C_n  }}

\subsection{Quantized enveloping algebra $\Uqq$}
\label{ss.qea}
The quantized enveloping algebra $\Uqq$ is a Hopf algebra over the field $\BQ(v)$ with explicit presentation given in \cite[Section 6.1.2]{KS}. 

Let $V=\BQ(v)^n$ be the defining representation of $\Uq$. Its dual $V^*$ is the simple $\Uqq$-module with highest weight the $(n-1)$-st fundamental weight. Let $e_1, \dots, e_n$ be the standard basis of $V$ with $e_n$ being the highest weight vector, see  \cite[Section 8.4.1]{KS}. 
Let $e^1, \dots, e^n$ be the dual basis of $V^*$, defined by
$e^i(e_j) = \delta_{i,j}$. As $\Uqq$ is a Hopf algebra, the category of finite dimensional $\Uq$-modules is monoidal.
Let $\Cn$ be the full subcategory of $\Uq$-modules consisting of objects isomorphic to tensor products of copies of $V$ and $V^*$.

It is known that $\Hom_\Uqq(V^{\ot n},\Qv)$ has dimension 1 and is generated by the 
 $q$-antisymmetrizer $\cal A_-: V^{\ot n} \to \Qv$ defined by
 \beq\label{e.A-}
{\cal A}_- (e_{\sigma(1)} \otimes \dots\otimes  e_{\sigma(n)})= t^{n/2} a  (-q)^{\ell(\sigma)},\  \text{for any }  \sigma: \{1, \dots, n\} \to \{1, \dots, n\},
\eeq
where   $(-q)^{\ell(\sigma)}=0$ if $\sigma$ is not a permutation (according to Eq. \eqref{e.ell}), and  $a,t$ are given in Subsection~\ref{ss.notation}, cf. \cite{Gy}. Similarly, $\Hom_\Uqq(\Qv,V^{\ot n})$ has dimension 1 and  is generated by
$\cal A_+: \Qv \to V^{\ot n}$, given by
 \beq\label{e.A+}
{\cal A}_+(1)=  a  \sum_{\sigma\in S_n} (-q)^{l(\sigma)} 
e_{\sigma(1)}\otimes e_{\sigma(2)}\otimes ... \otimes e_{\sigma(n)},
\eeq

\subsection{Braiding and the Iwahori-Hecke algebra} 
\label{ss.Hecke}
The algebra $\Uqq$ has a topological completion which is a topological ribbon Hopf algebra, making the category of  finite dimensional $\Uqq$-modules a  ribbon category, see  \cite{CP,Tu1}.  The ribbon structure defines (through the universal $R$-matrix) for any two $\Uqq$-modules $V_1, V_2$ a braiding $\hR_{V_1, V_2}: V_1 \ot V_2 \to V_2 \ot V_1
$, which is an invertible $\Uqq$-morphism satisfying certain conditions discussed below. Let us record here the formula for $\hR_{V,V}$, which will be simply denoted by $\hR$. 
An operator $A: V\ot V \to V \ot V$ is given by its matrix entries $A^{ij}_{mk}\in \Qv$, with $i,j,m,k\in \{1,\dots, n\}$, which are defined such that
$$ A(e_m\ot e_k) = \sum_{i,j} A^{ij}_{mk} e_i \ot e_j.$$

The braiding $\hR: V \ot V \to V\ot V$ and the matrix $\cR$ are given by
\beq\label{e.R}
\hR^{ji}_{lk} =  \cR^{ij}_{lk} = q^{-\frac 1n} \left(    q^{ \delta_{i,j}} \delta_{j,k} \delta_{i,l} + (q-q^{-1})
    \delta_{j<i} \delta_{j,l} \delta_{i,k}\right),
\eeq
cf. \cite[Section 8.4.2(60) and Section 9.2]{KS},
where $\delta_{j<i}=1$ if $j<i$ and $\delta_{j<i}=0$ otherwise, as in Sec. \ref{ss.notation}. 

For an integer $k\ge 2$ and $i=1, \dots, k-1$ we define $\widehat R_i : V^{\ot k} \to V^{\ot k}$  by
\beq 
\widehat R_i = \id ^{\ot i-1} \ot \hR \ot \id^{\ot^{k-i-1}}.
\notag
\eeq
Then the operators $\widehat R_i$ satisfy the following relations
\begin{align*}
q^{\frac 1n} \widehat R_i - q^{-\frac 1n} \widehat R_i & = (q-q^{-1}) \id, \quad \text{for} \ i=1, \dots , k-1\\   
\widehat R_i \widehat R_j &= \widehat R_j \widehat R_i  \quad \text{for} \ 1 \le i < j-1  \le k-1\\   
\widehat R_i \widehat R_{i+1} \widehat R_i &= \widehat R_{i+1} \widehat R_i \widehat R_{i+1} \quad \text{for} \ i=1, \dots , k-2.  
\end{align*}
The last equation, known as the braid relation, is a consequence of $\hR$ being induced by the $R$-matrix of $\Uqq$.

\def\HH{  \mathcal H}

By \cite{Gy, Si}, 
\beq\label{e.H2}
\widehat R_i \circ \cal A_+ =  - q^{-\frac 1n -1} \cal A_+, \quad 
\cal A_- \circ \widehat R_i = - q^{-\frac 1n -1} \cal A_-.
\eeq

\subsection{The dual module} 
\label{ss.dual}
For $x\in V^*$ and $y\in V$ let $\la x, y \ra$ denote $x(y) \in \Qv$. 
There is an invertible element $g_0\in \Uqq$ called the charmed element, whose action on $V$ is given by
$$g_0(e_i) = q^{2i-n-1} e_i = q^{2 d_i} e_i.$$
The ribbon structure implies that the following $\Qv$-linear maps are $\Uqq$-morphisms:

\begin{align*}
&\eva: V^* \ot V \to \Qv,&  \quad & \eva(e^i \ot e_j)  = \la e^i, e_j \ra = \delta_{i,j}\\
& \coev: \Qv \to V \ot V^* ,& & \coev(1) = \sum_{i=1}^n e_i \ot e^i \\
&\tev: V \ot V^* \to \Qv,& & \tev(e_i \ot e^j) = q^{2i-n-1} \delta_{i,j}= \la e^j, g_0(e_i) \ra \\
&\tcoev: \Qv \to V^* \ot V ,&  &\tcoev(1) = \sum_{i=1}^n  q^{n+1-2i} e^i \ot e_i = \sum_{i=1}^n e^i \ot (g_0)^{-1} e_i.
\end{align*}

\subsection{The quantized coordinate algebra $\cO_q(SL(n))$}
\label{ss.oqsln} 

The {\bf algebra of quantum matrices} $\cO_q(M(n))$ is the associative $\Zv$-algebra generated by elements $u^i_j$, for  $i,j=1,2,\dots, n,$
subject to relations 
\beq
(\buu \ot \buu) \cR = \cR (\buu \ot \buu),  \label{e.OqM}
\eeq
where $\cR$ is the $R$-matrix given by Eq. \eqref{e.R}, and $\buu \ot \buu$ is the $n^2\times n^2$ matrix with entries $(\buu \ot \buu)^{ik}_{jl} = u^i_j u^k_l$ for $i,j,k,l\in \{1,\dots, n\}$. 
We call $\buu:= (u^i_j)$ as well as its images under  algebra homomorphisms {\bf quantum matrices}. Any square submatrix of $\buu$ is  a quantum matrix (of smaller size). The element 
$$ {\det}_q(\buu):=\sum_{\sigma\in S_n} (-q)^{\ell(\sigma)}u_1^{\sigma(1)}\cdots u_n^{\sigma(n)} = \sum_{\sigma\in S_n} (-q)^{\ell(\sigma)}u^1_{\sigma(1)}\cdots u^n_{\sigma(n)}$$
is central and called the {\bf  quantum determinant} of the quantum matrix $\buu$, cf.  \cite[9.2.2]{KS}.

The {\bf quantized coordinate algebra of $SL(n)$} is the quotient  
$$\Oq =\cO_q(M(n))/(\text{det}_q \buu-1).$$
It is a Hopf algebra with comultiplication, counit, and antipode given by
\begin{align}\label{e-Oq-ops}
\Delta(u^i_j) =\sum_k u^i_k \otimes u^k_j,\quad 
\epsilon(u^i_j) =\delta_{i,j},\quad 
S({u}^i_j)= (-q)^{i-j}M^j_i(\buu), 
\end{align}
where $M^j_i(\buu)$ is the quantum determinant of the minor of $\buu$ obtained by removing the $j$-th row and $i$-th column of the matrix $\buu$, see eg \cite[9.2.2]{KS} or \cite{Ta}. 

For technical convenience we have defined $\Oq$ over the  ring $\Zv$.
The dual $\Uqq^*$, consisting of all $\Qv$-linear maps $\Uqq \to \Qv$, has a $\Qv$-algebra structure, dual to the coalgebra structure of $\Uqq$, and is considered as a $\Zv$-algebra via  $\Zv \embed \Qv$.

\bpro[Hopf duality between $\Oq$ and $\Uq$, {\cite[Sec. 4]{Takeuchi0}}] \label{p.Hopfdual} 
There is a unique pairing $\la \cdot, \cdot \ra: \Oq \times \Uqq\to \Q(v),$ such that $\la u^i_j, x\ra= e^i(x(e_j))$ for $x\in \Uqq$ and $i,j=1,\dots, n.$ It is a Hopf pairing which induces an embedding of $\Zv$-algebras $\Oq \embed \Uqq^*$.
\epro

For the convenience of the reader, we recall that $\la \cdot, \cdot \ra$ being a Hopf pairing means that $\Oq$ and $\Uqq$ are in Hopf duality: for all $u,u'\in \Oq$ and $x,x'\in \Uqq$
\begin{align*}
\la uu', x \ra & = \sum \la u, x_{(1)} \ra \la u', x_{(2)} \ra, \quad \text{where} \ \Delta(x) = \sum x_{(1)} \ot x_{(2)}\\
\la u, x x'\ra & = \sum \la u_{(1)},x  \ra \la u_{(2)},x' \ra, \quad \text{where} \ \Delta(u) = \sum u_{(1)} \ot u_{(2)} \\
 \la 1, x \ra & = \ve(x), \quad 
 \la u,1 \ra  = \ve(u), \quad 
 \la S(u), x \ra  = \la u, S(x) \ra.
 \end{align*}

\def\Line{\Lambda}
\def\Lr{\Line_r}
\def\Ll{\Line_l}

\def\tal{{\tilde \al}}
\def\bal{{\bm{\al}}}

\def\rot{\mathsf{rot}}
\newcommand{\eqRT}{\overset{\mathsf{RT}}{\, =\joinrel= \, }}
\newcommand{\eqRTz}{\overset{\mathsf{RT}_0}{\, =\joinrel= \, }}
\def\Dij{\raisebox{-10pt}{\incl{1 cm}{Dij_no}}}
\def\Dijo{\raisebox{-10pt}{\incl{1 cm}{Dij_no1}}}
\def\Cij{\raisebox{-10pt}{\incl{1 cm}{Cij_no-orient}}}
\def\Cijo{\raisebox{-10pt}{\incl{1 cm}{Cij_no-orient1}}}
\def\dual{\raisebox{-25pt}{\incl{2.5 cm}{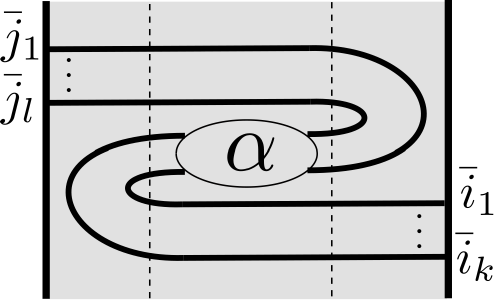}}}

\def\corss{\raisebox{-10pt}{\incl{1 cm}{cross}}}
\def\corssn{\raisebox{-10pt}{\incl{1 cm}{crossn}}}
\def\corsseq{\raisebox{-10pt}{\incl{1 cm}{crosseq}}}
\def\idd{\raisebox{-10pt}{\incl{1 cm}{idd}}}
\def\trivial{\raisebox{-10pt}{\incl{1 cm}{trivial}}}
\def\emptyt{\raisebox{-10pt}{\incl{1 cm}{emptyt}}}
\def\sinkn{\raisebox{-10pt}{\incl{1 cm}{sinkn}}}
\def\sinknnn{\raisebox{-10pt}{\incl{1 cm}{sinknnn}}}
\def\raa{\raisebox{-10pt}{\incl{1 cm}{raa}}}
\def\rab{\raisebox{-10pt}{\incl{1 cm}{rab}}}
\def\rac{\raisebox{-10pt}{\incl{1 cm}{rac}}}

\section{Review of the Reshetikhin-Turaev theory}
\label{s.RT}

\subsection{Based $n$-tangles and their operator invariants} 
\label{ss.based-t}
Reshetikhin-Turaev theory associates with every ribbon category an operator invariant of ribbon graphs, see \cite{Tu2}. Let us make this construction explicit for the category of left $\Uq$-modules and a special class of ribbon graphs, called based $n$-tangles, defined below.

The cube $Q:= [-1,1]\times (-1,1)^2$ in the 3-space, see Figure \ref{f.cube}(a), has boundary consisting of the right face $\{1\} \times (-1,1)^2$ and the left face $\{-1\} \times (-1,1)^2$. The intersection of $Q$ and the XY-plane is $S= [-1,1]\times (-1,1) \times \{0\}.$ It is depicted as the shaded square in Figure \ref{f.cube}(a),   
with its sides $\Lambda_l=\{-1\} \times \{0\} \times (-1,1)$ and $\Lambda_r=\{1\} \times \{0\} \times (-1,1)$.
We will say that vectors of the form $(0,0,z)$ with $z>0$ are in the {\bf Z-direction}.

\FIGc{cube}{(a) the cube $Q$, 
the square $S$ (shadowed) and its sides $\Lambda_r, \Lambda_l$. (b) an example of a diagram of a based $3$-tangle. The order at the source is counterclockwise, beginning with lowest branch. (c) the tensor product $\al \ot \beta$. (d) the composition $\beta\circ \al$. (Orientations of $\alpha,\beta$ are not shown.}{3.5cm}

For the sake of the definition below and for later use, we say that directed graph $\alpha$ is {\bf properly embedded} into a $3$-manifold $M$, if its set of $1$-valent vertices, $\partial \al$, coincides with $\al \cap \partial M$ and  $\al$ is transversal to $\partial M$. 

\bdf\label{d.based-t}
 A {\bf based  $n$-tangle} $\al$ is a disjoint union of finitely many oriented circles and directed graphs properly embedded into the cube $Q$, such that 
\benu
\item  The graphs of $\alpha$ have finitely many vertices only. Every vertex of $\alpha$ is either a sink or a source
and either $1$-valent or $n$-valent. We denote  
the set of $1$-valent vertices, called {\bf endpoints} of $\al$, by $\pal$. 
\item Each edge of the graph is a smooth embedding of the closed interval $[0,1]$ into $Q$.  
\item $\alpha$ is equipped with a {\bf framing} which is a continuous non-vanishing vector field transversal to $\alpha$. In particular, the framing at a vertex is  transversal to all incident edges.
\item The set of half-edges at every $n$-valent vertex is linearly ordered. 
\item The endpoints of $\alpha$ lie in $\Line_l \cup \Line_r$, and the framing at these endpoints has a $Z$-direction.
\eenu
\edf
We  consider based  $n$-tangles up to {\bf isotopies} which are continuous deformations of $n$-webs in their class.

\brem
Our notion of an $n$-tangle is broader than that of a traditional tangle, because it allows for $n$-valent vertices. It is a version of the notion of an $n$-web of \cite{Si}. 
However, we use a different name for it here to distinguish it from $n$-webs which we will introduce in Subsec. \ref{ss.marked} and which are unbased and have a different framing setup near their boundaries.
\erem

Suppose $\al$ is a based $n$-tangle. The {\bf sign} $\sgn(e)$ of an endpoint  $e\in \pal$ is positive if the direction of $\al$ goes from left to right at $e$, and negative otherwise. Let $\sgnl(\al)$ (respectively $\sgnl(\al)$) be the sequence of signs of endpoints of $\al$ on $\Lambda_l$ (respectively $\Lambda_r$) appearing from the bottom to the top.  

Let $\fCn^b$ be the $\Zv$-linear monoidal category whose objects $\boeta$ are finite sequences of signs $\pm$, and the set of morphisms $\Hom_{\fCn^b}(\boeta, \bomu)$ is the  $\Zv$-module freely spanned by isotopy classes of based $n$-tangles $\al$ such that $\sgnr(\al)=\boeta$ and $\sgnl(\al)=\bomu$. Here the tensor product of two sequences $\boeta$ and $\bomu$ is the concatenation of $\boeta$ followed by $\bomu$. The empty sequence is the unit. The tensor product $\al \ot \beta$ of two based $n$-tangles is the result of stacking $\beta$ above $\alpha$, as in  Figure \ref{f.cube}(c). If $\sgnl(\al) = \sgnr(\beta)$ then the composition $\beta\circ \alpha$  is obtained by placing $\beta$ to the left of $\al$ (after an isotopy to match endpoints), as in Figure \ref{f.cube}(d).

Let $\Cn$ be the category of 
left $\Uq$-modules isomorphic to tensor products of finite numbers of modules $V$ and $V^*$. The morphisms 
of $\Cn$ are $\Uq$-module homomorphisms. It is a ribbon category. Since based tangles can be viewed as tangles with coupons, in the sense of Reshetikhin-Turaev, their theory (\cite{RT}) defines a monoidal functor $ \RT: \fCn^b \to \Cn$ constructed as follows. For a sequence $\boeta =(\eta_1, \dots, \eta_k)$ of signs $\pm$ let
$$ \RT(\boeta) = V^{\boeta}:= V^{\eta_1} \ot \dots \ot V^{\eta_k},$$
where $V^+= V \ \text{and}\ V^-= V^*$. 

We will define values of $\RT$ for based $n$-tangles through their diagrams. 
For that purpose, we will identify the XY-plane in $\R^3$ with the pages of this paper and we will point the $z$-axis towards the reader, as in Figure \ref{f.cube}(a). 
Any based $n$-tangle $\alpha$ can be represented by its diagram obtained by isotoping $\alpha$ first so that its framing is in the Z-direction everywhere and then by putting it in a general position with respect to the projection 
$p: \BR^3 \to \BR^2$ onto the XY-plane, cf. Figure \ref{f.cube}(b). We further assume that the projection of $\al$ near an $n$-valent vertex consists of $n$ lines directed from left to right, and that their linear order is 
counterclockwise beginning from the lowest line to the highest for the source, and from the highest to the lowest line, for the sink. 

Such a projection of $\al$ onto the square $S=[-1,1]\times (-1,1)\times \{0\}$ (shaded in Figure \ref{f.cube}(a)), together with the over/under information at every crossing is called a {\bf  diagram of $\al$.} 

In Eqs. \eqref{e.id}--\eqref{e.RTT} we list {\bf elementary} based  $n$-tangles $\alpha$ and the corresponding operators $\RT(\alpha).$ The associated operators $\eva, \tev, \coev, \tcoev$ were defined in Subsection \ref{ss.dual}, while $\cal A_-, \cal A_+$ and $\hR$ were given by Eqs. \eqref{e.A-}--\eqref{e.R}  respectively.
Since  every based $n$-tangle can be built of them through tensor products and compositions,
these operators totally determine $\RT$.

\def\idi{\raisebox{-8pt}{\incl{.8 cm}{idi}}}
\def\ida{\raisebox{-8pt}{\incl{.8 cm}{id}}}

\beq\label{e.id}
\begin{array}{cccc}
\edgewall[ ]{}{>}{} & \edgewall[ ]{}{<}{} &  \capwall[w]{}{>}{}{} &   \wallcup[w]{}{>}{}{}   \\
 \text{\small $\id: V \to V$} & \text{\small $\id: V^* \to V^*$} & \text{\small $\eva: V^*\otimes V\to R$}&
  \text{\small $\coev: R\to V\otimes V^*$}  
\end{array}
\eeq
\beq\label{e.cap2}
\begin{array}{cccc}
  \capwall[w]{}{<}{}{} &   \wallcup[w]{}{<}{}{}  \\
 \text{\small $\tev: V\otimes V^*\to R$} &
\text{\small $\tcoev: R\to V^*\otimes V$}
\end{array}
\eeq
\beq\label{e.RTT}
\begin{array}{cccc}
\cross{}{}{p}{>}{>}{}{}{}{} & \cross{}{}{n}{>}{>}{}{}{}{} & \vertexnearwall[ ]{t}{>} & \vertexwall[w]{}{b}{>}{}{}{}\\
\text{\small $\hat R: V\otimes V\to V\otimes V$} & \text{\small $\hat R^{-1}: V\otimes V\to V\otimes V$} &
\text{\small $\cal A_+: R\to V^{\otimes n}$} & \text{\small $\cal A_-: V^{\otimes n}\to R$}\\ 
\end{array}
\eeq

\def\cyclica{\raisebox{-10pt}{\incl{1.1 cm}{cyclica}}}
\def\raab{\raisebox{-10pt}{\incl{1.1 cm}{raab}}}

Based $n$-tangles and the Reshetikhin-Turaev functor on them were considered in \cite{Si}, albeit with a different normalization of $\cA_\pm$. 

\def\cC{{\mathcal C}}

\subsection{Kernel of functor $\RT$} 
Recall that a monoidal ideal in a monoidal category $\cal C$ is a subset $I\subset \Hom(\mathcal C)$ such that for $x\in I$ and $y\in \Hom(\cC)$ we have $x\ot y, y \ot x \in I$ and  $x \circ y, y\circ x\in I$, whenever such compositions can be defined.

The following are well-known elements of $\ker \RT$ (see \cite{Si}):

\begin{align} 
q^{\frac{1}{n }} \cross{}{}{p}{>}{>}{}{}{}{}
 -  q^{-\frac{1}{n }} \cross{}{}{n}{>}{>}{}{}{}{}
 &  \eqRTz   (q-q^{-1})\, \walltwowall{}{}{>}{>}{}{}{}{}  \label{e.ma10}\\
\kink[w] &\eqRTz  t_0\,  \edgewall[ ]{}{>}{}{}
\label{e.ma20}\\
\circlediag[w]{>} &\eqRTz   [n] \, 
\emptyd[w] \label{e.ma30} \\
\sinksourcethree[w]{>}
&\eqRTz  (-q)^{\binom n2} \sum_{\sigma \in S_n}   (-q ^{\frac 1n -1} )^{\ell(\sigma)} \, 
  \coupon[w]{$\sigma_+$}{>}{}{}{}{},  \label{e.ma40}
\end{align}

where $t_0$ is given by Eq. \eqref{e.t}, $\sigma_+$ is the minimum crossing positive braid representing a permutation $\sigma$ and $\ell(\sigma)$ is the length of $\sigma$ defined in Subsection \ref{ss.notation}. Here $x\eqRTz y$ means $x-y\in \ker(\RT)$.
 
\bcon \label{c.1} The kernel $\ker \RT$ is the monoidal ideal $I_0$ generated by elements given in Eq. \eqref{e.ma10}--\eqref{e.ma40}. 
\econ

\def\cyclic{\raisebox{-10pt}{\incl{1.1 cm}{cyclic}}}

For example, the arguments in \cite{Si} indeed imply that the following identities are consequences of  \eqref{e.ma10}--\eqref{e.ma40}:
\begin{align}
\label{e.cyclic0}
 \vertexwallcyclic[w]{>} \, &\eqRTz  (-1)^{n-1}\, 
 \vertexwall[w]{}{t}{>}{}{}{} 
 \, , \qquad 
  {}
  \scalebox{-1}[1]{\vertexwallcyclic[w]{<}}\,  \eqRTz   (-1)^{n-1}\,  
  \vertexnearwall[ ]{t}{>}
  , \\
\vertexfourcross[w]{>} &  \eqRTz -q^{-(1+\frac 1n)} 
\vertexfour[w]{>}
\label{e.vertextwist0}\\
 [n-2]!\cross{}{}{p}{>}{>}{}{}{}{}  &\eqRTz [n-2]!\, q^\frac{n-1}{n} 
  \walltwowall{}{}{>}{>}{}{}{}{} - (-1)^{\binom n2} q^{-\frac{1}{n}}
\twoonetwowall[w]{}{>}{>}{\hspace*{-.05in}$n$-$2$}{}{}, \label{e.crossing0}
\end{align}
where the tangle on the right has $n-2$ parallel edges in the middle.

Conjecture \ref{c.1} is analogous to Morrison's conjecture \cite[Sec. 5.5]{Mor} which was proved in \cite{CKM}. Note that our formalism of $n$-webs leads to a simpler set of kernel generators than that in \cite{Mor}.
After this manuscript was posted on arxiv, A. Poudel announced a proof of the above conjecture, \cite{Po}.

\def\cyclic{\raisebox{-10pt}{\incl{1.1 cm}{cyclic}}}

\def\raab{\raisebox{-10pt}{\incl{1.1 cm}{raab}}}
\def\raa{\raisebox{-10pt}{\incl{1 cm}{raa}}}
\def\cyclica{\raisebox{-10pt}{\incl{1.1 cm}{cyclica}}}

\subsection{(Unbased) $n$-tangles} 
\label{ss.unbased-t}
Now we will consider a modified ribbon structure on the category of  $\Uq$-modules, giving rise to a new Reshetikhin-Turaev functor, denoted by $\RTp$, with simpler skein properties. We will explain in 
 Subsection \ref{ss.half-ribbon} how this new ribbon structure comes from the theory of quantized enveloping algebras. For now we simply declare $\RTp$ to coincide on all elementary based $n$-tangles with $\RT$, except that we multiply the values of $\tev$ and $\tcoev$ of \eqref{e.cap2} by $(-1)^{n-1}$. Thus, if 
 $D$ is a diagram of a based $n$-tangle $\al$ and $\#\downarrow(D)$ is the number of its downward critical points, i.e. points where the tangent is parallel to the vertical $y$-axis and pointing downward, then
\be \RTp(D) = (-1)^{(n-1)\#\downarrow(D)} \RT(\al).  \label{e.RTp}
\ee
It is an easy exercise to show that $\RTp(D)$ is an isotopy invariant of based $n$-tangles. 
 
From Eq. \eqref{e.cyclic0} it follows that $\RTp$ is invariant under cyclic changes of an order at a vertex: 
\beq\label{e.cyclic1}
 \vertexwallcyclic[w]{>} \, \eqRT  
 \vertexwall[w]{}{t}{>}{}{}{} 
 \, , \qquad 
  {}
  \scalebox{-1}[1]{\vertexwallcyclic[w]{<}}\,  \eqRT   
  \vertexnearwall[ ]{t}{>},
\eeq
 where $x \eqRT y$ means $\RTp(x)= \RTp(y)$.
  
An (unbased) {\bf $n$-tangle} is defined exactly as a based $n$-tangle except that half-edges incident to its every $n$-valent vertex are required to be cyclically ordered only.
(Such cyclic orderings of edges around each vertex are called a ribbon structure.) Eq. \eqref{e.cyclic1} shows that $\RTp$ is an invariant of $n$-tangles. In a diagram of an $n$-tangle the cyclic order at every $n$-valent vertex is the counterclockwise order. 
 
Let $\fCn$ be a monoidal category obtained from $\fC_n^b$ by replacing based $n$-tangles with $n$-tangles.
Then $\RTp$ is a $\Zv$-linear monoidal function from $\fCn$ to $\Cn$.

By Eqs. \eqref{e.ma10}--\eqref{e.ma40}, we have
\begin{align} 
q^{\frac{1}{n }} \cross{}{}{p}{>}{>}{}{}{}{} -  q^{-\frac{1}{n }} \cross{}{}{n}{>}{>}{}{}{}{}  &  \eqRT   (q-q^{-1})\,  \walltwowall{}{}{>}{>}{}{}{}{}  \label{e.ma1}\\
\kink[w] &\eqRT t\,  \edgewall[ ]{}{>}{}{}   \label{e.ma2}\\
\circlediag[w]{>}  &\eqRT  (-1)^{n-1} [n] \, \emptyd[w] \label{e.ma3} \\
\sinksourcethree[w]{>} &\eqRT  (-q)^{\binom n2} \sum_{\sigma \in S_n}   (-q ^{\frac 1n -1} )^{\ell(\sigma)} \, \coupon[w]{$\sigma_+$}{>}{}{}{}, \label{e.ma4}
\end{align}

The following two are  consequences of Eqs. \eqref{e.ma1}--\eqref{e.ma4}:
\begin{align}
\vertexfourcross[w]{>} 
&  \eqRT -q^{-(1+\frac 1n)} 
\vertexfour[w]{>},\label{e.vertextwist1}\\
  [n-2]!\cross{}{}{p}{>}{>}{}{}{}{} &\eqRT [n-2]!q^\frac{n-1}{n} \walltwowall{}{}{>}{>}{}{}{}{} 
   - (-1)^{\binom n2} q^{-\frac{1}{n}}
\twoonetwowall[w]{}{>}{>}{\hspace*{-.05in}$n$-$2$}{}{},
\label{e.crossing1}
\end{align}

\def\comm{\raisebox{-13pt}{\incl{1.1 cm}{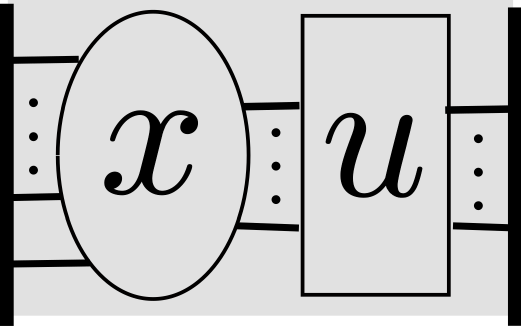}}}
\def\commb{\raisebox{-13pt}{\incl{1.1 cm}{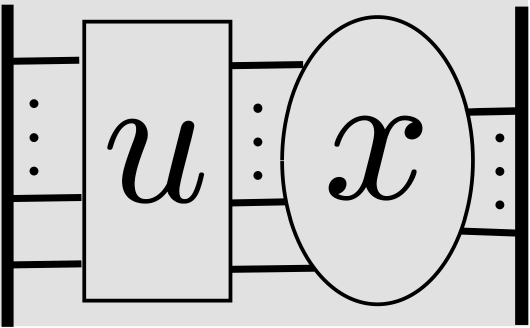}}}

\subsection{Linear functionals on $\Uqq $ from $n$-tangles} 
\label{ss.Gamma}

For $x\in \Hom_\fCn(\boeta, \bomu)$ and an element $u\in \Uq$  define 
$\RTp(x)\circ  u: V^\boeta \to V^\bomu$ as the composition of the action of $u$ on $V^\boeta$ with $\RTp(x)$.
Similarly one defines  $u \circ \RTp(\al): V^\boeta \to V^\bomu$ where $u$ acts on $V^\bomu$ now. 
We will use ovals for depicting morphisms of $\fCn$ and rectangles for elements of $ \Uq$ throughout the paper.
Since $\RTp(x)$ is a $\Uq$-morphism, we have
\beq \RTp(x)\circ  u = u \circ \RTp(x), \quad \text{or, pictorially,} \quad \comm \eqRT \commb,
\label{e.com1}
\eeq

A vector space equipped with a basis is called {\bf based}. The tensor product of based vector spaces $V_i$ with bases $B_i$ for $i=1,\dots, k$, has the natural tensor product basis $\prod_{i=1}^k B_i$. A linear operator $A: V_1 \to V_2$ between based vector spaces with bases 
$\{ e_i^{(1)}\} $ and $\{ e_j^{(2)}\} $ defines a matrix with elements $ A^j_i$, also denoted by $ \la j | A |i\ra$, such that
$$A(e^{(1)}_i) = \sum_j A^j _i  e^{(2)}_j, \quad \la j | A | i\ra : = A^j_i.$$

We consider the $\Uq$-module $V$  as a based $\Qv$-vector space with basis $\{ e_1,\dots, e_n\}$,
defined in Subsec. \ref{ss.qea}. The dual $V^*$ will be considered as a based vector space with the basis
$\{ f^1, \dots , f^n\}$ where 
\beq f^i = c_{\bi } e^\bi,\ \text{for}\ i=1,\dots, n, \ \text{where $c_i$ are given by \eqref{e.ci}}. \label{e.fi}
\eeq
This basis allows for a simplification of our theory and, in particular, makes the $\RTp$ functor independent of web orientations, cf. Subsection \ref{ss.dualop}. Now for any sign sequence $\boeta$, the vector space $V^\boeta$ is based with the {\bf tensor basis}, 
indexed by $\{1, \dots, n \}^{|\boeta|}$, induced by the above bases of $V$ and $V^*$.

\def\bx{\mathbf{x}}

A {\bf right state} (respectively  a {\bf left state}) of a morphism $x\in \Hom_\fCn(\boeta, \bomu)$, for some $\boeta$ and $\bomu$, is an assignment of $i_1, \dots, i_{|\boeta|}\in \{1,\dots, n\}$ (respectively $j_1, \dots, j_{|\bomu|}\in \{1,\dots, n\}$) to the right (respectively left) endpoints of $x$, listing them from the bottom to the top. A right (respectively: a left) stated morphism $x$ as above will be denoted by $\bx =(x,\boi)$ (respectively: $\bx =(\boj,x)$), where $\boi=\{i_1, \dots, i_{|\boeta|}\}\in \{1,\dots, n \}^{|\boeta|}$ and $\boj=\{j_1, \dots, j_{|\bomu|}\}\in \{1,\dots, n \}^{|\bomu|}$. Similarly, a {\bf stated morphism} is a triple $\bx =(\boj,x,\boi)$, as above. For such $\bx$ define a $\Qv$-linear function $\Gamma(\bx): \Uq \to \Qv$ whose value at $u\in \Uq$ is
$$\la \Gamma(\bx), u \ra := \la \boj |\RTp(x) \circ  u |\boi \ra= \la \boj |u \circ \RTp(x) |\boi \ra.$$
\def\ep{{\epsilon}}
For example, for $x= \emptyset$ 
\begin{align}
\Gamma(\emptyset)= \ep : \Uq \to \Qv,\quad \text{is the counit}.
\label{e.G0}
\end{align}

Let $\cT$ be the set of isotopy classes of stated $n$-tangles. The module $\Zv \cT$ over $\Zv$ 
 freely spanned by  $\cT$ is an algebra with the product $\boal \bobeta= \boal \ot \bobeta$ obtained by placing $\boal$ below $\bobeta$ and concatenating the states.
We extend $\Gamma$ linearly onto a $\Zv$-linear map 
$$\Gamma: \Zv \cT \to \Uq^*,$$ 
\bpro \label{p.Gamma}
The map $\Gamma$ is a $\Zv$-algebra homomorphism.
\epro
\begin{proof} Since the $\Uq$-action on tensor product of $\Uq$-modules is given by the coproduct on $\Uq$, for any $n$-tangles $\al$ and $\beta$ and any $u\in \Uq$ we have
$$ (\al \ot \beta)\circ u = \sum  (\al \circ u') \ot (\beta \circ u''), \quad \text{where} \ \Delta(u)= \sum u' \ot u''.$$
By assigning states, we get
$$ \la \boal \ot \bobeta, u \ra = \sum \la \boal, u' \ra \la \bobeta , u'' \ra,
$$
which means that $\Gamma$ maps the product $\boal \ot \bobeta$ to the dual of the coproduct in $\Uq$, which is the product on $\Uq^*$. Hence, $\Gamma$ is an algebra homomorphism.
\end{proof}

\subsection{Dual operator, orientation reversal invariance}
\label{ss.dualop}
For a stated morphism $\bx = (\boi, x, \boj)$ let $\la \bx \ra$ be the $(\boi, \boj)$-element of the matrix of $ \RTp(x)$,
$$\la \bx \ra = \la \boi \mid \RTp(x) \mid \boj \ra \in \Qv.$$
 {}

For a stated $n$-tangle, $\boal = (\boi, \al, \boj)$, let $\cev{{\boal}}=(\boi, \cev{\al}, \boj)$, where $\cev{\al}$ is obtained from $\al$ by reversing the orientation of all its edges and and of all circle components.

 \def\ro{\mathsf{ro}}
Furthermore, let 
$$\boal^*= (c_{\boi})^{-1}\,  c_\boj\, (\boj^*, \ro(\al),  \boi^*),$$ 
where $\ro( \al)$ is obtained by $180^\circ$ rotation of $\alpha$,
$$c_\boi = \prod_{m=1}^k c_{i_m} \ \ \text{and}\ \  (i_1, \dots, i_k)^*=(\bar i_k, \dots, \bar i_1),\ \ \text{where}\ \ \bar i=n+1-i.$$
Pictorially,
\def\dualx{\raisebox{-16pt}{\incl{1.6 cm}{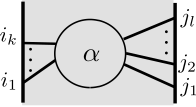}}}
\def\dualy{\raisebox{-16pt}{\incl{1.6 cm}{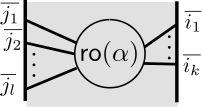}}}
$$\boal^*=\left( \dualx \right)^*= (c_{\boi})^{-1} c_\boj \, \dualy.$$

The above operations extend linearly onto $\Z[v^{\pm 1}]\cal T.$ 

It should be noted that for an $n$-tangle $\al$ the operators  $\RTp(\al), \RTp(\cev\al) $, and $\RTp(\ro(\al))$ have different domains and different target spaces. The following statement shows an important benefit the basis $\{ f^1,\dots, f^n\}$ of $V^*$:

\bpro \label{p.reverse} For any stated morphism $\bx = (\boi, x, \boj)$ one has
$\la \bx \ra = \la \cev{\bx} \ra=\la \bx^* \ra$.  
\epro

\begin{proof} Checking $\la \cev{\alpha}\ra=\la \alpha\ra$ for cups and caps is quite easy:
\begin{align}
 \capwall[w]{}{>}{$i$}{$j$} \eqRT 
  \capwall[w]{}{<}{$i$}{$j$} \eqRT \delta_{i, \bar j}  c_i , \qquad 
   \wallcup[w]{}{>}{$i$}{$j$} \eqRT 
    \wallcup[w]{}{<}{$i$}{$j$} \eqRT \delta_{i, \bar j} (c_{\bi})^{-1}.
\label{e.capcup}
\end{align}
 
(1) The identity $\la \alpha^*\ra=\la \alpha\ra$ follows from
$$ \la \boj^*  |\RTp(\ro(\al)) | \boi^* \ra \eqRT \dual =  \left( \prod_{m=1}^k c_{i_m}\right) \left( \prod_{m=1}^l c_{j_m}\right)^{-1} \la \boi |\RTp(\al) | \boj \ra,$$
where the first identity is by an isotopy. For the second identity, we decompose the tangle above along the dashed lines and use the values of cups and caps in  Eq. \eqref{e.capcup}.

(2) To prove $\la \cev{\alpha}\ra=\la \alpha\ra$ one needs to check it for 
 the elementary $n$-tangles. For cups and caps it have been done in Eq. \eqref{e.capcup}. 
The statement for the positive crossing (Eq. \eqref{e.RTT}) follows from part (1) and the fact that the $R$-matrix formula \eqref{e.R} is preserved under the involution $i\leftrightarrow \bar l, j\leftrightarrow \bar k.$
The statement for the sink and the source follows by a straightforward computation. 
\end{proof}

%
\subsection{
Annihilators}
\label{ss.annihilators}
%

In this and the next two subsections we will analyze the kernel of $\Gamma: \Zv \cT \to \Uq^*$.

An {\bf internal annihilator} is a $\fCn$-morphism $x$  such that $\RTp(x)=0$. 
From the definition we have:

\bpro\label{r.int_ann}
(a) Internal annihilators form a monoidal ideal in $\fC_n.$\\
(b) For any stating $(\boi,x, \boj)$ of an internal annihilator $x$, we have $\Gamma((\boi,x, \boj))=0$.
\epro
  
Eqs. \eqref{e.ma1}--\eqref{e.ma4} are internal annihilators, called the {\bf basic internal annihilators}. 

\def\tHom{\widetilde {\Hom} }
\def\HHr{\mathcal H_r}
\def\HHl{\mathcal H_l}

 For a sequence $\boeta$ of signs $\{\pm \}$ let $\HHr(\boeta)$ be the $\Zv$-module freely spanned by right stated $n$-tangles $(\al, \boi)$ such that the sequence of left ends of $\al$ has type $\boeta.$
 Note that  $\HHr(\boeta)$ contains 
$\Hom_\fCn(\emptyset, \boeta)$, and we extend the map $\RTp$ to 
 $$\RTp: \HHr(\boeta) \to \Hom_\Qv(\Qv, V^{\boeta}) \ \text{by } \RTp((\al, \boi))(1) = \RTp(\al)(v_\boi),$$
where $v_\boi$ is the basis vector of $V_r(\al)$ with index $\boi$. For example, for $\eta=(+,-,+,-)$, 
$\alpha=
\tangleeg{$-$}{$+$}{$-$}{$+$}{$-$}{$+$}$ and $\boi=(i_1,i_2)$, we have
$$\RTp((\alpha,\boi))=e_{i_1}\otimes  \tcoev(1) \otimes f^{i_2}\in V\otimes V^*\otimes V\otimes V^*.$$

Note that for $x\in \HHr(\boeta) \setminus \Hom(\emptyset,\boeta)$, the operator $\RTp(x)$ might  not be a $\Uq$-morphism.

For $y \in \HHr(\boeta')$ define $x \ot y \in \HHr(\boeta\ot \boeta')$ by placing $y$ atop  $x$ and by concatenating the right states. For $z\in \Hom_{\fCn}( \boeta, \bomu)$ we can define the composition $z \circ x\in \HHr(\bomu)$. Clearly 
$$\RTp(x \ot y)= \RTp(x) \ot \RTp(y)\quad \text{and}\quad \RTp(x \circ y)= \RTp(x) \circ \RTp(y).$$

A {\bf right annihilator} is an element $x\in \HHr(\boeta)$, for  certain $\boeta$, such that $\RTp(x)=0$. By the 
above equation 
and 
by Eq. \eqref{e.com1}, we have

\bpro\label{r.right_ann}
Let $x\in \HHr(\boeta)$, $y\in \HHr(\boeta')$ be right annihilators and let $z\in \Hom(\boeta,\bomu)$.

(a) For any left state $\boi$ of $x$  we have $\Gamma((\boi,x))=0$.

(b) All  $x \ot y,y\ot x$, and $z \circ x$ are right annihilators.
\epro

We do not draw the left boundary vertical edge in pictures of right annihilators, to indicate that $\eqRT$ holds for their composition with any web on the left (for which such composition is possible).

\bpro \label{r.ra}
We have the following identities for the values of the function $\RTp$:
 \begin{align}
   \vertexnearwall{b}{white}
  & \eqRT 
   a \sum_{\sigma \in S_n} (-q)^{\ell(\sigma)}\,  \nedgewall{}{b}{white}{$\sigma(n)$}{$\sigma(2)$}{$\sigma(1)$}
  \label{e.ra1}\\
\capwall{}{white}{$i$}{$j$}  & \eqRT  \delta_{\bar j,i }\,  c_i\ \twowallpic{}{right}{}{}{},  
 \label{e.ra2}\\
\capnearwall{white}
&\eqRT  \sum_{i=1}^n  (c_{\bar i})^{-1}\, \twowall{}{white}{black}{$i$}{$\overline{i}$}
\label{e.ra3}\\
\crosswall{}{p}{white}{white}{$i$}{$j$} & \eqRT q^{-\frac{1}{n}}\left(\delta_{{j<i} }(q-q^{-1})\twowall{}{white}{white}{$i$}{$j$}+ q^{\delta_{i,j}}\twowall{}{white}{white}{$j$}{$i$}\right),\label{e.ra4} 
\end{align}
for any $i,j=1,\dots, n,$ 
where the white circle represents the orientation, left-to-right or right-to-left, and is the same for all white circles in one identity. The black circle stands for the opposite orientation of the white one. The values $\delta_{j<i}, \delta_{i,j}, a, c_1,\dots, c_n$ were defined in Subsec. \ref{ss.notation}.
\epro

By subtracting the left side from the right side in the equations above we obtain right annihilators which we call {\bf basic}.

\begin{proof}[Proof of Proposition \ref{r.ra}:] By applying the total orientation reversion and Proposition \ref{p.reverse} if necessary, we can assume all the white circles indicate the left-to-right orientation. 
Identity \eqref{e.ra1}  is the defining equation \eqref{e.A+} of the operator $\cal A_+$, while Identity \eqref{e.ra4} is the defining equation \eqref{e.R} of the braiding. Identities \eqref{e.ra2} and \eqref{e.ra3} are consequences of Eq.~\eqref{e.capcup}.
\end{proof}

\def\coeff{\mathrm{coeff}}

 We now  define left annihilators.  For a sign sequence $\boeta$, let $\HHl(\boeta)$ be the $\Zv$-module freely spanned by left-stated $n$-tangles $(\boi,\al)$ such that $V_r(\al) = V^\boeta$. Then  $\HHl(\boeta)$ contains $\Hom_\fCn(\boeta, \emptyset)$, and we extend $\RTp$ to 
$$\RTp: \HHl(\boeta) \to \Hom_\Qv( V^{\boeta}, \Qv) \ \text{by }\RTp((\boi,\alpha))(z) = \coeff_\boi(\RTp(\al)(z)),$$
for $z\in V^\eta,$ where $\coeff_\boi: V^\bomu \to \Qv$ is coefficient of the basis vector 
of $V^\bomu$ indexed by $\boi$.

A {\bf left annihilator} is an element $x\in \HHl(\boeta)$, for a certain $\boeta$, such that $\RTp(x)=0$. 
In analogy to Proposition \ref{r.right_ann}, we have:

\bpro \label{r.left_ann}
Let  $x\in  \HHl(\boeta), y\in \HHl(\boeta')$ be left annihilators and let $z\in \Hom(\bomu, \boeta)$.

(a) For any right state $\boi$ of $x$  we have $\Gamma((x, \boi))=0$.

(b) All  $x \ot y,y\ot x$, and $x\circ z$ are left annihilators.
\epro

\def\boal{{\boldsymbol{\alpha}}}
\def\boi{{\mathbf{i}}}

\def\hd{\mathsf{hd}}

\def\htwzzz{\raisebox{-13pt}{\incl{1.2 cm}{htw000}}}
\def\htwzzo{\raisebox{-13pt}{\incl{1.2 cm}{htw001}}}
\subsection{Turning right annihilators to left ones}  \label{ss.ltr}

For an integer $k \ge 2$ let $H_k$ be the positive half-twist of $k$ strands and let $\bar H_k$ be its inverse:

$$ H_k = \phalftwist{b}{>}{}{}{}, 
\quad \bar H_k = \nhalftwist{b}{>}{}{}{}.
$$
(Note that $H_k$ does not twist the framing, which always points towards the reader. This applies to all half-twists considered in this paper.)
By abuse of notation,  for any $n$-tangle $\al$ with $k$ left endpoints, denote by $\bar H_k \circ \al$ the composition of $\alpha$ 
with a version of $\bar H_k$ in which the orientation of some of its components was reversed so that it is composable with~$\alpha.$ 

Let $\hd: \HHr(\boeta) \to \HHl(\boeta)$ be a $\Zv$-linear map given by
$$ \hd((\al,\boi))= (\boi, \bar H \circ \ro (\al)),$$
where $\ro (\al)$ denotes the $180^o$ rotation of $\al$ about the center, as before.

In Subsection \ref{ss.dualop}, we showed that our basis $\{f^i\}$ of $V^*$ makes the matrices $\RTp(x)$ invariant under the total reversal of orientation of $x$. It also makes the following statement hold:

\bpro \label{p.left-right0}
 If $x$ is a basic right annihilator then $\hd(x)$ is a left annihilator.
\epro 

\noindent{\it Proof:} The statement for each basic right annihilator can be checked by a direct computation.
The calculation for Relations \eqref{e.ra2}-\eqref{e.ra3} follows from the framing change, Eq. \eqref{e.ma2}. 
The calculation for Relation \eqref{e.ra4} utilizes Eq. \eqref{e.ma1}. Finally, let us show that 
the image of \eqref{e.ra1} under $\hd$, pictured below, is a left annihilator. (We assume its orientation to the right, since the statement for the opposite orientation follows from Proposition \ref{p.reverse}.)
$$ 
\vertexwall[w]{}{t}{>}{}{}{} \  \eqRT a \sum_{\sigma\in S_n} (-q)^{ \ell(\sigma) }\, 
\coupon[w]{$\bar{H}$}{>}{$\sigma(n)$}{$\sigma(2)$}{$\sigma(1)$}.$$
Note that since $H_k$ is invertible in $End(V^{\otimes n})$ we can consider the above equality composed with $H$ on the right instead. Then, by Eq. \eqref{e.H2}, it reduces to
$$(-q^{-\frac 1n-1})^\frac{n(n-1)}{2}\cal A_-=a\sum_{\sigma\in S_n} (-q)^{\ell(\sigma)} e^{\sigma(1)}\otimes \dots \otimes e^{\sigma(n)},$$
which is indeed equivalent to the definition of $\cal A_-,$ Eq. \eqref{e.A-}, since $(-q^{-\frac 1n -1})^\frac{n(n-1)}{2}=t^{-n/2}.$
\qed

A stronger statement, valid for all right annihilators,  will be shown in Subsection \ref{ss.half-ribbon} using a more conceptual approach. 

We call the left annihilators of Proposition \ref{p.left-right0} {\bf basic}.

\subsection{Kernel of $\Gamma$}

\def\cI{\mathcal I}
\begin{theorem}[Proof in Subsection \ref{s.kernel}] \label{t.kernelG}
The kernel of $\Gamma: \Zv\cT \to \Uq^*$ is generated by internal basic annihilators, right basic annihilators, and left basic annihilators.
\end{theorem}
This means that if we begin with the internal, left, and right  basic annihilators, and use procedures described in Propositions \ref{r.int_ann}, \ref{r.right_ann}, and \ref{r.left_ann}, we obtain the entire kernel $\ker \Gamma$.
This theorem 
is analogous to Conjecture \ref{c.1}, except that it describes the kernel of a stated version of RT.

\def\tUq{\widetilde{\Uq}}
\def\tot{\tilde{\ot}}
\def\vt{\vartheta}
\def\rev{\mathrm{rev}}

%
\subsection{Half-ribbon Hopf algebra} 
\label{ss.half-ribbon}

In this subsection we explain conceptually some technical aspects of this paper. In particular, we interpret our sign modification of the Reshetikhin-Turaev functor, Eq. \eqref{e.RTp}, through a modification of  the ribbon element in a completion of $\Uq.$ Coincidently, that modification leads to a ``half-ribbon" element which makes it possible to prove a stronger version of Proposition \ref{p.left-right0}, below. Although the content of this subsection is more technical, it is not 
necessary for the paper.

\bpro\label{p.left-right} If $x$ is a right annihilator then $\hd(x)$ is a left annihilator.
\epro 

We precede the proof with a few preliminaries: the quantized enveloping algebra $\Uqq$ is a topological ribbon Hopf algebra, meaning it has an $R$-matrix $\cR$  in a completion of $\Uqq\ot \Uqq$ and a ribbon element $\vt_0$ in a completion of $\Uq$, satisfying certain conditions.  It is proved in \cite{ST} that there is a completion  $\tUqq$ of $\Uq$ having the same $R$-matrix  but a new ribbon element $\vt$, which acts on $V$ and $V^*$ the same way as $(-1)^{n-1}\vt_0$. (That completion was studied in different context also in \cite{Lu-Zform}.) Consequently the new charmed element  $g$ is $(-1)^{n-1}g_0$ on $V$ and $V^*$, explaining the sign correction in \eqref{e.RTp}, which we used as the definition of $\RTp$. 

\def\fl{\mathrm{fl}}
Additionally, there is an invertible element $X \in \tUq$, called the half-twist, such that $X^2 = \vt$ and the universal $R$-matrix $\cR$ satisfies
\begin{align}
\cR &= (X^{-1}\ot X^{-1}) \Delta(X) =  ((\fl\circ \Delta)(X))  (X^{-1}\ot X^{-1}), \ \text{where}\ \fl(x\ot y) := y\ot x.  \label{e.X}
\end{align}

For a sign sequence $\boeta = (\eta_1, \dots, \eta_k)$ let $\cev{\boeta} = (\eta_k,\dots, \eta_1)$ and let $\rev_k: V^\boeta \to V^{\cev{\boeta}}$ be the $R$-linear operator given by $\rev_k(x_1 \ot \dots \ot x_k) = (x_k \ot \dots \ot x_1)$.
From Eq. \eqref{e.X} by induction on $k$ we get the following, see  \cite[Proposition 4.18]{ST}: If $\bar H$ stands for $\bar H_k$ with an orientation on the strands on the right given by $\boeta$ then we have an equality of transformations $V^{\boeta}\to  V^{\cev{\boeta}}:$ 
\beq
 \RTp(\bar H)=     \rev_k \circ  X ^{\ot k} \circ \Delta^{[k]} (X^{-1}).
 \label{e.XXX}
\eeq 
Here $\Delta^{[k]}$ is defined inductively by $\Delta^{[2]}= \Delta$ and $\Delta^{[k+1]} = (\Delta\ot \id^{\ot k }) \circ \Delta^{[k]}$.  

An additional special feature of the basis $\{f^i\}$ of $V^*$ (besides those discussed already) is: 

\bpro\label{p.Xaction}
The actions of $X$ on $V$ and $V^*$ are given by the same matrix
\be 
X^i_j = \delta_{i,\bar j}\  c_{i}.  \label{e.Xij}
\ee
\epro

As this is not proved in \cite{ST}, we give a proof of this result in the appendix. Similarly, the actions of the charmed element $g$ on both $V$ and $V^*$ are given by the same diagonal matrix with entries 
\beq 
g^i_j = \delta_{i,j} g_i, \quad \ \text{where}\ g_i= (-1)^{n-1} q^{2i-n-1} = (-1)^{n-1} q^{2d_i}.
\label{e.gi}
\eeq

\def\htwz{\raisebox{-28pt}{\incl{2cm}{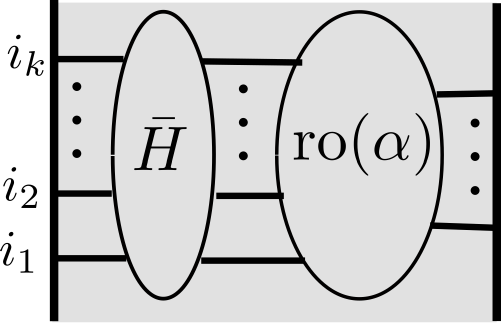}}}
\def\htwo{\raisebox{-28pt}{\incl{2cm}{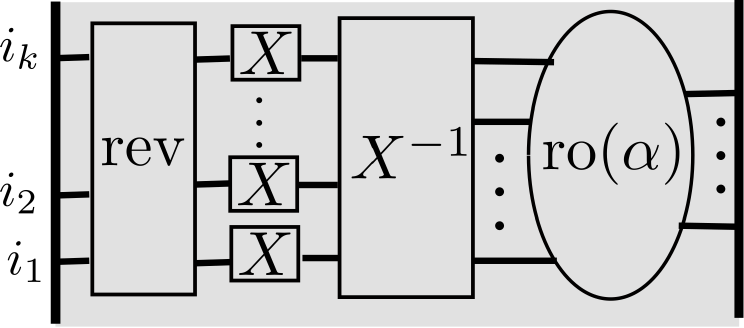}}}
\def\htwt{\raisebox{-28pt}{\incl{2cm}{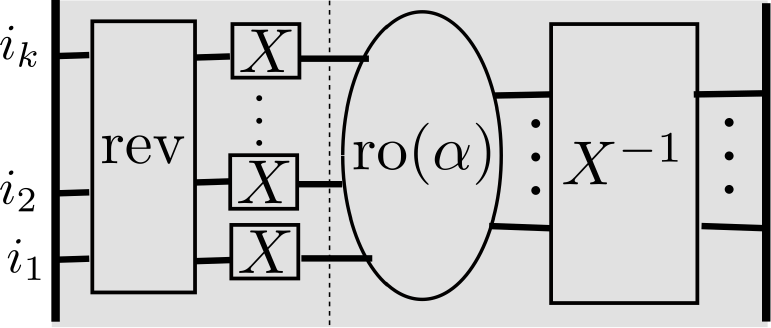}}}
\def\htwth{\raisebox{-28pt}{\incl{2cm}{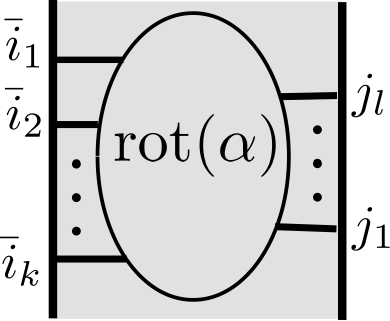}}}
\def\htwfi{\raisebox{-38pt}{\incl{3cm}{htw5}}}
\def\htws{\raisebox{-38pt}{\incl{3cm}{htw6}}}

\begin{proof}[Proof of Proposition \ref{p.left-right}] Let $x\in \HHr(\boeta)$ for some $\boeta$ be a right annihilator.
We need to show that $\RTp(\hd(x))=0$. 
Since $X$ is invertible, this is equivalent to  $\RTp (\hd(x)) \circ X = 0$, which, in turn, is equivalent to:  $$ \la \RTp (\hd(x)) \circ X, \boj\ra =0 \qquad \text{for all } \ \boj=(j_1,\dots, j_l)\in \{1,\dots, n \}^l, l:=|\boeta|.$$

Suppose $(\al, \boi)$  is a  right stated $n$-tangle.
By Eq. \eqref{e.XXX} and then by Eq. \eqref{e.com1}, we have
$$ \hd((\al,\boi))    =  \htwz =  \htwo =\htwt.$$
By composing with $X$ on the right, then decomposing along the dashed line and using the values of 
$X^i_j$ from Eq. \eqref{e.Xij}, we get
$$\la \hd((\al, \boi)) \circ X, \boj \ra  \eqRT  c_{\boi}  \,  \htwth=  c_\boj  \,  \la \boj^* | (\al, \boi)\ra $$ 
where  $c_\boi=\prod_{k} c_{i_k}$. The second identity follows from Proposition \ref{p.reverse}. By linearity,
$$\la \hd(x ) \circ X , \boj \ra  \eqRT  c_\boj   \la \boj^* | x\ra =0,$$
for every $\boj\in \{1,\dots, n\}^{|\boeta|}.$
This proves $\hd(x) \eqRT  0$.
\end{proof}
\brem The use of half-ribbon element came up in a discussion of the first author with F. Costantino and J. Korinman. A full fledged theory of stated skein algebra based on half-ribbon category will be developed in an upcoming work \cite{CKL}.
\erem

\section{Stated SL($n$)-skein modules}
\label{s.stated}
%

\def\pM{\partial M}
\def\vk{\varkappa}
\def\drQ{\partial_r(Q)}
\def\pal{\partial \al}
\def\bpp{\beta^\perp}
\subsection{Marked $3$-manifolds and $n$-webs}
\label{ss.marked} 

A {\bf marked \pmb{$3$}-manifold} is a pair $(M, \cN),$ where $M$ is a smooth oriented $3$-manifold with (possibly empty) boundary $\p M$ and $\cN \subset \p M$ consists of open intervals, 
called {\bf markings}. The topological closure of each marking is required to be the closed interval $[0,1]$, disjoint from the closure of other markings.

Roughly speaking an $n$-web in $\MN$ is like $n$-tangle, except that the framing at boundary points is different. Here is the precise definition, where the first four requirements are the same as those in the definition of an $n$-tangle.

\bdf\label{d.n-web}
An  {\bf $n$-web} $\alpha$ in $\MN$ is a disjoint union of finite number of oriented circles and a finite directed graph properly embedded into $M$ such that
\benu[nosep]
\item  Every vertex of $\alpha$ is either a sink or a source and either $1$-valent or $n$-valent. 
We denote the set of 1-valent vertices, called {\bf endpoints} of $\al$, by $\pal$. 
\item Each edge of the graph  is a smooth embedding of the closed interval $[0,1]$ into $M$.
\item $\alpha$ is equipped with a {\bf framing} which is a continuous non-vanishing vector field transversal to $\alpha$. In particular, the framing at a vertex is  transversal to all incident edges.
\item[(4')] The set of half-edges at every $n$-valent vertex is cyclically ordered. 
\item[(5')] The endpoints of $\alpha$ lie in $\cN$ and the framing at these endpoints is a tangent vector of $\cN$, pointing in the direction of the orientation of $\cN$. We call such tangent vector {\bf positive}.
\eenu
\edf

Webs are considered up to continuous isotopy within their space.

Note that the only difference between the unbased $n$-tangles of Subsection \ref{ss.unbased-t} and $n$-webs in the cube $Q$ marked with $\Lambda_l, \Lambda_r$ is the framing at their endpoints. The difference explains why half-twists appear in our theory. 

The  {\bf height order} on $\p \al$  is the  partial order in which  two points $x,y\in \p \al$ are comparable if and only if they belong to the same marking, and $x > y$, or {\bf $x$ is higher than $y$}, if going along the positive direction of the marking we encounter $y$ first. We say $x$ and $y$ are {\bf consecutive} if there is no $z\in \p \al$ such that $ x > z > y$ or $ y > z > x$.

To depict a local part of an $n$-web $\al$ we consider the intersection of $\al$ with the cube  $Q=[-1,1]\times (-1,1)^2$ embedded into $M$, presented in Figure \ref{f.cubep}(a). The cube $Q$ can be either in the interior of $M$ or its right side, $\{1\}\times (-1,1)^2$, lies in $\pM$. 
\bite[leftmargin=*]
\item  If $Q$ is in the interior of $M$ then we assume that $\al\cap Q$ is an $n$-tangle, and depict $\al\cap Q$ by its $n$-tangle diagrams on the shaded square, as in Subsection \ref{ss.unbased-t}. In particular, for all drawn diagrams the framing is perpendicular to the page and pointing to the reader, and the cyclic order of half-edges at each $n$-valent vertex is counterclockwise.

 \FIGc{cubep}{
 (a) Cube $Q$ with a marking $\beta$ and a perpendicular line $\beta^\perp$ on the right face. The shaded square $S$ is in the $XY$-plane. (b) An example of a web $\alpha$ with three strands (depicted by different colors) in $\al\cap Q$. (c) The projection of $\al\cap Q$ onto $S$. (d) A diagram of $\al \cap Q$, with the height order indicated by numeric labels: $i>j$ means $i$ is higher than $j$. (e) Another diagram of $\al\cap Q$ obtained by using a different height preserving deformation. (f) This is the same diagram of (e), with the height order indicated by the direction of the boundary line.}{3cm} 

\item In the second case, we assume that $Q\cap \pM$ is equal to the right face of $Q$, and $Q\cap \cN$ is a subinterval of a marking $\beta$ depicted pointing in the direction of the $z$-axis, as in Figure \ref{f.cubep}(a). In the $Q$ coordinates, $\beta \cap \pM = \{1\}\times  \{0\} \times (-1,1)$.  Let $\beta^\perp=\Lambda_r$, the right side of the shaded square. 
As in the previous case, we assume that the framing points to the reader.
(Note the difference between webs in $M$ and tangles in $Q$:  the boundary points of $\al$ in $Q$ are in $\beta$, while the right endpoints tangles in $Q$ are in $\Lambda_r=\bpp$.)
By an isotopy,  we can bring $\al\cap Q$ to a general position with respect to the projection $p: \BR^3 \to \BR^2$, except that all the points in $\p \al \cap \beta$ project to the same point, see Figure~\ref{f.cubep}(c).   To resolve this issue, first we define a {\bf  height-preserving deformation of $Q$} as a continuous family of diffeomorphisms  $\phi_t: Q \to Q, t\in [0,1],$ supported in  a small neighborhood of the right face of $Q$ and preserving the $z$-coordinate, i.e. the height above the page. We use such a height-preserving deformation to bring $\al$ to $\al'$ whose endpoints on the right face of $Q$ have distinct projections (through $p$).  The image $p(\al'\cap Q)$ together with the usual over- and undercrossing data and with the linear order of its boundary points on $\bpp$ (induced from the height order) is a {\bf diagram of $\al\cap Q$}. For example, Figure \ref{f.cubep}(d) shows a diagram of Figure \ref{f.cubep}(b). Note however that a different height-preserving deformation can give rise to different diagram, see for example Figure \ref{f.cubep}(e). Note that $\al'$ is not an $n$-web because its endpoints are not in $\cN$ in general.
\eite

Although the height order of web ends in $\alpha\cap Q$ can be always indicated by integers as in Figure \ref{f.cubep}(d)-(e), in this paper we will always use the following convention: when presenting $\alpha\cap Q\subset M$ diagrammatically, we will choose a direction of $\bpp$ (indicating it by an arrow down or up) and arrange for the height order of web ends in $\alpha\cap Q$ to increase monotonically (without gaps) in the indicated direction.
For example, Figure \ref{f.cubep}(f) indicates the web in part (e). 
Note that the height order of endpoints of the $\al$ outside the drawn part can be arbitrary.

\def\sign{\mathrm{sign}}

%
\subsection{Skein relations for $n$-webs}
\label{ss.skein_rel} 

A {\bf state} of an $n$-web $\al$ is a map $s: \pal \to \{1,2,\dots, n\}$. The value $s(x)$, for $x\in \pal$, is called the {\bf state of $x$}.
A web with a state $s$ is called {\bf stated}.

We will consider stated $n$-webs up to isotopy (in the space of all stated $n$-webs) and denote the set of their isotopy classes by $\cal W_n(M,\cN)$.

Recall that the ground ring $R$ is commutative and it comes with a distinguished invertible $v=q^{1/2n}\in R$.
The {\bf stated $SL(n)$-skein module of $(M,\cN)$}, denoted by $\S_n(M,\cN)$, is the quotient of the free $R$-module $R \cal W_n(M,\cN)$ by the submodule $SkRel_n(M,\cN)$ generated by the following {\bf internal relations} (\ref{e.pm})-(\ref{e.crossp-wall}), which are the basic internal annihilators, and {\bf boundary relations} \eqref{e.vertexnearwall}-\eqref{e.crossp-wall}, which comes from the basic right annihilators:
\beq\label{e.pm}
q^{\frac 1n} \cross{n}{n}{p}{>}{>}{}{}{}{} 
- q^{-\frac 1n}\cross{n}{n}{n}{>}{>}{}{}{}{} 
= (q-q^{-1})\walltwowall{n}{n}{>}{>}{}{}{}{} 
\eeq

\beq\label{e.twist}
\kink
= t\horizontaledge{>},\quad \text{where}\  t=(-1)^{n-1} q^{n-\frac 1n} 
\eeq
\beq\label{e.unknot}
\circlediag{<} 
= (-1)^{n-1} [n]_q\ \emptyd,\ \text{where}\ [n]_q=\frac{q^n-q^{-n}}{q-q^{-1}},
\eeq
\beq\label{e.sinksource}
\sinksourcethree{>}=(-q)^{\binom n2}\cdot \sum_{\sigma\in S_n}
(-q^{\frac{1-n}n})^{\ell(\sigma)} \coupon{$\sigma_+$}{>}{}{}{}
\eeq
where the ellipse enclosing $\sigma_+$  is the minimum crossing positive braid representing a permutation $\sigma\in S_n$ and $\ell(\sigma)$ is the length of $\sigma\in S_n$, as before.

The remaining relations in $SkRel_n(M,\cN)$ take place near markings where we use the convention in Subsection \ref{ss.marked} about height order. Thus, the bold boundary line of a shaded rectangle is a part of  $\bpp$, orthogonal to a marking $\beta$, and if it has a direction, then the endpoints on that part are consecutive in the height order, given by the direction.  The height order outside the drawn part of $\bpp$ can be arbitrary. Here are the boundary relations:
\def\ori{{\mathrm{or}}}
  \begin{align}
   \vertexnearwall{b}{white}
  & = 
   a \sum_{\sigma \in S_n} (-q)^{\ell(\sigma)}\,  \nedgewall{<-}{b}{white}{$\sigma(n)$}{$\sigma(2)$}{$\sigma(1)$}
  \label{e.vertexnearwall}\\
\capwall{<-}{white}{$i$}{$j$}  & = \delta_{\bar j,i }\,  c_i\ \twowallpic{}{right}{}{}{},  
 \label{e.capwall}\\
\capnearwall{white}
&= \sum_{i=1}^n c_{\bar i}^{-1}\, \twowall{<-}{white}{black}{$i$}{$\overline{i}$}  
\label{e.capnearwall}\\
\crosswall{<-}{p}{white}{white}{$i$}{$j$} &=q^{-\frac{1}{n}}\left(\delta_{{j<i} }(q-q^{-1})\twowall{<-}{white}{white}{$i$}{$j$}+q^{\delta_{i,j}}\twowall{<-}{white}{white}{$j$}{$i$}\right),\label{e.crossp-wall}
\end{align} 
where the values $\delta_{j<i}, \delta_{i,j}, a, c_1,\dots, c_n$ were defined in Sec. \ref{ss.notation} and
the small white circles represent an arbitrary direction of the edges (left-to-right or right-to-left), consistent for the entire equation, as before. The black circle represents the opposite direction.

We show in Proposition \ref{p.crossingrels} that if $[n-2]!$ is invertible in $R$ then Relation \eqref{e.crossp-wall} is a consequence of Relations \eqref{e.pm}-\eqref{e.capnearwall}. 

\subsection{Eliminating sinks and sources}
\label{ss.nosink}
\bpro\label{c.nosinks}
 For any marked $3$-manifold, $\S_n(M,\cN)$ is spanned by stated $n$-webs with no sinks nor sources.
\epro
\begin{proof} If $\cN=\emptyset$ then the numbers of sinks and sources in any $n$-web coincide and they can be eliminated by
Relation \eqref{e.sinksource}. If $\cN\ne \emptyset$ then 
 sinks and sources can be eliminated by Relation \eqref{e.vertexnearwall}.
\end{proof}

Nonetheless, the use of sinks and sources in our theory makes it much more manageable.

%
\subsection{Change of ground ring}
\label{ss.change-coef} 
We will use the notation $\S_n(M,\cN,R)$ when we need to make the coefficient ring $R$ explicit. 
By our assumptions, $R$ is an algebra over $\BZ[v^{\pm 1}]$. 
The right exactness of tensor product gives a natural isomorphism
$$\S_n(M,\cN,\BZ[v^{\pm 1}])\otimes_{\BZ[v^{\pm 1}]} R\xrightarrow{ \cong}  \S_n(M,\cN,R).$$
Therefore, many properties of $\S_n(M,\cN,R)$ follow from those of $\S_n(M,\cN,\BZ[v^{\pm 1}])$. 

\subsection{Functoriality}
\label{ss.functor} 
An {\bf embedding} of a marked 3-manifold $\MN$ into a marked 3-manifold $(M', \cN')$ is an orientation preserving proper embedding $f: M \embed M'$   which maps $\cN$ into $\cN'$ preserving their orientations. Clearly $f$ induces an $R$-module homomorphism $\S_n(f): \S_n\MN \to \S_n(M', \cN')$ 
mapping each $n$-web $\alpha$ to $f(\alpha)$ with its framing transformed by the differential $
f_*: T M\to T M'.$ That homomorphism depends only on the isotopy class of $f$ (in the embeddings).
A {\bf morphism} from  $\MN$ to $(M', \cN')$ is an isotopy class of embeddings from $\MN$ to $(M', \cN')$.
Hence,  $\S_n(\cdot )$ defines a functor from the category of marked $3$-manifolds to the category of $R$-modules.

\begin{example} \label{e.functoriality}
Let $\MN$ be a marked 3-manifold. For any closed subset $X$ of \mbox{$\pM-\cN$}, its complement $(M-X, \cN)$ is a marked $3$-manifold as well and the natural embedding $\iota: (M-X, \cN) \embed \MN$ is a morphism called a {\bf pseudo-isomorphism}. It induces an $R$-module isomorphism $\iota_*: \S_n(M',\cN) \xrightarrow{\simeq} \S_n(M,\cN)$.
\end{example}

In this paper we will consider certain geometric operations on $3$-manifolds, like cutting and gluing them along disks, which produce new manifolds defined up a diffeomorphisms only. 
We will address this issue with the aid of the following notion:

A {\bf strict isomorphism class of marked 3-manifolds} is a family of marked 3-manifolds $(M_i, \cN_i), i\in I$ equipped with isomorphisms $f_{ij}  : (M_i, \cN_i) \to (M_j, \cN_j)$ for any two indices $i,j$ such that $f_{ii}=\id$ and $f_{jk}\circ f_{ij} = f_{ik}$. For a strict isomorphism class of marked 3-manifolds we can identify all $R$-modules $\S_n(M_i, \cN_i)$ via the isomorphisms $\S_n(f_{ij})$. 

For example, to glue a  pair of boundary edges $e_1$ and $e_2$ we first fix an orientation reversing diffeomorphism $\phi: e_1\to e_2$ and then identify $x\equiv \phi(x)$ for all $x\in e_1$. Various $\phi$'s give various surfaces, but they belong to the same strict isomorphism class.

For a disjoint union of $M_1$ and $M_2$, the map
$$\S_n(M_1,\cN_1)\otimes \S_n(M_2,\cN_2)\to \S_n(M_1\sqcup M_2,\cN_1\sqcup \cN_2)$$ 
sending $\alpha_1 \otimes \alpha_2$ to $\alpha_1\sqcup \alpha_2$ is an isomorphism. We will identify
$\S_n(M_1\sqcup M_2,\cN_1\sqcup \cN_2)$ with $\S_n(M_1,\cN_1)\otimes \S_n(M_2,\cN_2)$ though this map.

\subsection{Grading} 
\label{ss.grading}

For a stated $n$-web $\alpha$ in $(M,\cN)$ and a marking $\beta \subset\cN$ we define the {\bf $\beta$-degree}
$$\deg_{\beta}(\al) = \sum_{x\in \al \cap \beta} d_{s(x)} =\sum_{x\in \al \cap \beta} \left(s(x) - \frac{n+1}{2}\right)\in \frac{1}{2}\Z,$$
where $s(x)$ denotes the state of $\alpha$ at $x.$
Note that the $\beta$-degree is preserved by the skein relations \eqref{e.pm}-\eqref{e.crossp-wall} and, therefore, it descends to $\frac{1}{2}\Z$-valued grading on $\S_n(M,\cN)$.

\subsection{Useful Identities} Recall that $a,t, c_i$ were defined in Subsection \ref{ss.notation}.

  \def\tCijrotate{\raisebox{-12pt}{\incl{1.2 cm}{tCij180}}}
  \def\tCij{\raisebox{-17pt}{\incl{1.2 cm}{tCij}}}
  \def\tCijTwist{\raisebox{-12pt}{\incl{1.2 cm}{tCijTwist}}}
  \def\tDijrotate{\raisebox{-12pt}{\incl{1.2 cm}{tDij180}}}
  \def\tDij{\raisebox{-17pt}{\incl{1.2 cm}{tDij}}}
  \def\Cji{\raisebox{-12pt}{\incl{1.2 cm}{Cji}}}
  \def\hookup{\raisebox{-12pt}{\incl{1.2 cm}{hookup}}}
  \def\hookupr{\raisebox{-12pt}{\incl{1.2 cm}{hookupr}}}
  \def\decompuprr{\raisebox{-12pt}{\incl{1.2 cm}{decompuprr}}}
  \def\hookdown{\raisebox{-12pt}{\incl{1.2 cm}{hookdown}}}
  \def\hookdownr{\raisebox{-12pt}{\incl{1.2 cm}{hookdownr}}}
  \def\circle{\raisebox{-12pt}{\incl{1.2 cm}{circle}}}
   \def\circleSplit{\raisebox{-12pt}{\incl{1.2 cm}{circleSplit}}}

\def\sourcea{\raisebox{-25pt}{\incl{2 cm}{sourcea}}}
\def\sourceb{\raisebox{-25pt}{\incl{2 cm}{sourceb}}}
\def\sourceleft{\raisebox{-25pt}{\incl{2 cm}{sourceleft}}}
\def\outgoingleft{\raisebox{-25pt}{\incl{2 cm}{outgoingleft}}}
 \def\Cijright{\raisebox{-13pt}{\incl{1.2 cm}{Cijright}}}
 \def\decompright{\raisebox{-13pt}{\incl{1.2 cm}{decompright}}}
 \def\decompppright{\raisebox{-13pt}{\incl{1.2 cm}{decompppright}}}
\def\ad{{a_\downarrow}}
\def\au{{a_\uparrow}}
\begin{proposition} \label{r.maintech} The following identities hold in any stated skein module $\S_n\MN$:
\begin{equation} \label{e.crossing}
 [n-2]! \cross{n}{n}{p}{>}{>}{}{}{}{} = [n-2]!q^\frac{n-1}{n}
  \walltwowall{n}{n}{>}{>}{}{}{}{} - (-1)^{\binom n2} q^{-\frac{1}{n}}
\twoonetwowall{n}{>}{>}{\hspace*{-.05in}$n$-$2$}{}{},
\end{equation}
where the label in the diagram on the right indicates $n-2$ parallel horizontal edges.
\begin{equation} \label{e.vertextwist}
\vertexfourcross{>}  = -q^{-(1+\frac 1n)} \vertexfour{>}, 
\end{equation}
\begin{align}
\vertexwall{<-}{b}{white}{$\sigma(n)$}{$\sigma(2)$}{$\sigma(1)$} & = a t^{n/2} (-q)^{\ell(\sigma)} 
\label{e.vertexwall}\\
\vertexwall{->}{t}{white}{$\sigma(1)$}{$\sigma(2)$}{$\sigma(n)$} & = a (-q)^{\ell(\sigma)}\label{e.vertexwall3}
\end{align}

\begin{align}
\wallvertex{t} \ & = a t^{n/2}
\sum_{\sigma\in S_n}   (-q)^{\ell(\sigma)}
    \walledges{b}{$\sigma(n)$}{$\sigma(2)$}{$\sigma(1)$}\label{e.wallvertex}.\\
\wallcup{<-}{white}{$j$}{$i$}= \capwall{->}{white}{$i$}{$j$} & = 
\delta_{\bar i, j}c_i^{-1}\label{e.capwallup} \\
\wallnearcap{white} & = \sum_{i=1}^n   c_{i} \, \walltwo{<-}{white}{black}{$i$}{$\overline{i}$}\label{e.wallnearcap}\\
\wallcross{<-}{p}{white}{white}{$i$}{$j$}\ & =q^{-\frac 1n}\left(\delta_{j<i}(q-q^{-1})\walltwo{<-}{white}{white}{$i$}{$j$}+ q^{\delta_{i,j}}\walltwo{<-}{white}{white}{$j$}{$i$}\right)\label{e.wallcross}
\end{align}
\end{proposition}

\begin{proof} Identities \eqref{e.crossing} and \eqref{e.vertextwist} are respectively \eqref{e.crossing1} and \eqref{e.vertextwist1}. As remarked in Subsec. \ref{ss.unbased-t}, these identies are consequences of the basic internal annihilators, which are skein relations for  $\S_n(M,\cN).$

\noindent{\it Proof of \eqref{e.vertexwall}:} $\vertexwall{<-}{b}{white}{$\sigma(n)$}{$\sigma(2)$}{$\sigma(1)$} = 
\vertexwallturn{white}{$\sigma(n)$}{$\sigma(2)$}{$\sigma(1)$} = a\sum_{\tau\in S_n} (-q)^{\ell(\tau)}
\threecaps{white}{$\sigma(n)$}{$\sigma(2)$}{$\sigma(1)$}{$\tau(1)$}{$\tau(2)$}{$\tau(n)$}.$
The web on the right is non-zero only for $\tau(i)=\overline{\sigma(\overline{i})}$ for $i=1,...n.$
Since $\ell(\tau)=\ell(\sigma)$ then,
the above equals
$$a(-q)^{\ell(\sigma)}\cdot c_1\cdot ...\cdot c_n=
a t^{n/2}(-q)^{\ell(\sigma)},$$
by \eqref{e.prodc}. 
\vspace*{.1in}

\noindent{\it Proof of
\eqref{e.vertexwall3}}:
\vertexwall{->}{t}{white}{$\sigma(1)$}{$\sigma(2)$}{$\sigma(n)$}$=$
\raisebox{-.4\height}{\includegraphics[height=.7in]{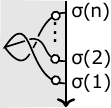}}
$=(-q^{-(1+\frac 1n)})^{\binom n2}$ \vertexwall{<-}{b}{white}{$\sigma(n)$}{$\sigma(2)$}{$\sigma(1)$}
 by \eqref{e.vertextwist}. Since 
\beq\label{e.a2}
(-q^{-(1+\frac 1n)})^{\binom n2}=(-1)^{\binom n2}q^{-\frac{n^2-1}{2}}, 
\eeq
the statement follows by \eqref{e.vertexwall}.

\noindent{\it Proof of \eqref{e.wallvertex}:} By composing \eqref{e.vertexnearwall} with 
\phalftwist{n}{white}{}{}{}\ on the left, we obtain 
$$(-q^{-(1+\frac 1n)})^{\binom n2} \vertexnearwall{b}{white} =  \sum_{\sigma \in S_n} a\, (-q)^{\ell(\sigma)}\,  
\phalftwist{<-}{white}{$\sigma(n)$}{$\sigma(2)$}{$\sigma(1)$}\ 
=\sum_{\sigma \in S_n} a\, (-q)^{\ell(\sigma)}\,  
\nedgewall{->}{t}{white}{$\sigma(1)$}{$\sigma(2)$}{$\sigma(n)$}
,$$
by \eqref{e.vertextwist}.
Now the statement follows by \eqref{e.a2} and by rotating these skeins $180^o$.\vspace*{.1in}

\noindent{\it Proof of \eqref{e.capwallup}:} The left side equals
$$\capcrosswall{->}{m}{white}{$j$}{$i$}=t^{-1}\cdot \capwall{<-}{white}{$j$}{$i$}=t^{-1}\delta_{\bar i, j}c_j=
t^{-1}\delta_{\bar i, j}c_{\bar i}=\delta_{\bar i, j}c_i^{-1},$$
by \eqref{e.twist}.
 
\noindent{\it Proof of \eqref{e.wallnearcap}:} By \eqref{e.twist} and \eqref{e.capnearwall}, 
$t^{-1}\capnearwall{white}= \kinkwall{n}{white}=\sum_{i=1}^n (c_{\bar i})^{-1} \crosswall{<-}{p}{white}{black}{$i$}{$\bar i$}$.

Now the statement follows by $180^o$ rotation and the fact that $tc_{\bar i}^{-1}=c_i,$ cf. \eqref{e.prodc}.

\noindent{\it Proof of \eqref{e.wallcross}:} $\wallcross{<-}{p}{white}{white}{$i$}{$j$}=\crosswall{->}{p}{black}{black}{$j$}{$i$}= 
\doublecrosswall{<-}{black}{black}{$i$}{$j$}
= q^{-\frac 1n}\left(\delta_{j<i}(q-q^{-1})\crosswall{<-}{p}{black}{black}{$i$}{$j$}+ q^{\delta_{i,j}}\crosswall{<-}{p}{black}{black}{$j$}{$i$}\right)=$
$q^{-\frac 1n}\left(\delta_{j<i}(q-q^{-1})\twowall{->}{black}{black}{$j$}{$i$}+q^{\delta_{i,j}}
\twowall{->}{black}{black}{$i$}{$j$}\right),$
by \eqref{e.crossp-wall}. 
\end{proof}

\subsection{Splitting homomorphism}
\label{ss.split}
%

As mentioned in Introduction, an important property of stated skein modules is that they behave in a simple manner under the splitting of 3-manifolds along disks. This property, known as the Splitting Theorem, was first proved by \cite{Le-triang, BL} for the Kauffman bracket skein modules ($n=2$) and then by \cite{Hi} for $n=3$. We formulate it now and prove for all $n$ below.

Suppose $(M,\cN)$ is a marked $3$-manifold and $D$ is a properly embedded closed disk in $M$ (and, hence, $\p D\subset \p M$), disjoint from the closure of $\cN$. By removing a collar neighborhood of $D$ we get a topological 3-manifold $M'$ whose 
boundary contains two copies $D_1$ and $D_2$ of $D$ such that gluing $D_1$ to $D_2$ yields $M$ together with a surjective homomorphism $\pr: M' \to M$. The manifold $M'$ has a smooth structure with corners. However, these corners can be smoothed out uniquely up to isotopy.
  
Let $\beta\subset D$ be an oriented open interval, and  $\beta_1\subset D_1$ and $\beta_2\subset D_2$ be preimages of $\beta$. The {\bf splitting of $\MN$ along $(D,\beta)$}, denoted by  $\Cut_{(D,\beta)}\MN$, is  the
marked 3-manifold $(M', \cN')$,  where $\cN'= \cN \cup \beta_1 \cup \beta_2$. 
It is easy to see that $\Cut_{(D,\beta)}\MN$ is defined uniquely as a strict isomorphism class, cf. Subsection \ref{ss.functor}.

Let $\al$ be a stated $n$-web in $\MN$. This web in this subsection is given by a specific embedding and, hence, not considered up to isotopy.
It is said to be {\bf $(D,\beta)$-transverse} if the vertices of $\al$ are not in $D$, $\al$  is transverse to $D$,  $\al \cap D \subset \beta$, and the framing at every point of $\al \cap \beta$ is a positive tangent vector of $\beta$. Note that every web $\alpha$ in $(M,\cN)$ can be isotoped 
so that it is $(D,\beta)$-transverse. 
Suppose in addition that $\al$ is stated. Then the $n$-web $\pr^{-1}(\al)$ of $(M', \cN')$ is stated everywhere except for its endpoints in  $\beta_1 \cup \beta_2$,  cf. Fig. \ref{f.splitexample}. Given any map $s:\al \cap \beta \to \{\pm\}$, let $\al(s)$ denote the (partially stated) $n$-web $\pr^{-1}(\al)$ in $(M', \cN')$ with additional states $s(\pr(x))$ for $x\in \pr^{-1}(\al) \cap (\beta_1 \cup \beta_2)$. Hence, $\al(s)$ is fully stated.
We call $\al(s)$ a {\bf lift} of $\al$. If $|\al\cap \beta|=k$ then $\al$ has $n^k$ distinct lifts.

\begin{figure}[h] 
   \centering \diagg{splitexample1.pdf}{0.7in} $\xrightarrow{\Theta_{D,\beta}} \sum_{i,j=1}^n$
   \centering \diagg{splitexample2.pdf}{0.7in}
    {}
   \caption{An example of a splitting of an $n$-web (in green) intersecting the splitting disk $D$ twice.}
   \label{f.splitexample}
\end{figure}

\bthm\label{t.splitting}
Let $D$ be a closed disk properly embedded in a marked $3$-manifold $\MN$ and let $\beta$ be an oriented open arc in $D$. Let $\Cut_{(D,\beta)}\MN$ 
be the splitting of $(M, \cN)$ along $(D,\beta)$, as described above.
Then there is a unique $R$-module homomorphism
$$ \Theta_{(D,\beta)}: \S_n(M,\cN) \to \S_n(\Cut_{(D,\beta)}\MN)$$
sending every stated $(D,\beta)$-transverse $n$-web $\alpha$ in $\MN$ to the sum of all of its lifts,
\beq
\Theta_{(D,\beta)}(\al) = \sum_{s:\, \al \cap \beta \to \{\pm\}} \al(s).
\label{e.split0}
\eeq
\ethm

Note that for any arcs $\beta, \beta'$ in $D$ there is an isomorphism between marked $3$-manifolds 
$\Cut_{(D,\beta)}\MN \simeq \Cut_{(D,\beta')}\MN$, inducing an isomorphism of stated skein modules which commutes with the splitting homomorphisms.
Consequently, we will often denote $\Theta_{(D,\beta)}$ and $\Cut_{(D,\beta)}\MN$  by $\Theta_D$ and $\CutD\MN$ when it does not lead to confusion.

\def\pr{\mathrm{pr}}
\def\tW{\tilde {\mathcal W}}

\begin{remark}\label{r.two-splittings}
It is easy to see that splitting homomorphisms along any two disjoint splitting disks $D_1$ and $D_2$ commute,
\[\Theta_{D_1}\circ \Theta_{D_2}=\Theta_{D_2}\circ \Theta_{D_1}.\]
\end{remark}

\begin{remark}
\label{r.notcloseddisk} 
By removing a closed subset of $\partial M$ disjoint from $\cN$ and using the pseudo-isomorphism
of Example \ref{e.functoriality},  we can apply the theorem to many cases when $D$ is a closed disk with some closed intervals on its boundary removed. This fact will be useful in Section~\ref{s.surfaces} where we will apply the Splitting Theorem to thickened surfaces $\Sigma \times (-1,1)$ cut along open disks, $(-1,1)\times (-1,1)$.
\end{remark}

In general the splitting homomorphism is not injective. For an example in the $n=2$ case see \cite{CL2}. 
We will discuss the injectivity and the image of the splitting homomorphism for thickened surfaces in Section~\ref{s.surfaces}.\vspace*{.1in}

\noindent{\it Proof of the Splitting Theorem:} We identify a closed collar neighborhood of $D$ with the closed cube $\bQ= [-1,1]^3$ so that $D= \{0\} \times [-1,1]^2$ and $\beta$ is an open interval subset of $ \{0\} \times  \{0\}\times  [-1,1]$, as in Figure \ref{f.cube_a}(a). 
For a stated $(D,\beta)$-transverse $n$-web $\al$ let
 $\Theta(\al)\in \S_n(\Cut_{(D,\beta)}\MN)$
 be the right side of Eq. \eqref{e.split0}. 
 To prove the theorem we need to show $\Theta(\al)$ is invariant under isotopies of $\al$.

\FIGc{cube_a}{
(a) The cube $\bar Q$. The disk $D$ is the middle square containing the lines $\beta$ and $\bpp$. (b) The splitting homomorphism $\Theta$.}{3cm}

An ambient isotopy of $\alpha$ in $M$ can be decomposed into a sequence of isotopies, each of which is supported in a small neighborhood of $D$ or supported outside of $D$. The latter clearly preserves $\Theta$, so we only need to check invariance of $\Theta$ under isotopies with support in the interior of $\bar Q$. By an isotopy outside $D$ we can assume that $\al\cap Q$ is an $n$-tangle. To get a diagram of $\al\cap Q$ we first use height preserving deformation near $D$ to move $\al$ to a general position with respect to the projection onto $[-1,1]^2\times \{0\}$ (as always considered in page). The points in $\al\cap \beta$, after that deformation, project to points on $\beta^\perp$. We will always choose a height preserving deformation such that the height order on $\beta^\perp$ is given by the direction from the top to the bottom, as in Figure \ref{f.cube_a}(b).

Now we can decompose a diagram of $\al$ into elementary tangle diagrams listed in \eqref{e.id}-\eqref{e.RTT}. If  $\al'$ is isotopic to $\al$ by an isotopy in $Q$, then its diagram can be obtained by 
 a sequence of operations moving elementary tangles through $ \beta,$ and the height exchange move, discussed in point (d) below. Therefore, it is enough to verify that $\Theta(\alpha)$ is preserved by the following  four moves:

\begin{enumerate}
\item[(a)] passing a cap through $\beta$. The invariance of $\Theta(\alpha)$ under this move is a consequence of skein relations (\ref{e.capwall}) and (\ref{e.wallnearcap}): 
$$\Theta\left(\capthroughwall\right)=\sum_{i,j} \capseparatedwall{$i$}{$j$}=\sum_{i=1}^n c_i \walltwo{<-}{white}{black}{$i$}{$\overline{i}$}= \Theta\left(\wallnearcap[-.2]{white}\right).$$
By the same argument, $\Theta\left(
\scalebox{-1}[1]{
\wallnearcap[-.2]{white}}\hspace*{-.15in}\right)=
\Theta\left(
\scalebox{-1}[1]{\capthroughwall}\right).$

\item[(b)]  passing a sink or a source through $\beta$. The invariance of $\Theta$ under this move is a direct consequence of skein relations \eqref{e.vertexwall} and  \eqref{e.wallvertex}:
$$\Theta\left(\vertexwallext{b}\right)=\sum_{\sigma\in S_n} 
\vertexwall{<-}{b}{white}{$\sigma(n)$}{$\sigma(2)$}{$\sigma(1)$}\hspace*{-.13in}
\scalebox{1.09}{\walledges{b}{}{}{}}
=  at^{n/2}
\sum_{\sigma\in S_n} (-q)^{\ell(\sigma)}  \walledges{b}{$\sigma(n)$}{$\sigma(2)$}{$\sigma(1)$}=\Theta\left(\wallvertex[-.4in]{b}\right).$$

\item[(c)]  passing a positive crossing through $\beta$:
\begin{multline*} 
\Theta\left(\crosswallext{<-}{p}{white}\right)=
\sum_{i,j} \crosswall{<-}{p}{white}{white}{$i$}{$j$} \walltwo{<-}{white}{white}{$i$}{$j$}=
q^{-\frac{1}{n}}\left(\sum_{j<i}  (q-q^{-1})\twowall{<-}{white}{white}{$i$}{$j$} \walltwo{<-}{white}{white}{$i$}{$j$}+q^{\delta_{i,j}}\sum_{i,j}\twowall{<-}{white}{white}{$j$}{$i$}\walltwo{<-}{white}{white}{$i$}{$j$}\right)=\\
q^{-\frac{1}{n}}\left(\sum_{j<i}  (q-q^{-1})\twowall{<-}{white}{white}{$i$}{$j$} \walltwo{<-}{white}{white}{$i$}{$j$}+q^{\delta_{i,j}}\sum_{i,j}\twowall{<-}{white}{white}{$i$}{$j$}\walltwo{<-}{white}{white}{$j$}{$i$}\right)=
\sum_{i,j} \twowall{<-}{white}{white}{$i$}{$j$} \wallcross{<-}{p}{white}{white}{$i$}{$j$}=
\Theta\left(\raisebox{.15in}{\rotatebox{180}{\crosswallext{->}{p}{white}}}
\right),
\end{multline*}
by \eqref{e.crossp-wall} and \eqref{e.wallcross}.\vspace*{.1in}

\item[(d)] passing a negative crossing through $\beta$ follows from (c) by composing the fragments of diagrams on the left and the right side of the above identity with a negative crossing on their both sides.

\item[(e)] height exchange of two consecutive points of $\al\cap \beta$ as in Figure \ref{f.cube_b}. The invariance of $\Theta$ under this moves follows from the move in (c) or (d) if the arcs involved have coinciding orientations. If Figure \ref{f.cube_b} involves arcs in opposite directions, then the left side of the diagram on the right can be decomposed into elementary diagrams and all of them can be moved to the right side by (a)-(d).

\FIGc{cube_b}{Height exchange move}{1.5 cm}
\end{enumerate}

The above argument shows that $\Theta(\alpha)$ is preserved by isotopies of $\alpha$.
To finish the proof off, observe that $\Theta$ maps the defining relations \eqref{e.pm}-\eqref{e.crossp-wall} to $0$ in $\cal S_n(M',N'),$ because they are all local and can be moved away from $D$.\qed

%
\subsection{Reversing orientations of $3$-manifolds and of webs}
\label{ss.symmetries}
An \underline{orientation} of a web consists of orientations of all its loop components and directions of all its edges.
Let $\cev{\alpha}$ denote an $n$-web $\alpha$ with its orientation reversed (and unchanged framing). 
Since the defining relations  \eqref{e.pm}-\eqref{e.crossp-wall} of $\S_n(M,\cN)$ are invariant under the total orientation inversion, we have 
\bcor\label{c.orient-rev}
$$\cev {\,\cdot\,}: \S_n(M,\cN)\to \S_n(M,\cN)$$
is a well defined $R$-module automorphism.
\ecor 

\def\bal{\bar \al}

Let $\overline{(M,\cN)}$ denote $M$ and $\cN$ with reversed orientations. Let $\bar R$ be the ring $R$ with the distinguished element $v^{-1}$ instead of $v$. For an $n$-web $\al$ of $\MN$ let $\overline{\alpha}$ be the $n$-web in $\overline{(M,\cN)}$ obtained from $\al$ by negating its framing, $f \to -f$, but retaining the orientation.

\bthm\label{t.framing-rev} 
(1) Any ring isomorphism $\vk: R\to \bar R$ sending $v$ to $v^{-1}$ extends to an isomorphism of $R$-modules
$\vk_{(M,\cN)}: \S_n(M,\cN,R)\xrightarrow{\cong} \S_n(\overline{M,\cN}, \bar R)$ sending every stated $n$-web $\alpha$ to $\overline{\alpha}$, where $\S_n(\overline{M,\cN}, \bar R)$ is an $R$-module via $\vk: R\to \bar R.$

(2) The composition   $ \vk_{(\overline{M,\cN})} \circ \vk_{(M,\cN)}$ is  the identity on $\S_n(M,\cN,R)$.
\ethm

The above isomorphism is called  the {\bf orientation reversion isomorphism}. (Note that for some rings $R$, an isomorphism $\vk:R\to \bar R$ as above may not exist or be non-unique.)\vspace*{.1in}

\noindent{\it Proof of Theorem \ref{t.framing-rev}:}  
By abuse of notation, we define an $R$-linear map $\vk_{(M,\cN)}$ first as
$$\vk_{(M,\cN)}: R{\cal W}_n(M,\cN)\to \S_n(\overline{M,\cN}),\quad \vk_{(M,\cN)}(\alpha)=\overline{\alpha}.$$ 
One checks immediately that that map factors through relations \eqref{e.pm}-\eqref{e.unknot}. 

By our graphical convention, a diagram of $\overline{\alpha}$ 
near a marking is given by switching all crossings in a diagram $\alpha$ and by reversing of the direction of the vertical line $\beta^\perp.$ For example, if 
$\alpha=\crosswall{->}{p}{<}{<}{$i$}{$j$}$ then $\overline{\alpha}=\crosswall{<-}{n}{<}{<}{$i$}{$j$}.$
It is clear that $\vk_{(M,\cN)}$ maps \eqref{e.sinksource} to the equality of Lemma \ref{l.sinksource-involution} (below) and, therefore, it preserves that relation.

To see that $\vk_{(M,\cN)}$ factors through \eqref{e.vertexnearwall} substitute $\sigma'$ for $\sigma$ in \eqref{e.wallvertex}, where 
$\sigma'(i)=\sigma(\bar i)$, for $i=1,...,n,$ 
and rotate that equation $180^o.$ Since $\ell(\sigma')=\binom n2-\ell(\sigma),$ we get
$$\vertexnearwall{b}{white} = t^{n/2}a \sum_{\sigma \in S_n} (-q)^{\binom n2-\ell(\sigma)}\,  \nedgewall{->}{b}{white}{$\sigma(n)$}{$\sigma(2)$}{$\sigma(1)$}.$$
Since 
$$(-q)^{\binom n2}t^{n/2}=q^{\binom n2+\frac{n^2-1}{2}} =a^{-2},$$
by Eq. \eqref{e.a}, we get the desired relation
$$\vertexnearwall{b}{white} = a^{-1} \sum_{\sigma \in S_n} (-q)^{-\ell(\sigma)}\,  \nedgewall{->}{b}{white}{$\sigma(n)$}{$\sigma(2)$}{$\sigma(1)$}.$$

Furthermore, $\vk_{(M,\cN)}$ maps \eqref{e.capwall} and \eqref{e.capnearwall} to \eqref{e.capwallup} and \eqref{e.wallnearcap}.

Let us show now that $\vk_{(M,\cN)}$ factors through \eqref{e.crossp-wall}.
We need to verify that 
$$\crosswall{->}{n}{white}{white}{$i$}{$j$} =q^{\frac 1n}\left(\delta_{j<i}(q^{-1}-q)\twowall{->}{white}{white}{$i$}{$j$}+ q^{-\delta_{i,j}}\twowall{->}{white}{white}{$j$}{$i$}\right).$$
By \eqref{e.pm}, the left side is
$$q^\frac{2}{n} \crosswall{->}{p}{white}{white}{$i$}{$j$} -q^\frac{1}{n}(q-q^{-1}) \twowall{->}{white}{white}{$i$}{$j$}$$ 
and since $1-\delta_{j<i}=\delta_{\bar j<\bar i} +\delta_{i,j},$ the above equation reduces to 
$$q^\frac{2}{n} \crosswall{->}{p}{white}{white}{$i$}{$j$}=q^{\frac 1n}\left(\delta_{\bar j<\bar i}(q-q^{-1})\twowall{->}{white}{white}{$i$}{$j$}+ q^{\delta_{i,j}}\twowall{->}{white}{white}{$j$}{$i$}\right),$$
which is \eqref{e.wallcross} rotated $180^o$ (and with $i$ and $j$ interchanged).

Hence, we have shown that the above map factors to
$$\vk_{(M,\cN)}: \S_n(M,\cN)\to \S_n(\overline{M,\cN}).$$
It is an $R$-module homomorphism by definition.

Part (2) is obvious.

\blem\lb{l.sinksource-involution} One has
$$\sinksourcethree{>} = (-q)^{-\binom n2} \sum_{\sigma \in S_n} (-q^{-\frac{1-n}n})^{\ell(\sigma)}  \coupon{$\sigma_-$}{>}{}{}{},$$
where $\sigma_-$ is the minimal crossing negative braid representing $\sigma\in S_n.$
\elem

\bpr
Let $\tau(i)=n+1-i$, for $i=1,...,n$. Then $\tau_-$ is the negative half-twist $n$-braid.
By applying it to the left side of \eqref{e.sinksource} we obtain
\beq\lb{e.sinksource-twist}
(-q^{-\frac{n+1}n})^{-{\binom n2}}\sinksourcethree{>} =(-q)^{\binom n2}\cdot \sum_{\sigma\in S_n}
(-q^{\frac{1-n}n})^{\ell(\sigma)} \coupon{\hspace*{-.1in}$\tau_-\sigma_+$\hspace*{-.05in}}{>}{}{}{},
\eeq
by \eqref{e.vertextwist}. Note that $\tau_-\sigma_+=(\tau\sigma)_-$ for every $\sigma\in S_n$ and that by \eqref{e.length},
$$\ell(\sigma)+\ell(\tau\sigma)=\ell(\tau)={\binom n2}.$$
Therefore, by denoting $\tau\sigma$ by $\sigma'$, the right side of \eqref{e.sinksource-twist} reduces to
$$(-q)^{\binom n2}\cdot \sum_{\sigma'\in S_n}
(-q^{\frac{1-n}n})^{{\binom n2}-\ell(\sigma')} \coupon{$\sigma_-'$}{>}{}{}{}=q^{\binom n2} q^{\frac{1-n}n \cdot\binom n2}\sum_{\sigma'\in S_n}(-q^{\frac{1-n}n})^{-\ell(\sigma')} \coupon{$\sigma_-'$}{>}{}{}{}.$$
and, hence, \eqref{e.sinksource-twist} becomes
$$\sinksourcethree{>} =(-1)^{\binom n2} q^D\cdot \sum_{\sigma'\in S_n}
(-q^{\frac{1-n}n})^{\ell(\sigma)} \coupon{$\sigma_-'$}{>}{}{}{},$$
where $$D=-\frac{n+1}n\cdot\binom n2+ \binom n2+\frac{1-n}n\binom n2=
-\binom n2.$$
\epr

%
\subsection{Marking automorphisms}
\label{ss.edgeauto}
%
Consider a function $\eta: \{1,\dots, n\}\to R^*$ such that 
$$\prod_{i=1}^n \eta(i)=1\ \text{and}\ \eta(i)\eta(\bar i)=1\ \text{for every}\ i,$$
where the bar denotes the conjugation, $\overline{i}=n+1-i,$ as before.
It is easy to see that for every such $\eta$ and every marking $\beta$ in $\cN$ there is an $R$-module automorphism $\phi_{\eta, \beta}$ of $\S_n(M,\cN)$ sending stated $n$-webs $\alpha$ to
$$\phi_{\eta, \beta}(\alpha)=\prod_{x\in \alpha\cap \beta} \eta(s(x))\cdot \alpha$$ 
where $s(x)$ is the state of the endpoint $x$ of $\alpha.$ We call $\phi_{\eta, \beta}$ a {\bf marking automorphism} of $\S_n(M,\cN)$. 

When $\eta_i = g_i= (-1)^{n-1} q^{2i-n-1}$, as in Eq. \eqref{e.gi}, we denote $\phi_{\eta,\beta}$ by $g_\beta$.

\subsection{Half-twist automorphisms} 
\label{ss.htw}

\def\tw{\text{tw}}

\begin{proposition}\label{p.twist} 
 For any marking $\beta$ in $\cN$ there exist unique $R$-linear isomorphisms 
 $$\htw_{\beta},\wt\htw_\beta:\S_n(M,\cN) \to \S_n(M,\cN)$$
 sending any stated $n$-web $\alpha$ in $(M,\cN)$ with $k$ endpoints on $\beta$ to
 \beq\label{e.htw} \htw_{\beta}   \left(\nedgewalltall{<-}{b}{}{${i_k}$}{${i_2}$}{${i_1}$}\right)= 
\left( \prod_{j=1}^k c_{\overline{i_j}}\right) \cdot
\nedgewalltall{->}{b}{}{$\overline{i_k}$} {$\overline{i_2}$}{$\overline{i_1}$} = \left( \prod_{j=1}^k c_{\overline{i_j}}\right) \cdot \htwf,
\eeq
and to 
\beq\label{e.htw-t} \wt\htw_{\beta}   \left(\nedgewalltall{<-}{b}{}{${i_k}$}{${i_2}$}{${i_1}$}\right)= 
\left( \prod_{j=1}^k c_{i_j}\right) \cdot
\nedgewalltall{->}{b}{}{$\overline{i_k}$} {$\overline{i_2}$}{$\overline{i_1}$} = \left( \prod_{j=1}^k c_{i_j}\right) \cdot \htwf,
\eeq
where $H$ is the positive half-twist, $\phalftwist{<-}{}{}{}{}$.
(The orientations of the horizontal edges are arbitrary.)
\end{proposition}
 
We call $\htw_{\beta}$ and $\wt\htw_\beta$ the {\bf half-twist automorphisms}. Note that $\wt\htw_\beta$ coincides with $\htw_\beta$ composed with $g_\beta$ of Subsec. \ref{ss.edgeauto}, because $c_i=c_{\bar i}(-1)^{n-1} q^{2\bar i -n-1}.$
 
\brem\label{r.Uq-action}
We will show in Subsection \ref{ss.coaction} that every marking in $\cN$ defines a left action of a completion $\widetilde{\UL}$ of the Lusztig integral version $\UL$ of $\Uq.$ That completion contains the half-twist element $X$ and the charmed element $g$ of Subsection \ref{ss.half-ribbon}.
The automorphisms $\htw_\beta$ and $g_\beta$ coincide with the actions of $X$ and $g$, respectively.
\erem

The map $\htw_{\beta}$ generalizes the inversion along an edge, for the stated Kauffman bracket skein algebras of thickened surfaces in \cite{CL}.  However it is the inverse of that map.

The inverse of $\htw_\beta$ is given by, 
$$\htw_{\beta}^{-1}   \left(  \nedgewalltall{<-}{b}{}{${i_k}$}{${i_2}$}{${i_1}$}\right)= 
 \left(\frac 1{ \prod_{j=1}^k c_{{i_j}}}\right) \cdot \htwn,
$$
where $\bar H$ denotes the negative half-twist, as before.\vspace*{.1in}
 
\noindent{\it Proof of Proposition \ref{p.twist}:}
 By abuse of notation, let us first consider a map $\htw_\beta:\cal W_n(M,\cN) \to \S_n(M,\cN)$ sending 
 any stated $n$-web $\alpha$ in $(M,\cN)$ with $k$ endpoints on $\beta$ to
 $$\htw_\beta\left(\nedgewalltall{<-}{b}{}{$i_k$}{$i_2$}{$i_1$}\right)= \prod_{j=1}^k 
 c_{\overline{i_j}} \cdot \nedgewalltall{->}{b}{}{$\overline{i_k}$}{$\overline{i_2}$}{$\overline{i_1}$}.$$

Obviously, $\htw_\beta$ preserves the internal skein relations, \eqref{e.pm}-\eqref{e.sinksource}. It maps \eqref{e.vertexnearwall} to
$$\vertexnearwall{b}{white} = 
a\cdot 
\left(\prod_{i=1}^n c_i\right)\cdot
\sum_{\sigma \in S_n} \, (-q)^{\ell(\sigma)}\,  \nedgewalltall{->}{b}{white}{$\overline{\sigma(n)}$}{$\overline{\sigma(2)}$}{$\overline{\sigma(1)}$}.$$
The right side equals 
$$a \left(\prod_{i=1}^n c_i\right) \sum_{\sigma \in S_n} \, (-q)^{\ell(\sigma)}\, 
\phalftwist{<-}{white}{$\overline{\sigma(1)}$}{$\overline{\sigma(2)}$}{$\overline{\sigma(n)}$}\ 
= 
a \left(\prod_{i=1}^n c_i\right) (-q^{-\frac{n+1}n})^\frac{n(n-1)}{2}\sum_{\sigma \in S_n} \, (-q)^{\ell(\sigma)}\, 
\nedgewalltall{<-}{t}{white}{$\overline{\sigma(1)}$}{$\overline{\sigma(2)}$}{$\overline{\sigma(n)}$},$$
by \eqref{e.vertextwist}.
By \eqref{e.prodc}, $\left(\prod_{i=1}^n c_i\right) \cdot (-q^{-\frac{n+1}n})^\frac{n(n-1)}{2}=1$ and, hence, the expression above coincides with the right side of \eqref{e.vertexnearwall} by the substitution $\sigma\to \sigma'$, where $\sigma'(i)\to \overline{\sigma(\overline{i})},$ for $i=1,...,n$, which does not affect the permutation length. Consequently, $\htw_\beta$ preserves \eqref{e.vertexnearwall}. 

The preservation of the remaining relations, \eqref{e.capwall}-\eqref{e.crossp-wall} by $\htw_{\beta}$ is an immediate consequence of the left boundary relations \eqref{e.capwallup} -- 
\eqref{e.wallcross}.
 
This shows that our map descends indeed to an $R$-module homomorphism 
$$\htw_{\beta}: \S_n(M,\cN)\to \S_n(M,\cN).$$

Similarly, it is straightforward to show that there is a well defined map
$\htw_\beta':\cal S_n(M,\cN) \to \S_n(M,\cN)$ sending 
 any $n$-web $\alpha$ in $(M,\cN)$ with $k$ endpoints on $\beta$ to
 $$\htw_\beta'\left(\nedgewalltall{->}{b}{white}{$i_k$}{$i_2$}{$i_1$}\right)= \left(\prod_{j=1}^k 
 c_{i_j}\right)^{-1} \cdot \nedgewalltall{<-}{b}{white}{$\overline{i_k}$}{$\overline{i_2}$}{$\overline{i_1}$}.$$
Since $\htw_\beta'$ is an inverse of $\htw_\beta$, both are isomorphisms.
 \qed

%

%

\subsection{Essential uniqueness of the skein relations of $\S_n(M,\cN)$}
\label{ss.uniqueness}
\def\bou{{\bm{u}}}

In the context of our theory it is natural to ask how arbitrary are the constants $a, c_i$ in Subsection  \ref{ss.notation}. 
For a tuple $\bou=(u, u_1, \dots, u_n)$ of $n+1$ invertible elements of $R$ let $\S_n(M,\cN; \bou)$ be defined the same as $\S_n(M,\cN)$, with $c_i$ and $a$ replaced  respectively by $c'_i = c_i (u_i u_{\bar i})^{-1}$ and $a'= a (\prod_{i=1}^n u_i)/u$, and with the right side of \eqref{e.sinksource} multiplied by $u^2$. 
We denote the set of $n$-valent vertices of $\alpha$ by $V_n(\alpha)$.
Then it is easy to see that the map
$$\al \to \al \,  u^{|V_n(\al)|} \prod_{x\in \partial \al} u_{s(x)},$$ 
defined on stated $n$-webs, extends to an $R$-linear isomorphism from $\S_n\MN$ to $\S_n(M,\cN; \bou)$.

One can show that 
the new stated skein module $\S_n(M,\cN; \bou)$ satisfies the splitting homomorphism if and only if the following holds:
$$u_i= \pm 1,\  \prod_{i=1}^n u_i =1\ u_i u_{\bar i} =1,\ \text{for every}\ i.$$
Furthermore, all properties of $\S_n(M,\cN)$ formulated so far have their version for $\S_n(M,\cN; \bou)$.

\section{Stated SL($n$)-skein algebras of surfaces}
\label{s.surfaces}
%

The theory of stated SL($n$)-skein modules is particularly rich for thickened surfaces $M=\Sigma \times (-1,1).$ Note that any finite set $B\subset \p \Sigma$ defines markings $\cN=B\times (-1,1)$ for which $\S_n(M,\cN)$ is an $R$-algebra with the product of webs $\alpha_1\cdot \alpha_2$ given by stacking $\alpha_1$ on top of $\alpha_2$.  It is convenient however to represent unmarked boundary components of $\Sigma$ by punctures and to separate points of $B$ by ideal boundary points. That leads to the notion of a punctured bordered surface, considered for example in \cite{Le-triang, CL} already.  In particular, a punctured bordered surface encapsulates information about the points $B$ in it.

\def\bS{\bar \Sigma}
\def\pS{\partial \Sigma}

\subsection{Punctured bordered surface} 
\label{ss.pbsurf}

 A {\bf punctured bordered surface} (a {\bf pb surface} for short) $\Sigma$ is an oriented surface with possibly empty boundary $\pS$ such that each connected component of $\pS$ is an open interval. 
These components are called {\bf boundary edges}.

For simplicity we will assume that $\Sigma$ is of finite type in the sense that $\Sigma=\bS\setminus \cP$, where $\bS$ is a compact oriented  surface and $\cP \subset \bS$ is a finite set, called the {\bf ideal points} of $\Sigma$. Note that each connected component of $\partial \bS$  meets $\cP$. (However, some of the points of $\cP$ may be in the interior of $\Sigma$.)

An {\bf ideal arc} in ${\Sigma}$ is the image of a proper embedding
 $c:(0,1)\embed {\Sigma}$. This means $c$ can be extended to an immersion $\bar c : [0,1] \to \Sigma$ such that $\bar c(0), \bar c(1) \in 
 \cP$. An ideal arc is {\bf trivial} if it bounds a disk in $\Sigma$.

In each boundary edge $e$ choose a point $b_e$.
Let $\S_n({\Sigma})= \S_n(M, \cN)$, where $M= {\Sigma} \times (-1,1)$ and $\cN$ is the union of all $b_e \times (-1,1)$, each having the natural orientation of the interval $(-1,1)$. 
Since up to a canonical isomorphism, $\S_n({\Sigma})$ does not depend 
on the specific choice of the points $b_e$, we do not specify them in our notation. An $n$-web in $\MN$ is simply called an $n$-web over ${\Sigma}$.

For stated $n$-webs $\al$ and $\beta$ over ${\Sigma}$  let their product $\al\beta\in \S_n({\Sigma})$ be the result of stacking $\al$ above $\beta$. This product turns $\S_n({\Sigma})$ into an $R$-algebra.  

According to the graphical convention of Subsec. \ref{ss.marked}, an $n$-web over ${\Sigma}$ is presented by its diagram on $\Sigma$, which is the projection of $\al$ onto ${\Sigma}$ with the over/undercrossing information at every double point. Before projecting, we use  height-preserving deformation near the markings $b_e \times (-1,1)$ to make the projections of endpoints of $\al$ distinct. As before, the height order at endpoints of the diagram on each boundary edge is part of the diagram.

The orientation of ${\Sigma}$ induces an orientation on its boundary. When part of ${\Sigma}$ is drawn on a page of paper, which is identified with the standard $XY$-plane,  the orientation of $\pS$ is the counterclockwise direction.
 A diagram where the height order on a boundary edge $e$ is given by the orientation of $e$ induced from that of ${\Sigma}$  (respectively: the opposite orientation) is called {\bf positively (respectively: negatively) ordered on $e$}.
 
Given two edges $e_1, e_2$ of a pb surface ${\Sigma}$, not necessarily connected, the gluing ${\Sigma}/(e_1=e_2)$ is the result of identifying $e_1$ with $e_2$ via a diffeomorphism $e_1 \to e_2$ such that the resulting surface has an orientation induced from that of ${\Sigma}$. Such a surface is defined uniquely up to strict isomorphisms.

A pb surface ${\Sigma}$ is {\bf essentially bordered} if every connected component of it has non-empty boundary.

\blem\label{l.span} If a pb surface $\Sigma$ is essentially bordered then  $\S_n(\Sigma)$ is spanned by stated $n$-web diagrams without any of: sinks, sources, crossings, trivial loops, and trivial arcs.
\elem 

\bpr Crossings can be eliminated by bringing them to near a boundary edge as in Fig. \ref{f.cross-elim}(left), then expressing them as linear combinations of webs of the form Fig. \ref{f.cross-elim}(right) by Relation \eqref{e.capnearwall}, and finally eliminating them by Relation \eqref{e.crossp-wall}.

\begin{figure}[h] 
   \begin{center}
   \crosswithcaps{>}
   \quad \crosswithtwoextra
   \caption{}
   \label{f.cross-elim}
\end{center}   
\end{figure}
Sinks and sources, trivial loops, and trivial arcs  can be eliminated by Relations
\eqref{e.vertexnearwall}, \eqref{e.unknot}, \eqref{e.capwall}, respectively.
\epr

\subsection{Splitting homomorphism for surfaces} 
Let $c$ be an ideal arc in the interior of a pb surface $\Sigma$. The splitting $\Cut_c({\Sigma})$ is a pb surface having two boundary edges $c_1, c_2$ such that ${\Sigma}= \Cut_c({\Sigma})/(c_1=c_2)$. Let $\pr: \Cut_c({\Sigma}) \to {\Sigma}$ be the natural projection map. An $n$-web  diagram $D$  is {\bf $c$-transverse} if  $n$-valent vertices of $D$ are not in $c$ and $D$ is transerve to $c$.
Assume that $D$ is a stated $c$-transverse $n$-web diagram. Let $h$ be a linear order on the set $D \cap c$. For a map $s: D \cap c \to \{1,\dots, n\}$ let $D(h,s)$ be the stated $n$-web diagram over $\Cut_c({\Sigma})$ which is $\pr^{-1}(D)$ with the height order on $c_1 \cup  c_2$ induced (via $\pr$) from $h$, and the states on $c_1 \cup c_2$ induced (via $\pr$) from $s$. 
The Splitting Theorem (Thm. \ref{t.splitting}) for ${\Sigma}$ becomes

\begin{theorem}\label{t.splitting2} Let $c$ be an interior ideal arc of a pb surface ${\Sigma}$. There is a unique $R$-linear map $\Theta_c: \S_n({\Sigma}) \to \S_n(\Cut_c({\Sigma}))$ such that if $D$ is a diagram of a stated $n$-web $\al$ over ${\Sigma}$ which is $c$-transverse and $h$ is any linear order on $D \cap c$, then
$$\Theta_c(\al) =\sum_{s: D \cap c \to \{1,\dots, n\}} D(h, s).$$
The map $\Theta_c$ is an $R$-algebra homomorphism.
\end{theorem}
\begin{proof} The set $c \times (-1,1)$ is not a closed disk but we can still use Theorem \ref{t.splitting}, see Remark \ref{r.notcloseddisk}. More precisely, let us enlarge the ideal points of $\Sigma$ to open disks in $\bS$, and embed  ${\Sigma}$ into $\BR^3$. Let $\bar M$ be the topological closure of ${\Sigma} \times (-1,1)$ in $\BR^3$. 

 Then $(M, \cN)$ is pseudo-isomorphic to $(\bar M, \cN)$.  
 Applying Theorem \ref{t.splitting} to split $(\bar M, \cN)$ along the topological closure of $c \times (-1,1)$ in $\bar M$, we get the $R$-linear map $\Theta_c: \S_n({\Sigma}) \to \S_n(\Cut_c({\Sigma}))$ defined in the statement. 

From the definition it is clear that $\Theta_c$ is an algebra homomorphism.
\end{proof}

\subsection{Reflection anti-involution} 

For any pb surface $\Sigma$ we have an involution $\tau: \Sigma\times (-1,1) \to \Sigma\times (-1,1),\ 
(x,u) \to (x,-u)$ which maps webs $\alpha$ to $\tau(\alpha)$ (with their framing transformed by the tangent map $\tau_* :T (\Sigma\times (-1,1))\to T (\Sigma\times (-1,1))$. 
Given a stated web $\alpha$ in $\Sigma\times (-1,1)$ let $\bar \alpha$ be $\tau(\alpha)$ with its framing reversed, $f \to -f$.

\bpro For any commutative ring $P$ and $R=P[v^{\pm 1}]$ and for any pb surface $\Sigma$,
there is a unique $P$-algebra anti-involution $\bar \cdot: \SF \to \SF$ such that $\bar v = v^{-1}$ and $\bar\al$ for stated webs $\alpha$ is defined as above.
\epro

We call $\bar \cdot$ the {\bf mirror reflection} map.

\begin{proof} Let $\MN$ be defined as in Subsection \ref{ss.pbsurf} and $\overline {\MN}$ be defined as in Subsection~\ref{ss.symmetries}. Then the mirror reflection map is the composition of the orientation reversion $\vk_{(M,\cN)}$ with $\tau$ and hence, a $P$-linear isomorphism sending $v$ to $\bar v = v^{-1}$ by Theorem \ref{t.framing-rev}. It is easy to see that 
$$\bar{\bar \al} =\al \ \text{and}\ \overline{\alpha\cdot \alpha'}=\bar{\alpha'}\cdot \bar{\alpha}$$
for stated webs $\alpha, \alpha'.$
\end{proof}
If $\al$ is a stated $n$-web diagram over ${\Sigma}$ then $\bar \al$ is obtained from $\al$ by switching all the crossings and reversing the height order on each boundary edge.

\subsection{Embedding of punctured bordered surfaces}
\label{ss.surfacef}
Consider a proper embedding of a pb surface $\Sigma_1$ into $\Sigma_2$. Note that it can map several boundary edges of $\Sigma_1$ into one boundary edge of $\Sigma_2$. For a boundary edge $b$ of $\Sigma_2$ a linear order on the set of boundary edges of $\Sigma_1$ mapped into $b$ is called a {\bf b-order}. Fixing it for each $b$ defines a {\bf height ordered embedding} $f: \Sigma_1 \embed \Sigma_2$, inducing an $R$-module homomorphism 
$f_*: \S_n(\Sigma_1) \to \S_n(\Sigma_2)$, where $f_*(\al)$ is $\al$ with its height order on each $b$ determined by the $b$-order in addition to the height order of $\partial \al$. 
If the $b$-order is given by the positive (respectively negative) orientation of $b$, we say $f_*$ is positively (respectively negatively) induced from $f$, cf. Fig. \ref{f.orderedembed}. 

\begin{figure}[h]
\diagg{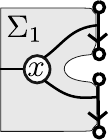}{.5in }\quad \diagg{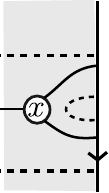}{.7in }\quad  \diagg{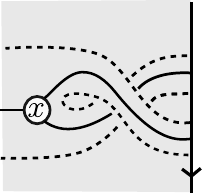}{.7in }
\caption{Left: A surface $\Sigma_1$ with a skein $x$. Middle: The image of $x$ under the negative height order embedding of $\Sigma_1$ into $\Sigma_2$. Right: The image of $x$ under the positive height order embedding of $\Sigma_2$.}
\label{f.orderedembed}
\end{figure}

Note that $f_*$ is an $R$-algebra homomorphism if and only if each boundary edge of $\Sigma_2$ contains the image of at most one boundary edge of $\Sigma_1$.

%
\section{Skein algebras of bigon and quantum groups}
\label{s.quantum-groups}
%

\def\OqM{{\mathcal O_q(M_n)}}

In this section we prove that the stated skein algebra $\S_n(\fB)$ of the bigon, $\fB$, has a natural structure of a co-braided Hopf algebra which is naturally isomorphic to  the quantized coordinate algebra $\Oq$. 
We also show that the stated skein algebra of the monogon, $\fM$, is the ground ring $R$. Finally, we prove Theorem \ref{t.kernelG} which identifies the kernel of the map $\Gamma$.

\subsection{Monogon} 
\label{ss.monobigon}
%

Let $D =\{ (x,y)\in \BR^2 \mid  x^2 + y^2\le 1\}$ be the standard disk with the counterclockwise orientation. The {\bf  monogon} $\fM$ is the pb surface obtained by removing the bottom point $(0,-1)$ from $D$.
 
\bthm[Proof in Subsec. \ref{ss.p.monogon}.]\label{t.monogon} The stated skein algebra $\S_n(\fM)$  of the monogon $\fM$ is isomorphic to the ground ring $R$ via the map 
$\mu: R\to S_n(\fM)$ given by $\mu(r)= r\cdot \emptyset$.
\ethm

%
\subsection{Bigon} 
\label{ss.bigon}

The {\bf  bigon} $\fB$ is the punctured bordered surface obtained from the standard disk $D$ by removing the top and the bottom points, $(0,1), (0,-1)$. The two edges of $\fB$ are denoted by $e_l$ and $e_r$  as in Figure \ref{f.bigon}. 
Up to isotopy there are two orientation preserving auto-diffeomorphisms of $\fB$, the identity and the rotation $\rot$ by 180 degrees about the center of $\fB$. The rotation $\rot$ induces an algebra involution 
$
\rot_* : \S_n (\fB) \to \S_n(\fB).
$

\FIGc{bigon}{(a) \& (b) Bigon $\fB$. (c) Stated arc $a^i_j$, (d) splitting of $\fB$}{2.5cm}

A {\bf directed bigon} is an oriented surface diffeomorphic to $\fB$, with one ideal vertex designated as the {\bf bottom vertex}. 
Equivalently the direction of a bigon can be specified by choosing the left (or right) edge. 
We often depict $\fB$ as the square $[-1,1] \times (-1,1)$, as in Figure \ref{f.bigon}(b). Let $a_{ij}$ be the stated $n$-web over $\fB$ given in Figure \ref{f.bigon}(c), and let $\ca^i_j$ be $a^i_j$ with the reverse orientation.

We will now define a Hopf algebra structure on $\S_n(\fB)$ geometrically.
By splitting $\fB$ along an interior ideal arc connecting its two ideal vertices we get two directed bigons $\fB_l$ and $\fB_r$, for each the  bottom vertex comes from the one of $\fB$. The splitting homomorphism becomes an algebra $R$-homomorphism
$$ \Delta: \S_n(\fB) \to \S_n(\fB) \ot \S_n(\fB).$$
The commutativity of the splitting homomorphisms at disjoint ideal arcs shows that $\Delta$ is a co-product.
For example, from the definition one has 
\beq\label{e.coprod}
 \Delta(a^i_j)=\sum_k a^i_k\otimes a^k_j,\quad \Delta(\ca^i_j)=\sum_k \ca^i_k\otimes \ca^k_j
\eeq

The natural embedding $\iota: \fB \to \fM$ (filling in the top ideal point) induces an $R$-linear map $\iota_{*}: \S_n(\fB) \to \S_n(\fM) $, where the left edge  $e_l$ is higher than the right edge $e_r$. 
 Let $\epsilon: \S_n(\fB) \to R$ be the composition
 \beq\label{e.counit}
 \epsilon: \S_n(\fB) \xrightarrow{ \wt\htw_{e_r}}\S_n(\fB) \xrightarrow{\iota_*} \S_n(\fM)\simeq R,
 \eeq
where $\wt\htw_{e_r}$ is a half-twist automorphism of Subsection \ref{ss.htw}.
 Explicitly, for a stated diagram $\al$,
$$\epsilon\left( \counit\right)= \left( \prod_{j} c_{ i_j}  \right ) \,  \mu ^{-1} \left(\counitb\right),$$
where $\mu: R \to \S_n(\fM)$ is the isomorphism of Theorem \ref{t.monogon}. For example, 
\beq\label{e.counit-a}
\epsilon(a^i_j)=
 \epsilon(\ca^i_j)=c_j\, \,  \counitc
=\delta_{i,j}.
\eeq

\def\antipode{\raisebox{-16pt}{\incl{1.6 cm}{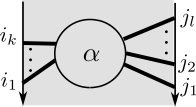}}}
\def\antipodea{\raisebox{-16pt}{\incl{1.6 cm}{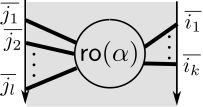}}}
 \def\ro{\mathsf{ro}}
 
Let the $R$-module automorphism $S: \S_n(\fB) \to \S_n(\fB)$ be the composition 
$$S= \rot_* \circ  \htw_{e_l }^{-1} \circ \wt\htw_{e_r},$$
Explicitly,  for a stated diagram $\al$, 
\beq S\left( \antipode \right)= \left( \prod_{m} c_{ i_m } \right) ^{-1} \left(   \prod_{m} c_{ j_m} \right ) \, \antipodea,
\label{e.antip}
\eeq
where $\ro(\al)$ is the result of rotating the planar diagram $\al$ about the center of the square by 180 degrees. (Here, we use the fact that $c_{\bar j}=(-1)^{n-1} q^{2j-n-1}c_j.$)

For example, we have
\beq
S(a^i_j) = (-q)^{i-j}\,\,  \ca^\bj _\bi.
\label{eq.Saij}
\eeq

\brem\label{r.tangles-webs}
Note that stated $n$-tangle diagrams can be identified with the diagrams of stated $n$-webs in $\fB$ with the downwards ascending height order on $\p_l \fB$ and $\p_r \fB,$  i.e. the height order is positive on the left edge but negative on the right edge.

That leads to a natural identification of stated $n$-tangles with stated $n$-webs in the thickened bigon, for which the basic internal, right, and left annihilators of Subsec. \ref{ss.annihilators} correspond to defining skein relations  \eqref{e.pm}--\eqref{e.crossp-wall} of $\S_n(\fB).$

The positive order on the left edge explains why there is a twist in the definition of the operation $\hd$ which turns right annihilators to basic annihilators of Subsection \ref{ss.ltr}.

By this identification $S$ corresponds to the dual operation $\alpha\to \alpha^*$ of Subsection \ref{ss.dualop}. 
\erem

Recall that $\fB$ is the standard bigon.

\bthm[Proof in Subsec. \ref{ss.p.t.Hopf}]\label{t.Hopf} 
(a) The algebra $\S_n(\fB)$ has the structure of a Hopf algebra over $R$ with the coproduct $\Delta$, the counit $\ep$, and the antipode $S$. 

(b) The map $ \Psi (u^i_j) = a^i_j$ extends to a unique Hopf algebra isomorphism  
$$\Psi: \OqR \xrightarrow{\cong} \S_n(\fB) .$$
\ethm
Here $\OqR:= \Oq \ot_\Zv R$ is the algebra $\Oq$ of Subsec. \ref{ss.oqsln} with the ground ring $R$.

%
\subsection{Cobraided structure}

 The Hopf algebra $\OqR$ is {\bf dual quasitriangular} (see  \cite[Section 2.2]{Maj}, 
 \cite[Section 10]{KS}, \cite[Section 10.3]{ES}), also known as {\bf cobraided} (see e.g. \cite[Section VIII.5]{Kass}). This means it has an $R$-form (i.e. a co-$R$-matrix), which is a bilinear form
 $$ \rho : \OqR \otimes \OqR \to R$$ satisfying certain properties, with the help of which one can make the category of $\OqR$-modules a braided category. 
 {}
 The following generalizes \cite[Theorem 3.5]{CL} from  $n=2$ to all $n$:
 
\def\xybar{  \raisebox{-9pt}{\incl{.9 cm}{xybar}} }
\def\yxbar{  \raisebox{-9pt}{\incl{.9 cm}{yxbar}} }
\def\letterxyz{  \raisebox{-10pt}{\incl{1 cm}{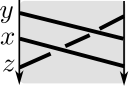}} } 
\def\braido{  \raisebox{-10pt}{\incl{1 cm}{braid1}} }
\def\pcB{\partial \cB}
 
\begin{theorem}[Proof in Subsec. \ref{ss.p.t.cobraid}]\label{t.cobraid} Under the identification of $\S_n(\cB)$ and $\OqR$ via the isomorphism $\Psi$,
 the $R$-form $\rho$ 
 {}
 has the following geometric description
 \begin{align}
 \rho \left (\letterx  \ot \lettery \right) & = \epsilon \left(  \letterxy \right),  \label{eq.cobraid} 
 {}
\end{align} 
for any $x,y\in \OqR.$  
 \end{theorem}

\def\tPhi{\tilde \Phi}
\def\ttPhi{\tilde{\tilde \Phi}}

\subsection{Ground ring} 

The remainder of this section is devoted to proving Theorems \ref{t.monogon}, \ref{t.Hopf}, \ref{t.cobraid} and \ref{t.kernelG}. 
Since it is enough to do it for $R=\Zv$, we will assume this ground ring for the rest of this section.

%
\subsection{Algebra homomorphism $\S_n(\fB)\to \Oq$}

 \blem\label{l.S_n-gens}
The webs $a^i_{j}$ for $i,j\in \{1,\ldots, n\},$ generate $\S_n(\fB)$ as an $R$-algebra. 
 \elem 

\begin{proof} By Lemma \ref{l.span}, $\S_n(\fB)$ is generated by $a^i_j$ and $\ca^i_j$ for 
$i,j=1,\dots ,n.$

Fix $i,j$ and choose a permutation $\tau\in S_n$   with $\tau(1)=\bar i$. By \eqref{e.wallnearcap} and \eqref{e.vertexwall3},
$$\ \hspace*{-.2in} \vertexnearwallcap{<-}{>}{$\tau(n)$}{$\tau(2)$}{$j$}{<}=c_{\bar i}\vertexnearwalledge{<-}{t}{>}{$\tau(n)$}{$\tau(n-1)$}{$\tau(1)$}{$i$}{$j$}{<}=
a\, c_{\bi } \,  (-q)^{\ell(\tau)}\,  \ca^i_j.$$

On the other hand, Eq. 
\eqref{e.vertexnearwall} expresses the left side in terms of $a^k_l$'s:
$$a\sum_{\sigma\in S_n} (-q)^{\ell(\sigma)}\wallnedgewallcap{b}{$\tau(n)$}{$\tau(2)$}{$\sigma(n)$}{$\sigma(2)$}{$\sigma(1)$}{$j$}=
a\cdot c_{\bar j}\cdot \hspace*{-.2in} \sum_{\sigma\in S_n:\, \sigma(1)=\bar j} \hspace*{-.1in} (-q)^{\ell(\sigma)} a^{\tau(2)}_{\sigma(2)}\cdot \dots \cdot a^{\tau(n)}_{\sigma(n)}.$$
 
Since $c_{\bar j} c_{\bi} ^{-1} = (-q)^{j-i} $, by comparing the two equalities, we have
\beq 
\ca^i_j =  (-q)^{j-i} \sum_{\sigma\in S_n: \sigma(1)=\bar j} (-q)^{\ell(\sigma)-\ell(\tau)} a_{\sigma(2)}^{\tau(2)}\cdot \dots \cdot 
a_{\sigma(n)}^{\tau(n)} 
\label{eq.ca}
\eeq
which shows $\ca^i_j$ is in the subalgebra generated by $a^i_j$.
\end{proof}

As the first step towards proving the theorems of this section we will construct an $R$-algebra homomorphism
$\Phi: \S_n(\fB) \to \Oq\subset \Uq^*$.

Let $\al \to T(\al)$ be the bijection of Remark \ref{r.tangles-webs} between the set of isotopy classes of stated $n$-webs over $\fB$ and stated $n$-tangles. It extends to an $R$-algebra isomorphism  $T: R {\mathcal W}_n (\fB) \to R \cT $.

By Remark \ref{r.tangles-webs}, the composition $\Gamma\circ T: R {\mathcal W}_n (\fB) \to \Uq^*$ of $T$ with $\Gamma$ defined in Subsection~\ref{ss.Gamma} preserves all the defining relations of $\S_n(\fB)$. 
Hence, $\Gamma\circ T$ descends to an $R$-linear homomorphism $\Phi: \S_n(\fB) \to \Uq^*$, which  by Proposition \ref{p.Gamma} is an algebra homomorphism. 
From Eq.  \eqref{e.G0} and Proposition \ref{p.Hopfdual} we have 
\begin{align}
\Phi (\emptyset) &= \epsilon, \quad \text{counit of } \ \Uq,   \label{e.eps}\\
\Phi (a^i_j)  &= u^i_j, \quad \text{generators of } \ \Oq, \label{e.aijtou}
\end{align}
for $i,j=1,\dots, n.$ As $u^i_j$ generate $\Oq$, Lemma \ref{l.S_n-gens} and Eq. \eqref{e.aijtou} show that 
$$\Phi(\S_n(\fB)) = \Oq.$$

%
\subsection{Proof of Theorem \ref{t.monogon}} \label{ss.p.monogon}

\def\monogon{\raisebox{-13pt}{\incl{1.3 cm}{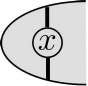}}}
\def\monogona{\raisebox{-13pt}{\incl{1.3 cm}{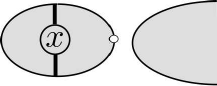}}}

 By  Lemma  \ref{l.span}, $\S_n(\fM)$ is spanned by the empty $n$-web. 
Therefore, the map $ 
 \mu: R \to \S_n(\fM)$ given by $\mu(r)=r\cdot \emptyset
 $
is surjective.  By removing the left edge of $\fB$ we get a monogon. This gives an embedding $\iota: \fM \embed \fB$, which induces an $R$-algebra homomorphism $\iota_*: \S_n(\fM) \to \S_n( \fB)$.  By Eq. \eqref{e.eps} the composition 
$$ R  \xrightarrow{ \mu} \S_n(\fM)\xrightarrow     { \iota_* } \S_n( \fB)   \xrightarrow     { \Phi }  \Uq^*$$
is an $R$-linear map sending 1 to $\ep$. As the free $R$-module generated by $\ep$ is a submodule of $\Uq^*$, the composition is injective.  Thus $\mu$ is injective, and hence, bijective. \qed.

Note that we have 
\beq\label{e.bigonscalar}
\monogon = \monogona,
\eeq
since the skein $x$ can be brought to a scalar with the same skein relations on the left as on the right.

%
\subsection{Proof that $\S_n(\fB)$ is a Hopf algebra}
%
We already noted that $\Delta$ is a coproduct. 
Since $\mu$ is an isomorphism, $\epsilon$ is well defined by Eq. \eqref{e.counit}.
A version of the argument of~\cite{CL} shows that $\epsilon$ is an $R$-algebra homomorphism as well:
for any webs $\alpha_1,\alpha_2,$ 
$$\epsilon(\alpha_1\alpha_2)= \epsilon\left(
\bigontwonodes{<}{<}{b}{$\alpha_2$}{$\alpha_1$}{$\p_l \alpha_2$}{$\p_l \alpha_1$}{$\p_r \alpha_2$}{$\p_r \alpha_1$}\right)=
\prod_{x\in (\alpha_1\cup \alpha_2)\cap \p_r \fB} \hspace*{-.2in} c_{\overline {s(x)}} \cdot 
\mu^{-1}\left(\bigontwonodes{<}{<}{m}{$\alpha_2$}{$\alpha_1$}{$\p_l \alpha_2$}{$\p_l \alpha_1$}{$\overline{\p_r \alpha_2}$}{$\overline{\p_r \alpha_1}$}\right)=$$ 
$$\prod_{x\in \alpha_1\cap \p_r \fB} \hspace*{-.2in} c_{\overline {s(x)}} \cdot \prod_{x\in \alpha_2\cap \p_r \fB} \hspace*{-.2in}  c_{\overline {s(x)}}  \cdot  \mu^{-1}\left(\bigonalpha{<}{}{m}{$\alpha_1$}{$\p_l \alpha_1$}{$\overline{\p_r \alpha_1}$}\right)\cdot \mu^{-1}\left(
\bigonalpha{<}{}{m}{$\alpha_2$}{$\p_l \alpha_2$}{$\overline{\p_r \alpha_2}$}\right)
=\epsilon(\alpha_1)\epsilon(\alpha_2),$$
where $\p_l \alpha, \p_r \alpha$ denote the sequences of left and right side states of $\alpha$ and
the third identity follows from \eqref{e.bigonscalar}.

We also have
$$(\epsilon \otimes id)\circ \Delta(x)= x= (id \otimes \epsilon)\circ \Delta(x)$$
for all $x\in \S_n(\fB)$. Indeed, since $\Delta$ and $\ep$ are algebra homomorphisms, it is enough to verify it for the generators $x= a^i_j$ and that follows  from the explicit values of $\Delta(a^i_j)$ and of $\ep(a^i_j)$ given by Eqs. \eqref{e.coprod} and \eqref{e.counit-a}.

Consequently, $(\S_n(\fB), \Delta,\epsilon)$ is an $R$-bialgebra. By ~\eqref{e.antip}, $S$ is $R$-algebra anti-isomorphism. Therefore, to prove that $S$ is an antipode for $(\S_n(\fB), \Delta,\epsilon)$ it remains to to be shown that 
\beq\label{e.S-property}
\sum S(x_{(1)})x_{(2)} =\ve(x)= \sum x_{(1)} S(x_{(2)}),\quad  \text{where} \ \Delta x=\sum x_{(1)}\otimes x_{(2)},
\eeq
As before, it suffices to be verified for the generators $a^i_j$, $i,j=1,\dots, n,$ only.
Since $\Delta(a^i_j)=\sum_k a^i_k\otimes a^k_j,$ the left side of \eqref{e.S-property} reduces to:
$$\sum_k S(a^i_k)a^k_j=\sum_k \frac{c_k}{c_i}
\bigontwoedges{<}{<}{b}{>}{<}{$k$}{$\bar k$}{$j$}{$\bar i$}
=  \frac{1}{c_i}\bigoncap{}{<}{b}{<}{$j$}{$\bar i$}
=\delta_{i,j}$$
by Eqs. \eqref{eq.Saij}, \eqref{e.wallnearcap}, and \eqref{e.capwall}. The proof of the right identity of \eqref{e.S-property} is analogous.

This completes the proof  that $(\S_n(\fB), \Delta, \epsilon, S)$ is a Hopf algebra.

\def\state{\mathrm{st}}
\def\dl{\partial_\ell}
\def\dr{\partial_r}
\def\ba{{\mathbf a}}
\def\buu{{\mathbf u}}
\def\boal{{\boldsymbol{\alpha}}}

%
\subsection{Proof of Theorem \ref{t.Hopf}}
\label{ss.p.t.Hopf}
%

\bpro\label{p.counit=RT} Suppose $\boal=(\boi, \al, \boj)$ is a stated $n$-web on $\fB$. Then $\ep(\boal)$ is equal to the matrix element of the corresponding modified Reshetikhin-Turaev operator:
\beq \ep (\boal) = \la  \boi \mid  \RTp(T(\al))  \mid \boj \ra.
\label{eq.epsilon}
\eeq
\epro
\begin{proof} The map $\S_n(\fB)\to \Qv$ given by $\boal \to \la \boi \mid  \RTp(T(\al))  \mid \boj \ra$ is clearly an $R$-algebra homomorphism whose values on $a^i_j$ coincide with those of $\ep$, by Eq. \eqref{e.eps}.  \end{proof}

\def\braid{  \raisebox{-10pt}{\incl{1 cm}{braid}} }
\def\braida{  \raisebox{-10pt}{\incl{1 cm}{braida}} }
\def\braidab{  \raisebox{-10pt}{\incl{1 cm}{braidab}} }
\def\braidac{  \raisebox{-10pt}{\incl{1 cm}{braidac}} }
\def\braidbb{  \raisebox{-10pt}{\incl{1 cm}{braidbb}} }
\def\braidbc{  \raisebox{-10pt}{\incl{1 cm}{braidbc}} }
\def\braidz{  \raisebox{-10pt}{\incl{1 cm}{braid0}} }
\def\sinkz{  \raisebox{-17pt}{\incl{1.5 cm}{sinkz}} }

In particular, we have 
\beq \epsilon \left(  
\cross{<-}{<-}{p}{>}{>}{$l$}{$k$}{$j$}{$i$}
\right) = \hR_{ij}^{kl}=  \cR_{ij}^{lk}.
\label{eq.Re}
\eeq
Let us show that $\ba = (a^i_j)$ is a quantum matrix. By isotopy,
$$\crossextr{<-}{<-}{p}{>}{>}{$l$}{$k$}{$j$}{$i$}=\crossextl{<-}{<-}{p}{>}{>}{$l$}{$k$}{$j$}{$i$}$$
Split along the dashed lines (coproduct), then apply the counit
$$  \ep \left(\cross{<-}{<-}{p}{>}{>}{$l$}{$k$}{$r$}{$s$}\right) \cdot \walltwowall{<-}{<-}{>}{>}{$r$}{$s$}{$j$}{$i$}   =  \walltwowall{<-}{<-}{>}{>}{$l$}{$k$}{$r$}{$s$}\ \cdot  \ep\left(\cross{<-}{<-}{p}{>}{>}{$r$}{$s$}{$j$}{$i$}\right),$$
 Using the value of $\cR$ in Eq. \eqref{eq.Re}, the above identity becomes
\be 
\cR (\ba \ot \ba) = (\ba  \ot \ba ) \cR ,  \notag
\ee
which is the defining relation Equ. \eqref{e.OqM} of a quantum matrix. Besides
$${\det}_q(\ba)=\sum_{\sigma\in S_n} (-q)^{\ell(\sigma)}\wallnedgewall{<-}{b}{>}
{$n$}{$2$}{$1$}
{$\sigma(n)$}{$\sigma(2)$}{$\sigma(1)$}
= a^{-1}\,   \left(\vertexnearwalll{<-}{b}{>}{$n$}{$2$}{$1$}
\right)=1$$
where the second equality is from Eq. \eqref{e.vertexnearwall} and the third is from Eq. \eqref{e.vertexwall3}.
Hence, the algebra map $\Psi:\Oq\to \S_n(\fB)$ given by $\Psi(u^i_j)= a^i_j$ is well-defined. 

Since $\Phi \circ \Psi(a^i_j)= a^i_j$,  we have $\Phi \circ \Psi =\id$. This shows $\Phi$ is injective, and hence  $\Phi: \S_n(\fB) \to\Oq $ is an algebra isomorphism. By checking the values of $\Delta, \ep$, and $S$ on the generators $a^i_j$ we see that $\Phi$ is a Hopf algebra homomorphism. 
This completes the proof of Theorem \ref{t.Hopf}.
\qed

%
\subsection{Proof of Theorem \ref{t.cobraid}}
\label{ss.p.t.cobraid}  
%
 
\def\letterx{  \raisebox{-9pt}{\incl{.9 cm}{x}} } 
\def\lettery{  \raisebox{-9pt}{\incl{.9 cm}{y}} }
\def\letterxy{  \raisebox{-9pt}{\incl{.9 cm}{xy}} }
\def\xybar{  \raisebox{-9pt}{\incl{.9 cm}{xybar}} }
\def\yxbar{  \raisebox{-9pt}{\incl{.9 cm}{yxbar}} }
\def\letterxyz{  \raisebox{-10pt}{\incl{1 cm}{xyz}} } 
\def\braido{  \raisebox{-10pt}{\incl{1 cm}{braid1}} }
\def\pcB{\partial \cB}

 \begin{proof} The $R$-form satisfies the following equalities (stated with Sweedler's notation for the coproduct): 
\begin{align}
\rho (xy \ot z)& = \sum \rho (x \ot z') \rho  (y \ot z'')  
\label{eq.cobraid3}\\
\rho (x \ot y z)& = \sum \rho  (x' \ot z) \rho (x'' \ot y).
\label{eq.cobraid4}
\end{align}
For $\Oq$ the values of $\rho$ are given by (see \cite[Section 2.2]{Maj} or \cite[Section 10.1.2]{KS})):
\be 
\rho(u^i_j \ot 1) = \rho(1 \ot u^i_j)= \delta_{ij}, \ \rho(u^i_j \ot u^k_l)= R^{ik}_{jl},  \label{eq.cobraid0}
\ee
which, together with  Relations \eqref{eq.cobraid3} and \eqref{eq.cobraid4}, totally determine $\rho$.

The first part of the proof follows that of \cite{CL}.
Let $\rho'$ be the map defined by the right side of \eqref{eq.cobraid}; we will show that $\rho'=\rho$. It is enough to show that $\rho'$ satisfies \eqref{eq.cobraid3}, \eqref{eq.cobraid4}, and the initial values \eqref{eq.cobraid0}, all with $\rho$ replaced by $\rho'$. 
We have, where a line labeled by, say $x$, stands for the stated $n$-tangle diagram $x$,
$$\rho'(xy \ot z)= \epsilon \left(\letterxyz   \right)$$
Splitting the bigon by the vertical middle ideal arc, then  using $\epsilon(u) =\sum \epsilon(u_{(1)}) \epsilon(u_{(2)}))$,
\begin{align*}
\rho'(xy \ot z) &= \sum \epsilon \left(\diagg{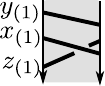}{.5in} \right) \cdot \epsilon \left(\diagg{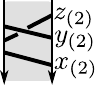}{.5in} \right)   \\
&= \sum \rho'  (x_{(1)} \ot z_{(1)}) \rho'  (y_{(2)} \ot z_{(2)}) \epsilon(x_{(2)}) \epsilon (y_{(1)})=  \sum \rho'  (x \ot z_{(1)}) \rho'   (y \ot z_{(2)}).
\end{align*}
This proves \eqref{eq.cobraid3} for $\rho'$. The proof of \eqref{eq.cobraid4} is similar.

Under the isomorphism, $u^i_j$ becomes $a^i_j$. Using Eq. \eqref{eq.Re}, we have
\begin{align*}
\rho'(a^i_j \ot a^k_l) &= \epsilon \left(  
\cross{<-}{<-}{p}{>}{>}{$i$}{$k$}{$l$}{$j$}
\right) = \hR_{jl}^{ki}=  \cR_{jl}^{ik},   \\
\rho'(a^i_j \ot 1) &=  \rho'(1\ot a^i_j) = \epsilon \left(  a^i_j \right) = \delta_{i,j}
\end{align*}
which proves  \eqref{eq.cobraid0}, completing the proof of the theorem.
 \end{proof}

\def\W{{\mathcal W}}
\subsection{Proof of Theorem \ref{t.kernelG}} \label{s.kernel}
\begin{proof} The proof of Theorem \ref{t.Hopf} shows that the kernel of the map $\Gamma\circ T: R\W_n(\fB) \to \Uq^*$ is generated by internal relations, boundary relations on the right side, and boundary relations on the left side. Transferring back to $\Gamma: R \cT \to \Uq^*$ via the isomorphism
 $T: R\W_n(\fB) \to R \cT$ we conclude that the kernel of $\Gamma$ is generated by the basic internal annihilators, basic right annihilators, and basic left annihilators. 
\end{proof}

%
\subsection{Additional facts}
%
\begin{enumerate}
\item 
Formula \eqref{eq.ca}, after a simple manipulation, has the form
$$\ca^i_j = {\det}_q M^{\bi}_{\bar j}(\mathbf a).$$
where $M^{\bi}_{\bar j}(\mathbf a)$ is the submatrix of $\ba$ obtained my removing the $\bi$-row and $\bar j$-column. Alternatively this formula is a consequence of the (geometric) antipode formula,
 $S(a^i_j)= (-q)^{i-j} \ca^{\bar j}_\bi,$ combined with the (algebraic) antipode formula, \eqref{e-Oq-ops}, in $\Oq$.

\item The antipode, given by \eqref{e.antip}, is equivalent to the dual map of Section \ref{ss.dualop} via $T$: $$ T(S(x)) = (T(x))^*.$$

\item The algebra involution $\rot_*: \Oq \to \Oq$, induced from the rotation by 180 degrees, is a coalgebra anti-homomorphism. It is easy to show that its dual restrict to an algebra anti-involution $\rot^*: \Uq \to \Uq$. One can check that $\rot^*$ is equal the anti-involution $\rho$ introduced by Lusztig \cite[Chapter 19]{Lu-book} in his study of canonical bases of quantized enveloping algebras.
\end{enumerate}

%
\section{Coaction of $\Oq$ on stated skein modules}
\label{s.coaction}

Similarly to the case $n =2$ considered in \cite{CL, BL}, we are going to show that every marking $\beta$ of a marked $3$-manifold $(M,\cN)$ defines a right coaction of $\Oq$ on $\S_n(M,\cN)$. Dually, it defines a left-$\tUL$ module structure on $\S_n\MN$, where $\tUL$ is  a completion of the Lusztig integral version $\UL$ of the quantum group $\Uq.$ We will observe that the actions of the charmed and the half-ribbon elements on $\S_n(M,\cal N)$ coincide with the marking automorphism $g_\beta$ and the half-twist automorphism $\htw_\beta$ of Subsections \ref{ss.edgeauto}-\ref{ss.htw}, respectively.

The above $\Oq$-coaction will be very important for the further development of the theory of stated skein algebras in the remainder of this paper.
For simplicity we assume $R=\Zv$ in this section.

\subsection{Module and Co-module structures}
\label{ss.coaction}
 
Suppose ${\Sigma}$ is a punctured bordered surface and $b$ is a boundary edge. Let $c$ be an interior ideal arc isotopic to $b$. This means that $b$ and $c$ cobound a bigon. By splitting ${\Sigma}$ along $c$ we get a surface ${\Sigma}'$ and a directed bigon with $b$ considered its right edge. As ${\Sigma}'$ is diffeomorphic to ${\Sigma}$ via a unique up to isotopy diffeomorphism, we identify $\S_n({\Sigma}') = \S_n({\Sigma})$. The splitting homomorphism gives an algebra homomorphism
\beq 
\Delta_b: \S_n({\Sigma}) \to \S_n({\Sigma}) \ot \Oq. \label{e.coact}
\eeq
The commutativity of splitting maps and the values of $\ep$ on horizontal stated arcs given by \eqref{e.eps} imply that
 $\Delta_b$ is a right coaction of $\Oq$ on $\S_n({\Sigma})$. Moreover,  the right coactions at different boundary edges commute. Since $\Delta_b$ in Eq. \eqref{e.coact} is an algebra homomorphism,  $\S_n({\Sigma})$ is a right {\em  comodule-algebra} over $\Oq$, as defined in \cite[Section III.7]{Kass}.
 
If we split off a bigon (as above) and identify $b$ with its left edge, we get a left $\Oq$-comodule structure on $\SF$.

The above construction of the $\Oq$-coactions on $\S_n(\Sigma)$ generalizes to $\Oq$-coactions on stated modules of marked 3-manifolds. Given a marking $\beta$ of a marked 3-manifold $\MN$, consider its closed disk neighborhood $D$ in $\pM$, disjoint from the other markings of $\MN$. By pushing the interior of $D$ inside $M$ we get a new disk $D'$ which is properly embedded in $M$. Splitting $\MN$ along $D'$, we get a new marked 3-manifold $(M', \cN')$ isomorphic to $\MN$, and another marked 3-manifold bounded by $D$ and $D'$. The latter,
after removing the common boundary of $D$ and $D'$, is isomorphic to the thickening of the 
bigon, with $\beta$ considered its right face marking, as depicted in Figure \ref{f.cubep}(a).
Hence, this construction yields an $R$-linear splitting map
$$\Delta_\beta: \S_n\MN \to \S_n\MN\otimes \Oq.$$
As in the surface case, this is a right coaction of $\Oq$ on $\S_n\MN$, and the right coactions at different markings commute.

The completion $\tUqq$ of $\Uq$ of \cite{ST}, cf. Subsection \ref{ss.half-ribbon}, has its integral version,
$\widetilde{\UL},$  which contains the half-twist element $X$ of Subsection \ref{ss.half-ribbon}, see \cite[Comment 3.7]{KT}.  Equivalently, this is a completion of the Lusztig integral version $\UL$ of $\Uq$ \cite[Sec. 1.3]{Lu-root1}. The Hopf algebras $\tUL$ and $\Oq$ are in Hopf duality over $\Zv$, which turns 
any right $\Oq$-comodule $W$ to a left $\tUL$-module as follows: for $u\in \tUL$ and $ x\in W$,
$$ u * x =  \sum x_{(1)} \la f_{(2)}, u\ra,  \quad \text{where}\  \Delta(x) =\sum x_{(1)} \ot f_{(2)}$$ is the $\Oq$-coaction map.

To make explicit the left action of $ \tUL$ on $\S_n\MN$ coming to the right coaction $\Delta_\beta$ we extend the states of an $n$-web at  marking $\beta$ as follows.
Suppose $\al$ is an $n$-web in $\MN$ with the sign sequence on $\beta$ equal to $\boeta=(\eta_1,\dots, \eta_k)\in \{\pm \}^k$. The set $\{ v_{\boi} \mid \boi \in \{1,\dots, n\}^k\}$ is the  $R$-basis of the based module $V^\boeta$. Assume $\al$ is stated at all markings except  $\beta$. For $x\in V^\boeta$ let $(\al, x)\in \S_n\MN$ be defined so that  $(\al, v_\boi)$ is $\al$ with states $\boi$ on $\beta$, and the map $x \to (\al, x)$ is $R$-linear. From the definition we have
\beq 
u * (\al,x) = (\al, ux)  \label{e.act0}
\eeq

\bexa The action of the charmed element $g$ on $\S_n\MN$ is exactly the map $g_\beta$ of Subsection \ref{ss.edgeauto}. In fact, $g$ is a group-like element, i.e. $\Delta^{[k]} (g) = g ^{ \ot k}$ for $k=1,2, \dots,$ and the actions of $g$ on the based $\Uq$-modules $V$ and $V^*$ are given by the same diagonal matrix with $g_1,\dots, g_n$
on the diagonal, see \eqref{e.gi}. That ensures that the action of $g$ on $\S_n\MN$ coincides with the map $g_\beta$.
\eexa

\def\htwx{\raisebox{-25pt}{\incl{2cm}{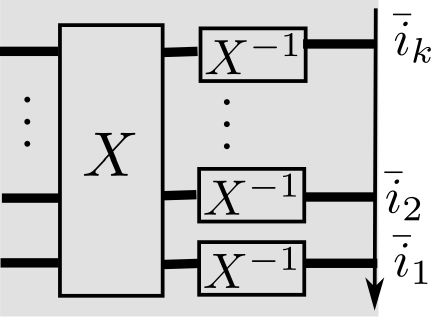}}}
\def\htwy{\raisebox{-25pt}{\incl{2cm}{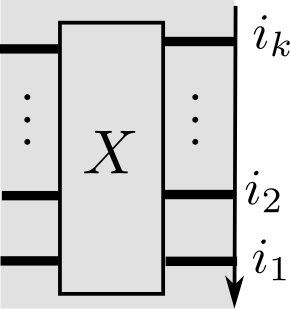}}}

\bexa The action of the half-ribbon element $X$ is the half-twist homomorphism $\htw_\beta$ of Subsection \ref{ss.htw}. Indeed, for the positive half-twist $H$ on $k$ strands (with arbitrary orientations) by \eqref{e.XXX} we have 
$$\RTp(H) = \Delta^{[k]}(X) \circ (X^{-1})^{\ot k} \circ \rev_k.$$ 
By applying this value of $\RTp(H)$ to \eqref{e.htw} we obtain
$$ \htw_{\beta}   \left(  \nedgewalltall{<-}{b}{}{${i_k}$}{${i_2}$}{${i_1}$}\right)
  = \left( \prod_{j=1}^k c_{\overline{i_j}}\right) \cdot \htwx = \htwy,
$$
where the second identity follows from \eqref{e.Xij}. This proves the statement.
\eexa

Formula \eqref{e.act0} makes it easy to study $\S_n\MN$ as an $\tUL$-module. For example, one can show that over the field $\Qv$ the $\Uq$-module $\S_n\MN\ot_\Zv \Qv$ is integrable and is a direct sum of finite-dimensional simple $\Uq$-modules. For the case $n=2$ see \cite{CL}.

\subsection{Boundary relations revisited} 
\label{ss.Drel}

\def\rrel{\raisebox{-25pt}{\incl{2cm}{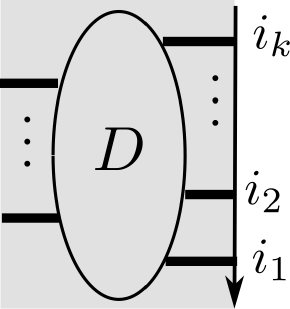}}}
\def\rrela{\raisebox{-25pt}{\incl{2cm}{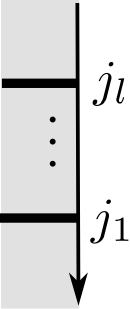}}}
\def\rrelb{\raisebox{-25pt}{\incl{2cm}{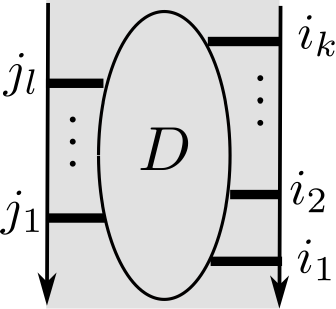}}}
\def\boD{{\bm{D}}}

Let $\boD=(D,\boi)$ be an $n$-web diagram $D$ over the bigon $\fB$ right stated by
$\boi=(i_1,\dots,i_k)$. 
Assume that $D$ has $l$ left endpoints, which are not stated.
Suppose further that $\al$ is a stated $n$-web in a marked 3-manifold $\MN$ and that in a cube $Q$ which intersects $\cN$ at a subinterval of a marking $\beta$ the intersection $\al\cap Q$ has diagram equal to $\boD$, as in the left side of \eqref{e.rrel}.  
\text{By the property of the counit of the coaction,} we have 
\beq
\rrel  = \sum _{j_1,\dots, j_l} \rrela \cdot \ep \left(   \rrelb \right).
\label{e.rrel}
\eeq
This identity provides a local relation in any stated skein module, called the {\bf $\boD$-relation}.
By \eqref{eq.epsilon}, the values of $\ep$ of stated $n$-webs are the entries of the matrix describing the Reshetikhin-Turaev operator $\RTp(D)$ and are not difficult to calculate. All the boundary relations \eqref{e.vertexnearwall}-\eqref{e.crossp-wall} are of this type. Since Relations \eqref{e.pm}-\eqref{e.crossp-wall} are sufficient for defining the $\Oq$-coation on $\S_n(\MN)$, any $\boD$-relation is a consequence of these relations.

\def\heightinvbb{\raisebox{-12pt}{\incl{1.2 cm}{height_inv_bb}}}
\def\heightinvc{\raisebox{-12pt}{\incl{1.2 cm}{height_inv_c}}}
\def\heightinvp{\raisebox{-12pt}{\incl{1.2 cm}{height_inv_p}}}
\def\heightinvbbb{\raisebox{-12pt}{\incl{1.2 cm}{height_inv_bbb}}}

\bexa For $\boD= \crosswall{<-}{n}{white}{black}{$i$}{$j$}$, the $\boD$-relation is
$$\crosswall{<-}{n}{white}{black}{$i$}{$j$}
=   q^{-\frac{1}{n}}  \left (   q^{\delta_{\bar j i} }   
\twowall{<-}{white}{black}{$j$}{$i$}
+   \delta_{\bar j i} \sum_{k > i} (-q)^{k-i}(q-q^{-1}) 
\twowall{<-}{white}{black}{$\bar k$}{$k$}
\right) .$$
\eexa

\bexa The following relations for $n=3$ will be useful later:
\begin{align}
\twoonetwowall{->}{>}{>}{}{$i$}{$j$}
 & = q^{\delta_{i,j}} \twowall{->}{>}{>}{$j$}{$i$} - \twowall{->}{>}{>}{$i$}{$j$}\cdot \begin{cases} q & \text{for}\ i\geq j\\
q^{-1} & \text{for}\ i< j\\ \end{cases}\label{e.Hib0} \\
\forkwall{->}{black}{black}{white}{$i$}{$j$}& = \delta_{i\ne j}(-q)^{\delta_{i>j}} q^{-\frac76}  \edgewall{}{white}{$i+j-2$}  \label{e.Hib1} \\
 \Hwall{<-}{white}{black}{}{$i$}{$i$} & = \twowall{<-}{white}{black}{$i$}{$i$}, \ \text{for } \ i=1,3. \label{e.FSboundary0}
 \end{align}
\eexa

%
\subsection{The last among the defining skein relations}
%
\bpro\label{p.crossingrels}
If $[n-2]!$ is invertible in $R$ then the last defining relation, \eqref{e.crossp-wall}, is a consequence of the other defining relations 
\eqref{e.pm}-\eqref{e.capnearwall}.
\epro
\begin{proof}
\def\OqR{{\mathcal O_q(sl_n;R)}} Note that Relation \eqref{e.crossp-wall} is the $\boD$-relation for $\boD= \crosswall{<-}{p}{white}{white}{$i$}{$j$}$.

Since $[n-2]!$ is invertible, Identity  \eqref{e.crossing} makes it possible to eliminate all crossings in every diagram, 
\beq 
\cross{n}{n}{p}{>}{>}{}{}{}{}
 = q^\frac{n-1}{n}\walltwowall{n}{n}{>}{>}{}{}{}{} - (-1)^{\binom n2} \frac{q^{-\frac{1}{n}}}{[n-2]!}
\twoonetwowall{n}{>}{>}{\hspace*{-.05in}$n$-$2$}{}{},
\label{e.cr2}
\eeq

For the purpose of this proof, let $\S_n'(M,\cN)$ be defined as $\S_n\MN$, only without Relation \eqref{e.crossp-wall}. 
The Splitting Theorem holds for $\S_n'(M,\cN)$, as using \eqref{e.cr2} we do not have to consider the invariance of the splitting homomorphism under moving a crossing through the splitting disk.
Lemma \ref{l.span} holds for $\S_n'(\fB)$ because all crossings of diagrams on $\fB$ can be eliminated. Consequently, Lemma \ref{l.S_n-gens} holds as well and the proof of the isomorphism 
$\S_n(\fB) \simeq \OqR $  extends to an isomorphism 
$$\S'_n(\fB) \xrightarrow{\simeq} \S_n(\fB)\xrightarrow{\simeq} \OqR. $$

\def\braidz{  \raisebox{-20pt}{\incl{1.8 cm}{braidz}} }
\def\braidzz{  \raisebox{-20pt}{\incl{1.8 cm}{braidzz}} }
\def\braidzzz{  \raisebox{-20pt}{\incl{1.8 cm}{braidzzz}} }
\def\braidzzzz{  \raisebox{-20pt}{\incl{1.8 cm}{braidzzzz}} }

Furthermore, every marking $\beta$ of $\MN$ defines a right coaction 
$\Delta_\beta': \S'_n\MN \to \S'_n\MN \ot_R \OqR$ as in previous subsection. Using the coaction, one sees that for every right stated $n$-web diagram $\boD=(D,\boi)$ on $\fB$, the  relation \eqref{e.rrel} is  a consequence of the defining relations for $\S'_n\MN$. In particular, for $\boD= \crosswall{<-}{p}{white}{white}{$i$}{$j$}$, we get 
the statement of the proposition.
\end{proof}

\section{Algebraic structure of skein algebras}

\subsection{Glueing over an ideal triangle} 
\label{ss.braidedtensor} 

\def\ufS{\underline{{\Sigma}}}
\def\OSL{\Oq}
\def\CSS{\S_n({\Sigma})}

\def\PP{{\mathcal P}}
\def\fl{{\rev}}

The standard ideal triangle $\fT\subset \BR^2$ is the closed triangle with vertices $(-1,0),(1,0)$ and $(0,1)$
 with these vertices removed.  We will denote its sides by $e_1, e_2$, and $\p_b \fT$ as in Figure \ref{f.brtensor0}. Suppose $a_1, a_2$ are two distinct boundary edges of a (possibly disconnected) pb surface $\Sigma$.
 Define 
$$\Sigma_{a_1\triangle a_2} = (\Sigma \sqcup \fT)/(e_1 =a_1, e_2 = a_2),$$ 
as in Figure\ref{f.brtensor0}. Define the $R$-linear homomorphism $\glue_{a_1,a_2}:\S_n(\Sigma) \to \S_n(\Sigma_{a_1\triangle a_2})$ so that if $\al$ is  a stated $n$-web diagram over ${\Sigma}$ with the negative height order on both $a_1$ and $a_2$ then $\glue_{a_1,a_2}(\al)$ is the result of  continuing the strands of $\al$ with endpoints on $a_1$  and $a_2$ until they reach $\p_b \fT $, as in Figure  \ref{f.brtensor0} (right). (As usual, the arrows indicate the height order.) 

\begin{figure}[htpb]
    \includegraphics[height=2.8cm]{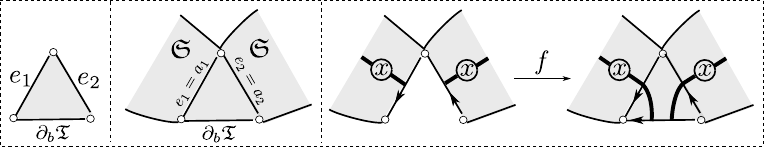} 
\caption{Left: The standard ideal triangle $\fT$. Middle: Glueing $\Sigma$ and $\fT$ by $a_1=e_1$ and $a_2=e_2$ to get $\Sigma_{a_1\triangle a_2}$. Right: tangle diagram $x\in \S_n(\Sigma)$ and its image $\glue_{a_1,a_2}(x)\in \S_n(\Sigma_{a_1\triangle a_2})$}
\lbl{f.brtensor0}
\end{figure}

We are going to show that $\glue_{a_1,a_2}$ is a linear isomorphism and to construct its inverse, $\Cut_{a_1,a_2}.$ 

Let $$\ve_{\fT}: \S_n(\fT)\to \S_n(\fM)=R,\quad \ve_{\fT}=\ve\circ \text{fill}_*$$
where fill embeds $\fT$ into a bigon by filling in the top vertex of $\fT$ and making the web ends at $e_1$ higher than those at $e_2.$ Let us consider also the homomorphism
$$\Cut_{a_1,a_2}:\S_n(\Sigma_{a_1\triangle a_2})\to \S_n(\Sigma),\quad 
\Cut_{a_1,a_2}=(\ve_{\fT}\otimes id_{\S_n(\Sigma)})\circ \Theta_{a_1}\circ \Theta_{a_2}.$$

\bpro\label{p.tri-glue} 
The homomorphisms $\Cut_{a_1,a_2}$ and $\glue_{a_1,a_2}$ are inverses of each other:
$$\Cut_{a_1,a_2}\circ \glue_{a_1,a_2}=id_{\S_n(\Sigma)}\ \text{and}\ \glue_{a_1,a_2}\circ \Cut_{a_1,a_2}=id_{\S_n(\Sigma_{a_1\triangle a_2})}$$
\epro

\bpr We follow Higgins' proof for $n=3$.
It is enough to check the above identities for the diagrams over $\Sigma$ and over $\Sigma_{a_1\triangle a_2}$ respectively.
 We have
\beq\label{e.triangle}
\ve_{\fT}\left(\diagg{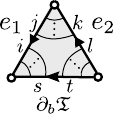}{.7in}\right)=\prod_{a=1}^{|\boj |} \left(\delta_{j_a,\overline{k_a}} c_{j_a}^{-1}\right)\cdot
\left(\prod_{b=1}^{|\boi |} \delta_{i_b, s_b}\right)\cdot \left(\prod_{d=1}^{|\bf l |} \delta_{l_d,t_d} \right),
\eeq
where the orientations of arcs in the triangle are arbitrary and $\boi, {\bf s}$ and $\boj, {\bf k}$ and ${\bf l}, {\bf t}$ are any three pairs are state sequences of equal length,  $|\boi |=|{\bf s}|,$ $|\boj |=|{\bf k}|,$ and $|{\bf l} |=|{\bf t}|.$ 

Let us call diagrams over $\fT$ which are like in Eq. \eqref{e.triangle} without the horizontal arcs, {\bf vertical}.
Since for every $x\in \S_n(\Sigma_{a_1\triangle a_2})$, the skein $\Theta_{a_1}\circ \Theta_{a_2}\circ \glue_{a_1,a_2}(x)$ is a linear combination of diagrams which are vertical on $\fT$, 
the first identity follows.

To prove the second identity observe that each diagram in $\S_n(\Sigma_{a_1\triangle a_2})$ can be positioned so that it intersects $\fT$ in disjoint arcs only. By applying Relation \eqref{e.capnearwall}, it can be presented as a linear combination of diagrams which are vertical on $\fT$. 
By Eq. \eqref{e.triangle}, for each of them we have
$$\glue_{a_1,a_2}\circ (\ve_\fT(W)\otimes id_{\S_n(\Sigma)})\circ \Theta_{a_1}\circ \Theta_{a_2}(W)=W.$$
Hence, the right identity of Proposition \ref{p.tri-glue} follows.
\epr

{
Note that the bijective map $\glue_{a_1,a_2}$ is not an algebra isomorphism. 
In fact, $\glue_{a_1,a_2}(yx)$ and $\glue_{a_1,a_2}(y) \glue_{a_1,a_2}(x) $ are depicted in 
Figure \ref{f.brtensorxy}, where $f= \glue_{a_1,a_2}$.

However, we are going to show that it is one with respect to the {\bf self braided tensor product} 
which we will define right now. The $n=2$ version of this product was considered in  \cite{CL}.

There are two right $\Oq$-comodule algebra structures on $\S_n(\Sigma)$ given by
$$\Delta_i:= \Delta_{a_i}: \S_n(\Sigma) \to \S_n(\Sigma)\ot \OSL,\quad i=1,2,$$
which  commute.

Define the $R$-linear map
$\bD:  \S_n(\Sigma)  \to \S_n(\Sigma) \ot \Oq$  by
$$\bD(x) = \sum x_{(1)} \ot u_{(2)} u_{(3)},$$
in Sweedler's notation, where
$$(\Delta_1\otimes \Id_\Oq)\circ \Delta_2(x) = \sum x_{(1)} \ot u_{(2)} \ot u_{(3)}.$$
For $x,y\in \S_n(\Sigma)$ define a new product by
\be \label{eq:selfproduct}
y\uast x = \sum  y_{(1)} x_{(1)}  \rho(u_{(2)} \ot w_{(2)}) 
\ee
where
$$\Delta_2 (y) =\sum y_{(1)} \ot u_{(2)},\quad  \Delta_1(x) =\sum x_{(1)} \ot w_{(2)},$$
and $\rho$ is the $R$-form.

It is proved in \cite{CL} that $\bD$ and $\uast$ together give $\SF$ a right $\Oq$-comodule algebra structure for $n=2$. That proof extends verbatim to all $n$. 
Denote by $\uot \S_n(\Sigma)$ the $R$-module $\S_n(\Sigma)$ with this $\Oq$-comodule algebra structure. On the other hand  $\S_n(\Sigma_{a_1\triangle a_2})$ has a right $\Oq$-comodule algebra structure coming from the boundary edge $\partial_b \fT$. Here is a stronger version of Proposition~\ref{p.tri-glue}.

\def\brtone{  \raisebox{-13pt}{\incl{1.3 cm}{brtensor1}} }
\def\brttwo{  \raisebox{-15pt}{\incl{1.5 cm}{brtensor2}} }
\def\brtthree{  \raisebox{-13pt}{\incl{1.3 cm}{brtensor3}} }
\def\brtfour{  \raisebox{-8pt}{\incl{.8 cm}{brtfour}} }
\def\brtonenew{  \raisebox{-13pt}{\incl{1.3 cm}{brtensor1-new}} }
\def\brttwonew{  \raisebox{-15pt}{\incl{1.5 cm}{brtensor2-new}} }

\bthm \label{r.braidedtensor}
The maps $\glue_{a_1,a_2} : \uot\, \S_n(\Sigma) \to \S_n({\Sigma}_{a_1 \triangle a_2})$,
$\Cut_{a_1,a_2}:\S_n(\Sigma_{a_1\triangle a_2})\to \uot\, \S_n(\Sigma)$ are
isomorphisms of right  $\Oq$-comodule algebras.
\ethm

\begin{proof} The geometric proof of \cite{CL} for $n=2$ carries over to all $n$ without modification. Here is a sketch. It is enough to show that $f= \glue_{a_1,a_2}$ is an algebra homomorphism.
 Let $x,y$ be stated $n$-web diagrams.

\FIGc{brtensorxy}{Diagrams of $yx,f(yx)$, and $f(y)f(x)$}{2.5cm}
We present $yx, f(yx), f(y)f(x)$ schematically as in Figure \ref{f.brtensorxy}. By splitting along the dashed line in the picture of $f(y)f(x)$ and by using the counit property,
$$ f(y) f(x) = \sum \diagg{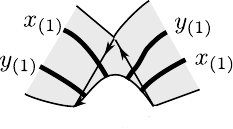}{.6in}
\ \ot \  \ep\left( \diagg{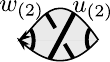}{.3in}
\right),$$ 
where $$\Delta_2(y)= \sum y_{(1)} \ot u_{(2)}, \ \text{and}\ \Delta_1(x)= \sum x_{(1)} \ot w_{(2)}.$$
The above equals
$\sum \,  f( y_{(1)} x_{(1)})\rho (u_{(2)} \ot  w_{(2)}),$
and, by  \eqref{eq:selfproduct}, it reduces to $f(y\uast x).$
Thus $f$ is an algebra homomorphism.
\end{proof}

A special case is when $\Sigma=\Sigma_1 \sqcup\Sigma_2$ and $a_i \subset \Sigma_i$ for $i=1,2$. In this case we say that ${\Sigma}_{a_1 \triangle a_2}$ is the result of gluing $\Sigma_1$ and $\Sigma_2$ over the triangle. 
Each $\S_n(\Sigma_i)$ is a right $\Oq$-comodule algebra via the coaction coming from the edge $a_i$. Then $\uot(\S_n(\Sigma))$ is the well-known {\em braided tensor product} $\S_n(\Sigma_1)$ and $\S_n(\Sigma_2)$ of the two $\Oq$-module algebras $\S_n(\Sigma_1)$ and $\S_n(\Sigma_2)$,  defined in \cite[Lemma 9.2.12]{Maj}.

\def\boa{{\bm{a}}}
\begin{example}[Ideal triangle]\label{eg.idealtriangle} Let $\Sigma_1= \Sigma_2=\fB$, where $a_1$ is the right edge of $\Sigma_1$ and $a_2$ is the left edge of $\Sigma_2$. Then ${\Sigma}_{a_1 \triangle a_2}$ is the triangle $\fT$. Hence, we have 
$$\S_n(\fT) \cong \Oq \uot \Oq,$$ 
where each copy of $\Oq$ is a right $\Oq$-comodule algebra via the coproduct. 
 From here one can easily write down an explicit presentation of the algebra $\S_n(\fT)$. Such presentation is used in the work \cite{LY2} on the quantum trace for stated $SL_n$-skein algebras.
\end{example}

Let $\Sigma_{g,p+1}$ be a $p$-punctured genus $g$ surface with a single loop boundary and let $\Sigma_{g,p}^*$ be $\Sigma_{g,p}$ with a boundary point removed. 

\bexa\label{ex.deltasum}
Let $\mathfrak S=\{\Sigma_{g,p}^*, g\geq 0, p\geq 1\}$ be the set of pb surfaces with a single arc boundary, considered up to a homeomorphism. For $\Sigma_1,\Sigma_2\in \mathfrak S,$
let $\Sigma_1\triangle \Sigma_2$ be the result of gluing over a triangle along $a_1=\p \Sigma_1$ and $a_2=\p \Sigma_2$, as above. Note that the $\triangle$ operation makes $\mathfrak S$ into a monoid with the identity $\fM$. 
\begin{figure}[h] 
   \centering \diagg{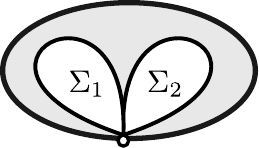}{0.6in} 
   \caption{The operation $\Sigma_1\triangle \Sigma_2$ on two surfaces with a single arc boundaries.}
   \label{f.deltasum}
\end{figure}
Theorem \ref{r.braidedtensor} implies that for any $\Sigma_1,\Sigma_2\in \mathfrak S,$ the algebra $\S_n(\Sigma_1\triangle \S_n)$ is the braided tensor product $\S_n(\Sigma_1) \uot \S_n(\Sigma_2).$ 

Therefore, 
$\S_n(\Sigma_{g,p}^*)$ is the braided tensor product of $p-1$ copies of $\S_n(\Sigma_{0,2}^*)$ and $g$ copies of $\S_n(\Sigma_{1,1}^*)$. 
We will analyze $\S_n(\Sigma_{0,2}^*)$ and $\S_n(\Sigma_{1,1}^*)$ in detail in Subsections \ref{ss.transmut} and \ref{ss.torus} below and we will see in particular that 
$$\S_n(\Sigma_{0,2}^*)\simeq \Oq\ \text{and}\ \S_n(\Sigma_{1,1}^*)\simeq \Oq^{\ot 2}$$ 
as $R$-modules.
Consequently,
$$\S_n(\Sigma_{g,p}^*)\simeq \Oq^{\otimes (p-1+2g)}$$ 
as an $R$-module. (A version of this formula appeared in \cite{BBJ}.)
This statement will be generalized by Theorem \ref{t.Aisom}.
\eexa

\subsection{Punctured Monogon and Majid's Transmutation} 
\label{ss.transmut} 

By attaching an ideal triangle to the bigon $\fB$ along its left and right edges $e_l$ and $e_r$,
as in Figure \ref{f.b_trigon}, we obtain the {\bf once-punctured monogon}, $\Sigma_{0,2}^*={\fB}_{a_1 \triangle a_2}$. 
\FIGc{b_trigon}{From the bigon to a punctured monogon by gluing over a triangle.}{2cm}

An algebraic description of the product on $\S_n(\Sigma_{0,2}^*)$ can be derived from that for $\Oq$ by the rule described in Eq. \eqref{eq:selfproduct}.  This allows to identify $\S_n(\Sigma_{0,2}^*)$ with Majid's transmutation of $\Oq$, as we explain now.

\def\tm{{\mathsf{tm}}}
Let $\tm$ be the composition of the inverse of the half-twist around the left edge with the above triangle gluing map,
\beq\label{e.tm}
\tm=\glue_{e_l,e_r}\htw_{e_l}^{-1}: \Oq=\S_n(\fB)\to \S_n(\Sigma_{0,2}^*).
\eeq
This map can be visualized as follows:  let a stated web $x$ in $\fB$ be represented by a diagram \bigonalpha{<}{<}{b}{$x$}{$\bf i$}{$\bf j$}, where the left horizontal line represents multiple horizontal edges (of possibly different directions) whose ends on the left are labeled by
$\bf i=\bmat
i_s\\ \vdots\\ i_1
\emat$ 
and, similarly, the right horizontal line represents multiple horizontal edges whose ends on the right are labeled by  $\bf j=\bmat
j_t\\ \vdots\\ j_1
\emat.$ 
Then $\htw_{e_l}^{-1}(x)=\frac1{\prod{c_{\bf i}}}\bigonalpha{>}{<}{b}{$x$}{$\overline{\bf i}$}{$\bf j$}$ 
and $\tm(x)=\frac1{\prod{c_{\bf i}}}\bigonpunctc{<}{<}{m}{$x$}{$\overline{\bf i}$}{$\bf j$}$.
By Propositions \ref{p.twist} and \ref{p.tri-glue}, $\tm$ is an $R$-linear isomorphism.
We will prove that this map defines Magid's {\em transmutation} on $\Oq$. (That was the reason for denoting the above map by $\tm$.)

\def\cH{{\mathcal H}}

Let us recall that notion first: for every Hopf algebra $\cH$, 
Majid proved that
\beq\label{e.adjco} 
Ad: \cH\to \cH\ot \cH,\quad Ad(x)=\sum x_{(2)}\otimes (Sx_{(1)})x_{(3)},
\eeq
defines a coaction on $\cH$ on itself \cite[Example 1.6.14]{Maj}, called the adjoint $\cH$-coaction and that there is
an associative {\em braided product} (or, {\em covariantised})
\beq\label{e.braidedprod}
x\underline{\cdot}\, y=\sum x_{(2)}y_{(2)}\rho((Sx_{(1)})x_{(3)}\otimes Sy_{(1)}),
\eeq
cf. \cite[Example 9.4.10]{Maj}. (This product should not be confused with the braided tensor product of Majid, which we discussed in Subsec. \ref{ss.braidedtensor}.)
Furthermore, he showed that $\cH$ with the braided product and the adjoint $\cH$-coaction is an $\cH$-comodule coalgebra. With these structures, $\cH$ is denoted by $\mathcal TH$ called the {\em transmutation} of $\cH$.

\bpro\label{p.tm}
(1)  $\tm$ is an isomorphism between the $\Oq$-comodule algebra $\cT\Oq$ 
and $\S_n(\Sigma_{0,2}^*)$. 
Hence, $(\tm \ot I)\circ Ad = \Delta_{\p \Delta_{0,2}^*} \circ \tm$.\\
(2) $\tm$ is a ring isomorphism. Hence, $\tm(x)\tm(y)=\tm(x\underline{\cdot}\, y).$
\epro

This statement was proved for $n=2$ in \cite{CL}.

\bpr[Proof of Prop. \ref{p.tm}](1)
Let a stated web in $\fB$ be given by the diagram \hspace*{-.1in}\bigonalpha{<}{<}{b}{$x$}{$\bf i$}{$\bf j$}, as above. 
Then $\Delta^2(x)=\sum_{\bf k, l} \bigonedge{<}{<}{b}{}{$\bf i$}{$\bf k$}\ot \bigonalpha{<}{<}{b}{$x$}{$\bf k$}{$\bf l$}\ot 
\bigonedge{<}{<}{b}{}{$\bf l$}{$\bf j$}$, where as above, horizontal arcs indicate multiple edges of possibly different directions. Consequently, by Eqs. \eqref{e.adjco} and \eqref{e.antip},
\beq\label{e.Adxbigon}
Ad\left( \bigonalpha{<}{<}{b}{$x$}{$\bf i$}{$\bf j$} \right)=\sum_{\bf k, l} \frac{\prod c_{\bf k}}{\prod c_{\bf i}}  \bigonalpha{<}{<}{b}{$x$}{$\bf k$}{$\bf l$}\otimes \bigontwonodes{<}{<}{b}{-}{\footnotesize{Ro}}{$\bf l$}{$\wh{\bf k}$}{$\bf j$}{$\wh{\bf i}$},
\eeq
where Ro denotes $180^\circ$ rotation in plane, $\hat{\bf i}$ denotes conjugation of components and the inversion of their order: $\wh{\bf i}=\bmat \overline{i_1}\\ \vdots\\ \overline{i_s}\emat$ for $\bf i=\bmat
i_s\\ \vdots\\ i_1\emat,$ and $\prod c_{\bf i}$ denotes $\prod_{t=1}^s c_{i_t}$. 

On the other hand, we have
$$\Delta_{\p \Sigma_{2,0}^*}\left(\bigonpunctc{<}{<}{m}{$x$}{$\overline{\bf i}$}{$\bf j$}\right)=
\sum_{\bf k, l} \bigonpunctc{<}{<}{m}{$x$}{$\overline{\bf k}$}{$\bf l$}\otimes \bigontwonodes{<}{<}{b}{-}{\footnotesize{Ro}}{$\bf l$}{$\wh{\bf k}$}{$\bf j$}{$\wh{\bf i}$}$$
and, hence,
$$\Delta_{\p \Sigma_{2,0}^*}\left( \frac1{\prod{c_{\bf i}}}
\bigonpunctc{<}{<}{m}{$x$}{$\overline{\bf i}$}{$\bf j$}\right)=
\sum_{\bf k, l} \frac{\prod c_{\bf k}}{\prod c_{\bf i}} \frac1{\prod{c_{\bf k}}}
\bigonpunctc{<}{<}{m}{$x$}{$\overline{\bf k}$}{$\bf l$}\otimes \bigontwonodes{<}{<}{b}{-}{\footnotesize{Ro}}{$\bf l$}{$\wh{\bf k}$}{$\bf j$}{$\wh{\bf i}$}.$$
Since this equality coincides with \eqref{e.Adxbigon} after replacing $x$ with $\tm(x)$ and $Ad$ by $\Delta_{\p \Sigma_{2,0}^*}$,
the statement follows. 

(2) Consider stated webs $x=\bigonalpha{<}{<}{b}{$x$}{$\bf i$}{$\bf j$}$ and  $y=\bigonalpha{<}{<}{b}{$y$}{$\bf k$}{$\bf l$}$ in $\fB.$
Then
\beq\label{e.tmxy}
 \tm(xy)=\tm\left(\bigontwonodes{<}{<}{b}{\small{$y$}}{$x$}{$\bf k$}{$\bf i$}{$\bf l$}{$\bf j$}\right)= \frac1{c_{\bf i}c_{\bf k}}
\diagg{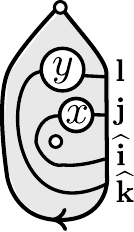}{1in}.
 \eeq
On the other hand,  
\beq\label{e.tmxy2} 
\tm(x)\tm(y)=\frac1{c_{\bf i}c_{\bf k}} \diagg{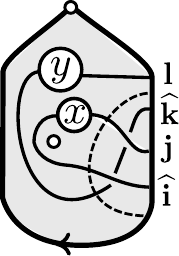}{1in}.
\eeq
Denoting the stated web diagram on the right by $z\in\S_n(\Sigma_{0,2}^*)$, we have 
\beq\label{e.Deltaepsilon}
\notag
 z=\sum z_{(1)}\ve(z_{(2)}), \  \text{where}\ \Delta(z)= \sum z_{(1)}\ot z_{(2)}\in \S_n(\Sigma_{0,2}^*)\ot \Oq
 \eeq 
 is the coaction along the dashed line and $\ve$ is the counit in $\Oq$. By applying this identify to 
\eqref{e.tmxy2} we obtain
$$\tm(x)\tm(y)=\frac1{c_{\bf i}c_{\bf k}} \sum \tm(x_{(2)} y_{(2)}) T(x_{(1)},x_{(3)}, y_{(1)}),$$
where $T(x_{(1)},x_{(3)}, y_{(1)})=\rho((Sx_{(1)})x_{(3)}\otimes Sy_{(1)}),$ by Theorem \ref{t.cobraid}.
Consequently, 
$$\tm(x)\tm(y)=\tm(x\underline{\cdot}\, y),$$
by \eqref{e.braidedprod}.
\epr

Consequently, our theory provides simple geometric proofs of the associativity of the (braided) product on $\cT\Oq$ and of $\cT\Oq$ being an $\Oq$-comodule algebra. (The proofs of these facts are quite technical and involved in \cite{Maj}.) Furthermore, our theory generalizes these statements to the boundary $\Oq$-coaction on the skein algebra of any essentially bordered punctured surface.

Let us discuss generators and relations of $\cT\Oq$ now. 
A {\em reflection equation algebra} $A_q(M_n)$ is an $R$-algebra generated by formal variables $x_{ij}$ for $i,j=1,...,n$ subject to the quadratic relations of the {\em reflection equation}:
\beq\label{e.refeq}
X_2 \hR X_2 \hR =\hR X_2 \hR X_2,
\eeq
where $X=(x_{ij})_{i,j=1,\dots, n}$, $X_2=X\ot Id,$ and $\hR$ is the braiding matrix of Subsec. \ref{ss.Hecke}.
(These equations are written explicitly out in \cite[Sec. 3]{DL}.) 

It is proved in \cite{KoS} that $\cT\Oq$ is the quotient of the reflection equation algebra by the {\em braided determinant} which is the image of the quantum determinant under the linear isomorphism $\cal{TO}_q(SL(n))\simeq \Oq$ above. An explicit polynomial expression in $x_{ij}$'s for it appears in \cite{JW}. Consequently, that expression together with the relations \eqref{e.refeq} are a complete set of relations for $\cT\Oq.$  

Let us relate this discussion to $\S_n(\Sigma_{0,2}^*)$ now. It is straightforward to verify that $\tm$ maps the generators $x_{ij}\in \cT\Oq$ to the arcs \diagg{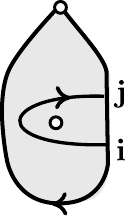}{.8in} which we will denote by $b_{i,j}.$ (Then $\tm(a_{i,j})=\frac1{c_{\bar i}}b_{\bar i,j}$ for the generators $a_{ij}$ for $\S_n(\fB)$ of Subsec. \ref{ss.bigon}. Independently of the above considerations, it is easy to see that $b_{ij}$'s for $i,j=1,\dots, n$ generate $\S_n(\Sigma_{0,2}^*)$, since any web in $\S_n(\Sigma_{0,2}^*)$ can be pushed towards the boundary of $\p \Sigma_{0,2}^*$ and simplified by the boundary relations to a polynomial expression in $b_{ij}$'s.)

Consequently, the above discussion provides a concrete finite presentation for $\S_n(\Sigma_{0,2}^*)$.

\subsection{On injectivity of splitting homomorphism}
\label{ss.injective}

\bpro \label{r.inj}
Suppose $\Sigma$ is an essentially bordered pb surface. Then for  any interior ideal arc $c$ of $\Sigma$, the splitting homomorphism $\Theta_c: \S_n(\Sigma) \to \S_n(\Cut_c\Sigma)$ is injective.
\epro

\begin{proof} First assume that an endpoint of $c$ is a boundary ideal point, which is an endpoint of a boundary edge $e\ne c$. In a small neighborhood of $e\cup c$ we can find an interior ideal arc $c'$ such that $e,c,c'$ cobound an ideal triangle; see Figure \ref{f.tri7}.

\FIGc{tri7}{The curve $c'$. Left: general case. Middle: $c$ and $e$ form a bigon. Right: $c$ cuts out a monogon.}{2.5cm}
Let $\mathring \fT$ be the interior of the triangle $\fT$  bounded by $c,c',$ and $e$, and let
 $\Sigma'= \mathring\Sigma \setminus (\mathring\fT \cup e)$. Then $(\Sigma')_{c \Delta c'}= \Sigma$.
By Theorem  \ref{r.braidedtensor} the map $\Cut_{c, c'} = (\ve_\fT \ot \id_{S_n(\Sigma')}) \circ \Theta _{c'} \circ \Theta_c$ is bijective.  It follows  that $\Theta_{c}$ is injective.

Now assume that both endpoints of $c$ are interior ideal points. Since $\Sigma$ is essentially bordered it contains an interior ideal arc $d$, disjoint from $c$, with one endpoint coinciding with an endpoint of $c$ and the other endpoint being a boundary ideal point. By the above case, the splitting map $\Theta_d: \S_n(\Sigma) \to \S_n(\Cut_d(\Sigma))$ is injective. Since the interior ideal arc $c\subset \Cut_d(\Sigma)$ has one endpoint on the boundary, $\Theta_c: \Cut_d(\Sigma) \to \Cut_{c,d}(\Sigma)$ is injective. From the commutativity $\Theta_c \circ \Theta_d = \Theta_d \circ \Theta_c$ we conclude that $\Theta_c: \S_n(\Sigma) \to \Cut_c(\Sigma)$ is injective.
\end{proof}

\bcon\label{con.inj} For any punctured bordered surface $\Sigma$ and any interior ideal arc $c$ the splitting homomorphism $\Theta_c$ is injective as well.
\econ

The conjecture is true when $n=2$ by \cite{CL} and for $n=3$ by Higgins \cite{Hi}. In both cases explicit bases of $\S_n(\Sigma)$ were used. Proposition \ref{r.inj} shows the conjecture is true if $\Sigma$ has non-trivial boundary. Furthermore, the argument of the proof reduces the conjecture to the empty boundary surfaces with a trivial ideal arc $c$, i.e. an ideal arc bounding a disk in $\Sigma$. Corollary \ref{c.inj} will establish a weaker version of this conjecture for all pb surfaces.
 
%
\subsection{Skein algebras of surfaces with boundary} 
\label{ss.surf-boundary}

Let $\Sigma$ be an essentially bordered pb surface. A collection 
$A=\{a_1,\dots , a_r\}$ of disjoint compact oriented arcs properly embedded into $\Sigma$ is {\bf saturated} if 
\begin{itemize}
\item[(i)] each connected component of ${\Sigma} \setminus \bigcup_{i=1}^r a_i$ contains exactly one ideal point (interior or boundary) of ${\Sigma}$, and
\item[(ii)] $A$ is maximal with respect to the above  condition.
\end{itemize}
Note that condition (i) does not imply (ii). For example, $A=\emptyset\subset \Sigma_{1,1}^*$ satisfies (i), but not (ii). Saturated $A$ consists of two ideal arcs in this surface.

Let $U(a_1),\dots, U(a_r)$ be a collection of disjoint open tubular neighborhoods of $a_1,\dots, a_r$, respectively. Each $U(a_i)$ is homeomorphic with $a_i\times (-1,1)$ (by an orientation preserving homeomorphism) and we require that $(\p a_i) \times (-1,1)\subset \p \Sigma.$

Recall from Subsec. \ref{ss.surfacef} that any embedding of pb surfaces $\Sigma'\subset \Sigma$ together with an ordering on the boundary edges of $\Sigma'$ in the boundary edges $b$ of $\Sigma$, called $b$-orders, 
defines a linear homomorphism $\S_n(\Sigma')\to \S_n(\Sigma).$ 
We will show that it is an isomorphism for a saturated system for arcs $a_1,\dots, a_r$ and $\Sigma'=U(A)=\bigcup_{i=1}^r U(a_i)$:

\def\tri{{\mathsf{tri}}}
\bthm\label{t.Aisom}
Assume $\Sigma$ is an essentially bordered pb surface and $A=\{a_1,\dots , a_r\}$ is a saturated system of arcs.

(1) We have $r= r(\Sigma):=\# \p \Sigma-\chi(\Sigma),$
where $\# \p \Sigma$ is the number of boundary components of $\Sigma$ and $\chi$ denotes the Euler characteristics.

(2) The embedding $U(A) \embed {\Sigma}$ with negative $b$-orderings for all boundary edges $b$ of $\Sigma,$ induces an $R$-module isomorphism $f_A:\S_n(U(A)) \to\S_n(\Sigma)$.
\ethm

Note that each $U(a_i) = a_i\times (-1,1)$ is naturally a directed bigon, with its sides $(\p a_i) \times (-1,1)$ oriented in the direction of $(-1,1)$.

\begin{figure}[h] 
   \centering \diagg{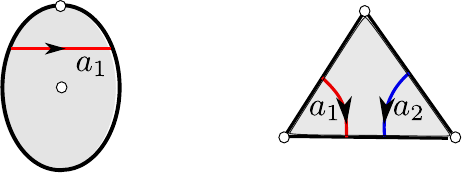}{0.7in}
   \caption{Examples of saturated systems. Left: $A=\{a_1\}$ in a punctured monogon, $\Sigma_{0,2}^*$. 
   Right: $A=\{a_1,a_2\}$ in an ideal triangle $\fT$ with $a_1$ in blue and $a_2$ in red.}
   \label{f.satsystems}
\end{figure}

\bexa\label{eg.sat-sys}
The saturated systems of Fig. \ref{f.satsystems} induce the linear isomorphisms 
\beq\label{e.satsys}
\tm: \S_n(\fB)\to \S_n(\Sigma_{0,2}^*) \quad \text{and}\quad \S_n(\fB)\ot \S_n(\fB)\to \S_n(\fT)
\eeq
of Eq. \eqref{e.tm} and Example \ref{eg.idealtriangle}.
\eexa

By the above theorem, for any essentially bordered pb surface we have an $R$-linear isomorphism
$$\Oq^{\ot r} \xrightarrow{\Psi^{\ot r}} \S_n(U(A)) \xrightarrow{ f_A} \S_n(\Sigma).$$

Since $\Oq$ has a Kashiwara-Lusztig's canonical basis over $\Z[v^{\pm 1}]$, cf. \cite[Prop 5.1.1]{Du}, 
 we have 
 
\bcor For any essentially bordered pb surface, 
$\S_n(\Sigma)$ is a free $R$-module with a basis given by the image of tensor product of Kashiwara-Lusztig's canonical bases on $\Oq^{\ot r}$ under $f_A\Psi^{\ot r}$.
\ecor

Remark \ref{r.projected-base} (below) generalizes the above theorem and corollary to all non-closed pb surfaces.

\brem  Part (1) implies that condition (ii) in the definition of a saturated system can be replaced by condition $|A|=r({\Sigma})$.
\erem

\begin{proof}[Proof of Theorem \ref{t.Aisom}] (1) By the maximality of $A$, its arcs cut $\Sigma$ into pieces of the following two types: 
$$\diagg{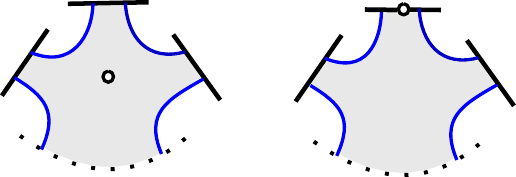}{.5in},$$ 
whose Euler characteristics are respectively 0 and 1. Hence the Euler characteristic of the result is the number of boundary ideal points, which is $\# \p \Sigma$. On the other hand each arc cut increases the Euler characteristic by 1. Hence $\chi({\Sigma}) + r = \# \p \Sigma$, proving part (1).

(2) We prove it by induction on $\tri({\Sigma})$ which is the number of ideal triangles in an ideal triangulation of ${\Sigma}$, defined as follows. Let $\cE$ be a maximal collection of nontrivial ideal arcs in ${\Sigma}$ which are pairwise disjoint and pairwise non-isotopic. The ideal arcs in $\cE$ not isotopic to boundary edges split ${\Sigma}$ into pieces, each is either a monogon, a bigon, or a triangle. Then $\tri({\Sigma})$ is the number of triangles, which is known to be independent on the choice of $\cE$. Note that $\tri({\Sigma})=0$ if and only if ${\Sigma}$ is a disjoint union of monogons and bigons, and the theorem is true for this case.

Suppose $\tri({\Sigma}) >0$. We can assume that ${\Sigma}$ is connected. 

\blem
There is a boundary edge of $\Sigma$ containing at least two endpoints of arcs in $A$.
\elem

\bpr Since arcs of $A$ are disjoint and simple, they have $2|A|$ endpoints and it is enough to prove that $2|A|> \# \p \Sigma$. Assuming otherwise, $2|A| \le \# \p \Sigma$, and by part (1) we have $1\le  \# \p \Sigma \le 2 \chi({\Sigma})$. The positivity of the Euler characteristic implies that $\chi({\Sigma}) =1$ and ${\Sigma}$ is a polygon. Then $ \# \p \Sigma \le 2 \chi({\Sigma})=2$ implies ${\Sigma}$ is a monogon or a bigon, contradicting the assumption $\tri({\Sigma}) >0$.
\epr

Let $b$ be a boundary edge containing at least two endpoints of $A$. Let $p$ the ideal end point of $b$, following the positive direction of $b$. Among all arcs in $A$ having endpoints in $b$ assume $a_1$ has an endpoint closest to $p$. When ${\Sigma}$ is cut by arcs in $A$, there are two pieces adjacent to $a_1$, one of them, denoted by $P_1$, contains the ideal point $p$. The other piece, denoted by $P_2$, contains an ideal point $p'$. Let $e_1$ be an ideal arc of ${\Sigma}$ lying in the interior of $P_1 \cup P_2$ connecting $p$ and $p'$ and intersecting $a_1$ once. No other arcs in $A$ intersects $e_1$.  Because the geometric intersection of $b$ with all arcs in $A$ is at least 2, $e_1$ cannot be isotopic to $b$. Pushing the union $e_1 \cup b$ slightly into the interior of ${\Sigma}$ yields an ideal arc $e_2$ such that $e_1, e_2, b$ bounds an ideal triangle $T$, as in Figure \ref{f.tri5}.  After an isotopy we can assume that $A$ is taught with respect to $e_1, e_2, b$ in the sense that for each $a_j\in A$ and each $e\in \{b,e_1, e_2\}$ the number $|a_j \cap e|$ is minimal when we replace $a_j$ by any isotopic arc. 

\FIGc{tri5}{Left: The arcs in $A$ are in red, the ideal arcs $b, e_1, e_2$ are in blue. Right: After pulling $e_1, e_2$ and arcs in $A$ taut.}{3cm}

Let ${\Sigma}'$ be the result of removing $b$ and the interior of $T$ from ${\Sigma}$, and let $A'$ be the collection $a_i':= a_i \cap {\Sigma}', i=1,\dots, r$. As $|A'|=|A|= r({\Sigma}) = r({\Sigma}')$, the system $A'$ is saturated for ${\Sigma}'$. Since each $a_i'$ is a shrinking of $a_i$, there is a natural isomorphism $f_{A\to A'}: \S_n(U(A)) \to  \S_n(U(A'))$. 
Since $\tri({\Sigma}')= \tri({\Sigma}) -1$ the induction hypothesis applies to ${\Sigma}'$. From the definition we see that $f_A$ is the composition of
$$\S_n(U(A))  \xrightarrow{ f_{A \to A'} } \S_n(U(A')) \xrightarrow{ f_{A'} } \S_n ({\Sigma}')  \xrightarrow{ \glue_{e_1,e_2} }  \S_n({\Sigma}).$$
Since each map in this composition is an $R$-linear isomorphism, so is $f_A$.
\end{proof}  

We will show now that the assumption about the negativity of all $b$-orderings in Theorem \ref{t.Aisom}(2) is unnecessary.

Let us enumerate the boundary edges of $\Sigma$ by $b_1,\cdots, b_s$ for bookkeeping purposes.
Let $o_1,...,o_s$ be some $b_1$-,..., $b_s$-orderings of the boundary intervals of $U(A)$ in the boundary intervals of $\Sigma$ and let $f_{A,o_1,\dots, o_s}: \S_n(U(A))\to \S_n(\Sigma)$ be the homomorphism induced by that height ordered embedding.

\def\brd{\mathsf{braid}}
To relate $f_{A,o_1,\dots, o_s}$ to $f_A$, note that each $b$-ordering $o$ is obtained by a certain permutation $\sigma$ of the negatively height ordered points $A\cap b$. Let us denote by $\sigma_1,...,\sigma_s$ the permutations corresponding to height orderings $o_1,...,o_s$.
Then $f_{A,o_1,\dots, o_s}(x)$ is induced by the embedding of $U(A)\times (-1,1)$ into $\Sigma \times (-1,1)$ with the boundary intervals of $U(A)$ braided by $(\sigma_1)_+,..., (\sigma_s)_+,$ cf. Figure \ref{f.orderedembed}.
Let us elaborate on it more detail now.

Let us call skeins of the form $f_A(x_1\ot ... \ot x_r)\in \S_n(\Sigma)$ {\bf pure}. 

\blem\label{l.braidedembed}
(i) For any braids $\tau_1\in B_{|A\cap b_1|},\cdots, \tau_s\in B_{|A\cap b_s|}$ there exists a unique linear transformation
$$\brd_{A,\tau_1,\cdots, \tau_s}: \S_n(\Sigma)\to \S_n(\Sigma)$$ 
which braids the endpoints of each pure skein in $\S_n(\Sigma)$ in $b_i$ by $\tau_i,$ for $i=1,...,s.$ (All skeins are considered with negative $b_i$-orderings for $i=1,...,s$.)\\
(ii) Let $\sigma_1,...,\sigma_s$ be permutations corresponding to height orderings $o_1,...,o_s$ on $b_1,...,b_s.$
Then for pure $x$,
$$f_{A,o_1,\dots, o_s}(x)=\brd_{A,(\sigma_1)_+,...,(\sigma_s)_+}\circ f_A.$$
(iii) $(\tau_1, ..., \tau_s)\to \brd_{A,\tau_1,\cdots, \tau_s}$ defines a group homomorphism from $B_{|A\cap b_1|}\times \cdots \times B_{|A\cap b_s|}$ to the group of $R$-linear automorphisms of $\S_n(\Sigma).$
In particular, each $\brd_{A,\tau_1,\cdots, \tau_s}$ is a linear isomorphism of $\S_n(\Sigma)$.
\elem

\bpr (i) 
For braids $\tau_1\in B_{|A\cap b_1|},\cdots, \tau_s\in B_{|A\cap b_s|}$ consider the embedding 
$$U(a_1)\cup ... \cup U(a_n) \subset \Sigma \times (-1,1)$$ 
modified by the braiding by $\tau_i$ of its components going towards $b_i\subset \Sigma$, for $i=1,...,s.$ This map is considered with the negative height order. It induces a linear map of skein algebras which we denote by 
$g_{A,\tau_1,\dots, \tau_s}: \S_n(U(A)) \to \S_n(\Sigma).$ Then for pure $x$, let
$$\brd_{A,\tau_1,\cdots, \tau_s}(x)=g_{A,\tau_1,\dots, \tau_s}\circ f_A^{-1}.$$
By Theorem \ref{t.Aisom}(2) pure skeins span $\S_n(\Sigma)$. Consequently, the condition of (i) determines $\brd_{A,\tau_1,\cdots, \tau_s}$ completely.

\noindent (ii) follows from the discussion above Lemma \ref{l.braidedembed}

\noindent (iii) By definition, 
$$\brd_{A,\tau_1,\cdots, \tau_s}\circ\brd_{A,\tau_1',\cdots, \tau_s'}=\brd_{A,\tau_1\tau_1',\cdots, \tau_s\tau_s'},$$ 
for any $\tau_1,\tau_1'\in B_{|A\cap b_1|},\cdots, \tau_s,\tau_s'\in B_{|A\cap b_s|}.$ Consequently, each
$\brd_{\tau_1,\cdots, \tau_s}$ is a linear isomorphism of $\S_n(\Sigma)$. Since $\brd_{A,id,...,id}=id$, the statement follows.
\epr

By Theorem \ref{t.Aisom}(2) and Lemma \ref{l.braidedembed}(2) and (3), we have: 

\bcor\label{c.Aisomanyorder} $f_{A,o_1,\dots, o_s}:\S_n(U(A)) \to\S_n(\Sigma)$ is a linear isomorphism for every $o_1,...,o_s.$
\ecor

%
\subsection{Products on skein algebras of surfaces with boundary} 
\label{ss.prod-surf-boundary}

In the previous subsection, we discussed $R$-module structures of skein algebras only. We will address the algebra products now.

Let $a_1,...,a_r$ be a saturated system of arcs in $\Sigma$ as before.
Note that the induced linear homomorphism $\S_n(U(a_i))\to \S_n(\Sigma)$ is an algebra homomorphism if and only if $a_i$ has its ends at different boundary intervals of $\Sigma.$ (We have seen this already in Example \ref{eg.sat-sys}, where 
the right map of Eq. \eqref{e.satsys} is an algebra homomorphism on each of the components, $\S_n(\fB)$, but the left map 
$tm: \S_n(\fB)\to \S_n(\Sigma_{0,2}^*)$ is not an algebra homomorphism.) 

Therefore, for the sake of studying algebra products on $\S_n(\Sigma)$ let us consider modified neighborhoods $U'(a_i)=U(a_i)\cup V$ for arcs $a_i$ with both their ends in the same boundary interval, where $V$ is a tubular neighborhood of the arc of $\p \Sigma$ connecting the endpoints of $a_i$. We assume that $V$ is small enough so that $U'(a_i)$ is homeomorphic to a punctured monogon. Note that the transmutation map is a linear isomorphism $\tm: \S_n(U(a_i))\to \S_n(U'(a_i))$ by
Proposition \ref{p.tm}.
We leave the chosen neighborhoods of the arcs with ends in different components of $\p \Sigma$ unchanged, $U'(a_i)=U(a_i).$

\def\mult{{\mathsf{mult}}}
Let us consider the map
$$\mult_A: \S_n(U'(a_1))\ot ... \ot \S_n(U'(a_r))\to \S_n(\Sigma),\quad \mult_A(x_1\ot ... \ot x_r)=x_1\cdot ... \cdot x_r.$$
Note that by the transmutation map for arcs in the same component of $\p \Sigma,$
$\mult_A$ coincides with $f_{A,o_1,\dots, o_s}$ for $b_1$-,..., $b_s$-orderings $o_1,...,o_s$, for which the boundary arcs of $U'(a_i)$ are higher than the boundary arcs of $U'(a_j)$ for $i>j$ in any boundary interval of $\p \Sigma$.

Consequently, by the above discussion and by Corollary \ref{c.Aisomanyorder}:
\bcor\label{c.mult} 
$\mult_A: \S_n(U'(a_1))\ot ... \ot \S_n(U'(a_r))\to \S_n(\Sigma)$ is an $R$-linear isomorphism and an algebra homomorphism on 
$$\S_n(U'(a_i))=R\ot ... \ot R\ot \S_n(U'(a_i))\ot R \ot ... \ot R\subset \S_n(U'(a_1))\ot ... \ot \S_n(U'(a_r))$$ 
for every $i$ (where each $R$ is spanned by the appropriate identity element).
\ecor

$\S_n(\Sigma)$ is not the tensor product of the algebras $\S_n(U'(a_1)),... ,\S_n(U'(a_n))$ because elements of different component algebras do not necessarily commute in $\S_n(\Sigma).$

We have seen in Examples \ref{eg.idealtriangle} and \ref{eg.sat-sys} already that the skein algebra of the ideal triangle, 
$\S_n(\fT)$ is the braided tensor product $\S_n(\fB)\underline{\ot} \S_n(\fB)$.
We will see in the next subsection however that our stated skein algebras are not braided products of their component algebras,
$\S_n(U'(a_1)), ..., \S_n(U'(a_r))$ in general.

\subsection{Torus with an arc boundary} 
\label{ss.torus}

Let us apply the approach of the above section to analyze 
the skein algebra of the torus with an arc boundary, $\S_n(\Sigma_{1,1}^*).$
Fig. \ref{f.torus} shows a torus (in black) with a saturated arc collection: $a_1$ in red and $a_2$ in blue. 

\begin{figure}[h] 
   \centering \diagg{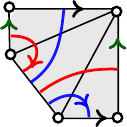}{0.7in}\quad \diagg{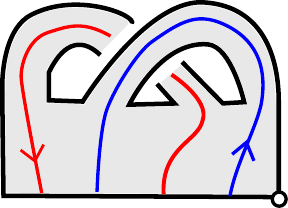}{0.7in}\quad 
   \caption{Left: A triangulation of $\Sigma_{1,1}^*$ (in black) with the horizontal edges identified and with the vertical edges identified. (Hence, all ideal vertices are identified.) The red and the blue arcs form a saturated arc collection. Right: Another presentation of $\Sigma_{1,1}^*$ with the corresponding red and blue arcs.}
   \label{f.toruswithboundary}
   \label{f.torus}
\end{figure}

\def\mult{{\mathsf{mult}}}
By Corollary \ref{c.mult},
$$\mult_A: \S_n(\Sigma_{0,2}^*)\ot \S_n(\Sigma_{0,2}^*)\to \S_n(\Sigma_{1,1}^*),
\quad \mult(x\ot y)=\diagg{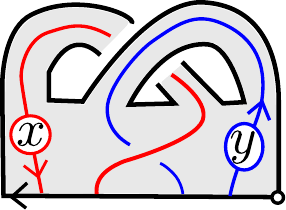}{0.6in}$$
is an $R$-linear isomorphism and an algebra homomorphism on each of the components algebras.


We described a method of finding an algebraic presentation of $\S_n(\Sigma_{0,2}^*)$ in Subsec. \ref{ss.transmut}. The above discussion allows for an algebraic description of the product on 
$\S_n(\Sigma_{0,2}^*)\ot \S_n(\Sigma_{0,2}^*)$ (induced from $\S_n(\Sigma_{1,1}^*)$ by $\mult_A$)
as follows: by the construction of $\mult_A,$ 
$$(x\ot 1)\cdot (x'\ot 1) =(x\cdot x')\ot 1,\quad  (1 \ot y)\cdot (1\ot y') =1\ot (y\cdot y'),\quad (x\ot 1)\cdot (1\ot y)=x\ot y$$ 
in $\S_n(\Sigma_{1,1}^*)$ for $x\in \S_n(U'(a_1))$ and $y\in \S_n(U'(a_2))$.
Therefore, to complete the algebraic description of the product in $\S_n(\Sigma_{1,1}^*)$ it remains to consider
$$(1\ot y)\cdot (x \ot 1)=\diagg{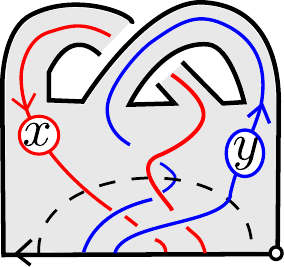}{.7in}.$$
Denoting this diagram by $z$, and applying the identity \eqref{e.Deltaepsilon}, where $\Delta$ is our $\Oq$-coproduct taken with respect to the dashed line, we see that
$$(1\ot y)\cdot (x \ot 1)=\sum x_{(2)}\cdot y_{(2)}\cdot T(x_{(1)},x_{(3)}, y_{(1)},y_{(3)}),$$
where $T(x_{(1)},x_{(3)}, y_{(1)},y_{(3)})$ is the counit value of the tangle in $\fB$ cut off from the diagram above by the dashed line. 

The above formulas completely determine the multiplication in $\S_n(\Sigma_{1,1}^*)$ 
and allow for writing a finite presentation of $\S_n(\Sigma_{1,1}^*)$ in terms of generators and relators.
Note that $\mult_A$ in this case is not a braided tensor product of the component algebras $\S_n(U'(a_1))$, for $i=1,2$, in the sense of \cite{Maj}.

For $R=\Bbbk(q)$, the skein algebra $\S_n(\Sigma_{1,1})$ is given by a semi-direct product 
$U_q(sl(n))\ltimes \uot \Oq$ and is called the ``elliptic double'' of $U_q(sl(n))$, and also the ``algebra of quantum differential operators on $SL(n,\Bbbk)$'', cf. \cite[Sec. 6.4]{BBJ}. 

\brem The finite presentations of $\S_n(\Sigma_{0,2}^*)$ and $\S_n(\Sigma_{1,1}^*)$ (discussed in Subsec. \ref{ss.transmut} and above) induce finite presentations of algebras $\S_n(\Sigma_{g,p}^*)$ for all $g\geq 0$, $p>0$ by the method of Example \ref{ex.deltasum}.

Furthermore, Corollary \ref{c.mult} allows for a generalization of the above method to provide a finite presentation of $\S_n(\Sigma)$ in terms of generators and relators for every essentially bordered surface $\Sigma.$
\erem

\section{ Kernel and image of the splitting homomorphism.}

\subsection{Kernel of the Splitting homomorphism} 
\label{ss.kernel}

Suppose ${\Sigma}$ is a connected pb surface with an ideal point $p$ and a trivial ideal arc $c_p$ at $p$. Then $\Cut_{c_p} {\Sigma}$ is the disjoint union of a monogon $\fM$ and of a new pb surface ${\Sigma}_p$ which has $c_p$ as its boundary edge, see Figure \ref{f.proj}. Let $\K_p({\Sigma})$ be the kernel of the composition
$$ \Theta_p: \SF \xrightarrow{ \Theta_{c_p}} \S_n({\Sigma}_p) \ot_R \S_n(\fM)  \xrightarrow \cong \S_n({\Sigma}_p).$$
Explicitly $\Theta_{p}$ is given as follows. Any stated $n$-web $\alpha$ over ${\Sigma}$ can be isotoped so that it is disjoint from $c_p$ and, hence, lying in $\Sigma_p$. Then $\Theta_{p}(\al)=\al$ as elements of $\S_n({\Sigma}_p)$.

\FIGc{proj}{From ${\Sigma}$ to ${\Sigma}_p$. Here $p$ is an interior ideal point. The picture when $p$ is a boundary ideal point is similar.}{2cm}

\bthm\label{t.K} 
For any two ideal points $p$ and $p'$ of a connected punctured bordered surface $\Sigma$ we have
 $\K_{p}({\Sigma}) = \K_{p'}({\Sigma})$. 
\ethm

\begin{proof} Assume the two trivial ideal arcs $c_p$ and $c_{p'}$ are disjoint. By splitting both $c_p$ and $c_{p'}$ we get two monogons and a pb surface ${\Sigma}_{p,p'}$. Let $\Theta: \SF \to \S_n({\Sigma}_{p,p'})$ be the composition of the two splittings, first along $c_p$ and then along $c_{p'}$. Since by Proposition \ref{r.inj} the second one is injective, we have $\ker \Theta_{p} = \ker \Theta$. By switching the order of the splitting we have $\ker \Theta_{{p'}}= \ker \Theta$. Thus $\K_p = \K_{p'}$.
\end{proof}

We denote this common ideal by $\K({\Sigma})$.
The quotient $\bSF:= \SF/\K(\Sigma)$ is called the {\bf projected stated skein algebra of ${\Sigma}$}. 
By Proposition \ref{r.inj}, $\K_p$ is trivial and $\bSF= \SF$ if $\pS\neq \emptyset$,

\bcor\label{c.inj} For any ideal arc $c$, the splitting homomorphism descends to an injective algebra homomorphism
$$\bar \Theta_c: \bSF\to \bar \S_n(\Cut_c\, \Sigma)=  \S_n(\Cut_c\, \Sigma).$$
\ecor

\bpr The proof is similar to that of Theorem \ref{t.K}. Assume $c$ is disjoint from a trivial arc $c_p$. Since the compositions 
$\Theta_c\Theta_{c_p}, \Theta_{c_p}\Theta_{c}: \SF \to \S_n (\Cut_c(\sF_p))$ coincide and for both of them the second map is injective, $\ker \Theta_{c_p} = \ker \Theta_c$.
\epr

\bcor\label{c.projected}
Conjecture \ref{con.inj} is equivalent to the projection $\S_n(\Sigma)\to \bar \S_n(\Sigma)$ being an isomorphism.
(And, hence, this projection is an isomorphism for $n=2$ and $3.$)
\ecor

In the next subsection we will prove the following:
\bthm\label{t-coinvariants}
For any $\Sigma$, $p$ and $c_p$ as above,  $\bSF$ coincides with the subalgebra of $\S_n(\Sigma_p)$ coinvariant under the coaction $\Delta_{c_p}: \S_n(\Sigma_p)\to \S_n(\Sigma_p)\otimes \S_n(\fB)$ at $c_p$: 
$$\bSF=\{x\in \S_n(\Sigma_p): \Delta_{c_p}(x)=x\otimes 1\}.$$
\ethm

\brem\label{r.projected-base} Let $\Sigma=\overline\Sigma -\cal P$, where $\cal P$ is a finite subset of compact surface $\overline\Sigma$, as in Subsec. \ref{ss.pbsurf}. 
Generalizing the setup of Subsec. \ref{ss.surf-boundary}, consider a collection $A$ of disjoint, oriented, arcs in $\Sigma$, each with endpoints in $\p \Sigma\cup \cal P,$ satisfying conditions (i) and (ii) above.  Theorem \ref{t.Aisom} and the discussion of the projected stated skein algebra implies that such $A$ defines an identification of $\bSF$ with $\Oq^{\ot r}$ and, hence, it determines a basis of $\bSF$.
\erem

%
\subsection{The Image of the Splitting homomorphism}
\label{ss.image}

Let $c$ be an interior oriented ideal arc of a pb surface $\Sigma$.
Denote the two copies of $c$ in $\Cut_c\,\Sigma$ by $a_1$ and $a_2$. We have the splitting $R$-algebra homomorphism
$$\Theta_c: \S_n(\Sigma)\to \S_n(\Cut_c\, \Sigma).$$
and $\S_n(\Cut_c\, \Sigma)$ is a $\Oq$-bi-comodule with the right and left coactions
\begin{align*} \Delta_{a_1}: \S_n(\Cut_c\, \Sigma)\to \S_n(\Cut_c\, \Sigma)\otimes \Oq \\
_{a_2}\Delta: \S_n(\Cut_c\, \Sigma)\to \Oq \ot  \S_n(\Cut_c\, \Sigma)
\end{align*} 
respectively, where $\Oq$ is identified with the skein algebra of the bigon directed by the orientation of $c$.
 {}
Recall that the Hochshild cohomology module is defined by
$$HH^0(\S_n(\Cut_c\, \Sigma))= \{ x \in \S_n(\Cut_c\, \Sigma) \mid  \Delta_{a_1}(x)= \fli \circ {}_{a_2}\Delta(x) \},$$
where $\fli$ is the transposition  
$$\fli : \Oq\otimes \S_n(\Sigma)\to \S_n(\Sigma)\otimes \Oq,\quad \fli(x\otimes y)=y\otimes x.$$

\bthm [See \cite{CL,KQ} for $n=2$ and \cite{Hi} for $n=3$] The image of $\Theta_c$ is equal to $ HH^0(\S_n(\Cut_c\, \Sigma))$.
\ethm
\begin{proof} Since the image of $\Theta_c$ is equal to the image of $\bar\Theta_c: \S_n(\Sigma)\to \S_n(\Cut_c\, \Sigma)$, 
we can work with projected skein algebras. More specifically, we will assume that one end $v$ of $c$ is a boundary ideal point of $\Sigma$, since we can remove a disk from $\Sigma$, adjacent to $v$ and disjoint from $c$, if necessary. We will present $c, a_1, a_2$ in $\Cut_c\,\Sigma_{a_1\triangle a_2}$ as in Fig. \ref{f.aac}, with $v$ in the bottom. 

\begin{figure}[h] 
   \diagg{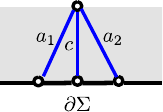}{.8in} 
   \caption{$(\Cut_c\,\Sigma)_{a_1\triangle a_2}$}
   \label{f.aac}
\end{figure}

Let $(\Cut_c\,\Sigma)_{a_1\wedge a_2}$ denote $(\Cut_c\,\Sigma)_{a_1\triangle a_2}-\p_0 \fT,$ 
for simplicity, where $\p_0 \fT$ is the bottom edge of $\fT$, as in Subsec. \ref{ss.braidedtensor}.
Note that we can identify the image of $\S_n((\Cut_c\,\Sigma)_{a_1\wedge a_2})\to \S_n((\Cut_c\, \Sigma)_{a_1\triangle a_2})$ with $\overline{\S}_n(\Sigma).$ We will use this identification below.

Let 
$$\nabla_{a_1,a_2}: \S_n(\Cut_c\, \Sigma)\to \S_n((\Cut_c\, \Sigma)_{a_1\triangle a_2}),\quad \nabla_{a_1,a_2}=\glue_{a_1,a_2}\circ \htw_{a_2}^{-1}.$$
It is an isomorphism by Propositions \ref{p.twist} and \ref{p.tri-glue}.

\blem\label{p.nabla}
(1) $\nabla_{a_1,a_2}\Theta_c(\S_n(\Sigma)) = \S_n((\Cut_c\, \Sigma)_{a_1\wedge a_2})\simeq \overline{\S}_n(\Sigma).$\\
(2) $\nabla_{a_1,a_2}$ restricted to $\im\, \Theta_c$ 
is the inverse to $\Theta_c.$
\elem

\bpr For every stated web diagram $D$ on $\Sigma$ we have
$$\diagg{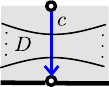}{0.6in}\xrightarrow{\Theta_c} \sum_{i_1,...,i_k} \diagg{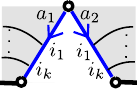}{0.6in} \xrightarrow{\nabla_{a_1,a_2}}\sum_{i_1,...,i_k} \left(\prod c_{i_j}^{-1}\right) \diagg{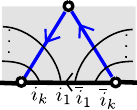}{0.7in}=D,$$ 
with the last equality by Relation \eqref{e.capnearwall}.
Hence,
$$\nabla_{a_1,a_2}\Theta_c=\id_{\overline{\S}_n(\Sigma)}.$$
Since $\nabla_{a_1,a_2}$ restricted to $\im \Theta_c$ is a bijection onto $\overline{\S}_n(\Sigma)$,
this identity implies (2). 
\epr

For any collection of boundary edges $X\subset \p\Sigma$, let  $\im_{\Sigma}\, \S_n(\Sigma-X)$ denote the image $\imath_*(\S_n(\Sigma-X))$ in $\S_n(\Sigma)$ of the homomorphism induced by $\imath: \Sigma-X\embed \Sigma$.

\blem\label{l.embedonto}
Let $b_1,b_2$ be two boundary components of $\Sigma$ separated by a puncture.
Then the embedding
$$\im_{\Sigma}\, \S_n(\Sigma-(b_1\cup b_2))\embed \im_{\Sigma}\, \S_n(\Sigma-b_1)\cap \im_{\Sigma}\, \S_n(\Sigma-b_2)$$
is onto.
\elem

\bpr
Consider an arc $\gamma$ parallel to $b_1\cup b_2$, as in Fig. \ref{f.t1t2gamma}.
\begin{figure}[h] 
   \diagg{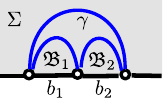}{0.6in} 
   \caption{}
   \label{f.t1t2gamma}
\end{figure}
Then the following diagram commutes:
$$\begin{array}{ccc}
\im_{\Sigma}\, \S_n(\Sigma-(b_1\cup b_2)) & \embed & \im_{\Sigma}\, \S_n(\Sigma-b_1)\cap \im_{\Sigma}\, \S_n(\Sigma-b_2)\\
\downarrow \Theta_\gamma & & \downarrow \Theta_\gamma\\
\im_{\Cut_\gamma\, \Sigma}\, \S_n(\Cut_\gamma\, \Sigma-(b_1\cup b_2)) & \embed & \im_{\Cut_\gamma\, \Sigma}\, \S_n(\Cut_\gamma\, \Sigma-b_1)\cap \im_{\Cut_\gamma\, \Sigma}\, \S_n(\Cut_\gamma\, \Sigma-b_2).
\end{array}$$
By Proposition \ref{r.inj}, both homomorphisms $\Theta_\gamma$ in the diagram are $1$-$1$ and, hence, it is enough to show that the embedding in the bottom line is onto. Since the skein algebra a surface is the tensor product of its connected components, 
it is enough to show the statement of the lemma for the triangle $\fT$ bounded by $b_1,b_2$ and $\gamma.$
By Proposition \ref{r.braidedtensor},  $\S_n(\fT)$ is isomorphic with $\S_n(\fB_1)\otimes \S_n(\cal \fB_2),$ as an $R$-module, where $\fB_i$ for $i=1$ are disjoint bigons in $\fT$ such that $b_i\subset \p \fB_i$, as in Fig. \ref{f.t1t2gamma}.
Through that isomorphism the statement of the lemma reduces to 
$$\Oq\otimes R\cdot 1\cap R\cdot 1\otimes \Oq=\S_n(\fM)=R,$$
where $1$ is the identify in $\Oq$ and $\fM$ is the monogon $\fT-b_1-b_2$. That follows from the fact that $R\cdot 1$ is a direct summand of $\Oq$, by \cite[Prop. 5.1.1]{Du}.
\epr

Let us continue with the proof of the theorem. To prove that the inclusion 
$$\Theta_c(\S_n(\Sigma))\subset Ker\, (\Delta_{a_1} - \fli\circ {}_{a_2}\Delta)$$ 
is an equality, we will show that for any $y\in Ker\, (\Delta_{a_1} - \fli\circ {}_{a_2}\Delta)$,
$$ x=\nabla_{a_1,a_2}(y)\in \S_n((\Cut_c\, \Sigma)_{a_1\triangle a_2})$$
lies in $\S_n((\Cut_c\, \Sigma)_{a_1\wedge a_2}).$
Then 
\beq\label{e.Th-xy}
\Theta_c(x)=y
\eeq by Proposition \ref{p.nabla}(2).
 
As mentioned above, it remains to be shown that $x\in \S_n((\Cut_c\, \Sigma)_{a_1\wedge a_2}).$
Recall that $\Delta_{a_1}$ and $\fli{}_{a_2}\circ \Delta$ map $y$ into $\S_n(\Cut_c\, \Sigma)\otimes \S_n(\fB)$, where the left and right edges of $\fB$ are denoted by $e_l,e_r,$ respectively.

Let $$z=\Delta_{a_1}(y)=\fli\circ {}_{a_2}\Delta(y).$$
By Proposition \ref{p.nabla}(2), 
$$\nabla_{a_1,b_l}(z)=\nabla_{a_1,b_l}\Delta_{a_1}(y)=y,$$
where $y$ at the end of the above equation is a skein in $\Cut_c\, \Sigma\sqcup \fB/(a_1=b_l)$ identified with $\Cut_c\, \Sigma$.
Then
$$\nabla_{b_r,a_2}\nabla_{a_1,b_l}(z)=\nabla_{b_r,a_2}(y)=\nabla_{b_r,a_2}\Theta_c(x)=x,$$
by \eqref{e.Th-xy}.

By Proposition \ref{p.nabla}(1), $\nabla_{a_1,b_l}\Delta_{a_1}(y)$ belongs to  $\S_n((\Sigma\sqcup \fB)_{a_1\wedge b_l}).$ By applying $\nabla_{b_r,a_2}$ to it, we see that
$$x\in \S_n((\Sigma\sqcup \fB)_{a_1\wedge b_l, b_r\triangle a_2}).$$
Since nablas for disjoint pairs of edges commute,
$$x=\nabla_{b_r,a_2}\nabla_{a_1,b_l}(z)$$
and by an analogous argument
$$x\in \S_n((\Sigma\sqcup \fB)_{a_1\triangle b_l, b_r\wedge a_2}).$$
Now the statement follows from
Lemma \ref{l.embedonto}. \end{proof}

The construction of the inverse of the splitting map (Lemma \ref{p.nabla}) implies the following:

\bcor\label{c.noendpoints}
For any union $\cal C$ of ideal boundary arcs of $\Sigma,$  
$$\Theta_c(\S_n(\Sigma-\cal C)) = \Theta_c(\S_n(\Sigma))\cap \S_n(\Cut_c\, \Sigma-\cal C)\quad \text{in}\ 
\S_n(\Cut_c \Sigma).$$
\ecor

\bpr
The inclusion $\subset$ is obvious and the opposite inclusion $\supset$ is obtained by applying the inverse map to $\Theta_c$ of Lemma \ref{p.nabla}.
\epr

\bpr[Proof of Theorem \ref{t-coinvariants}:]
$$\bSF\subset \{x\in \S_n(\Sigma): \Theta_{c_p}(x)=x\otimes 1\}$$ is obvious. 
The opposite inclusion, $\supset,$ is immediate for $\Sigma=\fM$: in that case $p$ is the only vertex of $\fM$, $\fM_{c_p}=\fB,$ where $c_p=\p_r \fB$ and
if $\Delta_{c_p}\, x= x\otimes 1$ for $x\in \fM_{c_p}$ then by applying $\ve\otimes 1$ we obtain $x=\ve(x)1\in R.$

More generally, let $\Sigma$ be a disjoint union of an essentially bordered $\Sigma'$ and of $\fM$ (with a vertex $p$ and an arc $c_p$ as above).
Let $\Delta_{c_p}\, x= x\otimes 1$ for $x\in \Sigma'\sqcup \fM_{c_p}$. Then by Theorem \ref{t.Aisom}, $x$ can be written as 
$x=\sum_{i=1}^N y_i\otimes z_i$, where $y_1,\dots, y_N\in \S_n(\Sigma')$ are linearly independent and $z_1,\dots, z_N\in \S_n(\fB)$.
Then $\Delta_{c_p} z_i=z_i\otimes 1$ for every $i$ and, hence, $z_1,\dots, z_N\in R.$ That concludes the proof of the inclusion $\supset$ in that case. 

Let $\Theta_{c_p}(x)=x\otimes 1$ now for some arbitrary $\Sigma, p, c_p$ (as above) and $x\in \S_n(\Sigma_p).$ We need to show that $x$ lies in the image of $\S_n(\Sigma_p-c_p)$ in $\S_n(\Sigma_p).$

Let $c_p'$ be an arc in $\Sigma_p$ parallel to $c_p,$ splitting $\Sigma_p$ into $\Sigma_p'$ and a bigon bounded by $c_p$ and $c_p'.$ Then 
$$\Theta_{c_p}\Theta_{c_p'}(x)=\Theta_{c_p'}\Theta_{c_p}(x)=\Theta_{c_p'}(x)\otimes 1.$$
Since 
$\Theta_{c_p'}(x)\in \S_n(\Sigma_p'\sqcup \fM_{c_p})$ the previous case implies that 
$\Theta_{c_p'}(x)$ is of the form $y\otimes 1\in \S_n(\Sigma_p')\otimes S_n(\fM_{c_p})$, for some $y\in \S_n(\Sigma_p')$. 
By Corollary \ref{c.noendpoints} above for $\cal C=c_p,$ we have $\Theta_{c_p'}(x)\in \Theta_{c_p'}(\S_n(\Sigma_p-c_p))$. Since $\Theta_{c_p'}$ is $1$-$1$, $x$ lies in $\S_n(\Sigma_p-c_p).$
\epr

%
\section{Relation to factorization homology, skein categories, and lattice gauge theory}
\label{s.facthom_etc}

%
\subsection{Factorization homology}
\label{ss.facthom}
%

Factorization homology was introduced by Beilinson and Drinfeld \cite{BD} in the setting of conformal field theory and then in \cite{Lu, AF, AFT} in the topological context. Given an algebraic object $\cal A$ called an $E_n$-algebra, it associates to oriented $n$-dimensional manifolds (with boundary) $M$ categories $\int_M \cal A,$ 
which are linear over a certain ring of coefficients $R$.

For $n=2$, the notion of $E_2$-algebra is equivalent to that of a braided tensor category. Important examples of such categories are the categories of finite dimensional representations of quantum groups $U_q(\mathfrak g).$ The factorization homology of surfaces for these categories was studied in \cite{BBJ}, where the authors proved that if
$\p \Sigma= S^1$ then $\int_M \cal A$ is equivalent to the category of left modules over a certain algebra $A_\Sigma$ (depending on the Lie algebra $\mathfrak g$). 

The factorization homology of \cite{BBJ} and skein categories (discussed below) are theories parallel to ours. 
We show:

\bthm\label{t.facthomo}
Let $R=\Bbbk(q)$ for a field $\Bbbk$ and let $E_2$ be the category of type $1$ finite dimensional representations of $U_q(sl(n))$ (over $R$). Then $A_{\Sigma_{g,p}}$  is isomorphic to $\S_n(\Sigma_{g,p}^*)$ (as $R$-algebras) for every $g\geq 0, p\geq 1$.
\ethm

On the one hand, factorization homology of \cite{BBJ} is more general in that it is defined for all semi-simple Lie algebras and it can be viewed as a quantization of the entire moduli stacks of representations, rather than just the character varieties.

On the other hand, one may consider our theory more elementary because it does not involve higher category theory. More importantly,
our stated skein modules are 
defined over non-field rings of coefficients, for all $3$-dimensional manifolds, and we worked out their theory for surfaces with multiple boundary components and multiple markings. Furthermore, unlike our stated skein algebras, the algebras $A_\Sigma$ of \cite{BBJ} are defined up to an isomorphism only.

The existence of our stated skein algebras over $\Z[q^\frac1{2n}]$ allows to construct quantum trace  
homomorphisms of the stated skein algebras into quantum tori, over any ground ring. It was done for $n=2$ by Bonahon-Wang, \cite{BW}, and more generally in \cite{CL, LY2}. 
The construction of quantum trace was generalized to all $n$ in \cite{LY3}, where two versions of quantum trace maps, quantizing respectively the length coordinates and the shear coordinates trace formulas, were introduced.

Embeddings into quantum tori allow to study algebraic properties $\S_n(\Sigma)$ and their representations.

The above works relate our algebras to the theory of quantum cluster algebras, which provide alternative quantizations of character varieties. Further connections to quantum cluster algebras are through \cite{CS, Sh, JLSS}.


For completeness, let us summarize briefly the construction of the factorization homology of \cite{BBJ} (in dimension $2$): 
it is based on the $(\infty,1)$-category $\mathsf{Mfld}^2$ whose objects are oriented surfaces (with boundary), morphisms are given by their embeddings, $2$-morphisms are isotopies between embeddings and higher order morphisms are isotopies between them. This category has a monoidal structure given by disjoint embeddings and it has a full subcategory $\mathsf{Disc}^2$ consisting of disks (with partial boundaries) and their embeddings. One can prove that any pivotal ribbon category defines a symmetric monoidal functor into a symmetric monoidal $(\infty,1)$-category $\cal C$ whose objects are certain presentable categories and its monoidal structure is given by the categorical product. 
Then $\int_M \cal A$ is the left Kan extension
$$\begin{tikzcd}
\mathsf{Disc}^2\  \arrow[r, "F"]
\arrow[d, hook]  
&  \cal C \\
\mathsf{Mfld}^2 \arrow[ru, "\int_- \cal A "'] 
& \\
\end{tikzcd},$$
which, in more concrete terms, is a certain colimit in $\cal C$ over all possible embeddings of collections of disks into a given surface.


\bpr[Proof of Theorem \ref{t.facthomo}]
Let $\mathfrak g=sl(n).$ Then the algebra $\mathfrak F_A$ of \cite{BBJ} is isomorphic with $\S_n(\Sigma_{0,2}^*)$, cf. \cite[Sec. 6.1]{BBJ} 
Furthermore, one can see that $A_{\Sigma_{1,1}}$ coincides with $\S_n(\Sigma_{1,1}^*),$ by comparing the ``gluing pattern'' of $\Sigma_{1,1}$ in \cite[Thm. 5.11]{BBJ} with ours in Fig. \ref{f.toruswithboundary}(right) or by \cite[Cor. 6.8]{BBJ}.

By \cite[Thm. 5.11]{BBJ}, $A_{\Sigma_{g,p}}$ is the braided tensor product of $p-1$ copies of $A_{\Sigma_{0,2}}=\S_n(\Sigma_{0,2}^*)$ and $g$ copies of $A_{\Sigma_{1,1}}=\S_n(\Sigma_{1,1}^*).$ Now the statement follows from Example \ref{ex.deltasum}.

\epr

%
\subsection{Skein categories}
\label{ss.skeincat}
%

Skein categories are categorical analogous of skein algebras introduced by Walker and Johnson-Freyd, \cite[p. 70]{Wa}, \cite[Sec. 9]{JF}. 
A framing of a point $p$ on surface $\Sigma$ is a choice of a non-zero vector $v\in T_p \Sigma.$
Let $\cal V$ be a ribbon category, linear over $R$. 
The {\bf ribbon category} $\mathsf{Rib}_{\cal V}(\Sigma)$ has objects given by finite sets of framed, signed disjoint points of $\Sigma$. 
Its morphisms are $R$-linear combinations of ribbon graphs in $\Sigma \times [0,1]$ (in the sense of Reshetikhin-Turaev) whose edges are decorated with objects of $\cal V$ and coupons are decorated with intertwiners. 
The ends of a ribbon graph $\Gamma$ in $\Sigma\times \{0\}$ (respectively: in $\Sigma\times \{1\}$) determine the source (and respectively: the target) of the morphism 
$\Gamma.$

For an oriented arc $C$, 
the Reshetikhin-Turaev construction defines a functor $\mathsf{RT}:$\\ ${\mathsf{Rib}_{\cal V}(C \times [0,1])\to \cal V.}$ (This functor was denoted by $\RT$ for the ribbon category $\cal C_n$ of Subsec. \ref{ss.based-t}.)

The {\bf skein category}, $\mathsf{Sk}_{\cal V}(\Sigma)$ is $\mathsf{Rib}_{\cal V}(\Sigma)$ modulo the relation on morphisms $\sum c_i\Gamma_i\sim 0$, whenever a restriction of $\sum c_i\Gamma_i$ to a certain cube $C\times [0,1]\times [0,1]$ is in the kernel of Reshetikhin-Turaev evaluation, $\mathsf{RT}.$

Cooke proved that for the category $\cal V$ of finite dimensional representations of a quantum group $U_q(\mathfrak g),$ the skein category of any surface coincides with its factorization homology $\int_\Sigma \cal V$, \cite{Co}. Furthermore, \cite{Fa,Ha,LY1} proved that for $\mathfrak g=sl(2)$ and $R$ a field, the skein category of $\Sigma$ with $\p \Sigma= S^1$ is equivalent to the category of left modules over $\S_2(\Sigma).$ 
By Theorem \ref{t.facthomo}, we obtain 
\bcor\label{c.skeincat}
For $\mathfrak g=sl(n)$ (for any $n$) and $R$ a field, the skein category of $\Sigma$ with $\p \Sigma= S^1$ is equivalent to the category of left modules over $\S_n(\Sigma).$ 
\ecor

Note that this equivalence is quite non-intuitive, as skein categories are built of unstated ribbon graphs with ends in $\Sigma\times \{0,1\}$ rather than of stated webs with ends in $\p \Sigma \times (-1,1)$ considered in stated skein algebras.

Corollary \ref{c.skeincat} asserts that the $sl(2)$-skein category of $\Sigma$ with $\p \Sigma= S^1$ has an internal algebra object isomorphic to $\S_2(\Sigma)$. In fact, by \cite[Thm. 1.1]{Ha} and \cite[Thm. 5.3]{Fa} this internal algebra object is isomorphic to $\S_2(\Sigma)$ as a $\mathcal Oq(SL(2))$-comodule algebra. We expect that this statement generalizes to $sl(n)$ for all $n$.

%
\subsection{Lattice Gauge Theory, Quantum moduli spaces}
\label{ss.latticegauge-quantmoduli}
%

A ciliated graph $\Gamma$ is a finite graph with an additional data specifying for each vertex of $\Gamma$ a linear order of half-edges adjacent to it. Each ciliated graph $\Gamma$ is ribbon and, hence, defines a surface which contracts onto $\Gamma.$
Inspired by an earlier Fock-Rosly's work \cite{FR}, Alekseev-Grosse-Schomerus and Buffenoir-Roche quantized moduli spaces of flat connections on such surfaces in \cite{AGS, AS, BR1,BR2}.  (See also \cite{BFK}.) Specifically, for each ciliated graph $\Gamma$ and a quantized coordinate Hopf algebra $O_q(G)$ they have defined an $O_q(G)$-comodule ${\cal L}(\Gamma)$, called quantum moduli space, quantizing a (properly defined) algebra of functions on the space of flat $G$-connections on $\Gamma.$

Let $\Sigma(\Gamma)$ be a surface without boundary realizing the ribbon structure on $\Gamma$ and let $\Sigma^0(\Gamma)$ be $\Sigma(\Gamma)$ with one of its punctures blown up into a disc, as in Fig. \ref{f.proj}. (Hence, $\Sigma^0(\Gamma)=\Sigma(\Gamma)_p$ for some puncture $p$, in the notation of Subsec. \ref{ss.kernel}. Note that $\Sigma^0(\Gamma)$ and $\Sigma(\Gamma)$ are uniquely determined up to a homeomorphism.)
Then, as observed in \cite{BBJ}, the defining equations for ${\cal L}(\Gamma)$ coincide with those induced by the gluing patterns of \cite{BBJ}. In other words, quantum moduli spaces are determined by the factorization homology of \cite{BBJ} and, consequently, for $G=SL(n),$ 
$${\cal L}(\Gamma)\simeq \S_n(\Sigma^0(\Gamma))$$ 
as $O_q(G)$-comodule algebras. This result was observed independently by the first author and proved by \cite{Ko} for $n=2$. 

By Theorem \ref{t-coinvariants},  the coinvariant subalgebra ${\cal L}(\Gamma)^{\Oq}$ is isomorphic with our projected skein algebra $\bar \S_n(\Sigma(\Gamma)).$ This result generalizes the results of \cite{BFK, Ko} for $n=2$.

%
\section{Relation to other known cases}
%

%
\subsection{Compatibility with stated Kauffman bracket skein modules of $3$-manifolds}
\label{ss.sKB}
%

The stated Kauffman bracket skein algebras (of  surfaces) of the first author \cite{Le-triang} were generalized to stated skein modules of marked $3$-manifolds in \cite{BL} (cf. also \cite{LY1}). We are going to prove that these modules are isomorphic with our $SL(2)$-skein modules, $\S_2(M,\cN)$.  

To relate these modules to ours, let us replace the variable $q$ of \cite{Le-triang} with $q^{1/2}$ and denote the resulted stated Kauffman bracket skein module by $\cS(M,\cN)_{q^{1/2}}$.  Let a {\bf framed link} in $\MN$ be a non-oriented $2$-web without sinks nor sources, stated by signs $\pm$. By definition $\cS(M,\cN)_{q^{1/2}}$ is the $R$-module freely spanned by isotopy classes of framed links subject to the relations \eqref{e.KB-Le}-\eqref{e.cap-near-w-Le} below. 

\bthm\label{t.su2} Suppose $(M,\cN)$ is a marked $3$-manifold. \\
(1) There is a unique  
 $R$-linear isomorphism $\Lambda: \cS(M, \cN)_{q^{1/2}}\to \cal S_2(M,\cN)$ which maps framed links $\al$ to stated $2$-webs by assigning arbitrary orientations to them, and changing the minus state to $1$ and the plus state to $2$.\\ 
(2) The splitting homomorphism of \cite{Le-triang, BL} coincides with ours through $\Lambda$. 
\ethm

\noindent{\it Proof of Theorem \ref{t.su2}:}
Let $\cal L(M,\cal N)$ be the set of all stated framed links in $(M,\cN)$.
First let us record the defining relations for $\cS(M,\cN)_{q^{1/2}}$:
\beq\label{e.KB-Le}
\cross{n}{n}{p}{}{}{}{}{}{}=q^{1/2} \walltwowall{n}{n}{}{}{}{}{}{} +q^{-1/2} \wallcapcupwall{n}{n}{}{}{}{}{}{}
\eeq
\beq\label{e.unknot-Le}
 W\cup \circlediag{} = -(q+q^{-1})\cdot W
\eeq
\beq\lb{e.caps-Le}
\capwall{<-}{}{$+$}{$+$}=\capwall{<-}{}{$-$}{$-$}=0,\quad \capwall{<-}{}{$+$}{$-$}=q^{1/4}, 
\capwall{<-}{}{$-$}{$+$}=-q^{5/4}\ \text{(by \cite[(18)]{Le-triang}).}    
\eeq
\beq\label{e.cap-near-w-Le}
\capnearwall{}= q^{-1/4}\twowall{<-}{}{}{$-$}{$+$}-q^{-5/4}\twowall{<-}{}{}{$+$}{$-$}.
\eeq
(This last equality is a consequence of applying a half-twist to \cite[(13)]{Le-triang}.)

For convenience, we draw diagrams with the arrow down, rather than up as in \cite{BL,Le-triang}, to make them compatible with the skein relations of our $\S_2(M,\cN).$ Since the half-twist is an invertible operation, they form an alternative set of defining skein relations of the stated skein module of \cite{Le-triang, BL}.

On the other hand, for $n=2$ our skein relations are: 
\beq\label{e.pm2} q^{1/2}\cross{n}{n}{p}{>}{>}{}{}{}{} - q^{-1/2}\cross{n}{n}{n}{>}{>}{}{}{}{} =(q-q^{-1}) 
\walltwowall{n}{n}{>}{>}{}{}{}{} \eeq
\beq\label{e.twist2} \kink=-q^{3/2}\horizontaledge{>}\eeq
\beq\label{e.unknot2} W\cup \circlediag{<} = -(q+q^{-1})\cdot W\eeq
\beq\label{e.sinksource2} \wallsinksourcetwowall{n}{n}{}{}{}{}
= -q 
\walltwowall{n}{n}{>}{>}{}{}{}{} +q^{1/2}
\cross{n}{n}{p}{>}{>}{}{}{}{} 
\eeq
\beq\label{e.vertex-wall-2}
\veenearwall= q^{-5/4}\left(\twowall{<-}{white}{black}{$2$}{$1$}-q\twowall{<-}{white}{black}{$1$}{$2$}\right)
\eeq
 {}
 \beq\label{e.caps-wall-2}
 \capwall{<-}{white}{$1$}{$1$}= \capwall{<-}{white}{$2$}{$2$}=0,  \capwall{<-}{white}{$1$}{$2$}=-q^{5/4} 
 ,\capwall{<-}{white}{$2$}{$1$}=q^{1/4}
 \eeq
 \beq\label{e.caps-wall-2a}
 \capnearwall{white}= -q^{-5/4}\twowall{<-}{white}{black}{$2$}{$1$}+q^{-1/4}\twowall{<-}{white}{black}{$1$}{$2$} 
 \eeq
 (By Proposition \ref{p.crossingrels}, Relation \eqref{e.crossp-wall} is redundant.)

\blem\label{l.2-orient}
The value of any framed link $T$ in $(M,\cN)$ considered as a $2$-web in $\cal S_2(M,\cN)$ does not depend on the orientation of $T.$
\elem

\bpr 
By \eqref{e.sinksource2}, \eqref{e.twist2} and \eqref{e.unknot2} we have
\beq\label{e.2kinks}
\twokinks{>}=q(q+q^{-1})\horizontaledge{>} +q^{\frac{1}{2}}\kink= \horizontaledge{>}.
\eeq
For any arc with a $2$-vertex near its end we have
\beq\label{e.arc-near-w}
\capveewall{$i$}=\capedgewall{$i$}{$\bar i$}{$i$}\cdot q^{-5/4}\cdot \begin{cases} -q & \text{for}\ i=1\\ 1 & \text{for}\ i=2\\ \end{cases} = -\edgewall{<-}{white}{$i$}.
\eeq

Hence, for any arc we have,\\
$$\walledgewall{}{}{white}{$i$}{$j$}= \wallcuspswall{<-}{<-}{white}{black}{$i$}{$j$} =
\walledgewall{}{}{black}{$j$}{$i$}$$
(Both ends may lie on the same marking in $\cN$.)
Similarly, by \eqref{e.2kinks}, for loops we have
$$
 \diagg{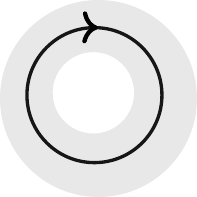}{0.6in}=   \diagg{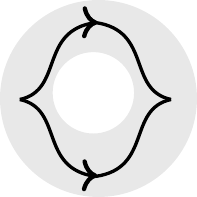}{0.6in} =  \scalebox{-1}[1]{\diagg{core-anul.pdf}{0.6in}}$$ 
\epr

 By Lemma \ref{l.2-orient}, assigning arbitrary orientations of links defines a map $\Lambda: R\cal L(M,\cN)\to \cal S_2(M,\cN)$ preserving relations \eqref{e.unknot-Le}-\eqref{e.cap-near-w-Le}. 
To see that \eqref{e.KB-Le} is preserved as well, we start with the following combinatorial observation: 

\def\kb#1{\hspace*{-.13in}\cross{n}{n}{#1}{}{}{}{}{}{}\hspace*{-.15in}}
\def\kbo#1{\hspace*{-.15in}\cross{n}{n}{#1}{>}{>}{}{}{}{}\hspace*{-.15in}}
\def\kbinfty{\hspace*{-.15in}\wallcapcupwall{n}{n}{}{}{}{}{}{}\hspace*{-.1in}} 
\def\kbsinksource{\hspace*{-.1in}\wallsinksourcetwowall{n}{n}{}{}{}{}\hspace*{-.11in}}
\def\kbsmooth#1{\hspace*{-.1in}\walltwowall{n}{n}{#1}{#1}{}{}{}{}\hspace*{-.1in}}

We say that a crossing \kb{p} or \kb{m}  
in an unoriented framed link $L$ in $(M,\cN)$ is of \kbinfty-{\bf type} if $L$ with that smoothing can be oriented so that it is a $2$-web with only two $2$-vertices, looking like \kbsinksource. We define a crossing of 
\kbsmooth{}-type analogously. It is straightforward to verify that every crossing is of one of these two types. (However, it can be of both types simultaneously, if the crossing involves an arc.) 

\blem
If  a crossing \kbo{p} or \kbo{m} in $L$ 
is of \kbinfty-{\bf type} then for that smoothing of $L$ we have $\Lambda(\,\kb{p}\,)= \kbsinksource .$
 {}
\elem

\bpr There are two possibilities:\\
(1) the NE end of  is connected to the NW or the SW end. Then the statement follows by introducing two $2$-valent vertices as in \eqref{e.2kinks}.\\ 
(2) the NE end is connected to the marking. Then one of the SE or SW ends must be connected to a marking and the statement follows by applying \eqref{e.arc-near-w} twice near the markings. 
\epr

Suppose that the crossing \kb{p} 
on the left side of \eqref{e.KB-Le} is of \kbinfty-type. 
Then by the above lemma, $\Lambda$ maps that skein relation to 
$$\kbo{p} = q^{1/2}\kbsmooth{>} + q^{-1/2}\Lambda\left(\kbinfty\right)=q^{1/2}\kbsmooth{>} + q^{-1/2}\kbsinksource,$$
which coincides with \eqref{e.sinksource2} in $\S_2(M,\cN)$.
The proof for a crossing of \kbsmooth{}-type is analogous. Thus the $R$-linear map $\Lambda: \cS(M, \cN)_{q^{1/2}}\to \cal S_2(M,\cN)$ is well-defined. 

We prove that $\Lambda$ is an isomorphism by constructing its inverse:  Consider first the map $R\cal W_2(M,\cN)\to \cS(M, \cN)_{q^{1/2}}$
sending webs $\alpha$ to $(-1)^{|V_2(\alpha)|} \bar \alpha$, where $\bar \al$ is the result of forgetting the orientation and of smoothing all the 2-valent vertices. It is immediate to see that it factors through relations \eqref{e.pm2}-\eqref{e.caps-wall-2a} into a homomorphism $\S_2(M,\cN)\to \cS(M, \cN)_{q^{1/2}}$.
 Since $\cS(M, \cN)_{q^{1/2}}$ and $\S_2(M,\cN)$ are spanned by links
 (i.e. webs with no sinks nor sources) and since $\Lambda$ and the above map 
$\S_2(M,\cN)\to \cS(M, \cN)_{q^{1/2}}$ are inverses of each other on links, the statement follows.

The proof of part (2) is straightforward.
\qed

\subsection{Compatibility with the SU(n)-skein modules}
\label{ss.SLn}
\def\UM{UM}

In this subsection we are going to show that for any $3$-manifold $M$ and any $n$ our skein module $\S_n(M, \emptyset)$ is isomorphic with the $SU(n)$-skein module introduced by the second author in \cite{Si}. That module is built of  {\bf based $n$-webs} in $M$ which are defined as our $n$-webs in $(M,\emptyset)$, except that the half-edges incident to any of their $n$-valent vertices are linearly ordered. We denote the set of all such webs up to isotopy by $\cal W_n^b(M)$.
Let $\S_n^b(M)$ be the quotient of the $R$-module freely generated by $\cal W_n^b(M)$ subject to relations \eqref{e.ma10}-\eqref{e.ma40}, which are the internal annihilators of the functor $\RT$. For an invertible $u\in R$ let $\S_n^b(M;u)$ be an $R$-module defined as $\S_n^b(M)$, except that the right side of \eqref{e.ma40} is multiplied by $u$. From the definition we see that $\S_n^b(M;u)$ is isomorphic to $\S_n^b(M)$ via the map $\al \to u^{\#{ sinks }(\al)} \al$. The $SU_n$-skein module defined in \cite{Si} is actually $\S_n^b(M;(-q)^{n(n-1)/2})$.

Given a based $n$-web $\alpha$, let $\al^\circ$ denote the underlying $n$-web in $\cal S_n\MN.$
Recall that every oriented 3-manifold has a spin structure.

\begin{theorem} Let $M$ be an oriented 3-manifold.\\
(a) For $n$ odd, the operation $\alpha\to \alpha^\circ$ on based $n$-webs extends to an isomorphism $\S_n^b(M) \cong \S_n(M, \emptyset)$.\\
(b) Every spin structure on $M$ defines a function $s: \cal W_n^b(M)\to \{1,-1\}$ such that the map $\al \to f(\al)=(-1)^{s(\al)} \al^\circ$ induces a unique $R$-linear isomorphism  $\S_n^b(M) \cong \S_n(M, \emptyset)$.
\end{theorem}

\begin{proof} 
 (a) For $n$ odd, Relations \eqref{e.cyclic0} which are consequences of the defining relations,  \eqref{e.ma10}-\eqref{e.ma40}, show that a based $n$-web $\al$, as an element of $\S_n^b(M)$, is determined by $\al^\circ$. Furthermore, the defining relations  \eqref{e.ma10}-\eqref{e.ma40} coincide with the defining relations \eqref{e.ma1}-\eqref{e.ma4}. 

(b) Let $n$ be even now.  Fix a Riemannian metric on $M$ and let $UM$ be a principal $SO(3)$-bundle associated to the tangent bundle of $M$.  A section at a point is the group $SO(3)$, which can be identified with the set of all ordered, positively oriented, orthonormal bases $(v_1, v_2, v_3)$ of the tangent space at the point. Any such ordered orthonormal basis is totally determined by the first two vectors. A smooth embedding $a: [0,1] \to M$ equipped with a normal vector field  defines a lift $\tilde a: [0,1] \to \UM$ where the first and the second vectors are respectively the velocity vector  and the  framing vector,  normalized to have length 1. 
For a based $n$-web $\al$ define $s(\al)\in \{0,1\}$ as follows:
 First isotope $\al$ so that the framing is normal everywhere, and at every $n$-valent vertex the $n$ half edges  have the same velocity vector. The latter condition implies the lift of the endpoints at all the half-edges at an $n$-valent vertex are the same. As $n$ is even   
 the lifts of all edges of $\al$ and of all its circle components form a $\BZ/2$ one-cycle $\tilde \al$ of $\UM$. 
 Recall that a spin structure $s$ of $M$ can be identified with a cohomology class in $H^1(UM, \BZ/2)$ which is non-trivial at the section at every point of $M$.
 Let $s(\al)$ be the evaluation of the spin structure, considered as an element of  $H^1(UM, \BZ/2)$ on $\tilde \al$.
Clearly $s(\al)$ depends only on the isotopy class of $\al$. From the definition, 
$s(\al)=1$ if $\al$ is the trivial loop. If $\al'$ is the result of adding a positive twist to an edge or loop of $\al$ then $s(\al') =- s(\al)$. Thus $f$ maps the defining relations  \eqref{e.ma10}-\eqref{e.ma40} respectively to the defining relations \eqref{e.ma1}-\eqref{e.ma4}. Hence it extends to a well-defined $R$-linear map $f:\S_n^b(M) \to \S_n(M, \emptyset)$.

For the inverse, note that Relations of \eqref{e.cyclic0} show that the map $R {\mathcal W}_n(M, \emptyset) \to \S_n^b(M)$ sending $\al^\circ \to (-1)^{s(\al)} \al$ for every based $n$-web $\alpha$ is well-defined. Since $\bar f$ maps the defining relations   \eqref{e.ma1}-\eqref{e.ma4} to the defining relations \eqref{e.ma10}-\eqref{e.ma40}, it descends to a well-defined $R$-linear map from $\S_n(M, \emptyset)$ to $\S_n(M)$, which is the inverse of $f$.
\end{proof}

%

\subsection{Compatibility with Higgins' SL3 skein algebras}
\label{ss.higgins}
%

In his recent work \cite{Hi} Higgins introduced his version of stated $SL_3$-skein algebras, denoted by $\S_q^{SL_3}(\Sigma)$, of punctured bordered surfaces $\Sigma$.  His skein algebra is the $R$-module freely generated by $3$-webs stated by $-1,0,1$, subject to his system of skein relations. 

Let us identify Higgins's states $1,0,-1$ of Higgins with our states $1,2,3,$ respectively.

\bthm\label{t.higgins}
For any punctured bordered surface $\Sigma$ there is an isomorphism  
$\phi: \S_q^{SL_3}(\Sigma)\to \S_3(\Sigma)$ sending every stated $3$-web $\alpha$ to
\beq (-1)^{h_-(\alpha)+v_3(\alpha)} \cdot q^{(3v_3(\alpha)+S_{in}(\alpha)-S_{out}(\alpha))/2}\cdot \alpha, 
\label{e.phi}
\eeq
where 
\bite
\item $v_3(\alpha)$ is the number of $3$-valent sources and sinks of $\alpha$, 
\item $h_-(\alpha)$ is the number of Higgins' $-1$ states in $\alpha,$ and
\item $S_{in}(\alpha)$ and $S_{out}(\alpha)$ are sums of Higgins' states 
of all edges
edges coming into and coming out of the boundary respectively.
\eite
\ethm


We thank V. Higgins for suggesting the above formula to us.

Furthermore, our theory recovers most of Higgins' for $n=3$. Specifically, Higgins constructed splitting homomorphisms for his skein algebras and an isomorphism $\S_q^{SL_3}(\Sigma)\simeq \mathcal O_q(sl_3)$. It is straightforward to check that these maps coincide with ours through $\phi.$ Higgins also proved 
a version of Theorem \ref{r.braidedtensor} for $n=3$. 
(However, additionally, he defined bases of his skein algebras $\S_q^{SL_3}(\Sigma)$ for all pb surfaces $\Sigma$.
There appears no easy generalization of these bases to $n>3$, as the relay on the confluence method of \cite{SW}, which works for $n=2$ and $3$ only.)\vspace*{.1in} 

\def\tphi{\tilde \phi}

\noindent{\it Proof of Theorem \ref{t.higgins}:}
By Identity \eqref{e.crossing},
\beq\label{e.crossp3}
\Iweb{>}{}=q^{\frac{1}{3}}\crossround{->}{p} - q\smoothA{>}{>}.
 \eeq
 By capping the skeins of Eq. \eqref{e.crossp3} from the top, we get
\beq\label{e.bigon}
\diagg{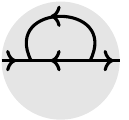}{.5in} = q^{\frac13}\kink - q\diagg{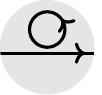}{.5in}= 
\left(q^3-q(q^2+1+q^{-2})\right)\horizontaledge{>}= -(q+q^{-1})\horizontaledge{>}.
\eeq
Taking the reflection of  Eq. \eqref{e.crossp3},
\beq\label{e.crossn3}
\Iweb{>}{}= q^{-\frac13}\crossround{->}{m} -q^{-1}\smoothA{>}{>}
\eeq
Hence, we have 
\begin{multline}\label{e.square}
\diagg{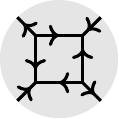}{.5in}=q^{\frac 13} \diagg{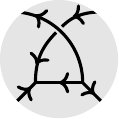}{.5in}-q \diagg{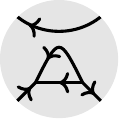}{.5in} =\\
q^{\frac 13}\left(q^{-\frac13} \smoothB{>}{<}-q^{-1}\diagg{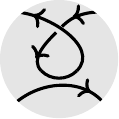}{.5in}\right) +q(q+q^{-1}))\smoothA{<}{>}=
\smoothA{<}{>}+\smoothB{>}{<},
\end{multline}
by Eq. \eqref{e.bigon}.

Let $\W_3^H({\Sigma})$ be the set of all isotopy classes of 3-webs over ${\Sigma}$ with states $-1,0,1$. Consider an $R$-linear isomorphism $\tphi: R \W_3^H({\Sigma}) \to R\W_3({\Sigma})$ given on stated $3$-webs $\alpha$ by \eqref{e.phi}.

By definition, the Higgins algebra $\S_q^{SL_3}(\Sigma)$ is $ R \W_3^H({\Sigma})$ modulo his skein relations \cite[(I1a)-(I4b),(B1)-(B4)]{Hi}. It is easy to see that Higgins' internal relations are pull backs under $\tphi$ of Relations \eqref{e.crossp3},
\eqref{e.crossn3}, \eqref{e.square}, \eqref{e.bigon}, and \eqref{e.unknot}, respectively.

Higgins' boundary relations (B1) and (B3) are pull backs of Relation  \eqref{e.vertexwall3}, (B4) is a pullback of \eqref{e.vertexwall3}, and, finally,  (B2) is a pull back of \eqref{e.crossn3} at the boundary combined with \eqref{e.wallcross}.


Hence, we showed that $\tphi$ descends to an $R$-linear homomorphism $\phi: \S_q^{SL_3}(\Sigma)\to \S_3(\Sigma).$ The definition of $\phi$ suggests an obvious inverse homomorphism 
$\S_3(\Sigma)\to \S_q^{SL_3}(\Sigma)$ and, indeed, one can verify that it is well defined. However, since checking that it respects our relation \eqref{e.sinksource} requires a lengthy calculation, we enclose an alternative proof of $\phi$ being an isomorphism:

Since it is clearly a surjective algebra homomorphism, it remains to show that $\phi$ is injective. 
From the definition it clear that $\phi$ commutes with the splitting homomorphism.
Since $\S_q^{SL_3}(\Sigma)$ satisfies the splitting homomorphism,  $\S_q^{SL_3}(\Sigma)=\mathcal O_q(sl_3)$, and the gluing over a triangle is given by the same isomorphism as described in Theorem \ref{r.braidedtensor}. Theorem \ref{t.Aisom} is also valid with $\S_3$ replaced by $\S_q^{SL_3}$. Part (2) of Theorem \ref{r.braidedtensor} shows that $\phi$ is an isomorphism when ${\Sigma}$ is essentially bordered.

Suppose ${\Sigma}$ is a connected, having empty boundary, and at least one puncture. Let $c$ be an ideal arc of ${\Sigma}$. In the commutative diagram
\be
\begin{tikzcd}
\S_q^{SL_3}({\Sigma})  \arrow[r,hook, " \Theta_c "]
\arrow[d,"\phi"]  
&  \S_q^{SL_3} (\Cut_c {\Sigma}) \arrow[d,"\cong"] \\
\S_3({\Sigma}) \arrow[r,  " \Theta_c "] & \S_3 (\Cut_c {\Sigma})
\end{tikzcd}
\notag
\ee
the upper $\Theta_c$ is injective by Higgins result, which forces $\phi$ to be injective.

\def\relz{  \raisebox{-6pt}{\incl{.6 cm}{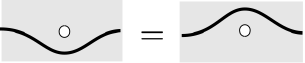}} }
\def\Rel{\mathrm{Rel}}

Consider the remaining case when $\Sigma$ is a closed surface without ideal point. Remove a point $p$  from $\Sigma$ to obtained a pb surface ${\Sigma}'$ having one puncture. Since for both $\S= \S_3$ and $\S= \S_q^{SL_3}$ we have
$\S({\Sigma}) = \S({\Sigma}') /\Rel $, where $\Rel$ is the relation $\relz$, we conclude that $\phi:\S_q^{SL_3}(\Sigma) \to \S_3(\Sigma)$ is an isomorphism. This completes the proof.
\qed

%
\subsection{Relation to the Frohman-Sikora SU(3)-skein algebras}
\label{ss.FS-SU3}
%

Frohman and the second author considered in \cite{FS} the ``reduced $SU(3)$-skein algebra'' of marked surfaces built of unstated $3$-webs, subject to the $SU(3)$-skein relations of \cite{Ku}, 
extended by certain boundary skein relations, which depend on an invertible parameter $a\in R$.
We denote that algebra by $\cal S_{FS}(\Sigma, B)$ for the value $1$ of that parameter.

For an unstated $3$-web $\alpha$ in $\Sigma$, let $\eta_+(\alpha)$ (respectively: $\eta_-(\alpha)$) denote $\alpha$ stated with $3$s (respectively: $1$'s) at all its ends. 

\bthm\label{t.SU3}\ \\
(1) For any punctured bordered surface $\Sigma$, the above operations extent to $R$-linear homomorphisms 
$$\eta_\pm: \cal S_{FS}(\Sigma)\to \S_3(\Sigma).$$ 
(2) $\eta_\pm$ are embeddings and $\eta_\pm(\cal S_{FS}(\Sigma))$ are direct summands of $\S_3(\Sigma).$
\ethm

\bpr
(1) 
$\eta_\pm$ maps the internal relations of \cite{FS} to \eqref{e.unknot},\eqref{e.crossp3}-\eqref{e.square}, and the boundary relations (for $a_{FS}=1$) to 
$$\capwall{<-}{white}{$i$}{$i$}=0=\forkwall{<-}{black}{black}{white}{$i$}{$i$},\quad 
\Hwall{<-}{white}{black}{}{$i$}{$i$}= \twowall{<-}{white}{black}{$i$}{$i$},
$$
for $i=1,3$, which are satisfied by \eqref{e.capwall}, \eqref{e.Hib1}, and \eqref{e.FSboundary0}.
 
(2) To prove that we identify $\S_3(\Sigma)$ with Higgins' skein algebra, through Theorem \ref{t.higgins}. Now it is easy to see that $\eta_\pm$ map the basis of reduced non-elliptic webs without British highways of \cite{FS} $1$-$1$ into the basis composed of irreducible webs of \cite{Hi}. 
\epr

\def\hR{\widehat{\mathcal R}}
\def\OqM{\cO_q(M(n))}
\def\Oq{\cO_q(SL(n))}
\def\OqZ{\cO_{q,\BZ}(SL(n))}
\def\Zv{{\BZ[v^{\pm1}]}}
\def\Uqq{{U_q(sl_n) } }
\def\tUqq{\widetilde{\Uqq}}
\def\Qv{{\BQ(v)}}
\def\End{\mathrm{End}}
\def\dl{\partial_\ell}
\def\dr{\partial_r}
\def\ba{{\mathbf a}}
\def\buu{{\mathbf u}}
\def\bX{{\mathbf X}}
\def\bY{{\mathbf Y}}

\appendix
\section{Proof of Proposition \ref{p.Xaction} (A calculation of matrices of $X$.)}
\label{s.appendixA} 

We need to prove Identity   \eqref{e.Xij}. The generators 
The quantized enveloping algebra $\Uq$ is generated by $E_i, F_i, K_i^{\pm 1}$ with relations given in \cite{KS}.
Its action on $V=\Q(q)^n$ with the standard basis $e_1,\dots, e_n$ is given by 
\beq E_i e_j = \delta_{i,j} e_{i+1}, \ F_i e_j = \delta_{i,j+1} e_i, \ K_i e_j = q^{\delta_{i,j+1} - \delta_{i,j}} e_j.
\label{e.EFK0}
\eeq
Note that $e_n$ is the highest weight vector.

By definition \cite{ST}, a half-ribbon element is an invertible element $X\in \tUqq$ satisfying
\be 
\cR=  (X^{-1} \ot X^{-1} ) \Delta(X),\  \text{and} \ X^2 = \vartheta, \text{the ribbon element}.
\label{e.htwist}
\ee
 In \cite[Section 4]{ST}, a half-ribbon element,  denoted here by $X_0$,  was constructed based on work of Kirillov-Reshetikhin \cite{KR} and Levendorskii-Soibelman \cite{LS}. 

To calculated the action of $X_0$, we use the following identities from \cite[Lemma 3.10]{ST}, 
\begin{align}
  X_0 F_i X_0^{-1} &= - E_{n-i },   \label{e.FF} \\
   X_0 K_i X_0^{-1} & = K_{n-i}^{-1}.
\label{e.EFK}
\\
X(T_{w_0}^{-1}(e_n)) & =  t_0^{1/2} e_n, 
\label{e.T0}
\end{align}
where $T_{w_0}\in \tUqq$ the quantum braid group element corresponding to the longest element $w_0$ of the symmetric group $S_n$ whose exact definition is not needed here.

Let  $\delta$ be the sum of all the fundamental weights. From \eqref{e.EFK} we get
\beq 
X_0 K_\delta X_0^{-1} = K_\delta^{-1}. 
\notag 
\eeq 
It follows that $X:= K_\delta^{-1} X_0$ also satisfies \eqref{e.htwist} and hence is a half-ribbon element. Actually $X$ is a half-ribbon element considered in \cite{KR,LS}.

If $x\in V$ has weight $\lambda$, then $T_{w_0}(x)$ has weight $w_0(\lambda)$. As each weight subspace of $V$ is 1-dimensional and $T_{w_0}$ is invertible, we have
$$T_{w_0}^{-1}(e_n) = c e_1, \quad 0 \neq c \in \BQ(v).$$

By \cite[Proposition 5.9]{KT}, there are positive integers $m_1,\dots, m_k$ and a sequence $i_1,\dots, i_k \in \{1,\dots, n-1\}$ such that
$$F_{i_1}^{(m_1)} \dots  F_{i_k}^{(m_k)} (e_n) =T_{w_0}^{-1}(e_n) = c e_1, \quad \text{where}\ F_i^{(m)} = F_i^m /[m]!$$ 
As $F_i^2=0$ on $V$, all the $m_j$ must be 1. Since $F_i e_j$ is either 0 or another $e_{j'}$ by Eq. \eqref{e.EFK0},  we must have $c=1$. Hence $T_{w_0}^{-1}(e_n) = e_1$, and Eq. \eqref{e.T0} becomes
$$X_0(e_1) = t_0^{1/2} e_n.$$
Applying $F_{n-1}$ to the above equation and using \eqref{e.FF}, we get $X_0 e_2 = - t_0^{1/2} e_{n-1}$. Continue applying $F_{n-2}, F_{n-3}, \dots$ and using \eqref{e.FF}, we get 
$$X_0 e_j = (-1)^{j-1} t_0^{1/2} e_\bj.$$

Since $K^{-1}_\delta$ acts on $V$ by 
$ K^{-1}_\delta (e_i) = q^{ \frac{n+1}{2}-i}   e_i,$
we the matrix of the action of $X$ on $V$ 
$$
X^i_j = \delta_{i,\bar j}\ (-1)^{n-i} t_0^{1/2} q^{ \frac{n+1}{2}-i} =  \delta_{i,\bar j}\ c_{i}. 
$$

The action of $X$ on a dual space is given by the antipode $S$. By \cite[Proposition 4.3]{ST} we have
$S(X) = gX$. From here one can easily calculate
$$X e^i = c_\bi e^\bi = f^i.$$
It follows that the action of $X$ on $V^*$ in the basis $\{f^1, \dots, f^n\}$ is given by $X^i_j = \delta_{i,\bar j} c_i$. This completes the proof of Identity  \eqref{e.Xij}. 



\begin{thebibliography}{99999}
\baselineskip15pt
%

\bibitem[AF]{AF} D. Ayala, J. Francis, Factorization homology of topological manifolds, {\em J. Topol.} {\bf 8} (2015), no. 4, 1045--1084, arXiv:1206.5522

\bibitem[AFT]{AFT} D. Ayala, J. Francis, H.L. Tanaka, Factorization homology of stratified spaces, {\em Selecta Mathematica,} {\bf 23} (2016) (1), 293--362. arXiv:1409.0848

\bibitem[AGS]{AGS} A. Alekseev, H. Grosse, V. Schomerus. Combinatorial quantization of the Hamiltonian Chern-Simons theory II. {\em Communications in Mathematical Physics} {\bf 174}(3) (1996) 561-- 604.

\bibitem[AS]{AS} A. Y. Alekseev, V. Schomerus, Representation theory of Chern-Simons observables, {\em Duke Math. J.} {\bf 85}(2) (1996), 447--510.

\bibitem[AB]{AB}  M.F. Atiyah, R. Bott, Yang-Mills Equations over Riemann Surfaces, {\em Philos. Trans. Royal Soc. London} Ser. A, Mathematical and Physical Sciences, Vol. 308, No. 1505 (1983), 523--615

\bibitem[BD]{BD} A. Beilinson, V. Drinfeld, Chiral algebras, vol. 51 of American Mathematical Society Colloquium Publications (AMS, 2004).

\bibitem[BBJ]{BBJ} D. Ben-Zvi, A. Brochier, D. Jordan, Integrating Quantum Groups Over Surfaces, {\em Journal of Topology}, 11(4), (2018), 873--916, arXiv: 1501.04652.

\bibitem[BR1]{BR1} E. Buffenoir, Ph. Roche, Two-dimensional lattice gauge theory based on a quantum group, {\em Comm. Math. Phys.} {\bf 170} (1995), no. 3, 669--698.

\bibitem[BR2]{BR2} E. Buffenoir, Ph. Roche, Link invariants and combinatorial quantization of Hamiltonian Chern-Simons theory, 
{\em Comm. Math. Phys.} {\bf 181} (1996), no. 2, 331--365.

\bibitem[BFK]{BFK} D. Bullock, C. Frohman, J. Kania-Bartoszy\'nska, Topological Interpretations of Lattice Gauge Field Theory, {\em Comm. Math. Physics}, {\bf 198} (1998) 47--81, arXiv:q-alg/9710003

\bibitem[BL]{BL} W. Bloomquist and T. T. Q. L\^ e, The Chebyshev-Frobenius homomorphism for skein module of marked 3-manifolds, arXiv: 2011.02130

\bibitem[BW]{BW} F. Bonahon, H. Wong, Quantum traces for representations of surface groups in $SL_2(\C)$, {\em Geom. Topol.} {\bf 15} (2011) 1569–-1615, arXiv:1003.5250.

\bibitem[CS]{CS} L.O. Chekhov, M. Shapiro, Darboux Coordinates For Symplectic Groupoid And Cluster Algebras, 2003.07499

\bibitem[Co]{Co} J. Cooke, Excision of Skein Categories and Factorisation Homology, \url{julietcooke.net}, arXiv:1910.02630.

\bibitem[CL]{CL} F. Costantino, T.T.Q. L\^ e, Stated skein algebras of surfaces, {\em Journal  of EMS}, to appear, arXiv:1907.11400, 

\bibitem[CL2]{CL2} F. Costantino, T.T.Q. L\^ e, In preparation.

\bibitem[CKL]{CKL} F. Costantino, J. Korinman, and T.T.Q. L\^ e, In preparation.

\bibitem[CP]{CP} V. Chari, A. Pressley, A guide to quantum groups, Cambridge University Press, Cambridge (1994), \MR{1300632}

\bibitem[CKM]{CKM} S. Cautis, J. Kamnitzer, S. Morrison, Webs and quantum skew Howe duality,
{\em Math. Ann.} 360 no. 1-2, (2014), 351--390.

\bibitem[DL]{DL} M. Domokos, T.H. Lenagan, Quantized trace rings, {\em Quart. J. Math.} {\bf 56} (2005), 507--523, arXiv:math/0407053


\bibitem[Du]{Du} J. Du, q-Schur algebras, asymptotic forms, and quantum SLn, {\em J. Algebra} {\bf 177} (1995), no. 2, 385--408.

\bibitem[ES]{ES}  P. Etingof and O. Schiffmann, Lectures on Quantum Groups, 2nd ed. Lectures in Mathematical Physics, International Press 2002.

\bibitem[Fa]{Fa} M. Faitg, Holonomy and (stated) skein algebras in combinatorial quantization, arXiv:2003.08992

\bibitem[FG1]{FG1} V.V. Fock, A. B. Goncharov, Moduli spaces of local systems and higher {T}eichm\"{u}ller
              theory. {Publ. Math. Inst. Hautes \'{E}tudes Sci.}, {\bf 103},
      ({2006}),
      {1--211}.

\bibitem[FG2]{FG2} V.V. Fock, A. B. Goncharov, Cluster ensembles, quantization and the dilogarithm,
Ann. Sci. \'{E}c. Norm. Sup\'{e}r. (4), {\bf 42} (2009), 865--930.
 arXiv:0311245

\bibitem[FR]{FR} V.V. Fock, A.A. Rosly, Poisson structures on moduli of flat connections on Riemann surfaces and r-matrices, {\em Am. Math. Soc. Transl.} {\bf 191} (1999) 67--86.

\bibitem[FS]{FS} C. Frohman, A.S. Sikora, SU(3)-skein algebras and webs on surfaces, {\em Math. Z.} {\bf 300}, (2022) 33--56, arXiv:2002.08151

\bibitem[Gy]{Gy} A. Gyoja, A $q$-analogue of Young symmetrizer, 
{\em Osaka J. Math.} {\bf 23} (1986), 841--852.

\bibitem[Ha]{Ha} B. Ha\"ioun, Relating stated skein algebras and internal skein algebras, aXiv:2104.13848

\bibitem[Hi]{Hi} V. Higgins, Triangular decomposition of $SL_3$ skein algebras, arXiv: 2008.09419, {\em Quantum Topol.} {\bf 14} (2023) 1--63.

\bibitem[JLSS]{JLSS} D. Jordan, I. Le, G. Schrader,  A. Shapiro, Quantum decorated character stacks, 
arXiv:2102.12283

\bibitem[JF]{JF} T. Johnson-Freyd, Heisenberg-picture quantum field theory, Progress in Mathematics volume in honour of Kolya Reshetikhin, 2020, arXiv:1508.05908.

\bibitem[JW]{JW} D. Jordan, N. White, The center of the reflection equation algebra via quantum minors,  {\em J. Algebra} {\bf 542} (2020), 308--342. 

\bibitem[Kash]{Kash} M. Kashiwara,  Global crystal bases of quantum groups, {\em Duke Math. J.} {\bf 69} (1993), no. 2, 455--485.

\bibitem[Kass]{Kass} C. Kassel, Quantum groups, Springer-Verlag, New York, 1995, Graduate Texts in Mathematics, No. 155.

\bibitem[KR]{KR} A. N. Kirillov and N.  Reshetikhin, $q$-Weyl group and a multiplicative formula for universal R-matrices, {\em Comm. Math. Phys.} {\bf 134} (1990), no. 2, 421--431.

\bibitem[KS]{KS} A. Klimyk, K. Schm\"udgen, Quantum groups and their representations, Texts
and Monographs in Physics, Springer-Verlag, Berlin (1997).

\bibitem[KoS]{KoS} S. Kolb, J. Stokman Reflection Equation Algebras, Coideal Subalgebras And Their Centres, {\em Selecta Math. (N.S.)} {\bf 15} (4) (2009) 621--664.

\bibitem[Ko]{Ko} J. Korinman, Finite Presentations For Stated Skein Algebras And Lattice Gauge Field Theory, arXiv: 2012.03237.

\bibitem[KQ]{KQ} J. Korinman, A. Quesney, Classical shadows of stated skein representations at roots of unity, arXiv:1905.03441, 2019.

\bibitem[KT]{KT} J. Kamnitzer, P. Tingley, The crystal commutor and Drinfeld's unitarized R-matrix,
{\em J. Algebraic Combin.} {\bf 29} (2009), no. 3, 315--335.

\bibitem[Ku]{Ku} G. Kuperberg, Spiders for rank 2 Lie algebras, {\em Comm. Math. Phys.} {\bf 180} (1996), no. 1, 109--151.

\bibitem[L\^e1]{Le-triang} T.T.Q. L\^ e, Triangular decomposition of skein algebras, {\em Quantum Topology}, {\bf 9} (2018), 591--632 arXiv: 1609.04987

\bibitem[L\^e2]{Le-qteich} T.T.Q. L\^ e, Quantum Teichm\"uller spaces and quantum trace map, {\em J.Inst.Math.Jussieu} (2019) 18(2), 249--291, arXiv:1511.06054.

\bibitem[LY1]{LY1} T.T.Q. L\^ e, T. Yu, Stated Skein Modules Of Marked 3-Manifolds/Surfaces, A Survey,  {\em Acta Math. Vietnam.} {\bf 46} (2021), no. 2, 265--287, arXiv: 2005.14577

\bibitem[LY2]{LY2} T.T.Q. L\^ e, T. Yu, Quantum traces and embeddings of stated skein algebras into quantum tori, arXiv:2012.15272.

\bibitem[LY3]{LY3} T.T.Q. L\^ e, T. Yu, Quantum Traces For $SL_n$-Skein Algebras, arXiv:2303.08082.

\bibitem[LS]{LS} S. Z.  Levendorskii, and Ya. S.  Soibelman, The quantum Weyl group and a multiplicative formula for the R-matrix of a simple Lie algebra, (Russian) {\em Funktsional. Anal. i Prilozhen.} {\bf 25} (1991), no. 2, 73--76; translation in {\em Funct. Anal. Appl.} {\bf 25} (1991), no. 2, 143--145.

\bibitem[Lu]{Lu} J. Lurie, Higher algebra, \url{http://www.math.harvard.edu/~lurie}

\bibitem[Lu1]{Lu-root1} G. Lusztig, Quatum Groups at Roots of $1$, {\em Geom. Dedicata}, {\bf 35} (1990), 89--114.
\bibitem[Lu2]{Lu-bases} G. Lusztig, Canonical bases in tensor products, {\em Proc. Natl. Acad. Sci. USA}
{\bf 89} (1992) 8177--8179.
\bibitem[Lu3]{Lu-book} G. Lusztig, Introduction to quantum groups, Progress in Mathematics 110. 
Birkh\"auser Boston Inc., (1993).

\bibitem[Lu4]{Lu-Zform}  G. Lusztig, Study Of A Z-Form Of The Coordinate Ring Of A Reductive Group, {\em J. of AMS}, {\bf 22}, 3, (2009), Pages 739--769, arXiv:0709.1286

\bibitem[Maj]{Maj} S. Majid, Foundations of quantum group theory, Cambridge U. Press, (1995).

\bibitem[Mor]{Mor} S.E. Morrison, A Diagrammatic Category for the Representation Theory of $U_q(sl_n)$, Ph.D. Thesis, 2007.

\bibitem[Po]{Po} A. Poudel, A Comparison Between $SL_n$ Spider Categories, 2022, arXiv:2210.09289

\bibitem[Pr]{Pr} J.H. Przytycki, Skein modules of {$3$}-manifolds, {\em Bull. Polish
  Acad. Sci. Math.} {\bf 39} (1991), no.~1-2, 91--100.

\bibitem[PS]{PS} J.H. Przytycki, A.S. Sikora,  On Skein Algebras and $SL_2(C)$-Character Varieties, {\em Topology}, {\bf 39} (1) 2000, 115--148.
    
\bibitem[RT]{RT} N.Y. Reshetikhin, V.G. Turaev, Ribbon graphs and their invariants derived
from quantum groups, {\em Comm. Math. Phys.} {\bf 127} (1990) 1--26.

\bibitem[Si]{Si} A.S. Sikora, Skein theory for SU(n)-quantum invariants, {\em Alg. Geom. Topol.} {\bf 5} (2005) 865-897, arXiv:math/0407299v5.

\bibitem[Sh]{Sh} L. Shen, Shen Cluster Nature Of Quantum Groups, arXiv:2209.06258.

\bibitem[SW]{SW} A.S. Sikora, B.W. Westbury, Confluence Theory for Graphs, {\em Alg. Geom. Topol.} {\bf 7} (2007) 439--478, arXiv: math/0609832

\bibitem[ST]{ST} N. Snyder, P. Tingley, The half-twist for $U_q(\mathfrak g)$ representations, {\em Algebra Number Theory} 3(7): 809--834 (2009).

\bibitem[Ta1]{Takeuchi0} M. Takeuchi, Some topics on $GL_q(n)$.
{\em J. Algebra} {\bf 147} (1992), no. 2, 379--410.

\bibitem[Ta2]{Ta} M. Takeuchi, A Short Course On Quantum Matrices,
{\em  Math. Sci. Res. Inst. Publ.} {\bf  43}, New directions in Hopf algebras, 383--435, Cambridge Univ. Press, Cambridge, 2002.

\bibitem[Tu1]{Tu1} V.G. Turaev, The Conway and Kauffman modules of a solid torus, {\em Zap.
  Nauchn. Sem. Leningrad. Otdel. Mat. Inst. Steklov.} (LOMI) \textbf{167}
  (1988), no.~Issled. Topol. 6, 190, 79--89.

\bibitem[Tu2]{Tu2} V.G. Turaev, Quantum Invariants of Knots and 3-Manifolds, De Gruyter Studies in Mathematics, Vol. 18, 1994.

\bibitem[Wa]{Wa} K. Walker, TQFTs [early incomplete draft, 2006, \url{http://canyon23.net/math/tc.pdf} 

\bibitem[Wi]{Wi} E. Witten, Quantum field theory and the Jones polynomial, {\em Comm. Math. Phys.} {\bf 121}(3) (1989) 351--399.

\end{thebibliography}
\end{document}